\documentclass[12pt]{amsart}
\usepackage{pgf,tikz,pgfplots} 
\pgfplotsset{compat=1.14}
\usepackage{mathrsfs}
\usetikzlibrary{arrows}
\usetikzlibrary{patterns}
\usepackage{pgfplots}
\usetikzlibrary{intersections, pgfplots.fillbetween}
\usepackage[colorlinks=true,urlcolor=blue, citecolor=red,linkcolor=blue,linktocpage,pdfpagelabels, bookmarksnumbered,bookmarksopen]{hyperref}
\usepackage[hyperpageref]{backref}
\usepackage{cleveref}
\usepackage{xcolor}
\usepackage{amsthm} 
\usepackage{latexsym,amsmath,amssymb}
\usepackage{accents}
\usepackage[colorinlistoftodos,prependcaption,textsize=tiny]{todonotes}
\usepackage{a4wide}
\usepackage{soul}
\usepackage{mathtools} 
\usepackage{xparse} 
  
\pgfdeclarelayer{ft}
\pgfdeclarelayer{bg}
\pgfsetlayers{bg,main,ft}

\title[$W^{s,\frac{n}{s}}$-harmonic maps in homotopy classes]{Minimal $W^{s,\frac{n}{s}}$-harmonic maps in homotopy classes}

\date{\today}

\author{Katarzyna Mazowiecka}
\address[Katarzyna Mazowiecka]{
Universit\'e catholique de Louvain, Institut de Recherche en Math\'ematique et Physique, Chemin du Cyclotron 2 bte L7.01.02, 1348 Louvain-la-Neuve, Belgium}
\email{katarzyna.mazowiecka@uclouvain.be}

\author{Armin Schikorra}
\address[Armin Schikorra]{Department of Mathematics,
University of Pittsburgh,
301 Thackeray Hall,
Pittsburgh, PA 15260, USA}
\email{armin@pitt.edu}

\definecolor{indigo}{rgb}{0.29, 0.0, 0.51}
\definecolor{p1}{gray}{0.4}
\definecolor{p2}{gray}{0.6}
\definecolor{p3}{gray}{0.98}
\definecolor{p4}{gray}{0.8}
\definecolor{p5}{gray}{0.9}

plus 6pt minus 12pt
\def\eps{\varepsilon}

\def\vp{\varphi}

\def\B{{B}}

\def\n{{\mathcal M}}
\newcommand{\dif}{\,\mathrm{d}}

\def\N{{\mathbb N}}
\def\n{{\mathcal{M}}}
\def\n{{\mathcal N}}

\def\S{{\mathbb S}}

\newtheorem{theorem}{Theorem}
\newtheorem{lemma}[theorem]{Lemma}
\newtheorem{corollary}[theorem]{Corollary}
\newtheorem{proposition}[theorem]{Proposition}

\newtheorem{remark}[theorem]{Remark}
\newtheorem{definition}[theorem]{Definition}

\def\diam{{\rm diam\,}}
\def\dist{{\rm dist\,}}

\def\lip{{\rm Lip\,}}

\def\supp{{\rm supp\,}}

\newcommand{\dx}{\dif x}
\newcommand{\dy}{\dif y}
\newcommand{\dz}{\dif z}

\newcommand{\R}{\mathbb{R}}

\newcommand{\Z}{\mathbb{Z}}

\newcommand{\brac}[1]{\left (#1 \right )}
\newcommand{\abs}[1]{\left |#1 \right |}

\newcommand{\norm}[1]{\left\|{#1}\right\|}

\newcommand{\barint}{
\rule[.036in]{.12in}{.009in}\kern-.16in \displaystyle\int }

\newcommand{\barcal}{\mbox{$ \rule[.036in]{.11in}{.007in}\kern-.128in\int $}}

\def\mvint_#1{\mathchoice
          {\mathop{\vrule width 6pt height 3 pt depth -2.5pt
                  \kern -8pt \intop}\nolimits_{\kern -3pt #1}}%
          {\mathop{\vrule width 5pt height 3 pt depth -2.6pt
                  \kern -6pt \intop}\nolimits_{#1}}%
          {\mathop{\vrule width 5pt height 3 pt depth -2.6pt
                  \kern -6pt \intop}\nolimits_{#1}}%
          {\mathop{\vrule width 5pt height 3 pt depth -2.6pt
                  \kern -6pt \intop}\nolimits_{#1}}}

\numberwithin{theorem}{section} \numberwithin{equation}{section}

\newcommand{\lap}{\Delta }
\newcommand{\aleq}{\precsim}
\newcommand{\ageq}{\succsim}
\newcommand{\aeq}{\approx}

\newcommand{\laps}[1]{(-\lap)^{\frac{#1}{2}}}

\let\latexchi\chi
\makeatletter
\renewcommand\chi{\@ifnextchar_\sub@chi\latexchi}
\newcommand{\sub@chi}[2]{
  \@ifnextchar^{\subsup@chi{#2}}{\latexchi^{}_{#2}}%
}
\newcommand{\subsup@chi}[3]{
  \latexchi_{#1}^{#3}%
}
\makeatother

\makeatletter
\def\tikz@arc@opt[#1]{
  {%
    \tikzset{every arc/.try,#1}%
    \pgfkeysgetvalue{/tikz/start angle}\tikz@s
    \pgfkeysgetvalue{/tikz/end angle}\tikz@e
    \pgfkeysgetvalue{/tikz/delta angle}\tikz@d
    \ifx\tikz@s\pgfutil@empty%
      \pgfmathsetmacro\tikz@s{\tikz@e-\tikz@d}
    \else
      \ifx\tikz@e\pgfutil@empty%
        \pgfmathsetmacro\tikz@e{\tikz@s+\tikz@d}
      \fi%
    \fi
    \tikz@arc@moveto
    \xdef\pgf@marshal{\noexpand%
    \tikz@do@arc{\tikz@s}{\tikz@e}
      {\pgfkeysvalueof{/tikz/x radius}}
      {\pgfkeysvalueof{/tikz/y radius}}}%
  }%
  \pgf@marshal%
  \tikz@arcfinal%
}
\let\tikz@arc@moveto\relax
\def\tikz@arc@movetolineto#1{%
  \def\tikz@arc@moveto{\tikz@@@parse@polar{\tikz@arc@@movetolineto#1}(\tikz@s:\pgfkeysvalueof{/tikz/x radius} and \pgfkeysvalueof{/tikz/y radius})}}
\def\tikz@arc@@movetolineto#1#2{#1{\pgfpointadd{#2}{\tikz@last@position@saved}}}
\tikzset{%
  move to start/.code=\tikz@arc@movetolineto\pgfpathmoveto,%
  line to start/.code=\tikz@arc@movetolineto\pgfpathlineto}
\makeatother
\begin{document}

\begin{abstract}
Let $\Sigma$ a closed $n$-dimensional manifold, $\mathcal{N} \subset \mathbb{R}^M$ be a closed manifold, and $u \in W^{s,\frac ns}(\Sigma,\mathcal{N})$ for $s\in(0,1)$. We extend the monumental work of Sacks and Uhlenbeck by proving that if $\pi_n(\mathcal{N})=\{0\}$ then there exists a minimizing $W^{s,\frac ns}$-harmonic map homotopic to $u$. If $\pi_n(\mathcal{N})\neq \{0\}$, then we prove that there exists a $W^{s,\frac{n}{s}}$-harmonic map from $\mathbb{S}^n$ to $\mathcal{N}$ in a generating set of $\pi_{n}(\mathcal{N})$. 

Since several techniques, especially Pohozaev-type arguments, are unknown in the fractional framework (in particular when $\frac{n}{s} \neq 2$ one cannot argue via an extension method), we develop crucial new tools that are interesting on their own: such as a removability result for point-singularities and a balanced energy estimate for non-scaling invariant energies. Moreover, we prove the regularity theory for minimizing $W^{s,\frac{n}{s}}$-maps into manifolds.
\end{abstract}

\sloppy

\maketitle
\tableofcontents
\sloppy

\section{Introduction}
In the geometric calculus of variations, it is of utmost importance to find and classify not only absolute minimizers, but one would like to understand the more subtle structure of critical points (local minimizers, saddle points, etc.) within topological classes --- with questions ranging from the Willmore conjecture recently solved by Marques and Neves \cite{MN14} to open questions on existence of critical points for knot energies by Freedman--He--Wang \cite{FHW94} and Kusner--Sullivan \cite{KS97}. In this paper we study the existence theory of minimal $W^{s,\frac{n}{s}}$-harmonic maps in homotopy between two manifolds $\Sigma$ and $\n$. 

Throughout the paper we assume that $\Sigma$ is a smooth compact $n$-dimensional Riemannian manifold without boundary, and $\n \subset \R^M$ is a connected smooth compact Riemannian manifold isometrically embedded into $\R^M$. 

The most fundamental result in existence theory for harmonic maps in homotopy classes is due to Sacks and Uhlenbeck, \cite{Sucks1,Sucks2}. Harmonic maps are critical points of the Dirichlet energy 
\begin{equation}\label{eq:dirichletenergy}
\int_{\Sigma} |\nabla u|^2\quad \text{such that}\quad u\in W^{1,2}(\Sigma,\n).
\end{equation}
We summarize the results of Sacks and Uhlenbeck, \cite[Theorem~5.1]{Sucks1} and \cite[Theorem~5.5]{Sucks1}, as follows.
\begin{theorem}[Sacks and Uhlenbeck]\label{th:sucks}
Let $\Sigma$ be a two-dimensional manifold.
\begin{enumerate}
 \item If $\pi_2(\n) = \{0\}$ then there exists a minimizing harmonic map in every homotopy class $C^0(\Sigma,\n)$.
 \item If $\Sigma = \S^2$ and $\pi_1(\n)=\{0\}$ then there exists a generating set of homotopy classes in $C^0(\S^2,\n)$ in which minimizing harmonic maps exist.
\end{enumerate}
\end{theorem}
\Cref{th:sucks} (1) was originally obtained independently by Lemaire \cite{Lemaire1978} and Schoen and Yau \cite{Schoen-Yau1979}. Also let us remark, that the condition $\pi_1(\n)=\{0\}$ in \Cref{th:sucks}(2) is for pure commodity of this introduction, for $\pi_1(\n) \neq \{0\}$ the same result holds up to the action of $\pi_1(\n)$ on $\pi_2(\n)$, see \Cref{th:main2}.

\Cref{th:sucks} (2) is sharp in the case $\pi_2(\n)\neq \{0\}$ in the following sense: harmonic maps may not exist in every homotopy class of $\pi_{2}(\n)$, a counterexample was provided by Futaki~\cite{Futaki80}. 

The motivation to study harmonic maps under topological assumptions is at least twofold. On the one hand, there is the geometric interest as the image of a harmonic map from $\S^2$ to $\n$ is a conformal branched immersion (which seems to have been the main motivation for Sacks and Uhlenbeck). On the other hand, there is an interest from the applications point of view, as the harmonic map energy can be interpreted as a model case for the Oseen--Frank theory of nematic liquid crystals, see, e.g., \cite{BZ11,LW14}.

In this work we develop an existence theory for $W^{s,\frac{n}{s}}$-harmonic maps in homotopy classes. For $s\in (0,1)$ such maps are defined to be minimizers or critical points of the energy
\begin{equation}\label{eq:Wsnsharmonicmaps}
 E_{s,\frac{n}{s}}(u) \coloneqq \int_{\Sigma}\int_{\Sigma} \frac{|u(x)-u(y)|^{\frac{n}{s}}}{|x-y|^{2n}}\dx\dy\quad \text{such that}\quad u\in W^{s,\frac ns}(\Sigma,\n).
\end{equation}
Our main result is the counterpart to \Cref{th:sucks} for the energy $E_{s,\frac{n}{s}}$.
\begin{theorem}\label{th:wesuck}
Let $\Sigma$ be a closed $n$-dimensional manifold, $n\geq 1$, $s \in (0,1)$, $\frac{n}{s} \geq 2$.\footnote{The condition $n/s \geq 2$ is trivially satisfied if $n \geq 2$. For $n=1$ it is mostly a technical assumption which plays only a role in the regularity theory, \Cref{s:regularity}. It should not be too much work to extend this theorem to the full case $s \in (0,1)$ for $n=1$ as well but we will not develop this point here.}
\begin{enumerate}
 \item If $\pi_n(\n) = \{0\}$ then there exists a minimizing harmonic map in every connected component of $C^0(\Sigma,\n)$.
 \item If $\Sigma = \S^n$, $n\geq 2$, and $\pi_1(\n)=\{0\}$ then there exists a generating set of homotopy classes in $\pi_n(\n)$ in which minimizing harmonic maps exist.
 \item If $\Sigma = \S^1$ then there exists a generating set of homotopy classes in $C^0(\S^n,\n)$ in which minimizing harmonic maps exist.
\end{enumerate}
\end{theorem}

In particular we obtain the existence of nontrivial $W^{s,\frac{n}{s}}(\S^n,\n)$-harmonic maps whenever $\pi_{n}(\n) \neq \{0\}$, see \Cref{co:existencenontrivial}. \Cref{th:wesuck} sheds some light on a question raised by Mironescu \cite{M07} on the existence of minimizing $W^{s,\frac{n}{s}}(\S^n,\S^n)$-maps in homotopy groups.

Here, we also remark that in \Cref{th:wesuck} (2) the condition $\pi_1(\n)=\{0\}$ is not necessary: the same result for $\pi_1(\n) \neq \{0\}$ holds up to the action of $\pi_1(\n)$ on $\pi_n(\n)$, see \Cref{th:main2}.

\Cref{th:wesuck}~(1) is proven in Section~\ref{s:sucks1}, see \Cref{th:main1}, and \Cref{th:wesuck}~(2) is proven in Section~\ref{s:sucks2}, see \Cref{th:main2}. 

Similarly as in the case of harmonic maps, there are at least two motivations for studying $W^{s,\frac{n}{s}}$-maps, one coming from geometry and the other one from applications.

Firstly, as an example of a geometric motivation, the $W^{\frac{1}{2},2}$-energy appears as trace energy and one can model the \emph{free boundary} of minimal surfaces with such energies, cf. Moser~\cite{M11}, Roberts \cite{R18}, Millot--Sire \cite{MS15}, and Pigati--Da~Lio \cite{DLP17}\footnote{As a curious sidenote let us mention that to a certain extent this was actually used in Douglas' proof of the Plateau problem in 1932, \cite{D31}.}.

Secondly, since the 1990's nonlocal energies have been used by applied topologists to define self-repulsive curvature energies for curves and surfaces. 
Self-repulsiveness is a property that is desirable for models of cells, DNA, etc., and one natural way to include this feature is a nonlocal energy. For example O'Hara's knot energies \cite{OH91,OH92}, one of which is the famous M\"obius energy \cite{FHW94}; or the tangent points energies proposed by Banavar et al. \cite{BGMM03}; or the Menger curvature suggested by Gonzalez and Maddocks \cite{GM99}. We refer to \cite{AMN16,StrzvdM-2013,svdm:somemenger,svdm:imc} for further details.

There is a close connection of these nonlocal repulsive curvature energies to $W^{s,\frac{n}{s}}$-harmonic maps, as was discovered in \cite{BRS16,BRS19} for the O'Hara energies: one can construct an energy $\tilde{E}_{s,1/s}$ reminiscent of $E_{s,1/s}$ such that critical knots $\gamma$ (with respect to their knot energy) induce via their derivative $\gamma'$ an $\tilde{E}_{s,1/s}$-critical map (as a map into $\S^2$). 
This link between self-repulsive curvature energies and fractional harmonic map energies (at least formally) seems to extend to existence theory; a famous theorem by Freedman--He--Wang \cite{FHW94} states that minimizers for the M\"obius energy exist in prime knot classes (which are generators of the ambient isotopy classes), whereas Kusner--Sullivan \cite{KS97} conjectured that minimizers may not exist in composite knot classes. This is a very similar statement to \Cref{th:sucks}(2) and \Cref{th:wesuck}(2) --- minimizers exist in a generating subset of the homotopy class. Moreover, as mentioned above, for harmonic maps it is known that minimizers may not exist in some elements of the homotopy group: there is an example due to Futaki \cite{Futaki80}.

The techniques by Freedman--He--Wang \cite{FHW94} are very geometric in nature, and it is completely unclear how to extend them to other topological curvature energies (especially for scale-invariant self-repulsive surfaces energies. There are very few techniques available, see~\cite{SvdM13,Sconformal}). One of the underlying motivations of the present work is to to develop an analytic foundation for techniques that hopefully will be applicable for the wide range of scale-invariant self-repulsive critical knot and surface energies proposed by the applied topology community.

\subsection*{Outline, strategy of the proof, and main results}
If we take a generic minimizing sequence in some homotopy class of the Dirichlet energy or the $E_{s,\frac{n}{s}}$ energy, then there is no reason that it converges to a minimizer. Indeed, 
e.g., if $\Sigma = \S^n$ and $u$ is a minimizer in some nontrivial homotopy group (suppose it exists), then we can conformally rescale without changing the energy. Namely for any $\lambda > 0$, $u_\lambda(x) \coloneqq u(\tau(\lambda \tau^{-1}(x)))$, where $\tau: \R^n \to \S^n\setminus \{N\}$ is the inverse stereographic projection, satisfies $E_{s,\frac{n}{s}}(u_\lambda) = E_{s,\frac{n}{s}}(u)$, see \Cref{s:balancing}. Then $(u_\lambda)_{\lambda > 0}$ is a minimizing sequence, but $u_{\lambda}$ weakly converges to a constant map as $\lambda \to 0$. In other words, the energy $E_{s,\frac{n}{s}}$ is not coercive in the set of $W^{s,\frac{n}{s}}(\Sigma,\n)$-maps belonging to one (nontrivial) homotopy class.

Sacks and Uhlenbeck mitigated this non-coercivity by introducing a special minimizing sequence. Namely they defined the minimizing sequence $(u_\alpha)_{\alpha > 1}$ as the minimizers of the approximate energy
\[
 E_\alpha(u) \coloneqq \int_{\Sigma} (1+|\nabla u|^2)^{\alpha}.
\]
As $\alpha \to 1^+$ one hopes that the sequence $(u_\alpha)_{\alpha > 1}$ converges to a minimizer of the Dirichlet energy $E_1$. In the case, when $\Sigma = \S^2$, since the energy $E_\alpha$ is not conformally invariant, Sacks and Uhlenbeck were able to obtain some control over the energy concentration that is likely to happen. Crucially they showed that in this case energy concentration cannot happen at only one point, but either happens nowhere or at least at two points. 

We follow a similar philosophy, but we have to develop several novel arguments to overcome the problem of nonlocality of the energies $E_{s,n/s}$. Specifically, there are only few available Pohozaev-type arguments and they seem not to be working in our case (in contrast to the case of local equations be it harmonic or $n$-harmonic maps). Indeed, the only case where such arguments (and consequently monotonicity estimates etc.) are known is the case $n/s = 2$, cf. Millot--Sire~\cite{MS15}.

Following the Sacks--Uhlenbeck approach, we will first construct the minimizing sequence for $E_{s,n/s}$ via minimizers $u_t$, $t>s$, of the energy
\[
 E_{t,\frac{n}{s}}(u,\Sigma) \coloneqq \int_{\Sigma}\int_{\Sigma} \frac{|u(x)-u(y)|^{\frac{n}{s}}}{|x-y|^{n+t\frac{n}{s}}}\dx\dy.
\]
That is, we try to approximate $W^{s,\frac{n}{s}}$-minimizing maps by $W^{t,\frac{n}{s}}$-minimizing maps and let $t \to s^+$.\footnote{It would be more in line with the original approach of Sacks-Uhlenbeck if we chose $W^{s,\frac{n}{s}\alpha}$-minimizers, $\alpha \to 1^+$. However, that would have the technical drawback that $W^{s,\frac{n}{s}\alpha} \not \hookrightarrow W^{s,\frac{n}{s}}_{loc}$ for $\alpha > 1$ and $s \in (0,1)$, see \cite{MiSi15}. But we do have the embedding $W^{t,\frac{n}{s}} \hookrightarrow W^{s,\frac{n}{s}}_{loc}$ for $t > s$, \cite{RunstSickel}.}

In Section~\ref{s:regularity} we develop a regularity theory for minimizers of $E_{t,\frac{n}{s}}$. More precisely, we show that on balls where the $E_{s,\frac{n}{s}}$-energy of $u_t$ is not concentrating we have regularity estimates in $W^{s_0,\frac{n}{s}}$ for some $s_0 > s$ which is independent of $t\to s$, see \Cref{th:reghomo}.

Then, from a standard covering argument, one obtains that the minimizing sequence $u_t$ converges strongly to $u_s$ outside of a singular set consisting of finitely many points. The crucial next result we need is that $u_s$ is regular. Sacks and Uhlenbeck remove the point singularities by using a Pohozaev-type argument. In the nonlocal situation Pohozaev-type arguments are still under development, see \cite{RS14,DLPoho,G20}. Another option is to show that $u_s$ is a critical point of the $W^{s,\frac{n}{s}}$-harmonic map equation (which is easy, see \Cref{pr:removepde}) and use regularity theory for critical points. Also this approach is not feasible for us, because the regularity theory for critical points into general target manifolds for the $E_{s,\frac{n}{s}}$-energy (and also for the local analogue $n$-harmonic maps) is a major open problem since the 1990s, cf. \cite{SS17}. Our approach is to show in Section~\ref{s:removability} that the limit $u_s$ is actually still a minimizer (but in its own homotopy class, which might be different from the homotopy class of $u_t$), see \Cref{th:removabilityminimizing}. With the regularity theory from Section~\ref{s:regularity} for \emph{minimizers} we then get the desired regularity for the limit map $u_s$, see \Cref{th:corollaryremovability}.

The probably most crucial novelty of our work is contained in \Cref{s:balancing}: we essentially show that if $\Sigma = \S^n$ the minimizers of the non-scaling invariant energy $E_{t,\frac{n}{s}}$, $t>s$, will not have energy concentration in only one point as $t \to s$ (it has to be either no point or at least two points of energy concentration). We establish this statement by showing in \Cref{th:fractionalestimatesmalldiskbyremaining} that the energy of a $E_{t,\frac{n}{s}}$-minimizer on a small ball is controlled by the energy of the complement of that ball. We are not aware of such a statement in the literature even in the local case of $p$-harmonic maps, see \Cref{th:fractionalestimatesmalldiskbyremainingW1p}. However, see \cite{LMM19}, where the authors seem to use a similar effect to show that not all harmonic maps can be obtained from a Sacks--Uhlenbeck approximation. In our case, we use \Cref{th:fractionalestimatesmalldiskbyremaining} to replace the role of Sacks and Uhlenbeck's \cite[Lemma~5.3]{Sucks1} which is based on a rather explicit computation of the Euler--Lagrange equation which we could not reproduce in our nonlocal setting.

The remaining outline is as follows: in \Cref{s:homotopy} we recall the basic notion of homotopy for Sobolev maps. In \Cref{s:sucks1} we prove the analogue of \Cref{th:sucks}(1), and in \Cref{s:sucks2} we prove the analogue of \Cref{th:sucks}(2).

As corollaries we obtain in \Cref{co:existencenontrivial} the existence of nontrivial $W^{s,\frac{n}{s}}(\S^n,\n)$-harmonic maps whenever $\pi_{n}(\n) \neq 0$ and in \Cref{co:existencenontrivialsmallenergy} existence of minimizers in any nontrivial homotopy class $\Gamma$ which has small minimal energy $\inf_{\Gamma} E_{s,\frac{n}{s}}$.

\subsection*{Remarks on earlier extensions of Sacks--Uhlenbeck theory}

The work by Sacks and Uhlenbeck has been extended to finding $n$-harmonic maps in $\pi_n(\n)$; a version of \cite[Theorem~5.1]{Sucks1} follows from White's \cite{White-homotopy}, for a version of \cite[Theorem~5.5]{Sucks1} see Kawai--Nakauchi--Takeuchi \cite{Kawai-Nakauchi-Takeuchi}. \cite[Theorem~5.5]{Sucks1} uses 
a removability theorem for $n$-harmonic maps, see for example \cite{Mou-Yang-1996} or \cite{Duzaar-Fuchs}. 
Recently, some of those results were also generalized to the polyharmonic case, see \cite{polyharmonic}. See also \cite{R17} for viscosity methods for minimal surfaces.

There also has been a tremendous amount of work dedicated to the analysis of bubbles forming in the process of the minimization procedure -- for harmonic maps \cite{LR02,LR14,Parker,Lamm2010}, for $H$-surfaces see \cite{BrezisCoronH}, for Willmore surfaces \cite{BR14}, for $n$-harmonic maps \cite{DK98}, for biharmonic maps \cite{LaurainRivierebiharmonic}, for Dirac-harmonic maps \cite{JLZ19,Branding20}, and for fractional harmonic maps \cite{DL15bubble,LaurainPetrides}. Let us also mention the flow-technique  developed for harmonic maps by Struwe \cite{S85} which he used to show the existence of nontrivial minimal harmonic maps (cf.  \Cref{co:existencenontrivial}). For results concerning 1-harmonic maps we refer to \cite{Lasic}. We also refer to \cite{RandomRemy}.

\subsection*{A brief history of fractional harmonic maps}
The theory of fractional harmonic maps, i.e., critical points and minimizers of the energy in \eqref{eq:Wsnsharmonicmaps}, can be traced back to the 1930s when Douglas \cite{D31} used them (implicitly) to solve the Plateau problem and win the Fields price. 
Analytically they were introduced in the pioneering work by Da~Lio and Rivi\`ere \cite{DLR11,DaLio-Riviere-1Dmfd} who coined the notion of fractional harmonic maps and developed the regularity theory for critical (i.e., not necessarily minimizing) $W^{\frac{1}{2},2}$-harmonic maps on lines into manifolds. This regularity theory for critical points was extended to various variations of the energy functional \cite{S12,DL13,DLS14,SLp15,Seps15,DLS17,MRS17} --- in particular the notion of (critical) $W^{s,p}$-harmonic maps and their regularity theory into spheres was introduced in \cite{S15} (see also \cite{MS18}). The question of existence of minimizing $W^{s,\frac{n}{s}}(\S^n,\S^n)$-maps of degree one was raised earlier, see Mironescu \cite{M07}.

Moser~\cite{M11} and Roberts \cite{R18} developed a theory of \emph{intrinsic} fractional harmonic maps and their regularity theory. Moser \cite{M11}, Roberts \cite{R18}, and Millot--Sire \cite{MS15} used the technique of harmonic extension to the upper half-plane to characterize $W^{\frac{1}{2},2}$-harmonic maps as a partial free boundary problem of a classical harmonic map and obtain regularity theory from arguments for free boundary harmonic maps due to Scheven \cite{S06} --- see also \cite{DLP17}. Millot--Sire \cite{MS15} obtained from this approach a monotonicity formula for fractional harmonic maps which lead to the partial regularity theory of stationary harmonic maps. Sadly, the harmonic extension technique is as of now only available for $L^2$-type functionals, i.e., $W^{s,2}$-harmonic maps, thus not applicable in our case.

The singular set of stationary and minimizing $W^{s,2}$-harmonic maps (in the supercritical dimension) was analyzed in \cite{MSY18,MP20,MPS19}.

One challenge that keeps appearing when analyzing fractional harmonic maps (e.g., with respect to monotonicity formulas) is the lack of understanding of the fine estimates known for local equations --- such as Pohozaev identities. There has been some important progress in this direction, \cite{RS14,DLPoho,G20}, but in many cases the techniques available are bound to some form of the harmonic extension technique introduced for fractional harmonic maps by \cite{M11,MS15} --- which is not available for $W^{s,\frac{n}{s}}$-harmonic maps we consider here (unless $n/s = 2$). This is very different to the situation of $n$-harmonic maps, where Pohozaev-type estimates are readily available.

\subsection*{Notation} 
Throughout the paper we assume that $\Sigma$ is a closed Riemannian $n$-manifold embedded into $\R^L$ and $\n$ is a closed Riemannian manifold embedded into $\R^M$.

For the fractional Gagliardo semi-norm we use the standard notation
\[
 [u]_{W^{s,p}(\Omega)} \coloneqq \brac{\int_{\Omega}\int_{\Omega} \frac{|u(x)-u(y)|^{p}}{|x-y|^{n+sp}}\dx\dy}^{\frac{1}{p}},
\]
for any open $\Omega\subseteq\Sigma$.

We will also write
\[
 E_{s,p}(u,\Omega) \coloneqq \int_{\Omega}\int_{\Omega} \frac{|u(x)-u(y)|^{p}}{|x-y|^{n+sp}}\dx\dy.
\]
We will write $\N_0 = \N\cup \{0\}$. We denote by $B(x,r)$ the geodesic ball about $x$ of radius $r$ in $\Sigma$. When the center of the ball will not play any role we will simply write  $B(r)$. 

For simplicity of notation, we write $\aleq$ if there exists a constant $C$ (not depending on any crucial quantity) such that $A \leq C\, B$. We use $\ageq$ in a similar way. Finally, $A \aeq B$ means that $A \aleq B$ and $B \aleq A$.

\subsection*{Acknowledgment}

We would like to thank Maciej Zdanowicz for consultations in Algebraic Geometry and Jean Van Schaftingen for helpful discussions on homotopy classes and Sobolev spaces.

\begin{itemize}
 \item K.M. was supported by FSR Incoming Post-doctoral Fellowship.
 \item A.S. was supported by Simons foundation grant no 579261.
\end{itemize}

\section{Preliminaries on homotopy theory for Sobolev maps}\label{s:homotopy}
The purpose of this section is to recall the definition for homotopy classes of maps $u \in W^{s,\frac{n}{s}}(\Sigma,\n)$. Here and henceforth we always assume that 
$\Sigma$ is a smooth $n$-dimensional compact manifold without boundary, and $\n \subset \R^M$ is a smooth embedded manifold, also without boundary. 

Let us stress that all of the definitions and statements in this section are well-known and we claim no originality whatsoever.

We make no effort to give the most general notion (e.g., considering $\Sigma$ with boundary), but concentrate on what is needed for our purposes. For more detailed exposition we refer, e.g., to \cite[Section~4]{HL03}. Maps in $u \in W^{s,\frac{n}{s}}(\Sigma,\n)$ may not be continuous, so one needs to define homotopy classes $W^{s,\frac{n}{s}}(\Sigma,\n)$ via approximation.

We first recall the usual notion of a homotopy for continuous maps. Two maps, $u,v \in C^0(\Sigma,\n)$ are \emph{homotopic}, in symbols $u \sim v$, if there exists a \emph{homotopy} $H \in C^0([0,1] \times \Sigma, \n)$, namely an $H$ which satisfies
\[
 H(0) = u, \quad H(1) = v.
\]
Since $\Sigma$ and $\n$ are smooth Riemannian manifolds, there is no difference between continuous homotopy and smooth homotopy --- and one does not distinguish between them. This is the content of the following lemma. 
\begin{lemma}\label{la:homotopygroupsequal}
Let $u,v \in C^\infty(\Sigma,\n)$. The following relations are equivalent:
\begin{itemize}
 \item $u \sim v$ in $C^0$, i.e., there exists $H \in C^0([0,1],C^0(\Sigma,\n))$ such that $H(0) = u$ and $H(1) = v$;
\item $u \sim v$ in $C^\infty$, i.e., there exists $H \in C^\infty([0,1], C^\infty(\Sigma,\n))$ such that $H(0) = u$ and $H(1) = v$.
\end{itemize}
\end{lemma}
The proof of \Cref{la:homotopygroupsequal} is by approximation (using a standard mollification argument in $[0,1] \times \Sigma$ by constant extension to $(-1,2) \times \Sigma$).
\begin{remark}
As a sidenote let us remark, that \Cref{la:homotopygroupsequal} may not be true on non-Riemannian manifolds, e.g., sub-Riemannian manifolds, Carnot groups, or more general metric spaces. See, e.g., \cite{WY14,HST14,HS14,So17,H18}.
\end{remark} 

The reason that we can make sense of the notion of homotopy for $W^{s,\frac{n}{s}}(\Sigma,\n)$-maps (recall: such maps may be even not continuous) is that they can be approximated by smooth maps in $C^\infty(\Sigma,\n)$.

Indeed, the following is going to be the definition of homotopy that we are going to use from now on.
\begin{definition}\label{def:homotopy}
Let $u,v \in W^{s,\frac{n}{s}}(\Sigma,\n)$. 
\begin{enumerate}
 \item We say $u \sim v$ ($u$ is homotopic to $v$) if the following holds: for any smooth approximation $u_\eps$ of $u$ and $v_\eps$ of $v$ in $W^{s,\frac{n}{s}}$ we find an $\eps_0 > 0$ such that for every $\eps \in (0,\eps_0)$ we have $u_\eps \sim v_\eps$ in $C^\infty$.
\item We define the homotopy class $[u]$ as
\[
 [u] \coloneqq \left \{ v \in W^{s,\frac{n}{s}}(\Sigma,\n)\colon\quad v \sim u \right \}.
\]
\end{enumerate}
\end{definition}
\begin{remark}
Let us remark that one can define equivalently the relation 
\[
 u\sim v \quad \text{ in } W^{s,\frac ns}(\Sigma,\n)
\]
for $u,\, v \in W^{s,\frac ns}(\Sigma,\n)$ as: there exists a path $H(t)\in C^0([0,1],W^{s,\frac ns}(\Sigma,\n))$ such that $H(0)=u$ and $H(1) = v$, cf. \cite[Section 4]{Brezis-theinterplay}.
\end{remark}

The justification for \Cref{def:homotopy} is contained in the following proposition. In particular it implies that we do not need to distinguish between $W^{s,\frac{n}{s}}(\Sigma,\n)$-homotopies and $C^0(\Sigma,\n)$-homotopy classes.
\begin{proposition}\label{pr:homotopiesallsame}
\Cref{def:homotopy} is well-defined in the following sense:
\begin{enumerate}
 \item Any map $u \in W^{s,\frac{n}{s}}(\Sigma,\n)$ can be approximated by maps $u_k \in C^\infty(\Sigma,\n)$ with respect to the $W^{s,\frac{n}{s}}$-norm.
 \item For any map $u \in W^{s,\frac{n}{s}}(\Sigma,\n)$ there exists a small number $\eps = \eps(u) > 0$ such that any map $v \in W^{s,\frac{n}{s}}\cap C^0(\Sigma,\n)$ with $\|u-v\|_{W^{s,\frac{n}{s}}(\Sigma)} < \eps$ is of the same $C^0$-homotopy type.
 \item If $u,v \in C^0\cap W^{s,\frac{n}{s}} (\Sigma,\n)$ then $u\sim v$ in the continuous sense, if and only if $u\sim v$ in the $W^{s,\frac{n}{s}}(\Sigma,\n)$-sense.
\end{enumerate}
\end{proposition}

For the convenience of the reader, we give the proof of \Cref{pr:homotopiesallsame} below. 

We begin by recalling the fact that the manifolds we work with have a tubular neighborhood on which there exists a smooth nearest point projection. For a proof we refer to \cite[Section 2.12.3]{S96}.
\begin{lemma}\label{la:tubular}
Let $\n \subset \R^M$ be a smooth, compact manifold without boundary. There exists $\delta =\delta(\n) > 0$ such that on the tubular neighborhood
\[
 B_{\delta}(\n) \coloneqq \left \{p \in \R^M:\, \dist(p,\n) < \delta\right \}
\]
there exists the nearest point projection $\pi_\n \in C^\infty(B_{\delta}(\n),\n)$ such that 
\[
 |\pi_{\n}(p)-p| = \dist(p,\n) \quad \forall p \in B_{\delta}(\n).
\]
Moreover, for $p \in \n$, $\Pi(p) \coloneqq D\pi_\n(p) \in \R^{M \times M}$ is the tangential projection which maps a vector $v \in \R^M$ orthogonally into the tangent plane $T_p \n$.
\end{lemma}

\Cref{pr:homotopiesallsame} (1) is a consequence of the following lemma, which was observed by Schoen and Uhlenbeck in their celebrated paper \cite[Section 3]{SU1}. We remark that as showed by Schoen and Uhlenbeck \cite[Section 4]{SU2} an approximation as in \Cref{la:smoothapprox} may not be possible if $u \in W^{s,p}$ if $sp<n$. We refer the interested reader to \cite{Bethuel-Zheng, Bethuel-approximation, HL03,BPVS15,Brezis-mironescu-fractionaldensity} for the theory of the approximation of manifold valued Sobolev maps by smooth maps. 
\begin{lemma}\label{la:smoothapprox}
Let $u \in W^{s,\frac{n}{s}}(\Sigma,\n)$ then $u$ can be approximated by smooth maps $u_\eps \in C^\infty(\Sigma,\n)$ in the $W^{s,\frac{n}{s}}(\Sigma,\R^M)$-norm.
\end{lemma}
\begin{proof}
For clarity of the proof we assume that $\Sigma=\R^n$. Let $\eps>0$ and let us first approximate $u$ by unconstrained smooth maps. To do so we mollify the function $u\in W^{s,\frac{n}{s}}(\R^n,\n)$ by considering
\[
 \tilde u_\eps(x) \coloneqq \int_{\R^n} \eta_\eps(x-y)u(y) \dy= \int_{\R^n} \eta(y)u(x-\eps y) \dy,
\]
where $\eta\in C^\infty_c(\R^n,[0,1])$, $\supp \eta \subset B(0,1)$, $\eta_\eps(x)\coloneqq \eps^{-n}\eta(\frac{x}{\eps})$. and $\int_{\R^n} \eta =1$. Then, $\tilde u_\eps \in C^\infty(\R^n,\R^M)$ and 
\[
 \tilde u_\eps \xrightarrow{\eps \to 0} u \quad \text{ strongly in } W^{s,\frac ns}(\R^n,\R^M).
\]
The smooth map $\tilde u_\eps$ may not map into $\n$, but we can ensure that the image of $\tilde u_\eps$ is close to the manifold $\n$.

Let $B_{\delta}(\n)$ be the tubular neighborhood of $\n$ from \Cref{la:tubular}, on which the nearest point projection $\pi_\n\colon B_{\delta}(\n)\to \n$ is well defined. 

Let $z\in\Sigma$ be an arbitrary point, then we estimate for every $x\in\Sigma$,
\[
 \begin{split}
  \dist(\tilde u_\eps(x),\n) 
  &\le \abs{ \tilde u_\eps(x) -  u(x-\eps z)}= \abs{\int_{\R^n}\eta(y) u(x-\eps y) \dy -  u(x-\eps z)}\\
  &= \abs{\int_{\R^n}\eta(y)( u(x-\eps y) - u(x-\eps z)) \dy}\\
  &\le \int_{\Sigma}\eta(y)\abs{ u(x-\eps y) -  u(x-\eps z)}\dy.
 \end{split}
\]
Thus, multiplying both sides by $\eta(z)$ and integrating over $\R^n$ with respect to the variable $z$, we obtain
\[
 \dist(\tilde u_\eps(x),\n) \le \int_{\Sigma} \int_{\Sigma}\eta(y)\eta(z)\abs{u(x-\eps y) - u(x-\eps z)}\dy \dz,
\]
which combined with the support of $\eta$ leads to the estimate
\[
 \dist(\tilde u_\eps(x),\n) \aleq \sup_{a\in\R^n} \barint_{B(a,\eps)}\barint_{B(a,\eps)} |u(y) - u(z)|\dy \dz.
\]
Applying H\"{o}lder's inequality we get
\[
\begin{split}
 \dist(\tilde u_\eps(x),\n) 
 &\aleq \sup_{a\in\R^n} \barint_{B(a,\eps)}\barint_{B(a,\eps)} |u(y) - u(z)|\dy \dz\\
 &\aleq \sup_{a\in\R^n}\brac{\barint_{B(a,\eps)}\barint_{B(a,\eps)} |y-z|^{\frac{2n}{\frac ns-1}}\dy \dz}^{1-\frac{s}{n}}\brac{\barint_{B(a,\eps)}\barint_{B(a,\eps)} \frac{|u(y)-u(z)|^\frac ns}{|y-z|^{2n}}\dy \dz}^\frac sn\\
 &\aleq \sup_{a\in\R^n} \brac{\int_{B(a,\eps)}\int_{B(a,\eps)} \frac{|u(y)-u(z)|^\frac ns}{|y-z|^{2n}}\dy\dz}^\frac sn \xrightarrow{\eps\to0}0,
\end{split}
 \]
where the last convergence is a consequence of the absolute continuity of the integral, and holds for a.e. $x \in \Sigma$.

Thus, for $\eps$ sufficiently close to 0, we know that 
$\tilde u_\eps \in B_{\delta}(\n)$. This implies that that the maps $u_\eps \coloneqq \pi_\n \circ \tilde u_\eps \in C^{\infty}(\R^n,\n)$ are well defined, \Cref{la:tubular}. Moreover, since $\pi_\n$ is smooth we also have 
\[
 u_\eps =\pi_\n \circ \tilde u_\eps \to \pi_n\circ u = u \quad \text{ strongly in } W^{s,\frac ns}(\Sigma,\R^M) \text{ as } \eps\to 0.
\]

\end{proof}
Next we state a helpful lemma that says that if two maps are uniformly close then they are homotopic.
\begin{lemma}\label{la:samehomotopycont}
Let $\n$ be a smooth manifold without boundary embedded into $\R^M$. Then there exists an $\eps = \eps(\n)> 0$ such that if $f,g \in C^0(\Sigma,\n)$ and $\|f-g\|_{L^\infty} < \eps$ then $f$ is homotopic to $g$. 
\end{lemma}
\begin{proof}
Let $\pi_{n}: B_\eps(\n) \to \n$ be the nearest point projection into the manifold which must exist for some $\eps > 0$, \Cref{la:tubular}. If $\|f-g\|_{L^\infty} < \eps$, then 
\[
\dist((1-t)f(x) + tg(x),\n) \leq \|f-g\|_{L^\infty} < \eps \quad \forall t \in [0,1], x \in \Sigma
\]
that is $H(t,x)= \pi_{\n}((1-t)f(x) + tg(x))$ is well-defined for all $t\in[0,1]$. It is easy to check that $H$ is a homotopy between $f$ and $g$.
\end{proof}

\Cref{pr:homotopiesallsame}(2) is a consequence of the following
\begin{lemma}\label{la:samehomotopy}
Let $u \in W^{s,\frac{n}{s}}(\Sigma,\n)$ then there exists $\eps=\eps(u) > 0$ such that whenever $g_1,g_2 \in C^0 \cap  W^{s,\frac{n}{s}}(\Sigma,\n)$ with
\begin{equation}\label{eq:sh:eps}
 \|u-g_i\|_{L^1(\Sigma)} + [u-g_i]_{W^{s,\frac{n}{s}}(\Sigma)} \leq \eps \quad \text{ for } i=1,2,
\end{equation}
then $g_1$ and $g_2$ are homotopic.
\end{lemma}
\begin{proof}
Again, for simplicity of notation we assume that $\Sigma = \R^n$.

By \Cref{la:tubular}, there exists $\gamma = \gamma(\n)$ and a smooth nearest point projection from a $\gamma$-neighborhood of $\n$  into $\n$ which we denote by $\pi_\n: B_\gamma(\n) \to \n$.

By the absolute continuity of the integral, for any $\theta > 0$ there exists a $\delta_0 = \delta_0(u,\theta)$ such that 
\begin{equation}\label{eq:laappr:theta}
 \sup_{B(2\delta_0)\subset \Sigma} [u]_{W^{s,\frac{n}{s}}(B(2\delta_0))} \leq \theta.
\end{equation}
For $i=1,2$ let $g_{i,\delta} \coloneqq \eta_\delta \ast g_i$, $\delta < \delta_0$ where $\eta \in C_c^\infty(B(0,1))$, $\int_{\R^n} \eta = 1$ is the usual mollifier.

As in the proof of \Cref{la:smoothapprox}, 
\[
\begin{split}
 \dist(g_{i,\delta}(x),\n) 
 &\le C(n)\, \mvint_{B(x,\delta)}\mvint_{B(x,\delta)} |g_i(y)-g_i(z)|\dif y \dif z\\
 &\leq C(n,s)\, [g_{i}]_{W^{s,\frac{n}{s}}(B(2\delta))}\\
 &\leq C(n,s)\brac{[u-g_{i}]_{W^{s,\frac{n}{s}}(B(2\delta))}+[u]_{W^{s,\frac{n}{s}}(B(2\delta))}}\\
 &\leq C(n,s)\brac{\eps + \theta},
 \end{split}
\]
where in the last inequality we used \eqref{eq:laappr:theta} and \eqref{eq:sh:eps}.

So if $\eps$ and $\theta$ are small enough so that $C(n,s)(\eps + \theta) < \gamma$, we have  $\tilde{g}_{i,\delta} \coloneqq \pi_{\n} \circ g_{i,\delta}$ is well defined for any $i=1,2$ and any $\delta < \delta_0$.

Observe that $\delta \ni [0,\delta_0] \mapsto \tilde{g}_{i,\delta}$ is a homotopy, so $g_i$ and $\tilde{g}_{i,\delta_0}$ are homotopic for each $i=1,2$.

Moreover,
\[
 \|g_{1,\delta_0}-g_{2,\delta_0}\|_{L^\infty} \leq C(n) \frac{1}{\delta_0^n} \|g_1-g_2\|_{L^1} \leq C(n) \frac{\eps}{\delta_0^n}.
\]
So if we choose $\eps=\eps(u) > 0$ possibly even smaller so that $C(n) \frac{\eps}{\delta_0^n} < \eps(\n)$, where $\eps(\n)$ is from \Cref{la:samehomotopycont}, then we know that $g_{1,\delta_0}$ is homotopic to $g_{2,\delta_0}$. That is, we have shown
\[
 g_{1} \sim g_{1,\delta_0} \sim g_{2,\delta_0} \sim g_2.
\]
This concludes the proof.
\end{proof}

\begin{proof}[Proof of \Cref{pr:homotopiesallsame}(3)]
Let $u,v \in C^0\cap W^{s,\frac{n}{s}}(\Sigma,\n)$ and assume $u \sim v$ with respect to continuous homotopy. Denote the usual convolution of $u$ and $v$ with the standard mollifier respectively by $u_\delta$ and $v_\delta$. Then $u_\delta$ converges uniformly to $u$. In particular, for all small $\delta$, we have that $\pi_{\n} \circ u_\delta$ is $C^0$-homotopic to $u$, by \Cref{la:samehomotopycont}. Similarly $\pi_{\n} \circ v_\delta$ is $C^0$-homotopic to $v$. Since $v$ and $u$ are $C^0$-homotopic, this implies that $\pi_{\n} \circ u_\delta$ and $\pi_{\n} \circ v_\delta$ are $C^0$-homotopic to each other for all small $\delta$. But $\pi_{\n} \circ u_\delta$ is a smooth approximation of $u$ with respect to the $W^{s,\frac{n}{s}}$-norm, and similarly $\pi_{\n} \circ v_\delta$ of $v$. By \Cref{la:samehomotopy} this means that any other smooth approximation of $v$ and $u$, respectively, are also eventually $C^0$-homotopic to each other. By \Cref{def:homotopy} this means that $u$ and $v$ are $W^{s,\frac{n}{s}}$-homotopic.

For the converse we argue similarly. If $u$ and $v$ are $W^{s,\frac{n}{s}}$-homotopic as defined in \Cref{def:homotopy}, $\pi_{\n} \circ u_\delta$ and $\pi_{\n} \circ v_\delta$ must be homotopic for all small $\delta$. For small $\delta$ we have $\pi_{\n} \circ u_\delta$ is $C^0$-homotopic to $u$ (by uniform convergence and \Cref{la:samehomotopycont}) and likewise $\pi_{\n} \circ v_\delta$ is $C^0$-homotopic to $v$. This implies that $u$ is $C^0$-homotopic to $v$.
\end{proof}
Similar to \Cref{la:samehomotopy} we also obtain
\begin{lemma}\label{la:smallenergytrivial}
For any manifold $\Sigma,\n$ as above and $s \in (0,1)$ there exist $\eps = \eps(\n,\Sigma)$ such that the following holds.

If $u \in W^{s,\frac{n}{s}}(\Sigma,\n)$ and
\begin{equation}\label{eq:set:1}
 [u]_{W^{s,\frac{n}{s}}(\Sigma)} < \eps,
\end{equation}
then $u$ is homotopic to a constant map in the sense of \Cref{def:homotopy}.
\end{lemma}
\begin{proof}
Let $(u)_{\Sigma} \coloneqq |\Sigma|^{-1} \int_{\Sigma} u$. From \eqref{eq:set:1} we obtain as in the proof of \Cref{la:smoothapprox} 
\[
 \dist((u)_{\Sigma},\n) \aleq \eps.
\]
If $\eps$ is small enough, this implies that $v \coloneqq \pi_{\n}((u)_\Sigma)$ is well-defined, by \Cref{la:tubular}.

Also, denoting by $u_\delta \coloneqq \eta_\delta \ast u$ the usual mollification, we have 
\[
 \dist(u_\delta,\n) \aleq \eps \quad \forall \delta \leq 1.
\]
Setting $w_\delta \coloneqq \pi_{\n}(u_\delta)$ we have that $w_1$ is homotopic to $u$ in the sense of \Cref{def:homotopy}. Moreover we have 
\[
 \|w_1-v\|_{L^\infty} \leq C(\Sigma,\n) \|u-(u)_{\Sigma}\|_{L^1(\Sigma)} \aleq [u]_{W^{s,\frac{n}{s}}(\Sigma)} < \eps.
\]
So choosing $\eps$ small enough, we have from \Cref{la:samehomotopycont} that $w_1$ and $v$ are homotopic. This implies that $u$ and $v$ are homotopic, and $v$ is a constant map.
\end{proof}

\section{Regularity theory for minimizers in homotopy}\label{s:regularity}
The main result of this section is the following regularity theory for minimizers.
\begin{theorem}\label{th:reghomo}
Let $\Sigma,\, \n$ be as above. If $n=1$ then assume that $s\le \frac12$, if $n\ge 2$, then assume that $s\in(0,1)$. There exists $\eps>0$ and $s_0 > s$ such that the following holds for any $t \in [s,s_0]$.

Assume that $u \in W^{t,\frac{n}{s}}(\Sigma,\n)$ and that or a geodesic ball $B(R) \subset \Sigma$ the following holds:
\begin{itemize}
 \item $u$ is a minimizing $W^{t,\frac{n}{s}}$-harmonic map in $B(R)$, that is
 \[
  E_{t,\frac ns}(u,\Sigma) \leq E_{t,\frac ns}(v,\Sigma)
 \]
holds for all $v \in W^{t,\frac{n}{s}}(\Sigma,\n)$ such that 
\begin{itemize}
 \item $u \equiv v$ in $\Sigma \setminus  B(R)$, and
 \item $u \sim v$ in homotopy (as defined in \Cref{def:homotopy}).
\end{itemize}
 \item $[u]_{W^{s,\frac{n}{s}}(B(R))} < \eps$.
\end{itemize}
Then $u \in W^{s_0,\frac{n}{s}}(B(R/2)) \cap C^{s_0-s}(B(R/2))$ and we have the estimate
\begin{equation}\label{eq:threg:estimate}
 [u]_{C^{s_0-s}(B(R/2))} + [u]_{W^{s_0,\frac{n}{s}}(B(R/2))} \leq C\, R^{s-s_0} [u]_{W^{s,\frac{n}{s}}(B(R))} \brac{[u]_{W^{s,\frac{n}{s}}(\Sigma)}^{1-\frac{s}{n}}+[u]_{W^{s,\frac{n}{s}}(\Sigma)}}.
\end{equation}
\end{theorem}

The important feature of \Cref{th:reghomo} is that the regularity estimate is uniform as $t\to s^+$. By Morrey--Sobolev embedding, any map $u \in W^{t,\frac{n}{s}}(\Sigma,\n)$ is $C^{t-s}$-continuous if $t>s$, but it may not be $C^{s_0-s}$-continuous.

Clearly, global minimizers (without any assumptions on homotopy type) also fall under the realm of \Cref{th:reghomo}, and we record the following.
\begin{corollary}
Let $u \in W^{s,\frac{n}{s}}(\Sigma,\n)$ be a minimizing harmonic map (without restriction to any homotopy class) in an open set $\Omega \subset \Sigma$, i.e., assume that
\[
E_{s,\frac ns}(u,\Sigma) \leq E_{s,\frac ns}(v,\Sigma)
\]
for all $v \in W^{s,\frac{n}{s}}(\Sigma,\n)$ with $u \equiv v$ on $\Omega^c$. Then $u$ is H\"older continuous in $\Omega$.
\end{corollary}
\begin{remark}
While the initial step in the proof of \Cref{th:reghomo}, namely \Cref{th:ininitreg}, relies on the minimizing property, this is probably only really necessary in the case $t=s$. Most likely,  for $t>s$ one could test the Euler--Lagrange equations to obtain a similar result (but due to the necessity for uniform H\"older exponents we did not attempt to do this). 

That is,
most likely \Cref{th:reghomo} holds for \emph{critical} $W^{t,\frac{n}{s}}$-harmonic maps as long as $t > s$. In particular it seems that a similar statement as in \Cref{th:reghomo} could be made, e.g., for maps $u \in W^{s,p}(\Sigma,\R^M)$, $s-\frac{n}{p} > 0$, satisfying
\[
 \left |\int_{\Sigma}\int_{\Sigma} \frac{|u(x)-u(y)|^{p-2} (u(x)-u(y)) (\varphi(x)-\varphi(y))}{|x-y|^{n+sp}}\dx\dy\right | \aleq \int_{\Sigma}\int_{\Sigma} |\varphi(x)| \frac{|u(x)-u(y)|^{p}}{|x-y|^{n+sp}}\dx\dy.
\]
This is not necessary for our purposes, so we do not follow this direction.
\end{remark}

\begin{remark}
For $t=s$ and ``round'' target spaces $\n = \S^{n-1}$ or $\n$ a compact Lie group \Cref{th:reghomo} holds also for (possibly non-minimizing, but only) \emph{critical} $W^{s,\frac{n}{s}}$-harmonic maps \cite{S15,MS18}. For non-round targets this is a major open question even for the $p$-harmonic map case $s=1$, $p\neq 2$, and only partial results are known under non-geometric assumptions \cite{DM10,K10,S13}, see also the survey \cite{SS17}.
\end{remark}

The proof of \Cref{th:reghomo} consists of three steps: 
\begin{itemize}
\item[Step 1.] We first prove in \Cref{th:ininitreg} local $C^\alpha$-regularity of the solution. We do not require a precise estimate of $[u]_{C^\alpha}$, but we crucially get that $\alpha$ is independent of $t$.
\item[Step 2.] In \Cref{th:BLdiff} we show that the $C^\alpha$-regularity from above translates into a $W^{s+\beta,\frac{n}{s}}$-regularity for $\beta = \beta(\alpha)$ using a technique developed by Brasco and Lindgren, \cite{BL17,BLS18}. We then choose $s_0 \coloneqq s+\beta$.
 \item[Step 3.] The \emph{estimate} in \Cref{th:reghomo} is a consequence of a priori estimates of the Euler--Lagrange equations, i.e., of the respective harmonic map equation, under the assumption that the solution already belongs to $W^{s_0,\frac{n}{s}}$. This will be done in  \Cref{th:apriorireg}, and is based on a stability estimate for the fractional $p$-Laplacian, see \cite{S16}.
\end{itemize}

\subsection{Step 1: Uniform H\"older continuity}
We begin the proof of \Cref{th:reghomo} by first proving H\"older continuity of minimizers --- with a uniform H\"older exponent $\alpha > 0$ which does not change as $t \to s^+$. Namely we obtain the following theorem.

\begin{theorem}\label{th:ininitreg}
Let $s \in (0,1)$ and $\Sigma,\, \n$ be as above. There exists $\alpha > 0$ and $s_1 > s$ such that the following holds for any $t \in [s,s_1)$.

Assume that $u \in W^{t,\frac{n}{s}}(\Sigma,\n)$ and for a geodesic ball $B(R) \subset \Sigma$  $u$ is a minimizing $W^{t,\frac{n}{s}}$-harmonic map in $B(R)$, that is
 \[
  E_{t,\frac ns}(u,\Sigma) \leq E_{t,\frac ns}(v,\Sigma)
 \]
holds for all $v \in W^{t,\frac{n}{s}}(\Sigma,\n)$ such that 
\begin{itemize}
 \item $u \equiv v$ in $\Sigma \setminus  B(R)$, and
 \item $u \sim v$ in homotopy (as defined in \Cref{s:homotopy}).
\end{itemize}

Then $u \in C^\alpha_{loc}(B(R))$.
\end{theorem}
Let us remark that Millot--Sire--Yu \cite{MSY18} already obtained partial regularity for $E_{s,2}$-minimizers for $n=1$, $s< \frac{1}{2}$.

The proof of \Cref{th:ininitreg} follows from a Cacciopoli type estimate (and the technique probably can be traced back to Morrey \cite{M48}).

The first step is to construct a suitable competitor map. For $t > s$ this is simply the interpolation between $u$ and the mean value $(u)_{B(r)\setminus B(r/2)}$. For $t=s$ we have to be more careful, and apply an argument similar to Luckhaus' lemma.
Namely we have
\begin{lemma}\label{la:competitor}
Let $\Sigma,\, \n$ be as above, $s \in (0,1)$ and $s_1 \in (s,1)$. There exists a constant $C > 0$ such that the following holds for any $t \in [s,s_1]$.

Let $u \in W^{t,\frac{n}{s}}(\Sigma,\n)$. There exists an $\eps > 0$ (possibly depending on $u$) such that the following holds.

Assume that for some $r \in (0,1)$ and some ball $B(10r) \subset \Sigma$,
\begin{itemize} 
 \item[If]  $t = s$ \begin{equation}\label{eq:comp:uabscont}
 \int_{B(10r)} \int_{\Sigma} \frac{|u(x)-u(y)|^{\frac{n}{s}}}{|x-y|^{2n}} \dx \dy < \eps^\frac ns.
\end{equation}
\item[If] $t>s$
\begin{equation}\label{la:comp:ucont}
 |u(x)-u(y)| < \eps \quad \forall x,y \in B(10r).
\end{equation}
\end{itemize}
Then, there exists a $v \in W^{t,\frac{n}{s}}(\Sigma,\n)$ such that 
\begin{enumerate}
 \item $v \equiv u$ in $\Sigma \setminus B(r)$,
 \item $v$ is homotopic to $u$, and   
 \item we have \begin{equation}\label{la:comp:west}
 \begin{split}
        \int_{B(r)} \int_{\Sigma} \frac{|v(x)-v(y)|^{\frac{n}{s}}}{|x-y|^{n+t\frac{n}{s}}} \dif x \dif y
%
        \leq C \int_{B(2r)\setminus B(r/2) } \int_{\Sigma} \frac{|u(x)-u(y)|^{\frac{n}{s}}}{|x-y|^{n+t\frac{n}{s}}} \dx \dy.
        \end{split}
       \end{equation}
\end{enumerate}
\end{lemma}

\begin{proof}[Proof of \Cref{la:competitor} for $t>s$]
 For simplicity of the presentation we assume that $\Sigma=\R^n$, moreover without loss of generality we can assume $B(r) = B(0,r)$.

Since $\n$ is a smooth compact manifold without boundary, there exists a tubular neighborhood $B_\eps(\n) \subset \R^M$ and a smooth projection $\pi_{\n}: B_\eps(\n) \to \n$, \Cref{la:tubular}. We choose $\eps$ possibly even smaller so that $\eps\|D\pi_\n\|_{L^\infty}  < \tilde \eps(\n)$, where $\tilde \eps(\n)$ is the small quantity from \Cref{la:samehomotopycont}.

Set $\eta \in C_c^\infty(B(\frac{9}{10} r),[0,1])$ and $\eta \equiv 1$ in $B(\frac{8}{10}r)$, $|\nabla \eta|\aleq r^{-1}$. Set
\[
 w\coloneqq(1-\eta)u + \eta (u)_{A(r)}.
\]
Here $A(r) = B(r) \setminus B(r/2)$ and $(u)_{A(r)}$ denotes the mean value on that set.
Observe that
\[
\begin{split}
 \dist((1-\eta(x))u(x) + \eta(x) (u)_{A(r)},\n) 
 &\leq |(1-\eta)u(x) + \eta (u)_{A(r)} - u(x)|\\
  &\leq \eta(x) |u(x) - (u)_{A(r)} |\\
  &< \eps,
 \end{split}
\]
we used \eqref{la:comp:ucont} in the last inequality.

Thus, we can compose $w$ with $\pi_\n$ and set 
\[
 v \coloneqq \pi_{\n}\circ w.
\]
Observe,
\[
 |u(x)-v(x)|\leq \|D\pi_\n\|_{L^\infty} |u(x) - (u)_{A(r)} | \chi_{B(r)}(x) \leq \|D\pi_\n\|_{L^\infty}\frac{\eps}{100}  < \tilde \eps(\n).
\]
From \Cref{la:samehomotopycont} we obtain that $u$ and $v$ are homotopic. 

It remains to prove the estimate \eqref{la:comp:west}. 
By Lipschitz continuity of $\pi_\n$ we have
\[
        \int_{B(r)} \int_{\R^n} \frac{|v(x)-v(y)|^{\frac{n}{s}}}{|x-y|^{n+t\frac{n}{s}}} \dif x \dif y
        \leq C(\pi_\n)\int_{B(r)} \int_{\R^n} \frac{|w(x)-w(y)|^{\frac{n}{s}}}{|x-y|^{n+t\frac{n}{s}}}.
\]        
Now we have 
\[
\begin{split}
w(x)-w(y)=&(1-\eta(x))u(x) + \eta(x) (u)_{A(r)}-\brac{(1-\eta(y))u(y) + \eta(y) (u)_{A(r)}}\\
=&(1-\eta(x)) (u(x)-u(y))- (\eta(x)-\eta(y)) \brac{u(y)-(u)_{A(r)}}.
\end{split}
\]
So we have, decomposing $\R^n=B(r)\cup \R^n\setminus B(r)$ (it is important to observe that we can choose all of the constants to be independent of $t$ as long as $t \in [s,s_1]$)
\[
\begin{split}
 \int_{B(r)} \int_{\R^n} \frac{|v(x)-v(y)|^{\frac{n}{s}}}{|x-y|^{n+t\frac{n}{s}}}    \dx \dy  
 &\aleq \int_{B(r)}\int_{B(r)} (1-\eta(x))^{\frac{n}{s}} \frac{|u(x)-u(y)|^{\frac{n}{s}}}{|x-y|^{n+t\frac{n}{s}}} \dx \dy\\
 &\quad +\int_{B(r)}\int_{B(r)} \frac{|\eta(x)-\eta(y)|^{\frac{n}{s}} |u(y)-(u)_{A(r)}|^{\frac{n}{s}}}{|x-y|^{n+t\frac{n}{s}}} \dx \dy
  \\
  &\quad +\int_{B(r)}\int_{\R^n \setminus B(r)} (1-\eta(x))^{\frac{n}{s}} \frac{|u(x)-u(y)|^{\frac{n}{s}}}{|x-y|^{n+t\frac{n}{s}}} \dx \dy\\
  &\quad+\int_{B(r)}\int_{\R^n \setminus B(r)} \frac{|\eta(x)-\eta(y)|^{\frac{n}{s}} |u(y)-(u)_{A(r)}|^{\frac{n}{s}}}{|x-y|^{n+t\frac{n}{s}}} \dx \dy.
 \end{split}
\]
In view of the support of $\eta$ and $1-\eta$ we find 
\[
\begin{split}
 \int_{B(r)} \int_{\R^n} \frac{|v(x)-v(y)|^{\frac{n}{s}}}{|x-y|^{n+t\frac{n}{s}}}    \dx \dy &\aleq \int_{B(r)}\int_{B(r)\setminus B(r/2)} \frac{|u(x)-u(y)|^{\frac{n}{s}}}{|x-y|^{n+t\frac{n}{s}}} \dx \dy\\
 &\quad +r^{-\frac{n}{s}}\int_{B(r)}\int_{B(r)} \frac{|u(y)-(u)_{A(r)}|^{\frac{n}{s}}}{|x-y|^{n+(t-1)\frac{n}{s}}} \dx \dy\\
 &\quad +\int_{B(r)}\int_{\R^n \setminus B(r)} \frac{|u(x)-u(y)|^{\frac{n}{s}}}{|x-y|^{n+t\frac{n}{s}}} \dx \dy\\
 &\quad +\int_{B(r)}\int_{\R^n \setminus B(r)} \frac{|\eta(x)-\eta(y)|^{\frac{n}{s}} |u(y)-(u)_{A(r)}|^{\frac{n}{s}}}{|x-y|^{n+t\frac{n}{s}}} \dx \dy.
 \end{split}
\]
Integrating in $x$ and then using Jensen's inequality we have 
\[
\begin{split}
 r^{-\frac{n}{s}}\int_{B(r)}\int_{B(r)} \frac{|u(y)-(u)_{A(r)}|^{\frac{n}{s}}}{|x-y|^{n+(t-1)\frac{n}{s}}} \dx \dy
 &\aleq r^{-t\frac{n}{s}} \int_{B(r)} |u(y)-(u)_{A(r)}|^{\frac{n}{s}} \dy\\
 &\aleq \int_{B(r)}\int_{B(r)/B(r/2)} \frac{|u(y)-u(z)|^{\frac{n}{s}}}{|y-z|^{n+t\frac{n}{s}}} \dz \dy.
 \end{split}
\]
Moreover,
\[
\begin{split}
 &\int_{B(r)}\int_{\R^n \setminus B(r)} \frac{|u(x)-u(y)|^{\frac{n}{s}}}{|x-y|^{n+t\frac{n}{s}}} \dx \dy\\
 &=\int_{B(r)}\int_{B(2r)\setminus B(r)} \frac{|u(x)-u(y)|^{\frac{n}{s}}}{|x-y|^{n+t\frac{n}{s}}} \dx \dy +\int_{B(r)}\int_{\R^n \setminus B(2r)} \frac{|u(x)-u(y)|^{\frac{n}{s}}}{|x-y|^{n+t\frac{n}{s}}} \dx \dy\\
 &\aleq\int_{B(2r)\setminus B(r)} \int_{\R^n}\frac{|u(x)-u(y)|^{\frac{n}{s}}}{|x-y|^{n+t\frac{n}{s}}} \dx \dy\\
 &\quad +\int_{B(r)}\int_{\R^n \setminus B(2r)} \frac{|u(y)-(u)_{A(r)}|^{\frac{n}{s}}}{|x-y|^{n+t\frac{n}{s}}} \dx \dy+\int_{B(r)}\int_{\R^n \setminus B(2r)} \frac{|u(x)-(u)_{A(r)}|^{\frac{n}{s}}}{|x-y|^{n+t\frac{n}{s}}} \dx \dy\\
 &\aleq\int_{B(2r)\setminus B(r)} \int_{\R^n}\frac{|u(x)-u(y)|^{\frac{n}{s}}}{|x-y|^{n+t\frac{n}{s}}} \dx \dy\\
 &\quad+r^{-t\frac{n}{s}} \int_{B(r)}|u(y)-(u)_{A(r)}|^{\frac{n}{s}}\dy+r^{-n}\int_{B(r)}\int_{B(r)/B(r/2)}\int_{\R^n \setminus B(2r)} \frac{|u(x)-u(z)|^{\frac{n}{s}}}{|x-y|^{n+t\frac{n}{s}}} \dx \dz \dy\\
 &\aleq\int_{B(2r)\setminus B(r/2)} \int_{\R^n}\frac{|u(x)-u(y)|^{\frac{n}{s}}}{|x-y|^{n+t\frac{n}{s}}} \dx \dy.
 \end{split}
\]
In the last step we used that if $x \in \R^n \setminus B(2r)$, $y \in B(r)$, and $z \in B(r) \setminus B(r/2)$ then $|x-y|\ge |r-|x||$ and
\[
 |x-z| \le \dist(x,B(r)) + \dist(z,B(r)) \le 2 \dist(x,B(r)) \le 2 |r-|x||.
\]
Similarly,
\[
 \begin{split}
  &\int_{B(r)}\int_{\R^n \setminus B(r)} \frac{|\eta(x)-\eta(y)|^{\frac{n}{s}} |u(y)-(u)_{A(r)}|^{\frac{n}{s}}}{|x-y|^{n+t\frac{n}{s}}} \dx \dy\\
  &\aleq \int_{B(r)}\int_{B(2r)\setminus B(r)} \frac{|\eta(x)-\eta(y)|^{\frac{n}{s}} |u(y)-(u)_{A(r)}|^{\frac{n}{s}}}{|x-y|^{n+t\frac{n}{s}}} \dx \dy \\
  &\quad+ \int_{B(r)}\int_{\R^n \setminus B(2r)} \frac{|\eta(x)-\eta(y)|^{\frac{n}{s}} |u(y)-(u)_{A(r)}|^{\frac{n}{s}}}{|x-y|^{n+t\frac{n}{s}}} \dx \dy\\
  &\aleq r^{-\frac{n}{s}} \int_{B(r)}\int_{B(2r)\setminus B(r)} \frac{|u(y)-(u)_{A(r)}|^{\frac{n}{s}}}{|x-y|^{n+(t-1)\frac{n}{s}}} \dx \dy 
  + \int_{B(r)}\int_{\R^n \setminus B(2r)} \frac{|u(y)-(u)_{A(r)}|^{\frac{n}{s}}}{|x-y|^{n+t\frac{n}{s}}} \dx \dy\\
  &\aleq \int_{B(2r)\setminus B(r/2)} \int_{\R^n}\frac{|u(x)-u(y)|^{\frac{n}{s}}}{|x-y|^{n+t\frac{n}{s}}} \dx \dy.
 \end{split}
\]

This establishes \eqref{la:comp:west} and the proof is complete (in the case $t  >s$).
\end{proof}

\begin{proof}[Proof of \Cref{la:competitor} for $t=s$]
In the following we restrict for simplicity to the case $n \geq 2$, but the statement remains true for $n=1$ with easy modifications. Again, for the simplicity of the presentation, we assume that $\Sigma=\R^n$.

We claim that there exists an radius $\rho\in(\frac{3}{4} r,\frac{5}{6} r)$ such that 
\begin{equation}\label{eq:comp:goodslice1}
\begin{split}
&r \int_{\partial B(\rho)}\int_{B(2r)\backslash B(r/2)} \frac{|u(\theta) - u(\omega)|^\frac ns}{|\theta-\omega|^{2n}} \dif \theta \dif \omega +
 r\int_{\partial B(\rho)}\int_{\partial B(\rho)} \frac{|u(\theta) - u(\omega)|^\frac ns}{|\theta-\omega|^{2n-1}} \dif \theta \dif \omega \\
 &\aleq \int_{B(r)\setminus B(r/2)}\int_{B(2r)\setminus B(r/2)} \frac{|u(x)-u(y)|^{\frac{n}{s}}}{|x-y|^{2n}} \dx \dy.
 \end{split}
\end{equation}
Indeed, by Fubini's theorem for any $\kappa \in (0,\frac{1}{2})$ there exists a set 
$A_{\kappa} \subset (\frac 34 r,\frac56 r)$, with $\mathcal{L}^1((\frac 34 r,\frac56 r)\setminus A_\kappa) \leq \kappa r$ such that for any $\tau \in A_\kappa$,
\begin{equation}\label{eq:Fubiniononeintegral}
 r \int_{\partial B(\tau)} \int_{B(2r)\setminus B(r/2)} \frac{|u(x) - u(\omega)|^\frac ns}{|x-\omega|^{2n}} \dif x \dif \omega  \leq \kappa^{-1} \int_{B(r)\setminus B(r/2)} \int_{B(2r)\setminus B(r/2)} \frac{|u(x) - u(y)|^\frac ns}{|x-y|^{2n}} \dif x \dif y.
\end{equation}
Indeed, if \eqref{eq:Fubiniononeintegral} was not true on a set $\bar{A} \subset (\frac{3}{4}r,\frac{5}{6}r)$ of $\mathcal{L}^1$-measure $> \kappa r$ we integrate both sides over that set and get
\[
 \int_{B(\frac{5}{6}r)\backslash B(\frac{3}{4}r)} \int_{B(2r)\setminus B(r/2)} \frac{|u(x) - u(\omega)|^\frac ns}{|x-\omega|^{2n}} \dif x \dif \omega  > \int_{B(r)\setminus B(r/2)} \int_{B(2r)\setminus B(r/2)} \frac{|u(x) - u(y)|^\frac ns}{|x-y|^{2n}} \dif x \dif y,
\]
a contradiction to the monotonicity of the integral with respect to its integration domain.

Also by Fubini's theorem for any $\sigma \in (0,\frac{1}{2})$ there exists a set 
$B_{\sigma} \subset (\frac 34 r,\frac56 r)$, with $\mathcal{L}^1((\frac 34 r,\frac56 r)\setminus B_\sigma) \aleq \sigma r$ such that for any $\tau \in A_\sigma$,
\begin{equation}\label{eq:Fubiniononeintegral2}
\begin{split}
 r \int_{\partial B(\tau)} \int_{\partial B(\tau)} & \frac{|u(\omega) - u(\theta)|^\frac ns}{|\omega-\theta|^{2n-1}} \dif \omega \dif \theta\\ 
 &\leq \sigma^{-1} \int_{B(r)\setminus B(r/2)} \int_{B(r)\setminus B(r/2)} \frac{|u(y) - u(x)|^\frac ns}{|y-x|^{2n}} \dif y \dif x
 \end{split}
\end{equation}
for every $\tau \in B_\sigma$. The arguments to obtain \eqref{eq:Fubiniononeintegral2} are a bit more complicated, although well-known to experts. 
For simplicity let $r=1$, the right power for the factor involving $r$ follows from scaling arguments. One argument for \eqref{eq:Fubiniononeintegral2} goes via the Gagliardo-extension\footnote{This was popularized in the PDE community by \cite{CS07}, see also the harmonic analysis side in \cite{S93,BC17} or, for an collection of identifications, \cite[Proposition 10.2]{LS20}.}, we have
\[
 \int_{B(1)\setminus B(1/2)} \int_{B(1)\setminus B(1/2)} \frac{|u(y) - u(x)|^\frac ns}{|y-x|^{2n}} \dif y \dif x \aeq \inf_{U} \int_{B(1)\backslash B(1/2)\times[0,\infty)} t^{\frac{n}{s}-1-n} |\nabla U|^{\frac{n}{s}},
\]
where the infimum is taken over all smooth maps $U: B(1)\backslash B(1/2) \times [0,\infty) \to \R^M$ with $U = u$ in the trace sense on $B(1)\backslash B(1/2) \times [0,\infty)$. Now we can apply the argument via Fubini's theorem in $B(1) \backslash B(1/2) \times [0,\infty)$ and find a large set $B_\sigma$ so that for each slice $\rho \in B_\sigma$
\[
 \int_{\partial B(\rho) \times[0,\infty)} t^{\frac{n}{s}-1-n} |\nabla U|^{\frac{n}{s}}  \aleq \sigma^{-1} \int_{B(1)\backslash B(1/2)\times[0,\infty)} t^{\frac{n}{s}-1-n} |\nabla U|^{\frac{n}{s}}.
\]
By the trace theorem we have 
\[
 [u]_{W^{s,\frac{n}{s}}(\partial B(\rho))}^{\frac{n}{s}} \aleq \int_{\partial B(\rho) \times[0,\infty)} t^{\frac{n}{s}-1-n} |\nabla U|^{\frac{n}{s}}.
\]
Suitably scaling this argument we obtain \eqref{eq:Fubiniononeintegral2} for any $r \in (0,1)$.

Combining \eqref{eq:Fubiniononeintegral} and \eqref{eq:Fubiniononeintegral2}, taking $\sigma$ and $\kappa$ small enough we can ensure that $A_\kappa \cap B_\sigma \neq \emptyset$ and for $\rho \in A_{\kappa}\cap B_\sigma$ \eqref{eq:comp:goodslice1} holds.

From now on we fix such ``good slice'', i.e., a $\rho \in (\frac{3}{4}r,\frac{5}{6}r)$ such that \eqref{eq:comp:goodslice1} holds.

We know from Morrey--Sobolev embedding that $W^{s,\frac ns}(\partial B(\rho)) \subset C^\frac sn(\partial B(\rho))$ with
\begin{equation}\label{eq:comp:goodslice2}
\begin{split}
|u(\theta) - u(\omega)| 
&\aleq |\theta-\omega|^\frac sn [u]_{W^{s,\frac ns}(\partial B(\rho))} \\
&\aleq |\theta-\omega|^\frac sn r^{-\frac sn} \brac{\int_{B(r)\setminus B(r/2)}\int_{B(2r)\setminus B(r/2)} \frac{|u(x)-u(y)|^{\frac{n}{s}}}{|x-y|^{2n}} \dx \dy}^{\frac sn}\\
&\aleq \brac{\int_{B(r)\setminus B(r/2)}\int_{B(2r)\setminus B(r/2)} \frac{|u(x)-u(y)|^{\frac{n}{s}}}{|x-y|^{2n}} \dx \dy}^{\frac sn} \quad \text{ for all } \theta,\, \omega\in \partial B(\rho).
\end{split}
\end{equation}
%

Set for a $\delta\in(0,\frac14)$ 
\[
w(x) \coloneqq
\begin{cases}
u(x) \quad &|x| > \rho\\
(1-\eta(|x|)) u(\theta) + \eta(|x|) (u)_{\partial B(\rho)} \quad &\theta = \rho\frac{x}{|x|},\ |x| \in ((1-\delta)\rho,\rho)\\
(u)_{\partial B(\rho)} \quad & |x|<(1-\delta)\rho,
 \end{cases}
\]
where $\eta: \R_+ \to [0,1]$ is smooth with $\eta(t) = 0$ for $t \geq (1-\frac{\delta}{2})\rho$ and $\eta \equiv 1$ on $(0,(1-\frac{3}{4} \delta)\rho]$, $|\eta'| \leq \frac{100}{\delta \rho}$.

We apply a Lemma reminiscent of Luckhaus' Lemma, namely \Cref{la:frluckhaus}. Observe from \eqref{eq:comp:goodslice2} and \eqref{eq:comp:uabscont} we obtain
\[
 \dist((u)_{\partial B(r)},\n) \aleq \eps.
\]
From \Cref{la:frluckhaus} and again \eqref{eq:comp:uabscont} we obtain 
\[
 \dist(w,\n) \aleq \eps.
\]
We choose the $\eps$ in the assumptions of \Cref{la:competitor} small enough so that the map $w$ lies in the tubular neighborhood of the manifold $\n$, cf. \Cref{la:tubular}.

We set \[v \coloneqq \pi_\n \circ w.\]

\underline{We need to show \eqref{la:comp:west}.}

Let $A(\rho)=B(\frac{3}{2}\rho) \setminus B(\rho)$ and denote by $(u)_{A(\rho)}=(v)_{A(\rho)}$ the mean value of $u=v$ on $A(\rho)$.
\begin{equation}\label{eq:competitor:ts:1}
\begin{split}
\int_{B(r)} \int_{\R^n} \frac{|v(x)-v(y)|^{\frac{n}{s}}}{|x-y|^{2n}} \dx \dy 
&\aleq \int_{B(2\rho)} \int_{B(2\rho)} \frac{|v(x)-v(y)|^{\frac{n}{s}}}{|x-y|^{2n}} \dx \dy\\
&\quad +\int_{B(\frac 43\rho)} \int_{\R^n \setminus B(2\rho)} \frac{|u(x)-(u)_{A(\rho)}|^{\frac{n}{s}}}{|x-y|^{2n}} \dx\, \dy\\
&\quad + \int_{B(\frac 43\rho)} \int_{\R^n \setminus B(2\rho)} \frac{|(v)_{A(\rho)} -v(y)|^{\frac{n}{s}}}{|x-y|^{2n}} \dx \dy.
\end{split}
\end{equation}
Now observe that if $z \in A(\rho)$, $x \in \R^n \setminus B(2\rho)$ and $y \in B(\frac43\rho)$ then
\[
 |x-z|\le \dist(x,B(4/3 \rho)) + \dist(z, B(4/3 \rho)) \le 2 \dist(x,B(4/3 \rho)) \le 2\abs{\frac43 -|x|} \le |x-y|.
\]
Consequently,
\begin{equation}\label{eq:competitor:ts:2}
\begin{split}
 \int_{B(\frac43\rho)} \int_{\R^n \setminus B(2\rho)} \frac{|u(x)-(u)_{A(\rho)}|^{\frac{n}{s}}}{|x-y|^{2n}} \dx \dy
 &\aleq\rho^{-n} \int_{B(\frac43\rho)} \int_{A(\rho)} \int_{\R^n \setminus B(2\rho)} \frac{|u(x)-u(z)|^{\frac{n}{s}}}{|x-y|^{2n}} \dx \dz \dy\\
 &\aeq\rho^{-n} \int_{B(\frac43\rho)} \int_{A(\rho)} \int_{\R^n \setminus B(2\rho)} \frac{|u(x)-u(z)|^{\frac{n}{s}}}{|x-z|^{2n}} \dx \dz \dy\\
 &\aeq\int_{A(\rho)} \int_{\R^n \setminus B(2\rho)} \frac{|u(x)-u(z)|^{\frac{n}{s}}}{|x-z|^{2n}} \dx \dz\\
 &\leq\int_{B(2r)\setminus B(r/2)} \int_{\R^n} \frac{|u(x)-u(z)|^{\frac{n}{s}}}{|x-z|^{2n}} \dx \dz.
 \end{split}
\end{equation}
Also, integrating in $x$ we have
\begin{equation}\label{eq:competitor:ts:3}
\begin{split}
\int_{B(\frac43\rho)} \int_{\R^n \setminus B(2\rho)} \frac{|(v)_{A(\rho)} -v(y)|^{\frac{n}{s}}}{|x-y|^{2n}} \dx \dy
&\aeq \rho^{-n}\int_{B(\frac 43\rho)}|(v)_{A(\rho)} -v(y)|^{\frac{n}{s}} \dy\\
&\aleq  \int_{B(2\rho)} \int_{B(2\rho)} \frac{|v(x)-v(y)|^{\frac{n}{s}}}{|x-y|^{2n}} \dx \dy. 
\end{split}
\end{equation}

Lastly, observe that since $v = \pi_{\n} \circ w$ and $\pi_{\n}$ is Lipschitz
\begin{equation}\label{eq:competitor:ts:323}
\int_{B(2\rho)} \int_{B(2\rho)} \frac{|v(x)-v(y)|^{\frac{n}{s}}}{|x-y|^{2n}} \dx \dy \leq C(\pi_{\n})
 \int_{B(2\rho)} \int_{B(2\rho)} \frac{|w(x)-w(y)|^{\frac{n}{s}}}{|x-y|^{2n}} \dx \dy.
\end{equation}

Plugging \eqref{eq:competitor:ts:3}, \eqref{eq:competitor:ts:2}, and \eqref{eq:competitor:ts:323} into \eqref{eq:competitor:ts:1} we arrive at 
\begin{equation}\label{eq:competitor:ts:4}
\begin{split}
\int_{B(r)} \int_{\R^n} \frac{|v(x)-v(y)|^{\frac{n}{s}}}{|x-y|^{2n}} \dx \dy
&\aleq \int_{B(2r)\setminus B(r/2) } \int_{\R^n} \frac{|u(x)-u(y)|^{\frac{n}{s}}}{|x-y|^{2n}} \dx \dy\\
&\quad +\int_{B(2\rho)} \int_{B(2\rho)} \frac{|w(x)-w(y)|^{\frac{n}{s}}}{|x-y|^{2n}} \dx \dy.
\end{split}
\end{equation}
From the estimate of \Cref{la:frluckhaus}, namely from \eqref{eq:frl:west}, combined with \eqref{eq:comp:goodslice1} and \eqref{eq:comp:goodslice2} we have
\begin{equation}\label{eq:comp:frlwest}
 \begin{split}
  [w]_{W^{s,\frac ns}(B(2\rho))}^\frac ns 
  &\aleq[u]_{W^{s,\frac ns}(B(2\rho)\setminus B(\rho))}^{\frac{n}{s}} + \int_{B(r)\setminus B(r/2)}\int_{B(2r)\setminus B(r/2)} \frac{|u(x)-u(y)|^{\frac{n}{s}}}{|x-y|^{2n}} \dx \dy.
\end{split}
 \end{equation}
Plugging \eqref{eq:comp:frlwest} into \eqref{eq:competitor:ts:4} we obtain 
\begin{equation}\label{eq:competitor:ts:10}
\int_{B(r)} \int_{\R^n} \frac{|v(x)-v(y)|^{\frac{n}{s}}}{|x-y|^{2n}} \dx \dy \aleq \int_{B(2r)\setminus B(r/2) } \int_{\R^n} \frac{|u(x)-u(y)|^{\frac{n}{s}}}{|x-y|^{2n}} \dx \dy.
\end{equation}
In particular \eqref{la:comp:west} is established.

\underline{It remains to show that $u$ and $v$ are homotopic}.

Since $u=v$ outside of $B(r)$ we have in particular from \eqref{eq:competitor:ts:10} and \eqref{eq:comp:uabscont}
\[
 [u-v]_{W^{s,\frac{n}{s}}(\R^n)} \aleq \eps.
\]
From Poincar\'e inequality we obtain 
\[
 \|u-v\|_{L^1(\R^n)} +[u-v]_{W^{s,\frac{n}{s}}(\R^n)} \aleq \eps.
\]
Choosing $\eps$ small enough we can conclude that for $\eps(u)$ from \Cref{la:samehomotopy} we have 
\[
 \|u-v\|_{L^1(\R^n)} +[u-v]_{W^{s,\frac{n}{s}}(\R^n)} < \frac{\eps(u)}{4}.
\]
In view of \Cref{la:samehomotopy}, we have that $u$ and $v$ are homotopic in the sense of \Cref{def:homotopy}. This finishes the proof of \Cref{la:competitor}.
\end{proof}

\begin{proof}[Proof of \Cref{th:ininitreg}]
H\"older continuity is a local property, and since we are not interested in any sort of estimate at this point, it suffices to prove H\"older continuity around any point $x_0 \in B(R)$. Without loss of generality we may assume that $x_0=0$.

Let $\eps$ be from \Cref{la:competitor}. 

If $t > s$, by Sobolev embedding any map $u \in W^{t,\frac{n}{s}}$ is uniformly continuous in $B(R)$, so there exists $r_0 > 0$ such $B(10 r_0) \subset B(R)$ and $|u(x)-u(y)|<\eps$ for all $x,y \in B(10r_0)$, that is \eqref{la:comp:ucont} is satisfied for any $r < r_0$.

If $t=s$, by absolute continuity of the integral, there exists $r_0 > 0$ such that \eqref{eq:comp:uabscont} is satisfied for any $r < r_0$.

For any $r < r_0$ and any ball $B(10r) \subset B(R)$ we apply \Cref{la:competitor} and obtain a competitor $v$ for the minimizer $u$, that is
\[
 E_{t,\frac{n}{s}}(u) \leq E_{t,\frac{n}{s}}(v).
\]
Since $u \equiv v$ in $\R^n \setminus B(r)$ this implies
\[
\int\int_{(\R^n)^2 \setminus (B(r)^c)^2} \frac{|u(x)-u(y)|^{\frac{n}{s}}}{|x-y|^{n+t\frac{n}{s}}} \dx \dy \\
 \leq \int\int_{(\R^n)^2 \setminus (B(r)^c)^2} \frac{|v(x)-v(y)|^{\frac{n}{s}}}{|x-y|^{n+t\frac{n}{s}}} \dx \dy. 
\]
Here we denote 
\begin{equation}\label{eq:forK}
\begin{split}
 (\R^n)^2 \setminus (B(r)^c)^2 =& \brac{\R^n \times \R^n} \setminus \brac{ \R^n \setminus B(r) \times \R^n \setminus B(r)}\\
 =& \brac{B(r) \times B(r)} \cup \brac{\R^n \setminus B(r) \times B(r)} \cup \brac{B(r)\times \R^n \setminus B(r)}\\
 =& \brac{\R^n \times B(r)} \cup \brac{B(r)\times \R^n \setminus B(r)}.
 \end{split}
\end{equation}
In particular we have 
\[
\begin{split}
 \int_{B(r)}\int_{\R^n} \frac{|u(x)-u(y)|^{\frac{n}{s}}}{|x-y|^{n+t\frac{n}{s}}} \dx \dy 
 &\leq \int\int_{(\R^n)^2 \setminus (B(r)^c)^2} \frac{|v(x)-v(y)|^{\frac{n}{s}}}{|x-y|^{n+t\frac{n}{s}}} \dx \dy \\
 &\aleq\int_{B(r)}\int_{\R^n} \frac{|v(x)-v(y)|^{\frac{n}{s}}}{|x-y|^{n+t\frac{n}{s}}} \dx \dy.
 \end{split}
\]
Applying \eqref{la:comp:west} we find
\begin{equation}\label{eq:initreg:prehole}
 \int_{B(r)}\int_{\R^n} \frac{|u(x)-u(y)|^{\frac{n}{s}}}{|x-y|^{n+t\frac{n}{s}}} \dx \dy 
 \leq C\int_{B(2r)\setminus B(r/2)} \int_{\R^n}\frac{|u(x)-u(y)|^{\frac{n}{s}}}{|x-y|^{n+t\frac{n}{s}}} \dx \dy,
 \end{equation}
 where $C$ is a constant depending only on $s$ and $s_1$, not on $t \in [s,s_1]$.
 
 We now perform the hole filling trick by adding 
 \[C\int_{B(r/2)} \int_{\R^n}\frac{|u(x)-u(y)|^{\frac{n}{s}}}{|x-y|^{n+t\frac{n}{s}}} \dx \dy\]
 to both sides of \eqref{eq:initreg:prehole}, and find for $\tau \coloneqq \frac{C}{C+1} \leq 1$
\[
 \int_{B(r/2)}\int_{\R^n} \frac{|u(x)-u(y)|^{\frac{n}{s}}}{|x-y|^{n+t\frac{n}{s}}} \dx \dy 
 \leq \tau \int_{B(2r)} \int_{\R^n}\frac{|u(x)-u(y)|^{\frac{n}{s}}}{|x-y|^{n+t\frac{n}{s}}} \dx \dy.
 \] 
This inequality holds for any $r < r_0$. Applying it to $r = 4^{-k} r_0$, we have 
\[
 \int_{B(4^{-k-1}r_0)}\int_{\R^n} \frac{|u(x)-u(y)|^{\frac{n}{s}}}{|x-y|^{n+t\frac{n}{s}}} \dx \dy 
 \leq \tau^k \int_{B(r_0)} \int_{\R^n}\frac{|u(x)-u(y)|^{\frac{n}{s}}}{|x-y|^{n+t\frac{n}{s}}} \dx \dy.
 \] 
Setting $\beta \coloneqq \log_4 \tau$, this implies
\[
 \int_{B(4^{-k-1}r_0)}\int_{\R^n} \frac{|u(x)-u(y)|^{\frac{n}{s}}}{|x-y|^{n+t\frac{n}{s}}} \dx \dy 
 \aleq 4^{-k \beta} \int_{B(r_0)} \int_{\R^n}\frac{|u(x)-u(y)|^{\frac{n}{s}}}{|x-y|^{n+t\frac{n}{s}}} \dx \dy.
 \] 
Since for any $r \in (0,r_0)$ we can find $k \in \N_0$ such that $\frac{r}{r_0} \aeq 4^{-k}$, we conclude for any $r \in (0,r_0)$
\[
 \int_{B(r)}\int_{\R^n} \frac{|u(x)-u(y)|^{\frac{n}{s}}}{|x-y|^{n+t\frac{n}{s}}} \dx \dy 
 \aleq \brac{\frac{r}{r_0}}^\beta \int_{B(r_0)} \int_{\R^n}\frac{|u(x)-u(y)|^{\frac{n}{s}}}{|x-y|^{n+t\frac{n}{s}}} \dx \dy.
 \] 
Observe that $\beta$ only depends on $\tau$ (and thus on $s$ and $s_1$, but not on $t \in [s,s_1]$).

Since $s \le t$ we have in particular for any $r < r_0$,
\[
 \int_{B(r)}\int_{B(r)} \frac{|u(x)-u(y)|^{\frac{n}{s}}}{|x-y|^{2n}} \dx \dy  \aleq r^{(t-s)\frac{n}{s}}\brac{\frac{r}{r_0}}^\beta\, \int_{B(r_0)}\int_{\R^n} \frac{|u(x)-u(y)|^{\frac{n}{s}}}{|x-y|^{n+t\frac{n}{s}}} \dx \dy. 
\]
This in turn readily implies
\[
 \mvint_{B(r)} |u-(u)_{B(r)}| \leq C(r_0,u)\, r^{(t-s)+\beta\frac{s}{n}}.
\]
If $t \in [s,s_1]$ then
\[
 \mvint_{B(r)} |u-(u)_{B(r)}| \leq C(r_0,u)\, r^{\beta\frac{s}{n}}.
\]
So $u$ belongs to a Campanato space in $B(r)$ and this implies that $u \in C^{\beta\frac{s}{n}}_{loc}(B(r))$, see \cite[III, Theorem 1.2]{Giaquinta}. Setting $\alpha \coloneqq \beta\frac{s}{n} > 0$ we conclude.
\end{proof}

\subsection{Step 2: Higher differentiability}
We show that the Euler--Lagrange equation for harmonic maps combined with the H\"older continuity from \Cref{th:ininitreg} implies higher differentiability. The following theorem is strongly inspired by the techniques for the fractional $p$-Laplacian due to Brasco and Lindgren \cite{BL17,BLS18}. 
\begin{theorem}\label{th:BLdiff}
Let $p \geq 2$. Assume that $u \in W^{s,p}(\R^n)$, $f \in L^1(\R^n)$ solve
\begin{equation}\label{eq:BLdiff:pde}
 \int_{\R^n}\int_{\R^n} \frac{|u(x)-u(y)|^{p-2}(u(x)-u(y))(\varphi(x)-\varphi(y))}{|x-y|^{n+sp}} \dx \dy = \int_{\R^n} f \varphi \quad \forall \varphi \in C_c^\infty(B(R)).
\end{equation}
If $u \in L^\infty \cap C^\alpha(B(R))$ for some $\alpha > 0$ then $u \in W^{s+\gamma,p}(B(R/2))$ for any $\gamma < \min\{\frac{\alpha}{p},1\}$.
\end{theorem}
\begin{proof}
Without loss of generality we can assume that $\frac{\alpha}{p} < 1$ because if $u \in C^\alpha(B)$ then $u \in C^\beta(B)$ for any $\beta < \alpha$, and we could simply work with $\beta$ instead of $\alpha$.

While this is not their statement, the proof of \Cref{th:BLdiff} is strongly motivated by the argument in \cite{BL17}, in particular we take inspiration in \cite[Proposition 3.1]{BL17}. This is why we also follow the notation in \cite{BL17}. 

For $h\in \R^n$ and $f\colon \R^n \to \R^M$ we set $f_h(x) \coloneqq f(x+h)$, $\delta_h f(x) \coloneqq f(x+h)-f(x)$. Also for $v \in \R^M$ set
\[
 J_p(v) \coloneqq |v|^{p-2} v.
\]
Set $\delta \coloneqq \frac{1}{100} R$.
Let $|h| < \delta$, and let $\varphi \in C_c^\infty(B(R-2\delta))$. Then $\varphi_h \in C_c^\infty(B(R))$ and thus we have by substitution and with the help of \eqref{eq:BLdiff:pde}
\[
\begin{split}
\int_{\R^n}\int_{\R^n} &\frac{J_{p}(u_h(x)-u_h(y)) (\varphi(x)-\varphi(y))}{|x-y|^{n+sp}} \dx \dy\\
 &=\int_{\R^n}\int_{\R^n} \frac{J_{p}(u(x)-u(y)) (\varphi_h(x)-\varphi_h(y))}{|x-y|^{n+sp}} \dx \dy \\
&= \int_{\R^n} f \varphi_h.
 \end{split}
\]
Subtracting \eqref{eq:BLdiff:pde} from the above equality, we find for any $\varphi \in C_c^\infty(B(R-2\delta))$
\[
\int_{\R^n}\int_{\R^n} \frac{\brac{J_{p}(u_h(x)-u_h(y))-J_{p}(u(x)-u(y))} (\varphi(x)-\varphi(y))}{|x-y|^{n+sp}} \dx \dy 
= \int_{\R^n} f (\varphi_h-\varphi).
\]
Let $\eta \in C_c^\infty(B(R-20\delta),[0,1])$ with $ \eta \equiv 1$ in $B(R-30\delta)$ and $|\nabla \eta|\aleq \frac 1 \delta$. 
By a density argument we may choose
\[
 \varphi \coloneqq \eta \delta_h u
\]
and obtain
\[
\begin{split}
&\int_{\R^n}\int_{\R^n} \frac{\brac{J_{p}(u_h(x)-u_h(y))-J_{p}(u(x)-u(y))} (\eta(x) \delta_h u(x)-\eta(y) \delta_h u(y))}{|x-y|^{n+sp}} \dx \dy \\
 &= \int_{\R^n} f \delta_h (\eta \delta_h u).
 \end{split}
\]
By assumption $u \in C^\alpha(B(R))$, so
\[
\begin{split}
&\int_{\R^n}\int_{\R^n} \frac{\brac{J_{p}(u_h(x)-u_h(y))-J_{p}(u(x)-u(y))} (\eta(x) \delta_h u(x)-\eta(y) \delta_h u(y))}{|x-y|^{n+sp}} \dx \dy\\
 &\aleq [u]_{C^\alpha(B(R))} \|f\|_{L^1(\R^n)}\, |h|^\alpha.
 \end{split}
\]
Similarly as in \cite[p.320]{BL17}, by analyzing the support of $\eta$ and using the symmetry of the integral, we split the integral on the left-hand side into two pieces:  
\[
 \int_{\R^n}\int_{\R^n} \frac{\brac{J_{p}(u_h(x)-u_h(y))-J_{p}(u(x)-u(y))} (\eta(x) \delta_h u(x)-\eta(y) \delta_h u(y))}{|x-y|^{n+sp}} \dx \dy \geq \mathcal{I}_1 - 2 \mathcal{I}_2,
\]
where
\[
 \mathcal{I}_1 = \int_{B(R)}\int_{B(R)} \frac{\brac{J_{p}(u_h(x)-u_h(y))-J_{p}(u(x)-u(y))} (\eta(x) \delta_h u(x)-\eta(y) \delta_h u(y))}{|x-y|^{n+sp}} \dx \dy
\]
and
\[
 \mathcal{I}_2 = \left |\int_{\R^n\setminus B(R)}\int_{B(R-20\delta)} \frac{\brac{J_{p}(u_h(x)-u_h(y))-J_{p}(u(x)-u(y))} \eta(x) \delta_h u(x)}{|x-y|^{n+sp}} \dx \dy \right |.
\]
We first estimate $\mathcal{I}_2$. Observe that in the integrand $|x-y| \ageq \delta$, so there is no singularity.
\[
\begin{split}
 \mathcal{I}_2 &\le |h|^\alpha [u]_{C^\alpha} \int_{\R^n\setminus B(R)}\int_{B(R-20\delta)} \frac{|u_h(x)-u_h(y)|^{p-1}+|u(x)-u(y)|^{p-1}}{|x-y|^{n+sp}} \dx \dy \\
  &\leq 2|h|^\alpha [u]_{C^\alpha} \int_{\R^n\setminus B(R-5\delta)}\int_{B(R-15\delta)} \frac{|u(x)-u(y)|^{p-1}}{|x-y|^{n+sp}} \dx \dy \\
  &\leq2|h|^\alpha [u]_{C^\alpha} [u]_{W^{s,p}(\R^n)}^{p-1}\, \brac{\int_{\R^n\setminus B(R-5\delta)}\int_{B(R-15\delta)} \frac{1}{|x-y|^{n+sp}} \dx \dy }^{\frac{1}{p}}\\
  &\aeq 2|h|^\alpha [u]_{C^\alpha} [u]_{W^{s,p}(\R^n)}^{p-1}\, \delta^{-s}\, R^{\frac{n}{p}}.
 \end{split}
\]
We now estimate $\mathcal{I}_1$.
\[
\begin{split}
 \mathcal{I}_1 &\geq \int_{B(R)}\int_{B(R)} \eta(x) \frac{\brac{J_{p}(u_h(x)-u_h(y))-J_{p}(u(x)-u(y))} (\delta_h u(x)- \delta_h u(y))}{|x-y|^{n+sp}} \dx \dy\\
 &\quad -\int_{B(R)}\int_{B(R)} |\delta_h u(y)| \frac{\abs{J_{p}(u_h(x)-u_h(y))-J_{p}(u(x)-u(y))} |\eta(x)-\eta(y)|}{|x-y|^{n+sp}} \dx \dy\\
&= \int_{B(R)}\int_{B(R)} \eta(x) \frac{\brac{J_{p}(u_h(x)-u_h(y))-J_{p}(u(x)-u(y))} (u_h(x)-u_h(y) - (u(x)- u(y)))}{|x-y|^{n+sp}} \dx \dy\\
 &\quad -\int_{B(R)}\int_{B(R)} |\delta_h u(y)| \frac{\abs{J_{p}(u_h(x)-u_h(y))-J_{p}(u(x)-u(y))} |\eta(x)-\eta(y)|}{|x-y|^{n+sp}} \dx \dy.
 \end{split}
\]
For the first term we use the famous p-Laplace inequality
  which holds for $p \geq 2$ and $v,w \in \R^M$
\[
 (J_p(v)-J_p(w))(v-w) \ageq |v-w|^p,
\]
see \cite[Section 12 (I)]{L06} or \cite[Lemma B.3, (B.4)]{BL17}.
For the second term we use Lipschitz continuity of $\eta$. Then we have (recall $\eta \geq 0$ everywhere)
\[
\begin{split}
 \mathcal{I}_1 &\geq 
\int_{B(R-30\delta)}\int_{B(R-30\delta)} \frac{\abs{u_h(x)-u_h(y)-(u(x)-u(y))}^p}{|x-y|^{n+sp}} \dx \dy\\
 &\quad -C|h|^\alpha [u]_{C^\alpha} \delta^{-1}\int_{B(R)}\int_{B(R)}  \frac{|u_h(x)-u_h(y)|^{p-1}+|u(x)-u(y)|^{p-1}}{|x-y|^{n+sp-1}} \dx \dy\\
 &\ge[\delta_h u]_{W^{s,p}(B(R-30\delta))}^p - 2|h|^\alpha [u]_{C^\alpha} \delta^{-1} \int_{B(R+\delta)}\int_{B(R+\delta)}  \frac{|u(x)-u(y)|^{p-1}}{|x-y|^{n+s(p-1)-(1-s)}} \dx \dy\\
 &\geq [\delta_h u]_{W^{s,p}(B(R-30\delta))}^p - 2|h|^\alpha [u]_{C^\alpha} \delta^{-1} [u]_{W^{s,p}(B(R+\delta))}^{p-1} 
 \brac{\int_{B(R+\delta)}\int_{B(R+\delta)}  \frac{1}{|x-y|^{n-(1-s)p}} \dx \dy}^{\frac{1}{p}}\\
 &\aeq [\delta_h u]_{W^{s,p}(B(R-30\delta))}^p - 2|h|^\alpha [u]_{C^\alpha} \delta^{-1} [u]_{W^{s,p}(B(R+\delta))}^{p-1}  R^{\frac{n}{p}+(1-s)}.
 \end{split}
\]
That is, we have shown that for any $|h| \leq \delta$,
\[
 \left[|h|^{-\frac{\alpha}{p}} \delta_h u\right]_{W^{s,p}(B(R-30\delta))}^p \leq C(u,\delta,R,f). 
\]
From \Cref{la:besovest} we obtain that $u \in W^{s+\gamma,p}$ for any $\gamma < \frac{\alpha}{p}$.
\end{proof}

Above we used the following difference quotient estimate for fractional Sobolev spaces. For $W^{1,p}$ a statement like the one below is well-known, see \cite[\textsection 5.8.2, Theorem 3]{E10}, for Sobolev spaces it can be argued via the characterization of Besov--Nikol'skii spaces $B^{\gamma}_{p,\infty}$.
\begin{lemma}\label{la:besovest}
Let $s \in (0,1)$, $p \in (1,\infty)$, $\alpha > 0$ and $\gamma < \min\{s+\alpha,1\}$. Assume that $u \in W^{s,p}(\Omega)$ and that for some $\Omega' \subset \Omega$ we have 
\[
 \sup_{|h| < \dist(\Omega',\partial \Omega)} [|h|^{-\alpha}\delta_h u]_{W^{s,p}(\Omega')} < \infty,
\]
then $u \in W^{\gamma,p}_{loc}(\Omega')$.
\end{lemma}
\begin{proof}
In view of \cite[Lemma 3.3]{BL17} we find that $u \in B^{s+\alpha}_{p,\infty}(\Omega')$, where $B^{s}_{p,\infty}$ denotes the Besov--Nikol'skii space.

Combining \cite[(3.20)]{BL17} with \cite[Proposition~2.7]{BL17} we obtain that $u \in W^{s+\gamma,p}_{loc}(\Omega')$.
\end{proof}

\begin{remark}
Let us conclude this subsection by remarking that in the proof of \Cref{th:reghomo} it seems likely that one could construct a different argument for higher differentiability which is based on the fractional Gehring's lemma, developed in \cite{KMS14}.
\end{remark}

\subsection{Step 3: A priori estimates}
\begin{theorem}[A priori estimates]\label{th:apriorireg}
Let $s \in (0,1)$ (if $n=1$ let $s \leq \frac{1}{2}$). There exists an $\bar{s} \in (s,1)$ such that for any $s_0 \in (s,\bar{s})$ there exists an $s_1 \in (s,s_0)$ such that the following holds\footnote{The relation between those numbers is {$0<s<s_1<s_0<\bar s<1$}.}.

Let $\Sigma$ be an $n$-dimensional compact manifold without boundary, and $\n \subset \R^M$ a compact manifold without boundary.

There exists an $\eps=\eps(\n,\Sigma,s,s_0,\bar{s}) > 0$ depending on the choices above such that the following holds for any $t \in [s,s_1]$.

Assume that $u \in W^{s,\frac{n}{s}}(\Sigma,\n)$ and for a geodesic ball $B(R) \subset \Sigma$ 
\begin{itemize}
\item $u \in W^{s_0,\frac{n}{s}}(B(R))$;
 \item $u$ is a critical $W^{t,\frac{n}{s}}$-harmonic map in $B(R)$ for some $t \in [s,s_1]$;
 \item $[u]_{W^{s,\frac{n}{s}}(B(R))} < \eps$.
\end{itemize}
Then we have the estimate
\begin{equation}\label{eq:th:apr:claim}
 [u]_{W^{s_0,\frac{n}{s}}(B(R/2))} \leq C\, R^{-(s_0-s)} [u]_{W^{s,\frac{n}{s}}(B(R))}^{\frac{s}{n}} \brac{[u]_{W^{s,\frac{n}{s}}(\Sigma)}^{1-\frac{s}{n}}+[u]_{W^{s,\frac{n}{s}}(\Sigma)}}.
\end{equation}
\end{theorem}

The proof of \Cref{th:apriorireg} is based on estimates of the Euler--Lagrange equations, and the stability estimates for the fractional $p$-Laplacian in \cite{S16}.

We will also need the following iteration lemma, see \cite[Chapter V, Lemma 3.1]{Giaquinta}.
\begin{lemma}[Iteration lemma]\label{la:iteration}
 Let $0 < a < b < \infty$ and $f \colon [a,b] \to [0,\infty)$ be a bounded function. Suppose that there are constants $\theta \in [0,1)$, $K_1,\,K_2,\,\alpha > 0$ such that 
\begin{equation}\label{eq:iterationinequality}
 f(r) \le \theta f(\rho) + \frac{K_1}{(\rho-r)^\alpha}+ K_2 
 \quad \text{for all } a \le r < \rho \le b.
\end{equation}
Then we obtain the bound 
\[
 f(r)\le C \brac{\frac{K_1}{(b-r)^\alpha} + K_2} 
 \quad \text{for all } a \le r \le b 
\]
for a constant $C=C(\theta,\alpha) > 0$. 
\end{lemma}

\begin{proof}[Proof of \Cref{th:apriorireg}]
For simplicity of notation 
we assume\footnote{Since $\Sigma$ is compact, the manifold has a bounded geometry 
so it is locally comparable to $\R^n$ and by an extension theorem this assumption changes mainly the notation.} that $\Sigma = \R^n$.

Below we will establish the following estimate for any $r,\rho$ with $\frac{R}{4} < r < \rho < \frac{3}{4} R$, 
\begin{equation}\label{eq:apr:goalineq}
\begin{split}
 [u]_{W^{s_0,\frac{n}{s}}(B(r))}^{\frac{n}{s}} &\leq \frac{1}{2}\, [u]_{W^{s_0,\frac{n}{s}}(B(\rho))}^{\frac{n}{s}} \\
 &\quad +C\, (\rho-r)^{-(s_0-s)\frac{n}{s}} \brac{\frac{R}{\rho-r}}^{\frac{n}{s}}\, 
\brac{[u]_{W^{s,\frac{n}{s}}(\R^n)}^{\frac{n}{s}-1}+[u]_{W^{s,\frac{n}{s}}(\R^n)}^{\frac{n}{s}}} [u]_{W^{s,\frac{n}{s}}(B(R))}.
\end{split}
\end{equation}

Once \eqref{eq:apr:goalineq} is established, we apply \Cref{la:iteration} to $f(r) \coloneqq [u]_{W^{s_0,\frac{n}{s}}(B(r))}^{\frac{n}{s}}$, which gives
\[
 [u]_{W^{s_0,\frac{n}{s}}(B(R/2))}^{\frac{n}{s}} \le C\, R^{-(s_0-s)\frac{n}{s}} \, 
\brac{[u]_{W^{s,\frac{n}{s}}(\R^n)}^{\frac{n}{s}-1}+[u]_{W^{s,\frac{n}{s}}(\R^n)}^{\frac{n}{s}}} [u]_{W^{s,\frac{n}{s}}(B(R))}.
\]
This readily implies \eqref{eq:th:apr:claim}.

We now need to establish \eqref{eq:apr:goalineq}. From now on we fix some $r,\rho$ such that $\frac{R}{4} < r < \rho < \frac{3}{4} R$. We denote $\delta \coloneqq \frac{\rho-r}{100} \in (0,R)$. 

As a $W^{t,\frac ns}$-harmonic map $u$ solves the following inequality for any $\varphi \in C_c^\infty(B(R))$, cf. \cite[Lemma 5.1]{MS18}.
\begin{equation}\label{eq:ap:pde}
\begin{split}
& \abs{\int_{\R^n}\int_{\R^n} \frac{|u(x)-u(y)|^{\frac{n}{s}-2} (u(x)-u(y)) (\varphi(x)-\varphi(y))}{|x-y|^{n+t\frac{n}{s}}} \dif y \dif x}\\
&\aleq 
\int_{\R^n}\int_{\R^n} |\varphi(x)| \frac{|u(x)-u(y)|^{\frac{n}{s}}}{|x-y|^{n+t\frac{n}{s}}}  \dy \dx + \int_{\R^n}\int_{\R^n} \frac{|u(x)-u(y)|^{\frac{n}{s}}\, |\varphi(x)-\varphi(y)|}{|x-y|^{n+t\frac{n}{s}}}\dif y \dif x \\
&\leq 3\int_{\R^n}\int_{\R^n} |\varphi(x)| \frac{|u(x)-u(y)|^{\frac{n}{s}}}{|x-y|^{n+t\frac{n}{s}}}  \dy \dx.
\end{split}
\end{equation}

Pick $\bar{\eta} \in C_c^\infty(B(\rho-10\delta))$ such that $\bar{\eta} \equiv 1$ in $B(r+10\delta)$, and $|\nabla \bar{\eta}| \aleq \frac{1}{\delta}$. Set $\eta \coloneqq \bar{\eta}^2$. Then $|\nabla \eta|+|\nabla \sqrt{\eta}| \aleq \frac{1}{\delta}$. We define the test function \[\varphi \coloneqq \eta (u-(u)_{B(R)}).\]

We collect the main \underline{estimates of $\varphi$:}

Firstly,
\begin{equation}\label{eq:apr:varphiest}
 [\varphi]_{W^{s_0,\frac{n}{s}}(\R^n)} \aleq [u]_{W^{s_0,\frac{n}{s}}(B(\rho))} + \frac{R^{1-s_0+s}}{\delta} [u]_{W^{s,\frac{n}{s}}(B(R))}.
\end{equation}
Indeed,
\[
\begin{split}
 [\varphi]_{W^{s_0,\frac{n}{s}}(\R^n)}^{\frac{n}{s}}
 &\aleq \int_{B(\rho-5\delta)}\int_{B(\rho-5\delta)} \frac{|\varphi(x)-\varphi(y)|^{\frac{n}{s}}}{|x-y|^{n+s_0 \frac{n}{s}}}\dx \dy +\int_{B(\rho-10\delta)}\int_{B(\rho-5\delta)^c} \frac{|\varphi(y)|^{\frac{n}{s}}}{|x-y|^{n+s_0 \frac{n}{s}}} \dx \dy\\
 &\aleq \int_{B(\rho-5\delta)}\int_{B(\rho-5\delta)} \frac{|u(x)-u(y)|^{\frac{n}{s}}}{|x-y|^{n+s_0 \frac{n}{s}}}\dx \dy\\
 &\quad +\int_{B(\rho-5\delta)}\int_{B(\rho-5\delta)} \frac{|\eta(x)-\eta(y)|^{\frac{n}{s}}\, |u(x)-(u)_{B(R)}|^{\frac{n}{s}}}{|x-y|^{n+s_0 \frac{n}{s}}}\dx \dy\\
 &\quad +\int_{B(\rho-10\delta)}\int_{B(\rho-5\delta)^c} \frac{|u(y)-(u)_{B(R)}|^{\frac{n}{s}}}{|x-y|^{n+s_0 \frac{n}{s}}} \dx \dy.
 \end{split}
\]
That is
\[
\begin{split}
 [\varphi]_{W^{s_0,\frac{n}{s}}(\R^n)}^{\frac{n}{s}}
 &\aleq \int_{B(\rho)}\int_{B(\rho)} \frac{|u(x)-u(y)|^{\frac{n}{s}}}{|x-y|^{n+s_0 \frac{n}{s}}}\dx \dy +\frac{\rho^{(1-s_0)\frac{n}{s}}}{\delta^{\frac{n}{s}}}
 \int_{B(R)} |u(x)-(u)_{B(R)}|^{\frac{n}{s}} \dx\\
 &\quad +\frac{1}{\delta^{s_0\frac{n}{s}}}\int_{B(R)} |u(y)-(u)_{B(R)}|^{\frac{n}{s}}\dy\\
 &\aleq [u]_{W^{s_0,\frac{n}{s}}(B(\rho))}^{\frac{n}{s}} + \frac{R^{(1-s_0+s)\frac{n}{s}}}{\delta^{\frac{n}{s}}} [u]_{W^{s,\frac{n}{s}}(B(R))}^{\frac{n}{s}}.
 \end{split}
\]
This establishes \eqref{eq:apr:varphiest}.

Recall, that the fractional $\tilde s$-Laplacian for $\tilde{s} \in (0,1)$ is defined as 
\[
 \laps{\tilde{s}} \varphi (x) \coloneqq c \int_{\R^n} \frac{\varphi(x)-\varphi(y)}{|x-y|^{n+\tilde{s}}} \dif y,
\]
or --- equivalently --- via the Fourier transform $\mathcal{F}(\laps{\tilde{s}} \varphi)(\xi) = c |\xi|^s \mathcal{F}(\varphi)(\xi)$, cf. \cite{DNPV12,G19}.

With this notation, for any $0\leq\tilde{s} < s$
\begin{equation}\label{eq:apr:varphies:2}
 [\laps{\tilde{s}} \varphi]_{W^{s-\tilde{s},\frac{n}{s}}(\R^n)} \aleq \frac{R}{\delta}[u]_{W^{s,\frac{n}{s}}(B(R))}.
\end{equation}
Indeed, from the theory of Triebel--Lizorkin spaces $F^{s}_{p,p} \aeq W^{s,p}$, \cite{RunstSickel}, we have 
\[
 [\laps{\tilde{s}} \varphi]_{W^{s-\tilde{s},\frac{n}{s}}(\R^n)} \aeq [\varphi]_{W^{s,\frac{n}{s}}(\R^n)}.
\]
Now the estimate \eqref{eq:apr:varphies:2} follows as \eqref{eq:apr:varphiest}.

Also, for any $0 \le \tilde{s} \leq s$ and any $\gamma \in (0,1)$
\begin{equation}\label{eq:apr:varphies:3}
 \brac{\int_{\R^n \setminus B(R)} \int_{\R^n} \frac{|\laps{\tilde{s}} \varphi(x)-\laps{\tilde{s}} \varphi(y)|^{\frac{n}{s}}}{|x-y|^{n+\gamma \frac{n}{s}}} \dx \dy}^{\frac{s}{n}}
 \aleq R^{-\gamma-\tilde{s}+s} [u]_{W^{s,\frac{n}{s}}(B(R))}.
\end{equation}
Indeed,
\[
\begin{split}
&\int_{\R^n \setminus B(R)} \int_{\R^n} \frac{|\laps{\tilde{s}} \varphi(x)-\laps{\tilde{s}} \varphi(y)|^{\frac{n}{s}}}{|x-y|^{n+\gamma \frac{n}{s}}} \dx \dy\\
&\aleq\int_{\R^n \setminus B(R)} \int_{B(y,\frac{1}{2}R)} \frac{|\laps{\tilde{s}} \varphi(x)-\laps{\tilde{s}} \varphi(y)|^{\frac{n}{s}}}{|x-y|^{n+\gamma \frac{n}{s}}} \dx \dy\\
&\quad +\int_{\R^n \setminus B(R)} \int_{\R^n \setminus B(y,\frac{1}{2}R)} \frac{|\laps{\tilde{s}} \varphi(x)-\laps{\tilde{s}} \varphi(y)|^{\frac{n}{s}}}{|x-y|^{n+\gamma \frac{n}{s}}} \dx \dy\\
\end{split}
\]
Now we employ the estimate 
\[
 |f(x)-f(y)| \aleq |x-y|\, \brac{\mathcal{M}\nabla f(x)+\mathcal{M}\nabla f(y)},
\]
where $\mathcal{M}$ is the Hardy-Littlewood maximal function, cf. \cite{BH93,H96}. Then,
\[
\begin{split}
&\int_{\R^n \setminus B(R)} \int_{\R^n} \frac{|\laps{\tilde{s}} \varphi(x)-\laps{\tilde{s}} \varphi(y)|^{\frac{n}{s}}}{|x-y|^{n+\gamma \frac{n}{s}}} \dx \dy\\
&\aleq\int_{\R^n \setminus B(R)} |\mathcal{M} \nabla \laps{\tilde{s}} \varphi(x)|^{\frac{n}{s}} \int_{|x-y| \aleq R} \frac{1}{|x-y|^{n+(\gamma -1)\frac{n}{s}}} \dy \dx\\
&\quad +\int_{\R^n \setminus B(R)} |\laps{\tilde{s}} \varphi(x)|^{\frac{n}{s}} \int_{|x-y| \ageq R} \frac{1}{|x-y|^{n+\gamma \frac{n}{s}}} \dx \dy\\
&\aleq R^{(1-\gamma)\frac{n}{s}}\, \|\varphi\|_{L^1(\R^n)}^{\frac{n}{s}}\, \int_{\R^n \setminus B(R)} |x|^{\brac{-n-1-\tilde{s}} \frac{n}{s}} \dx
+R^{-\gamma \frac{n}{s}}\, \|\varphi\|_{L^1(\R^n)}^{\frac{n}{s}}\,  \int_{\R^n \setminus B(R)} |x|^{(-n-\tilde{s})\frac{n}{s}}  \dx\\
&\aleq R^{(-\gamma-n-\tilde{s}+s)\frac{n}{s}}\, \|\varphi\|_{L^1(\R^n)}^{\frac{n}{s}}\\
&\aleq R^{(-\gamma-n-\tilde{s}+s)\frac{n}{s}}\, R^{n \frac{n}{s}}[u]_{W^{s,\frac{n}{s}}(B(R))}^\frac ns\\  
&= R^{(-\gamma-\tilde{s}+s)\frac{n}{s}} [u]_{W^{s,\frac{n}{s}}(B(R))}^{\frac{n}{s}}.  
\end{split} 
\]
Similarly, for any $0 \le \tilde{s} \leq s$ and any $\gamma \in (0,1)$
\begin{equation}\label{eq:apr:varphies:4}
 \brac{\int_{\R^n \setminus B(\rho-4\delta)} \int_{\R^n} \frac{|\laps{\tilde{s}} \varphi(x)-\laps{\tilde{s}} \varphi(y)|^{\frac{n}{s}}}{|x-y|^{n+\gamma \frac{n}{s}}} \dx \dy}^{\frac{s}{n}}
 \aleq \delta^{-(\gamma+\tilde{s})} R^s [u]_{W^{s,\frac{n}{s}}(B(R))}.
\end{equation}
Also for any $\tilde{\eta} \in C_c^\infty(B(\rho),[0,1])$ with $\tilde{\eta} \equiv 1$ in a $\delta$-neighborhood of $\supp \varphi$ and with $|\nabla \tilde{\eta}| \aleq \delta^{-1}$ we have for any $\gamma \in (0,1)$ and $\tilde{s} \in [0,1)$ such that $\gamma+\tilde{s} <1$,
\begin{equation}\label{eq:apr:varphies:5}
 [(1-\tilde{\eta})\laps{\tilde{s}} \varphi]_{W^{\gamma,\frac{n}{s}}(\R^n)} \aleq \delta^{-(\gamma+\tilde{s})} R^s [u]_{W^{s,\frac{n}{s}}(B(R))}.
\end{equation}
Indeed, observe that by the disjoint support of $1-\tilde \eta$ and $\varphi$ we have 
\[
 (1-\tilde{\eta}) \laps{\tilde{s}} \varphi(x) \approx \int_{|x-z| \ageq \delta} |x-z|^{-n-\tilde{s}}\, \varphi(z)  \dif z.
\]
In particular we have from Young's convolution inequality
\[
 \|(1-\tilde{\eta}) \laps{\tilde{s}} \varphi\|_{L^p(\R^n)} \aleq \delta^{-\tilde s} \|\varphi\|_{L^p(\R^n)}.
\]
In a similar way,
\[
 \abs{\nabla \brac{(1-\tilde{\eta}) \laps{\tilde{s}} \varphi}(x)} \aleq \abs{\nabla \tilde{\eta} \laps{\tilde{s}} \varphi(x)} +  \int_{|x-z| \ageq \delta} |x-z|^{-n-\tilde{s}-1}\, |\varphi(z)|  \dif z,
\]
so that 
\[
 \norm{\nabla \brac{(1-\tilde{\eta}) \laps{\tilde{s}} \varphi}}_{L^p(\R^n)} \aleq \delta^{-\tilde s-1} \|\varphi\|_{L^p(\R^n)}.
\]
From interpolation, \cite{Tartar}, we then have
\[
 [(1-\tilde{\eta}) \laps{\tilde{s}} \varphi]_{W^{\gamma,p}(\R^n)} \aleq \delta^{-\tilde s -\gamma} \|\varphi\|_{L^p(\R^n)}.
\]
Applying Poincar\'e inequality for $p=\frac ns$, this leads to
\[
 [(1-\tilde{\eta}) \laps{\tilde{s}} \varphi]_{W^{\gamma,\frac{n}{s}}(\R^n)} \aleq \delta^{-\tilde s -\gamma} R^s [u]_{W^{s,\frac{n}{s}}(B(R))}.
\]
\eqref{eq:apr:varphies:5} is established.

Now we begin to \underline{estimate $u$}: 

For $\tilde{\varphi} \coloneqq \sqrt{\eta} (u-(u)_{B(R)})$ 
\[
\begin{split}
 [u]_{W^{s_0,\frac{n}{s}}(B(r))}^{\frac{n}{s}} \aleq& \int_{B(\rho-3\delta)}\int_{B(\rho-3\delta)} \frac{|u(x)-u(y)|^{\frac{n}{s}-2} |\tilde{\varphi}(x)-\tilde{\varphi}(y)|^2}{|x-y|^{n+s_0\frac{n}{s}}} \dx \dy.
 \end{split}
\]
Now 
\[
 \tilde{\varphi}(x)-\tilde{\varphi}(y) = (\sqrt{\eta}(x)-\sqrt{\eta}(y)) (u(x)-(u)_{B(R)}) + \sqrt{\eta}(y) (u(x)-u(y)).
\]
So that  (recall $\varphi = \sqrt{\eta} \tilde{\varphi}$),
\[
\begin{split}
 |\tilde{\varphi}(x)-\tilde{\varphi}(y)|^2 
 &= (\sqrt{\eta}(x)-\sqrt{\eta}(y)) (u(x)-(u)_{B(R)}) \cdot(\tilde{\varphi}(x)-\tilde{\varphi}(y))\\
 &\quad +  (u(x)-u(y))(\sqrt{\eta}(y)-\sqrt{\eta}(x))\tilde{\varphi}(x)\\
 &\quad +  (u(x)-u(y))(\varphi(x)-\varphi(y)).
  \end{split}
\]
That is, 
\[
\begin{split}
 |\tilde{\varphi}(x)-\tilde{\varphi}(y)|^2 
 &\aleq \frac{|x-y|}{\delta}|u(x)-(u)_{B(R)}| \brac{|\tilde{\varphi}(x)-\tilde{\varphi}(y)|+|u(x)-u(y)|}\\
 &\quad +  (u(x)-u(y))(\varphi(x)-\varphi(y)).
  \end{split}
\]
Then 
\begin{equation}\label{eq:eqeqeq}
\begin{split}
 [u]_{W^{s_0,\frac{n}{s}}(B(r))}^{\frac{n}{s}} &\aleq 
 \int_{B(\rho-3\delta)}\int_{B(\rho-3\delta)} \frac{|u(x)-u(y)|^{\frac{n}{s}-2} (u(x)-u(y))\, (\varphi(x)-\varphi(y))}{|x-y|^{n+s_0\frac{n}{s}}} \dx \dy\\
&\quad + \delta^{-1} \int_{B(\rho-3\delta)}\int_{B(\rho-3\delta)} \frac{|u(x)-u(y)|^{\frac{n}{s}-1} |u(x)-(u)_{B(R)}|}{|x-y|^{n+s_0\frac{n}{s}-1}} \dx \dy\\
&\quad + \delta^{-1} \int_{B(\rho-3\delta)} \int_{B(\rho-3\delta)} \frac{|u(x)-u(y)|^{\frac{n}{s}-2} |\tilde{\varphi}(x)-\tilde{\varphi}(y)| |u(x)-(u)_{B(R)}| }{|x-y|^{n+s_0\frac{n}{s}-1}} \dx \dy. 
 \end{split}
\end{equation}
Now observe that for any $w: \R^n \to \R^M$ we have
\begin{equation}\label{eq:duhast-duhast-duhast}
\begin{split}
 &\delta^{-1} \int_{B(\rho-3\delta)}\int_{B(\rho-3\delta)} \frac{|u(x)-u(y)|^{\frac{n}{s}-2} |w(x)-w(y)| |u(x)-(u)_{B(R)}|}{|x-y|^{n+(s_0-\frac{s}{n})\frac{n}{s}}} \dx \dy\\
 &=\delta^{-1} \int_{B(\rho-3\delta)}\int_{B(\rho-3\delta)} \frac{|u(x)-u(y)|^{\frac{n}{s}-2}} {|x-y|^{s(\frac{n}{s}-2)}} \frac{ |w(x)-w(y)|}{|x-y|^s} \frac{ |u(x)-(u)_{B(R)}|}{|x-y|^{s
+(s_0-s)\frac{n}{s} -1 }} \, \frac{\dx \dy}{|x-y|^n}\\
&\aleq \delta^{-1} [u]_{W^{s,\frac{n}{s}}(B(R))}^{\frac{n}{s}-2} [w]_{W^{s,\frac{n}{s}}(B(R))} \brac{\int_{B(R)}\int_{B(R)} \frac{ |u(x)-(u)_{B(R)}|^{\frac{n}{s}}} {|x-y|^{n+(s+(s_0-s)\frac{n}{s} -1)\frac{n}{s} }}  \dx \dy}^{\frac{s}{n}}\\
&\aleq\delta^{-1} [u]_{W^{s,\frac{n}{s}}(B(R))}^{\frac{n}{s}-2} [w]_{W^{s,\frac{n}{s}}(B(R))} \brac{\int_{B(R)} \frac{ |u(x)-(u)_{B(R)}|^{\frac{n}{s}}} {R^{(s+(s_0-s)\frac{n}{s} -1)\frac{n}{s} }} \dx  }^{\frac{s}{n}}\\
&\aleq\delta^{-1} R^{1-(s_0-s)\frac{n}{s}} [u]_{W^{s,\frac{n}{s}}(B(R))}^{\frac{n}{s}-1} [w]_{W^{s,\frac{n}{s}}(B(R))}.
  \end{split}
\end{equation}

Here we ensure that $s_0$ is close enough to $s$ so that $s+(s_0-s)\frac{n}{s} -1 <0$.

Applying \eqref{eq:duhast-duhast-duhast} to the last two terms in \eqref{eq:eqeqeq} first for $w= u$ and then for $w= \tilde \varphi$ we obtain in view of \eqref{eq:apr:varphies:2} with $\tilde s = 0$
\begin{equation}\label{eq:oneoftheuestimatesthatwillneedalabel}
\begin{split}
 [u]_{W^{s_0,\frac{n}{s}}(B(r))}^{\frac{n}{s}} &\aleq 
 \int_{B(\rho-3\delta)}\int_{B(\rho-3\delta)} \frac{|u(x)-u(y)|^{\frac{n}{s}-2} (u(x)-u(y))\, (\varphi(x)-\varphi(y))}{|x-y|^{n+s_0\frac{n}{s}}} \dx \dy\\
&\quad +\delta^{-1} R^{1-(s_0-s)\frac{n}{s}} \brac{1+\frac{R}{\delta}} [u]_{W^{s,\frac{n}{s}}(B(R))}^{\frac{n}{s}}.
 \end{split}
\end{equation}
For the remaining term we observe
\[
\begin{split}
 & \int_{B(\rho-3\delta)}\int_{B(\rho-3\delta)} \frac{|u(x)-u(y)|^{\frac{n}{s}-2} (u(x)-u(y))\, (\varphi(x)-\varphi(y))}{|x-y|^{n+s_0\frac{n}{s}}} \dx \dy\\
  &=\int_{B(\rho-3\delta)}\int_{B(\rho-3\delta)} \frac{|u(x)-u(y)|^{\frac{n}{s}-2} (u(x)-u(y))\, \frac{\varphi(x)-\varphi(y)}{|x-y|^{(s_0-t) \frac{n}{s}}}}{|x-y|^{n+t\frac{n}{s}}} \dx \dy.
 \end{split}
\]
We now follow the ideas in \cite{S16}. We use the following identity which holds for any $n >\beta >  \alpha \geq 0$,
\begin{equation}\label{eq:standardtrick}
 \laps{\alpha} \varphi(x) = c\int_{\R^n} |x-z|^{\beta-\alpha-n}\laps{\beta} \varphi(z)\, \dif z,
\end{equation}
for a constant $c>0$ which from now on changes from line to line.

Let $\gamma \geq 0$, $0 \leq t \leq s_1 $ such that $\gamma+(s_0-t)\frac{n}{s} \in (0,n)$, and set
\[
 k(x,y,z) = \brac{\frac{|x-z|^{\gamma+(s_0-t)\frac{n}{s}-n} - |y-z|^{\gamma+(s_0-t)\frac{n}{s}-n}}{|x-y|^{(s_0-t)\frac{n}{s}}} } - \brac{|x-z|^{\gamma-n}-|y-z|^{\gamma-n}}.
\]
Then \eqref{eq:standardtrick} implies
\[
\begin{split}
 \int_{\R^n} k(x,y,z) \laps{\gamma+(s_0-t)\frac{n}{s}} \varphi (z)\, dz
 =c\frac{\varphi(x)-\varphi(y)}{|x-y|^{(s_0-t)\frac{n}{s}}} - c\brac{\laps{(s_0-t)\frac{n}{s}} \varphi(x)-\laps{(s_0-t)\frac{n}{s}} \varphi(y)}.
 \end{split}
\]
\begin{equation}\label{eq:someinequalitythatdidnthaveanumber}
\begin{split}
&\int_{B(\rho-3\delta)}\int_{B(\rho-3\delta)} \frac{|u(x)-u(y)|^{\frac{n}{s}-2} (u(x)-u(y))\, (\varphi(x)-\varphi(y))}{|x-y|^{n+s_0\frac{n}{s}}} \dx \dy\\
  &=c\int_{B(\rho-3\delta)}\int_{B(\rho-3\delta)} \frac{|u(x)-u(y)|^{\frac{n}{s}-2} (u(x)-u(y))\, \brac{\laps{(s_0-t)\frac{n}{s}} \varphi(x)-\laps{(s_0-t)\frac{n}{s}} \varphi(y)}}{|x-y|^{n+t\frac{n}{s}}} \dx \dy\\
  &\quad +cR(u,\varphi),
 \end{split}
\end{equation}
where
\[
 R(u,\varphi) = \int_{B(\rho-3\delta)}\int_{B(\rho-3\delta)} \int_{\R^n} \frac{|u(x)-u(y)|^{\frac{n}{s}-2} (u(x)-u(y))\,k(x,y,z) \laps{\gamma+(s_0-t)\frac{n}{s}} \varphi(z)} {|x-y|^{n+t\frac{n}{s}}} \dz\dx  \dy.
\]
The main observation is that for $s_0-t = 0$ we have $k \equiv 0$ and thus $R \equiv 0$. In \cite[Theorem~1.1]{S16} an estimate of $R$ for small $|s_0-t|$ was obtained. Namely\footnote{while this follows from the statement of \cite[Theorem 1.1]{S16}, it might be more instructive at first to look into \cite[Proof of Theorem~1.1]{S16}, where the same notation is used as we use here.},
\begin{equation}\label{eq:Ruvarphiest}
\begin{split}
  R(u,\varphi) 
  &\aleq \brac{s_0-t} [u]_{W^{s_0,\frac{n}{s}}(B(\rho))}^{\frac{n}{s}-1}\, [\varphi]_{W^{s_0,\frac{n}{s}}(\R^n)}\\
  &\aleq |s_0-s|\, [u]_{W^{s_0,\frac{n}{s}}(B(\rho))}^{\frac{n}{s}-1}\, [\varphi]_{W^{s_0,\frac{n}{s}}(\R^n)}\\
  &\aleq |s_0-s|\brac{\, [u]_{W^{s_0,\frac{n}{s}}(B(\rho))}^{\frac{n}{s}} +  \frac{R^{1-s_0+s}}{\delta} [u]_{W^{s,\frac{n}{s}}(B(R))}^{\frac{n}{s}}},
  \end{split}
\end{equation}
in the last inequality we used \eqref{eq:apr:varphiest}.

Next, we begin the estimate of the first term of the right-hand side of \eqref{eq:someinequalitythatdidnthaveanumber}
\begin{equation}\label{eq:Armerikaistwunderbar}
\begin{split}
 &\int_{B(\rho-3\delta)}\int_{B(\rho-3\delta)} \frac{|u(x)-u(y)|^{\frac{n}{s}-2} (u(x)-u(y))\, \brac{\laps{(s_0-t)\frac{n}{s}} \varphi(x)-\laps{(s_0-t)\frac{n}{s}} \varphi(y)}}{|x-y|^{n+t\frac{n}{s}}} \dx \dy\\
 &=\int_{\R^n}\int_{\R^n} \frac{|u(x)-u(y)|^{\frac{n}{s}-2} (u(x)-u(y))\, \brac{\laps{(s_0-t)\frac{n}{s}} \varphi(x)-\laps{(s_0-t)\frac{n}{s}} \varphi(y)}}{|x-y|^{n+t\frac{n}{s}}} \dx \dy\\
 &\quad -\iint_{(\R^n)^2 \setminus (B(\rho-3\delta))^c)^2} \frac{|u(x)-u(y)|^{\frac{n}{s}-2} (u(x)-u(y))\, \brac{\laps{(s_0-t)\frac{n}{s}} \varphi(x)-\laps{(s_0-t)\frac{n}{s}} \varphi(y)}}{|x-y|^{n+t\frac{n}{s}}} \dx \dy.
\end{split}
 \end{equation}
We will estimate now the last term in the above inequality. We observe
\begin{equation*}
 \begin{split}
&\iint_{(\R^n)^2 \setminus (B(\rho-3\delta)^c)^2} \frac{|u(x)-u(y)|^{\frac{n}{s}-2} (u(x)-u(y))\, \brac{\laps{(s_0-t)\frac{n}{s}} \varphi(x)-\laps{(s_0-t)\frac{n}{s}} \varphi(y)}}{|x-y|^{n+t\frac{n}{s}}} \dx \dy  \\
&\le \int_{\R^n}\int_{\R^n\setminus B(\rho-3\delta)} \frac{|u(x)-u(y)|^{\frac{n}{s}-2} (u(x)-u(y))\, \brac{\laps{(s_0-t)\frac{n}{s}} \varphi(x)-\laps{(s_0-t)\frac{n}{s}} \varphi(y)}}{|x-y|^{n+t\frac{n}{s}}} \dx \dy
\end{split}
\end{equation*}
and
\[
\begin{split}
&\R^n\times \R^n\setminus \brac{B(\rho-3\delta)^c  \times B(\rho-3\delta)^c }\\
&\subseteq  \brac{B(\rho-4\delta)\times \R^n\setminus B(\rho-3\delta)}\, \cup\,  \brac{\R^n\setminus B(\rho-4\delta)\times \R^n\setminus B(\rho-4\delta)}.
\end{split}
\]
Thus, choosing $s_0$ so close to $s$ so that $(s_0-t)\frac{n}{s} < s$ for all $t \in [s,s_0]$, we have
\begin{equation}\label{eq:I}
\begin{split}
&\iint_{(\R^n)^2 \setminus (B(\rho-3\delta)^c)^2} \frac{|u(x)-u(y)|^{\frac{n}{s}-2} (u(x)-u(y))\, \brac{\laps{(s_0-t)\frac{n}{s}} \varphi(x)-\laps{(s_0-t)\frac{n}{s}} \varphi(y)}}{|x-y|^{n+t\frac{n}{s}}} \dx \dy  \\
&\aleq\iint\limits_{|x-y| > \delta} \frac{1}{|x-y|^{(s_0-s)\frac{n}{s}}}\, \frac{|u(x)-u(y)|^{\frac{n}{s}-2} (u(x)-u(y))\, \brac{\laps{(s_0-t)\frac{n}{s}} \varphi(x)-\laps{(s_0-t)\frac{n}{s}} \varphi(y)}}{|x-y|^{n+
s(\frac{n}{s}-1) 
+\brac{s-(s_0-t)\frac{n}{s}}}} \dx \dy\\ 
&\quad +
\int_{\R^n \setminus B(\rho-4\delta)} \int_{\R^n \setminus B(\rho-4\delta)}\frac{|u(x)-u(y)|^{\frac{n}{s}-2} (u(x)-u(y))\, \brac{\laps{(s_0-t)\frac{n}{s}} \varphi(x)-\laps{(s_0-t)\frac{n}{s}} \varphi(y)}}{|x-y|^{n+t\frac{n}{s}}} \dx \dy.
\end{split}
\end{equation}

As for the first term on the right-hand side of \eqref{eq:I} we first use that $|x-y|>\delta$ and then apply H\"{o}lder's inequality to get
\begin{equation}\label{eq:II}
 \begin{split}
  &\iint\limits_{|x-y| > \delta} \frac{1}{|x-y|^{(s_0-s)\frac{n}{s}}}\, \frac{|u(x)-u(y)|^{\frac{n}{s}-2} (u(x)-u(y))\, \brac{\laps{(s_0-t)\frac{n}{s}} \varphi(x)-\laps{(s_0-t)\frac{n}{s}} \varphi(y)}}{|x-y|^{n+s(\frac{n}{s}-1) +\brac{s-(s_0-t)\frac{n}{s}}}} \dx \dy\\
&\aleq \delta^{-(s_0-s)\frac ns} \int_{\R^n}\int_{\R^n} \frac{|u(x)-u(y)|^{\frac ns-1}}{|x-y|^{s(\frac{n}{s}-1)}} \frac{\abs{\laps{(s_0-t)\frac{n}{s}} \varphi(x)-\laps{(s_0-t)\frac{n}{s}} \varphi(y)}}{|x-y|^{s-(s_0-t)\frac{n}{s}}} \frac{\dif x \dif y}{|x-y|^n}\\
&\aleq \delta^{-(s_0-s)\frac ns}[u]^{\frac ns-1}_{W^{s,\frac ns}(\R^n)}[\laps{(s_0-t)\frac{n}{s}} \varphi]_{W^{s-(s_0-t)\frac{n}{s},\frac{n}{s}}(\R^n)}\\
&\aleq\delta^{-(s_0-s)\frac{n}{s}}\, \frac{R}{\delta}\, [u]_{W^{s,\frac{n}{s}}(B( R))}^{\frac{n}{s}-1} 
[u]_{W^{s,\frac{n}{s}}(B(R))},
 \end{split}
\end{equation}
where in the last estimate we used \eqref{eq:apr:varphies:2}.

To estimate the second term on the right-hand side of \eqref{eq:I}, we again apply H\"{o}lder's inequality
\begin{equation}\label{eq:III}
 \begin{split}
  &\int_{\R^n \setminus B(\rho-4\delta)} \int_{\R^n \setminus B(\rho-4\delta)}\frac{|u(x)-u(y)|^{\frac{n}{s}-2} (u(x)-u(y))\, \brac{\laps{(s_0-t)\frac{n}{s}} \varphi(x)-\laps{(s_0-t)\frac{n}{s}} \varphi(y)}}{|x-y|^{n+t\frac{n}{s}}} \dx \dy\\
  &\le \int_{\R^n \setminus B(\rho-4\delta)} \int_{\R^n}\frac{|u(x)-u(y)|^{\frac{n}{s}-1}}{|x-y|^{n-s}} \frac{\abs{\laps{(s_0-t)\frac{n}{s}} \varphi(x)-\laps{(s_0-t)\frac{n}{s}} \varphi(y)}}{|x-y|^{(t-s)\frac ns +s}} \frac{\dx \dy}{|x-y|^{n}}\\
  &\le [u]^{\frac ns-1}_{W^{s,\frac ns}(\R^n)}\brac{\int_{\R^n \setminus B(\rho-4\delta)} \int_{\R^n} \frac{|\laps{(s_0-t)\frac{n}{s}} \varphi(x)-\laps{(s_0-t)\frac{n}{s}} \varphi(y)|^{\frac{n}{s}}}{|x-y|^{n+\brac{(t-s)\frac{n}{s}+s}\frac{n}{s} }} \dx \dy}^{\frac{s}{n}}\\
  &\aleq \delta^{-s-(s_0-s)\frac{n}{s}} R^s\, 
[u]_{W^{s,\frac{n}{s}}(\R^n)}^{\frac{n}{s}-1} [u]_{W^{s,\frac{n}{s}}(B(R))} ,
 \end{split}
\end{equation}
in the last estimate we applied \eqref{eq:apr:varphies:4} with $\gamma=(t-s)\frac{n}{s}+s<1$ (the latter condition can be satisfied by a good choice of $s_1$).

We observe that
\begin{equation}\label{eq:IV}
\begin{split}
&\brac{\delta^{-(s_0-s)\frac{n}{s}}\, \frac{R}{\delta} + \delta^{-s-(s_0-s)\frac{n}{s}} R^s} [u]_{W^{s,\frac{n}{s}}(\R^n)}^{\frac{n}{s}-1} 
[u]_{W^{s,\frac{n}{s}}(B(R))}\aleq \delta^{-1-(s_0-s)\frac{n}{s}} R\, 
[u]_{W^{s,\frac{n}{s}}(\R^n)}^{\frac{n}{s}-1} [u]_{W^{s,\frac{n}{s}}(B(R))}.
\end{split}
\end{equation}
Thus, combining \eqref{eq:I} with \eqref{eq:II}, \eqref{eq:III}, and \eqref{eq:IV} we obtain
\begin{equation}\label{eq:VI}
 \begin{split}
 &\iint_{(\R^n)^2 \setminus (B(\rho-3\delta)^c)^2} \frac{|u(x)-u(y)|^{\frac{n}{s}-2} (u(x)-u(y))\, \brac{\laps{(s_0-t)\frac{n}{s}} \varphi(x)-\laps{(s_0-t)\frac{n}{s}} \varphi(y)}}{|x-y|^{n+t\frac{n}{s}}} \dx \dy\\
&\aleq  \delta^{-1-(s_0-s)\frac{n}{s}} R\, 
[u]_{W^{s,\frac{n}{s}}(\R^n)}^{\frac{n}{s}-1} [u]_{W^{s,\frac{n}{s}}(B(R))}.
 \end{split}
\end{equation}

Bringing together the estimates
\eqref{eq:oneoftheuestimatesthatwillneedalabel}, \eqref{eq:someinequalitythatdidnthaveanumber}, \eqref{eq:Ruvarphiest}, \eqref{eq:Armerikaistwunderbar}, and \eqref{eq:VI} we have shown
\begin{equation}\label{eq:almostlastestimateofus0onsmallball}
\begin{split}
 &[u]_{W^{s_0,\frac{n}{s}}(B(r))}^{\frac{n}{s}}\\
 &\aleq |s_0-s|\, [u]_{W^{s_0,\frac{n}{s}}(B(\rho))}^{\frac{n}{s}} +  \delta^{-1-(s_0-s)\frac{n}{s}} R\, 
[u]_{W^{s,\frac{n}{s}}(\R^n)}^{\frac{n}{s}-1} [u]_{W^{s,\frac{n}{s}}(B(R))}\\
&\quad  +\delta^{-1} R^{1-(s_0-s)\frac{n}{s}} \brac{1+\frac{R}{\delta}} [u]_{W^{s,\frac{n}{s}}(B(R))}^{\frac{n}{s}}\\
&\quad +\int_{\R^n}\int_{\R^n} \frac{|u(x)-u(y)|^{\frac{n}{s}-2} (u(x)-u(y))\, \brac{\laps{(s_0-t)\frac{n}{s}} \varphi(x)-\laps{(s_0-t)\frac{n}{s}} \varphi(y)}}{|x-y|^{n+t\frac{n}{s}}} \dx \dy.
\end{split}
\end{equation}
Let us now estimate the last term of the inequality above. Take $\tilde{\eta} \in C_c^\infty(B(\rho-3\delta))$ and $\tilde{\eta} \equiv 1$ in $B(\rho-4\delta)$, with $|\nabla \tilde{\eta}| \aleq \delta^{-1}$, then
\begin{equation}\label{eq:DeutschalandDeutschland}
 \begin{split}
&\int_{\R^n}\int_{\R^n} \frac{|u(x)-u(y)|^{\frac{n}{s}-2} (u(x)-u(y))\, \brac{\laps{(s_0-t)\frac{n}{s}} \varphi(x)-\laps{(s_0-t)\frac{n}{s}} \varphi(y)}}{|x-y|^{n+t\frac{n}{s}}} \dx \dy\\
&=\int_{\R^n}\int_{\R^n} \frac{|u(x)-u(y)|^{\frac{n}{s}-2} (u(x)-u(y))\, \brac{\tilde{\eta}(x)\laps{(s_0-t)\frac{n}{s}} \varphi(x)-\tilde{\eta}(y)\laps{(s_0-t)\frac{n}{s}} \varphi(y)}}{|x-y|^{n+t\frac{n}{s}}} \dx \dy\\  
&+\int\limits_{\R^n}\int\limits_{\R^n} \frac{|u(x)-u(y)|^{\frac{n}{s}-2} (u(x)-u(y))\, \brac{(1-\tilde{\eta}(x))\laps{(s_0-t)\frac{n}{s}} \varphi(x)-(1-\tilde{\eta}(y))\laps{(s_0-t)\frac{n}{s}} \varphi(y)}}{|x-y|^{n+t\frac{n}{s}}} \dx \dy.  
 \end{split}
\end{equation}
As for the second term of \eqref{eq:DeutschalandDeutschland}, we use, similarly as in \eqref{eq:III}, H\"{o}lder's inequality 
\begin{equation}\label{eq:DeutschalandDeutschland1}
\begin{split}
 &\int\limits_{\R^n}\int\limits_{\R^n} \frac{|u(x)-u(y)|^{\frac{n}{s}-2} (u(x)-u(y))\, \brac{(1-\tilde{\eta}(x))\laps{(s_0-t)\frac{n}{s}} \varphi(x)-(1-\tilde{\eta}(y))\laps{(s_0-t)\frac{n}{s}} \varphi(y)}}{|x-y|^{n+t\frac{n}{s}}} \dx \dy\\ 
 & \leq [u]_{W^{s,\frac{n}{s}}(\R^n)}^{\frac{n}{s}-1}\, \left[(1-\tilde{\eta})\laps{(s_0-t)\frac{n}{s})} \varphi\right]_{W^{(t-s)\frac ns +s,\frac ns}(\R^n)}\\
 &\aleq\delta^{-(s_0-s)\frac{n}{s}-s} R^s \, [u]_{W^{s,\frac{n}{s}}(\R^n)}^{\frac{n}{s}-1} [u]_{W^{s,\frac{n}{s}}(B(R))},
 \end{split}
\end{equation}
in the last inequality we used \eqref{eq:apr:varphies:5} with $\gamma = (t-s)\frac ns +s$ and $\tilde s = (s_0-t)\frac ns$.

As for the first term of \eqref{eq:DeutschalandDeutschland}, we use the PDE \eqref{eq:ap:pde} with test function $\eta \laps{(s_0-t)\frac ns \varphi}$ and arrive at
\begin{equation}\label{eq:onlysecondtermtocheckandimdonewiththisproposition}
\begin{split}
 &\abs{\int_{\R^n}\int_{\R^n} \frac{|u(x)-u(y)|^{\frac{n}{s}-2} (u(x)-u(y))\, \brac{\tilde{\eta}(x)\laps{(s_0-t)\frac{n}{s}} \varphi(x)-\tilde{\eta}(y)\laps{(s_0-t)\frac{n}{s}} \varphi(y)}}{|x-y|^{n+t\frac{n}{s}}} \dx \dy}\\
 &\aleq  \int_{\R^n}\int_{\R^n} \abs{\tilde{\eta}(x)\laps{(s_0-t)\frac{n}{s}} \varphi(x)} \frac{|u(x)-u(y)|^\frac ns}{|x-y|^{n+t\frac ns}} \dif y \dif x \\
 &\leq \int_{B(\rho)}\int_{B(\rho)} \abs{\tilde{\eta}(x)\laps{(s_0-t)\frac{n}{s}} \varphi(x)} \frac{|u(x)-u(y)|^{\frac{n}{s}}}{|x-y|^{n+t\frac{n}{s}}}  \dy \dx\\
&\quad +\int_{B(\rho-\delta)}\int_{\R^n \setminus B(\rho)} \abs{\tilde{\eta}(x)\laps{(s_0-t)\frac{n}{s}} \varphi(x)} \frac{|u(x)-u(y)|^{\frac{n}{s}}}{|x-y|^{n+t\frac{n}{s}}}  \dy \dx.
 \end{split}
\end{equation}
We estimate the first term of the last inequality. We will use $t_0 < s_1 < s_0$ (here we choose $s_1$ so that $s+(s_1-s)\frac{s}{n-s} < s_0$ and $s_1 < s+(s_0-s)(1-\frac{s}{n}$). Using H\"{o}lder's inequality twice we get
\begin{equation}\label{eq:Imnotevensurewhatmusiclistento}
 \begin{split}
&\int_{B(\rho)}\int_{B(\rho)} \abs{\tilde{\eta}(x)\laps{(s_0-t)\frac{n}{s}} \varphi(x)} \frac{|u(x)-u(y)|^{\frac{n}{s}}}{|x-y|^{n+t\frac{n}{s}}}  \dy \dx\\
&=\int_{B(\rho)} \abs{\tilde{\eta}(x)\laps{(s_0-t)\frac{n}{s}} \varphi(x)} \int_{B(\rho)}\frac{|u(x)-u(y)|^{\frac{n}{s}-1}}{|x-y|^{t\frac{n}{s}-s}} \, \frac{|u(x)-u(y)|}{|x-y|^s} \frac{\dy\dx }{|x-y|^{n}} \\
&\leq \int_{B(\rho)} \abs{\tilde{\eta}(x)\laps{(s_0-t)\frac{n}{s}} \varphi(x)} \brac{\int_{B(\rho)}\frac{|u(x)-u(y)|^{\frac{n}{s}}}{|x-y|^{n+\brac{\brac{t\frac{n}{s}-s}\frac{s}{n-s}} \frac{n}{s}}} \dy}^{1-\frac{s}{n}} \, \brac{\int_{B(\rho)}\frac{|u(x)-u(y)|^{\frac{n}{s}} }{|x-y|^{2n}} \dy}^{\frac{s}{n}} \dx \\
&\leq \norm{\laps{(s_0-t)\frac{n}{s}} \varphi}_{L^{\frac{n}{(s_0-s)\brac{\frac{n}{s}-1} -(t-s)\frac{n}{s}}}(\R^n)}\, [u]_{W^{s,\frac{n}{s}}(B(\rho)}\\
&\quad
\brac{\int_{B(\rho)} \brac{\int_{B(\rho)}\frac{|u(x)-u(y)|^{\frac{n}{s}}}{|x-y|^{n+(
t\frac{n}{n-s}-s\frac{s}{n-s} 
) \frac{n}{s}}} \dy}^{\frac{s}{n}
\frac{n}{2s-s_0 +(t-s)\frac{n}{(n-s)}}
} 
\dx}^{\frac{2s-s_0 +(t-s)\frac{n}{(n-s)}}{n} \frac{n-s}{s}}.
\end{split}
\end{equation}
In the last inequality we applied (generalized) H\"{o}lder's inequality with exponents $\frac{n}{(s_0-s)\brac{\frac ns-1}-(t-s)\frac ns}$, $\frac{n}{2s-s_0+(t-s)\frac{n}{n-s}}\frac{s}{n-s}$, $\frac ns$.

By Sobolev embedding, \Cref{th:sobolev}, \eqref{eq:sob4} (applied for  $s\coloneqq t\frac{n}{n-s} - s\frac{s}{n-s}$, $p\coloneqq \frac ns$, $p^* \coloneqq \frac{n}{2s-s_0+(t-s)\frac{n}{n-s}}$, $t\coloneqq s_0$) we obtain 
\begin{equation}\label{eq:sameproblem}
 \brac{\int_{B(\rho)} \brac{\int_{B(\rho)}\frac{|u(x)-u(y)|^{\frac{n}{s}}}{|x-y|^{n+(
t\frac{n}{n-s}-s\frac{s}{n-s} 
) \frac{n}{s}}} \dy}^{\frac{s}{n}
\frac{n}{2s-s_0 +(t-s)\frac{n}{(n-s)}}
} \dx}^{\frac{2s-s_0 +(t-s)\frac{n}{(n-s)}}{n}} \aleq [u]_{W^{s_0,\frac{n}{s}}(B(\rho))}.
\end{equation}
Moreover, also by Sobolev embedding, \Cref{th:sobolev}, \eqref{eq:sob2b} (applied for $s\coloneqq (s_0-t)\frac ns$, $p^*\coloneqq \frac{n}{(s_0-s)\brac{\frac ns-1}-(t-s)\frac ns}$, $p=\frac ns$, $t\coloneqq s_0$) we have
\begin{equation}\label{eq:asbeforenochanges}
\begin{split}
 \|\laps{(s_0-t)\frac{n}{s}} \varphi\|_{L^{\frac{n}{(s_0-s)(\frac{n}{s}-1) -(t-s)\frac{n}{s}}}(\R^n)} 
 &\aleq [\varphi]_{W^{s_0,\frac{n}{s}}(\R^n)}\\
 &\aleq [u]_{W^{s_0,\frac{n}{s}}(B(\rho))} + \frac{R^{1-s_0+s}}{\delta} [u]_{W^{s,\frac{n}{s}}(B(R))},
\end{split}
 \end{equation}
 we used \eqref{eq:apr:varphiest} in the last estimate.
%

Thus, collecting the estimates \eqref{eq:Imnotevensurewhatmusiclistento}, \eqref{eq:sameproblem}, and \eqref{eq:asbeforenochanges} and recalling that by assumption $[u]_{W^{s,\frac ns}(B(\rho))}\le \eps$ we have
\[
 \begin{split}
&\int_{B(\rho)}\int_{B(\rho)} \abs{\tilde{\eta}(x)\laps{(s_0-t)\frac{n}{s}} \varphi(x)} \frac{|u(x)-u(y)|^{\frac{n}{s}}}{|x-y|^{n+t\frac{n}{s}}}  \dy \dx\\
& \aleq [u]_{W^{s,\frac{n}{s}}(B(\rho))} [u]_{W^{s_0,\frac{n}{s}}(B(\rho))}^{\frac{n}{s}-1}  \brac{[u]_{W^{s_0,\frac{n}{s}}(B(\rho))} + \frac{R^{1-s_0+s}}{\delta} [u]_{W^{s,\frac{n}{s}}(B(R))}}\\
&\aleq \eps [u]_{W^{s_0,\frac{n}{s}}(B(\rho))}^{\frac{n}{s}}+\eps [u]_{W^{s_0,\frac{n}{s}}(B(\rho))}^{\frac{n}{s}-1}\, \frac{R^{1-s_0+s}}{\delta} [u]_{W^{s,\frac{n}{s}}(B(R))}\\
&\aleq \eps [u]_{W^{s_0,\frac{n}{s}}(B(\rho))}^{\frac{n}{s}} + \eps^{\frac{n}{n-s}} [u]_{W^{s_0,\frac{n}{s}}(B(\rho))}^{\frac{n}{s}} + \brac{\frac{R^{1-s_0+s}}{\delta} [u]_{W^{s,\frac{n}{s}}(B(R))}}^{\frac{n}{s}},
\end{split}
\]
we used Young's inequality in the last estimate.

It remains to treat the second term of the right-hand side of \eqref{eq:onlysecondtermtocheckandimdonewiththisproposition}. 
\begin{equation}\label{eq:apr:234}
\begin{split}
&\int_{B(\rho-\delta)}\int_{\R^n \setminus B(\rho)} \abs{\tilde{\eta}(x)\laps{(s_0-t)\frac{n}{s}} \varphi(x)} \frac{|u(x)-u(y)|^{\frac{n}{s}}}{|x-y|^{n+t\frac{n}{s}}}  \dy \dx\\
&\aleq \iint_{|x-y|\ge \delta} \frac{1}{|x-y|^{(s_0-s)\frac ns}} \abs{\laps{(s_0-t)\frac{n}{s}} \varphi(x)}\frac{|u(x)-u(y)|^{\frac{n}{s}}}{|x-y|^{n+(s+t-s_0)\frac ns}} \dif y \dif x\\
&\aleq \delta^{-(s_0-s)\frac{n}{s}}\, \norm{\laps{(s_0-t)\frac{n}{s}} \varphi}_{L^{\frac{s}{s_0-t}}(\R^n)}\, 
\brac{\int_{\R^n}\brac{\int_{\R^n}
\frac{|u(x)-u(y)|^{\frac{n}{s}}}{|x-y|^{n+(s+t-s_0))\frac{n}{s}}} \, \dy}^{\frac{s}{n}\frac{n}{s+t-s_0}} \dx}^{\frac{s+t-s_0}{n} \,\frac{n}{s}},
\end{split}
\end{equation}
in the last estimate we have used H\"{o}lder's inequality with exponents $\frac{s}{s_0-t}$ and $\frac{s}{s+t-s_0}$.

Again applying the Sobolev embedding, \Cref{th:sobolev} we obtain
\begin{equation*}
\brac{\int_{\R^n}\brac{\int_{\R^n}
\frac{|u(x)-u(y)|^{\frac{n}{s}}}{|x-y|^{n+(s+t-s_0))\frac{n}{s}}} \, \dy}^{\frac{s}{n}\frac{n}{s+t-s_0}} \dx}^{\frac{s+t-s_0}{n}} \le [u]_{W^{s,\frac ns}(\R^n)}.
\end{equation*}
and
\begin{equation*}
\norm{\laps{(s_0-t)\frac{n}{s}} \varphi}_{L^{\frac{s}{s_0-t}}(\R^n)} \aleq [\varphi]_{W^{s,\frac ns}(\R^n)} \aleq \frac R \delta [u]_{W^{s,\frac ns}(B(R))},
\end{equation*}
were in the second inequality we used \eqref{eq:apr:varphies:2}.

Thus,
\begin{equation}\label{eq:BrhominusdeltaRnsetminusblahblah}
 \begin{split}
 &\int_{B(\rho-\delta)}\int_{\R^n \setminus B(\rho)} \abs{\tilde{\eta}(x)\laps{(s_0-t)\frac{n}{s}} \varphi(x)} \frac{|u(x)-u(y)|^{\frac{n}{s}}}{|x-y|^{n+t\frac{n}{s}}}  \dy \dx \\
 &\aleq \delta^{-(s_0-s)\frac{n}{s}} \frac{R}{\delta}[u]_{W^{s,\frac{n}{s}}(B(R))}\, [u]_{W^{s,\frac{n}{s}}(\R^n)}^{\frac{n}{s}}\\
 &\aleq \delta^{-(s_0-s)\frac{n}{s}} \brac{\frac{R}{\delta}}^{\frac ns}[u]_{W^{s,\frac{n}{s}}(B(R))}\, [u]_{W^{s,\frac{n}{s}}(\R^n)}^{\frac{n}{s}}.
 \end{split}
\end{equation}

Finally, combining \eqref{eq:almostlastestimateofus0onsmallball} with \eqref{eq:DeutschalandDeutschland}, \eqref{eq:DeutschalandDeutschland1}, \eqref{eq:onlysecondtermtocheckandimdonewiththisproposition}, \eqref{eq:Imnotevensurewhatmusiclistento}, and \eqref{eq:BrhominusdeltaRnsetminusblahblah} we obtain
\[
\begin{split}
[u]_{W^{s_0,\frac{n}{s}}(B(r))}^{\frac{n}{s}}
 &\aleq |s_0-s|\, [u]_{W^{s_0,\frac{n}{s}}(B(\rho))}^{\frac{n}{s}} +  \delta^{-1-(s_0-s)\frac{n}{s}} R\, 
[u]_{W^{s,\frac{n}{s}}(\R^n)}^{\frac{n}{s}-1} [u]_{W^{s,\frac{n}{s}}(B(R))}\\
&\quad  +\delta^{-1} R^{1-(s_0-s)\frac{n}{s}} \brac{1+\frac{R}{\delta}} [u]_{W^{s,\frac{n}{s}}(B(R))}^{\frac{n}{s}}\\
&\quad + \delta^{-(s_0-s)\frac{n}{s}-s} R^s \, [u]_{W^{s,\frac{n}{s}}(\R^n)}^{\frac{n}{s}-1} [u]_{W^{s,\frac{n}{s}}(B(R))}\\
&\quad + \eps [u]_{W^{s_0,\frac{n}{s}}(B(\rho))}^{\frac{n}{s}} + \eps^{\frac{n}{n-s}} [u]_{W^{s_0,\frac{n}{s}}(B(\rho))}^{\frac{n}{s}} + \brac{\frac{R^{1-s_0+s}}{\delta} [u]_{W^{s,\frac{n}{s}}(B(R))}}^{\frac{n}{s}}\\
&\quad + \delta^{-(s_0-s)\frac{n}{s}} \brac{\frac{R}{\delta}}^\frac ns [u]_{W^{s,\frac{n}{s}}(B(R))}\, [u]_{W^{s,\frac{n}{s}}(\R^n)}^{\frac{n}{s}}\\
&\aleq \brac{|s_0 - s| + \eps + \eps^{\frac {n}{n-s}}}[u]^{\frac ns}_{W^{s_0,\frac ns}(B(\rho))}\\
&\quad + \delta^{-(s_0-s)\frac ns}\brac{\frac{R}{\delta}}^{\frac ns} \brac{[u]^{\frac ns-1}_{W^{s,\frac ns}(\R^n)}+[u]_{W^{s,\frac ns}(\R^n)}^\frac ns}[u]_{W^{s,\frac ns}(B(R))}.
\end{split}
\]
For $|s_0-s|$ and $\eps$ small enough we thus have shown \eqref{eq:apr:goalineq}.

This concludes the proof of \Cref{th:apriorireg}.
\end{proof}

\subsection{Proof of Theorem~\ref{th:reghomo}}
\begin{proof}[Proof of \Cref{th:reghomo}]
Let $s_0 \coloneqq \min\left\{\frac{s+\alpha/p}{2},\frac{s+\bar{s}}{2} \right\}$, where $\alpha$ is taken from \Cref{th:ininitreg} (without loss of generality we may assume that $\alpha/p < 1$) and $\bar{s}$ is taken from \Cref{th:apriorireg}. Take $s_1$ and $\eps$ from \Cref{th:apriorireg}.

Assuming $u$ is a $W^{t,\frac{n}{s}}$-minimizer in $B(R)$, we get from \Cref{th:ininitreg} that $u \in C^\alpha_{loc}(B(R))$. Since $u$ is a minimizer, it satisfies the Euler--Lagrange equations, cf. \eqref{eq:ap:pde}. So we can apply \Cref{th:BLdiff} and obtain that $u \in W^{s+\beta,\frac{n}{s}}_{loc}(B(R),\R^M)$ for any $\beta < s+\frac{\alpha}{p}$. In particular, $u \in W^{s_0,\frac{n}{s}}_{loc}(B(R),\R^M)$. From \Cref{th:apriorireg} we obtain for all $t \in [s,s_1]$,
\[
  [u]_{C^{s_0-s}(B(R/2))} + [u]_{W^{s_0,\frac{n}{s}}(B(R/2))} \leq C\, R^{s-s_0} [u]_{W^{s,\frac{n}{s}}(B(R))} \brac{[u]_{W^{s,\frac{n}{s}}(\Sigma)}^{1-\frac{s}{n}}+[u]_{W^{s,\frac{n}{s}}(\Sigma)}}.
\]
In particular we have
\[
  [u]_{C^{s_1-s}(B(R/2))} + [u]_{W^{s_1,\frac{n}{s}}(B(R/2))} \leq C\, R^{s-s_1} [u]_{W^{s,\frac{n}{s}}(B(R))} \brac{[u]_{W^{s,\frac{n}{s}}(\Sigma)}^{1-\frac{s}{n}}+[u]_{W^{s,\frac{n}{s}}(\Sigma)}}.
\]
So \Cref{th:reghomo} is proven taking $s_0$ in the statement of the theorem to be $s_1$.
\end{proof}

\subsection{Consequences}
We will need the following generalization of \cite[Lemma 4.3]{Sucks1}. 
\begin{corollary}\label{co:strongconvergencesmallenergy}
Let $\Sigma$ and $\n$ be as above. There exists $\eps > 0$ and $s_0 \in (s,1)$ such that the following holds.

Let $\{u_t\}_{t >s}$ be a sequence of $W^{t,\frac ns}(\Sigma,\n)$-harmonic maps minimizing in a fixed homotopy class. Let us assume that $u_t\rightharpoonup u_s$ converges weakly in $W^{s,\frac ns}(\Sigma, \n)$.
 
 There exists an $\eps>0$ such that if $E_{s,\frac ns}(u_t, B(x_0,\rho))<\eps$ for some ball $B(x_0,\rho)$ then $u_t \to u_s$ strongly in $W^{s_0,\frac ns}(B(x_0,\rho/2),\n))$. The number $s_0>s$ is taken from Proposition \ref{th:reghomo}.
\end{corollary}
\begin{proof}
Let $s_1$ be the ``$s_0$'' from \Cref{th:reghomo} and set $s_0 \coloneqq \frac{s+s_1}{2}$.

From \Cref{th:reghomo} we obtain
\[
 \sup_{t\in (s,s_1]} [u_t]_{W^{s_1,\frac{n}{s}}(B(x_0,\rho/2))}<\infty.
\]
Thus $u_t$ converges weakly to $u_s$ in $W^{s_1,\frac{n}{s}}(B(x_0,\rho/2))$. By Rellich--Kondrachov Theorem we obtain strong convergence in $W^{s_0,\frac{n}{s}}(B(x_0,\rho/2))$.
\end{proof}

The following theorem combines \Cref{co:strongconvergencesmallenergy} with a covering argument, and is a generalization of \cite[Proposition 4.3 \& Theorem 4.4]{Sucks1}.
\begin{theorem}\label{th:strongconvoutsideofpoints}
For any $s \in (0,1)$ there exists $s_0 > s$ such that the following holds.
 For $t \in (s,s_0]$ let $u_t\colon \Sigma\to \n$ be a sequence of minimizing $W^{t,\frac{n}{s}}$-harmonic maps in a fixed homotopy class of $C^0(\Sigma,\n)$. Then, there is a decreasing sequence $(t_j)_{j \in \N} \subset (s,s_0]$ such that $t_{j} \xrightarrow{j \to \infty} s$ and a finite number of points $A\coloneqq \{x_1,\ldots,x_K\}$, such that 
 \[
 \begin{split}
  u_{t_j} &\xrightarrow{j \to \infty} u_s \quad \text{ locally strongly in } W^{s_0,\frac ns}(\Sigma \setminus A).
 \end{split}
  \]
  Moreover, $u_s$ is a $E_{s,\frac{n}{s}}$-minimizer within its homotopy class in $\Sigma \setminus A$, {\it i.e.}, 
  \[
   E_{s,\frac ns}(u_s, \Sigma) \leq E_{s,\frac ns}(v, \Sigma) \quad \text{if $u \equiv v$ in a neighborhood of $A$ and $u \sim v$}.
  \]
 \end{theorem}

 \begin{proof}
 We can assume $E_{s,\frac ns}(u_t, \Sigma) <\Lambda$ for all $t\in [s,s_0]$.  
 
 Indeed,  since $\Sigma$ is compact and by minimality of $u_{t}$,
\[
\begin{split}
 \sup_{t\in[s,s_0]} E_{s,\frac{n}{s}}(u_t) &= \sup_{t\in [s,s_0]} \int_{\Sigma} \int_{\Sigma} \frac{|u_t(x) - u_t(y)|^\frac ns}{|x-y|^{2n}}\frac{|x-y|^{n\brac{\frac{t-s}{s}}}}{|x-y|^{n\brac{\frac{t-s}{s}}}} \dx \dy\\
 &\aleq \sup_{t\in[s,s_0]} E_{t,\frac{n}{s}}(u_t)\\
 &\aleq \sup_{t\in[s,s_0]} E_{t,\frac{n}{s}}(u_{s_0})\\
 &\aleq E_{s_0,\frac{n}{s}}(u_{s_0}) < \infty.
\end{split}
 \]

Thus, $E_{s,\frac{n}{s}}(u_t)$ is uniformly bounded.

Let $\alpha\in\N$, we define  
\[
\mathcal B_\alpha \coloneqq \{B(x_{i,\alpha},2^{-\alpha})\colon i\in I, x_{i,\alpha}\in\Sigma\} 
\]
a family of balls such that $\Sigma \subset \bigcup \mathcal B_\alpha$, each point $x\in \Sigma$ is covered at most $h$-times, and for which, for twice smaller radius we still have $\Sigma\subset \bigcup_{i\in I} B(x_{i,\alpha},2^{-\alpha-1})$. Then,
\[
\sum_{i\in I} \int_{B(x_{i,\alpha},2^{-\alpha})}\int_{\Sigma} \frac{|u_t(x) - u_t(y)|^\frac ns}{|x-y|^{2n}} \dx \dy < \Lambda h. 
\]
Let $\eps>0$ be taken from \Cref{co:strongconvergencesmallenergy}, then
for each $t\in[s,s_0]$ there exists at most $\frac{\Lambda h}{\eps}$ balls in $\mathcal B_\alpha$ on which 
\begin{equation}\label{eq:ballswithbigenergy}
 \int_{B(x_{i,\alpha}, 2^{-\alpha})} \int_{\Sigma} \frac{|u_t(x) - u_t(y)|^\frac ns}{|x-y|^{2n}} \dx \dy > \eps.
\end{equation}
Now, we claim that there exists a subsequence $\{t_{K,\alpha}\}\subset \{t\}$ for which 
\[
 u_{t_{K,\alpha}} \xrightarrow{t_{K,\alpha}\rightarrow s} u_s \quad \text{ strongly in } W^{s_0,\frac ns}(B(x_{i,\alpha},2^{-\alpha-1}),\n)
\]
except for $K$ balls from $\mathcal B_\alpha$, where $K< \frac{h\Lambda}{\eps} + 1$.

Indeed, suppose that we have already shown that we have a subsequence $\{t_{k,\alpha}\}\subset \{t\}$ for which
\[
 u_{t_{k,\alpha}} \xrightarrow{t_{k,\alpha}\rightarrow s} u_s \quad \text{ strongly in } W^{s_0,\frac ns}(B(x_{i,\alpha},2^{-\alpha-1}),\n)
\]
for $i=1,\ldots,k$ and that there are more than $\frac{h\Lambda}{\eps}$ balls remaining in $\mathcal B_\alpha \setminus \{B(x_{1,\alpha},2^{-\alpha}),\ldots, B(x_{k,\alpha},2^{-\alpha})\}$. Then, by \eqref{eq:ballswithbigenergy} there is at least one $j\in I\setminus\{1,\ldots,k\}$ for which on the ball $B(x_j,2^{-\alpha})\in\mathcal B_\alpha$ we have
\[
 \int_{B(x_{j,\alpha},2^{-\alpha})} \int_{\Sigma} \frac{|u_{t_{k,\alpha}}(x) - u_{t_{k,\alpha}}(y)|^\frac ns}{|x-y|^{2n}} \dx \dy < \eps.
\]
By \Cref{co:strongconvergencesmallenergy} we know that there is a subsequence $\{t_{k+i,\alpha}\}\subset \{t_{k,\alpha}\}$ such that on the smaller ball we have
\[
 u_{t_{k+1,\alpha}} \xrightarrow{t_{k+1,\alpha}\rightarrow s} u_s \quad \text{ strongly in } W^{s_0,\frac ns}(B(x_{j,\alpha},2^{-\alpha-1}),\n).
\]
We repeat this construction until there are $K < \frac{h\Lambda}{\eps}+1$ balls left.

Thus, we have shown that for some $\{y_{1,\alpha},\ldots, y_{K,\alpha}\}$ we have
\[
 u_{t_{K,\alpha}} \xrightarrow{t_{K,\alpha}\rightarrow s} u_s \quad \text{ strongly in } W^{s_0,\frac ns}\bigg(\Sigma \setminus \bigcup_{i\le K}(B(y_{i,\alpha},2^{-\alpha-1}))),\n\bigg).
\]
Moreover, we have $\{t_{K,\alpha}\}\subset \{t_{K,\alpha-1}\}\subset\{t\}$. Finally, we choose a diagonal subsequence $\tilde t$ of the sequences $\{t_{K,\alpha}\}$, then $u_{\tilde t} \rightarrow s$ in $W^{s_0,\frac ns}$ on 
\[
 \bigcup_{\alpha \in \N} \brac{\Sigma\setminus \bigcup_{i\le K}B(y_{i,\alpha},2^{-\alpha-1})} = \Sigma \setminus \bigcap_{\alpha\in \N} \bigcup_{i\le K}B(y_{i,\alpha},2^{-\alpha-1}) = \Sigma \setminus \{x_1,\ldots,x_K\}.
\]

\end{proof} 

\section{Removability of singularities}\label{s:removability}
In this section we show that in the case when limits of minimizing $W^{t,\frac ns}$-harmonic maps have isolated singularities, then those singularities can be removed.

\begin{theorem}\label{th:removabilityminimizing}
Let $\Sigma,\n$ be manifolds as above. Let $B=B(x_0,R) \subset \Sigma$ be a geodesic ball centered at a point $x_0 \in \Sigma$, then the following holds.

Assume $u\in W^{s,\frac ns}(\Sigma,\n)$ be a minimizing map in $B(x_0,R)$ in homotopy away from the point $x_0$. That is assume  for any $\eps>0$ and any $w\in W^{s,\frac ns}(\Sigma)$ satisfying
\begin{itemize}
 \item $u\equiv w$ on $B(x_0,\eps) \cup \brac{\Sigma \backslash B(x_0,R)}$ and
 \item $u\sim w$, 
\end{itemize}
we have
 \begin{equation}\label{eq:minimalityoutsideofapoint}
  E_{s,\frac ns} (u,\Sigma) \le E_{s,\frac ns}(w,\Sigma).
 \end{equation}
Then, $u$ is minimizing in all of $B(x_0,R)$, i.e., for any $v\in W^{s,\frac ns}(\Sigma)$ such that 
\begin{itemize}
 \item $u\equiv w$ on $\Sigma \backslash B(x_0,R)$ and
 \item $u\sim w$, 
\end{itemize}
we have
\[
 E_{s,\frac ns} (u,\Sigma) \le E_{s,\frac ns}(v,\Sigma).
\]
\end{theorem}
In particular we obtain regularity theory for maps as in \Cref{th:removabilityminimizing}, see \Cref{th:corollaryremovability}.

To prove \Cref{th:removabilityminimizing} will construct a comparison map, the construction will be very similar to the one in the paper by Monteil--Van~Schaftingen \cite[Proof of Theorem 3.1]{MonteilVanSchaftingen}. We will be using the following lemmata from \cite{MonteilVanSchaftingen}. The first lemma, is called \emph{the opening of maps} in the sense of Brezis--Li \cite{Brezis-Li}, and the purpose of it is to connect a given map continuously to a constant within the Sobolev space. 

\begin{lemma}[{{\cite[Lemma~2.1]{MonteilVanSchaftingen}}}]\label{le:21}
 Let $0<s\le1$, $p\ge 1$, $\lambda>1$, and  $\eta \in (0,\lambda)$. Then, there is a constant $C>0$ such that for any $\rho >0$, any measurable $u\colon B(\lambda \rho) \to \n$, and every Lipschitz continuous map $\phi \colon B((1+\eta)\rho) \to B((\lambda-\eta)\rho)$, there exists a point $a\in B(\eta\rho)$ such that
 \[
  E_{s,p}\brac{u\circ \brac{\phi(\cdot - a) + a}, B(\rho)} \le C \lip(\phi)^{sp} E_{s, p}(u, B(\lambda \rho)).
 \]
 \end{lemma}
 
 The next lemma allows to glue two maps along a ''buffering zone".
\begin{lemma}[{{\cite[Lemma~2.2]{MonteilVanSchaftingen}}}]\label{le:22}
Let $0<s\le 1$, $p\ge 1$. There exists a constant $C>0$ such that for every $\eta\in(0,1)$, $A\subset \Sigma$ open, every measurable $u\colon B(\lambda \rho) \to \n$, and every $\rho>0$ such that $B_\rho\setminus \overline{B(\eta \rho)} \subset A$ we have
\[
 E_{s,p}(u,A) \le \brac{1 + \frac{C}{(1-\eta)^{sp+1}}} E_{s,p}(u, A\cap B(\rho)) + \brac{1 + \frac{C\eta^n}{1-\eta}}E_{s,p}(u, A\setminus \overline{B(\eta \rho)}),
\]
where the constant $C = C(n,s,p)$ does not depend on the set $A$ nor on the radii $\rho,\, \eta$.
\end{lemma}

The next lemma says that a Sobolev map on a ball taking values in a manifold can be extended to a larger ball. This can e.g. be proven by an inversion, setting $v(x) := u(\rho^2 x/|x|^2)$ for $|x| > \rho$. 
\begin{lemma}[{{\cite[Lemma~2.4]{MonteilVanSchaftingen}}}]\label{le:24}
 Let $s\in (0,1]$, $p \ge 1$, $\lambda \ge 1$. There exists a constant $ C>0$ such that if  $\rho>0$, $u\colon B(\rho) \to \n$ is measurable, then there exists $v\colon B(\lambda \rho) \to \n$ such that $v=u$ on $B(\rho)$ and 
 \[
  \| v\|_{L^p(B(\lambda \rho))}^p \le C \|u\|^p_{L^p(B(\rho))}, \quad E_{s,p}(v, B(\lambda \rho)) \le C E_{s,p}(u, B(\rho)).
 \]
 \end{lemma}
 
 Finally, the last lemma is also well known and is often used to remove singularities in  critical Sobolev spaces (not necessarily of fractional order). The lemma basically says that a point in the critical Sobolev space has zero capacity. For the proof we refer, e.g., to \cite[Theorem 5.1.9]{Adams-Hedberg}, compare also with \cite[Lemma~3.2]{MonteilVanSchaftingen}. The proof is based on the existence of unbounded functions in the critical Sobolev space and truncation.
\begin{lemma}\label{le:32}
For any $s\in (0,1)$, $n \geq 1$ there exist $\{\zeta_{\ell}\}_{\ell \in \N} \subset C_c^\infty (\Sigma, [0,1])$ such that for all $\ell \in \N$, \[
\zeta_\ell \equiv 1 \quad \text{on $B(\rho_{\ell})$}, \quad \quad  \zeta_\ell \equiv 0 \quad \text{ outside } B(R_\ell)                                                                                                                                                                                                                                                                                        \]
for some $0<\rho_\ell < R_\ell\to 0$ as $\ell \to \infty$ and
\[
 \lim_{\ell \to \infty} E_{s,\frac ns}(\zeta_\ell, \Sigma) =0.
\]
\end{lemma}

We are ready to prove our Theorem. 
\begin{proof}[Proof of Theorem \ref{th:removabilityminimizing}]
We will construct a comparison map, we begin with a modification of $u$. We will simply write here $B(r)$ for $B(x_0,r)$.

\textsc{Step 1.} Let us take $B(\rho_\ell)$ from Lemma \ref{le:32} and extend $u\big\rvert_{B(\rho_\ell)}\colon B(\rho_\ell) \to \n$ with the help of Lemma \ref{le:24}. We know that there exists $u_1\colon B(3\rho_\ell) \to \n$ such that $u_1 = u$ on $B(\rho_\ell)$ and 
\begin{equation}\label{eq:u1estimate}
 E_{s,\frac ns}(u_1, B(3\rho_\ell)) \aleq  E_{s,\frac ns}(u, B(\rho_\ell)).
\end{equation}
\textsc{Step 2.} Next we again modify the map $u_1$, in such a way that we obtain a map that is constant outside the ball $B(4\rho_\ell)$. Take $\phi_1\colon B(6\rho_\ell) \to B(2\rho_\ell)$ Lipschitz continuous such that
\[
 \begin{split}
  \phi_1(x) &=x \quad \text{if } |x| \le 2\rho_\ell\\
  \phi_1(x) &=0 \quad \text{if } |x| \ge 3 \rho_\ell.
 \end{split}
\]
Then, by \Cref{le:21} there exists an $a_1\in B(\rho_\ell)$ such that
\begin{equation}\label{eq:u_1lipschitzcomposition}
 E_{s,\frac ns}(u_1 \circ (\phi_1(\cdot -a_1)+a_1),B(5\rho_\ell)) \le C\lip(\phi_1)^{n} E_{s,\frac ns}( u_1, B(3\rho_\ell)) \aleq E_{s,\frac ns}( u_1, B(3\rho_\ell))
\end{equation}
and we have

\begin{enumerate}
 \item if $|x| \le \rho_\ell$, then $|x-a_1|\le 2 \rho_\ell$,
 \item if $|x|\ge 4\rho_\ell$, then $|x-a| \ge 3 \rho_\ell$.
\end{enumerate}
Thus,
\[
 \phi_1(x-a_1) + a_1 = \left\{ \begin{array}{ll}
                                x & \text { if } |x| \le \rho_\ell,\\
                                a_1 & \text{ if } |x| \ge 4 \rho_\ell
                               \end{array}
\right.
\]
and
\[
 u_1 \circ(\phi_1(\cdot -a_1)+a_1)(x) = \left\{ \begin{array}{ll}
                                                 u_1(x) = u(x) & \text{ if } |x| \le \rho_\ell,\\
                                b_1 & \text{ if } |x| \ge 4 \rho_\ell,
                                                \end{array}
\right.
\]
where $b_1 \coloneqq u_1(a_1) \in \n$. We define 
\[
 u_2(x) \coloneqq \left\{ \begin{array}{ll}
                   u_1 \circ (\phi_1(x -a_1)+a_1 & \text{ if } |x|\le 4 \rho_\ell\\
                   b_1 & \text{ if } |x|\ge 4\rho_\ell.
                  \end{array}
\right.
\]
Combining Lemma \ref{le:22} (applied with $A = \Sigma$, $\rho = 5 \rho_\ell$, and $\eta = \frac 45$) with \eqref{eq:u_1lipschitzcomposition} and \eqref{eq:u1estimate} we get
\begin{equation}\label{eq:u2estimate}
\begin{split}
 E_{s,\frac ns}(u_2, \Sigma) 
 &\aleq  E_{s,\frac ns} (u_1 \circ (\phi_1(\cdot -a_1)+a_1), B(5\rho_\ell))\\
 &\aleq E_{s,\frac ns} (u, B(\rho_\ell)). 
\end{split}
 \end{equation}

\textsc{Step 3.} Now we modify the map $u_2$ in such a way that it connects on an annulus the constant $b_1\in \n$ with another constant $b_2\in\n$. The newly obtained map is again constant outside a bigger ball $B(R_\ell)$.

Since $\n$ is connected we know that there is a Lipschitz continuous map such that $\gamma \colon [0,1]\to \n$, $\gamma(0) = b_2$, where $b_2\in\n$ is point that will be chosen later, $\gamma(1) = b_1$, and the Lipschitz constant satisfies $\lip(\gamma)\le 2 d_{\n}(b_1,b_2)$. Then, 
\[
 \gamma \circ \zeta_\ell \colon \Sigma \to \n,
\]
where $\zeta_\ell$ is taken from Lemma \ref{le:32} (we just replaced $\rho_\ell$ by $6\rho_\ell$). For this function we have
\[
 \gamma \circ \zeta_\ell = b_1 \text{ on } B(6\rho_\ell), \quad \gamma \circ \zeta_\ell (x) = b_2 \text{ on } \Sigma\setminus B(R_\ell).
\]
and
\begin{equation}\label{eq:Lipschitzgamma}
 E_{s,\frac ns}(\gamma \circ \zeta_\ell,\Sigma) \le \lip(\gamma)^\frac{n}{s} E_{s,\frac ns}(\zeta_\ell,\Sigma) \le 2 d_{\n}(b_1,b_2)^\frac{n}{s} E_{s,\frac ns}(\zeta_\ell,\Sigma) \aleq  E_{s,\frac ns}(\zeta_\ell,\Sigma),
\end{equation}
where the last constant depends only on the manifold $\n$. 

We note that for sufficiently large $\ell$ we have $B(6\rho_\ell)\subset B(R_\ell)$. We define $u_3\colon \Sigma \to \n$
\[
 u_3(x) \coloneqq \left\{ \begin{array}{ll}
                     u_2(x) & \text{ if } x\in B(5\rho_\ell)\\
                     \gamma \circ \zeta_\ell(x) &\text{ if } x\in \Sigma\setminus B(5\rho_\ell).
                    \end{array}
 \right.
\]
Then, by Lemma \ref{le:22} (applied with $A = \Sigma$, $\rho = 6 \rho_\ell$, $\eta = \frac56$)
\[
 E_{s,\frac ns}(u_3, \Sigma) \aleq  E_{s,\frac ns}(u_2, B(6\rho_\ell)) + E_{s,\frac ns}(\gamma \circ \zeta_\ell, \Sigma).
\]
Which, combining with \eqref{eq:Lipschitzgamma} and \eqref{eq:u2estimate}, gives
\begin{equation}\label{eq:u3estimate}
 E_{s,\frac ns}(u_3, \Sigma) \aleq E_{s,\frac ns}(u, B(\rho_\ell)) +  E_{s,\frac ns}(\zeta_\ell, \Sigma). 
\end{equation}

\begin{center}
\begin{tikzpicture}[line cap=round,line join=round,>=triangle 45,x=1cm,y=1cm]
\fill[fill=p3] (-5,-5) rectangle (5,5);
\filldraw[fill=p4, draw=black] (0,0) circle (3cm);
\filldraw[fill=p5, draw=black] (0,0) circle (1.5cm);
\draw[pattern=dots] (0,0) circle (1.5cm);
\draw[pattern=grid] (0,0) circle (1.25cm);
\filldraw[fill=p2, draw=black] (0,0) circle (1cm);
\filldraw[fill=p1, draw=black] (0,0) circle (0.25cm);
\draw (0,0) circle (3cm);
\draw (0,0) circle (0.25cm);
\draw (0,0) circle (1.5cm);
\node[anchor=east] at (-2.9,0) {$R_\ell$};
\node[anchor=east] at (-1.4,0) {$6\rho_\ell$};
\node[left] at (-0.15,0) {$\rho_\ell$};
\node[below, scale=0.9] at (0,-4) {$\Sigma$};
\filldraw[fill=p1, draw=black] (6.25,-1) rectangle (6.5,-1.25);
\node[right,scale=1] at (6.5,-1.125) {$u$};
\filldraw[fill=p2, draw=black] (6.25,-1.5) rectangle (6.5,-1.75);
\node[right,scale=1] at (6.5,-1.625) {connection from $u$ to $b_1$};
\filldraw[fill=p5, draw=black] (6.25,-2) rectangle (6.5,-2.25);
\node[right,scale=1] at (6.5,-2.125) {$b_1$};
\filldraw[pattern=grid, draw=black] (6.25,-2.5) rectangle (6.5,-2.75);
\node[right,scale=1] at (6.5,-2.625) {''buffer zone`` for $u_3$};
\filldraw[pattern=dots, draw=black] (6.25,-3) rectangle (6.5,-3.25);
\node[right,scale=1] at (6.5,-3.125) {''buffer zone`` for $u_2$};
\filldraw[fill=p4, draw=black] (6.25,-3.5) rectangle (6.5,-3.75);
\node[right,scale=1] at (6.5,-3.625) {connection from $b_1$ to $b_2$};
\filldraw[fill=p3, draw=black] (6.25,-4) rectangle (6.5,-4.25);
\node[right,scale=1] at (6.5,-4.125) {$b_2$};
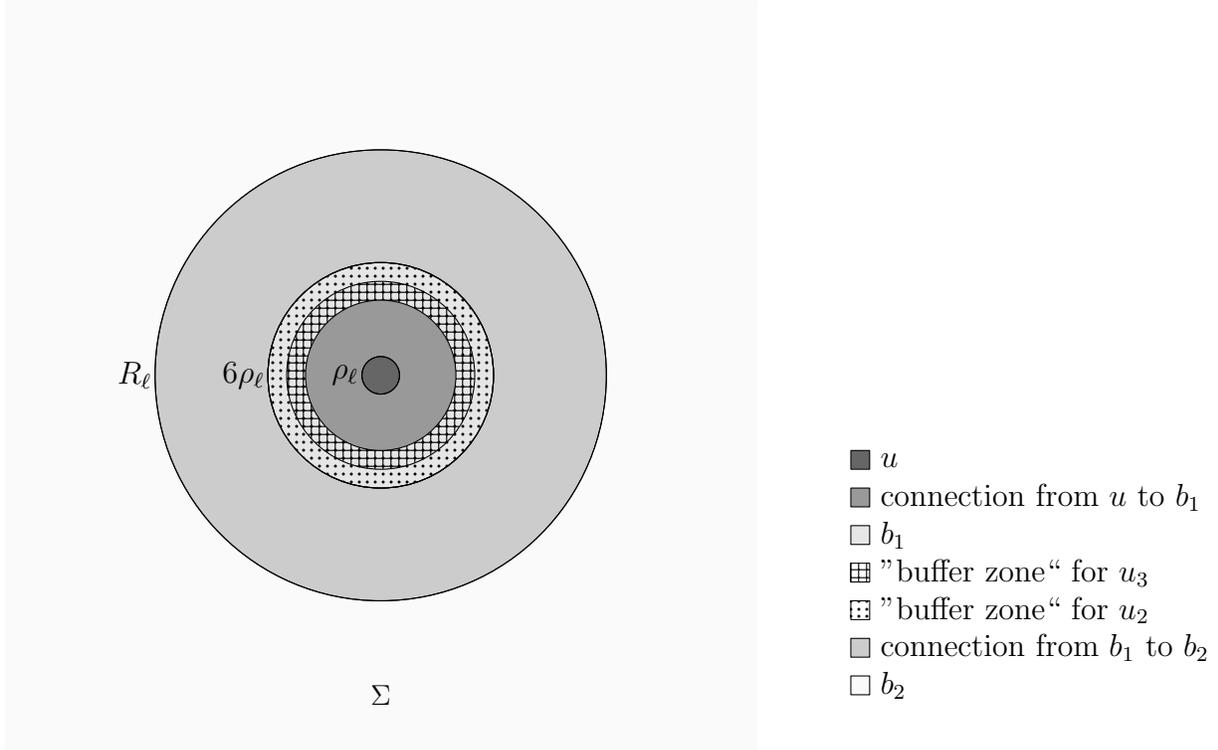
\node[below] at (0,-5) {The domain of the  map $u_3$}; 
\end{tikzpicture}
\end{center}

\textsc{Step 4. } 
Let $v\in W^{s,\frac ns}(\Sigma,\n)$ be any map such that $v\sim u$. We will modify $v$ in such a way that we will be able to compare the energy of the modified $v$ with $u$.

Let $\phi_2 \colon B(9 R_\ell) \to B(9 R_\ell)$ be a Lipschitz continuous function, such that $\phi_2(x) = x$ if $|x|\ge 5 R_\ell$, $\phi_2(x)=0$ if $|x|\le 3 R_\ell$. Then by Lemma \ref{le:21} with $\rho = 8 R_\ell$, $\lambda = \frac{10}{8}$, $\eta = \frac 18$, we obtain an existence of a point $a_2 \in B_{R_\ell}$ such that
\[
 E_{s, \frac ns} (v \circ (\phi_2(\cdot - a_2) +a_2), B_{8 R_\ell}) \le C E_{s,\frac ns} (v, B_{10 R_\ell}).
\]
We also have
\[
 \phi_2(x - a_2) +a_2 = \left\{ \begin{array}{ll}
                                 x & \text{ if } |x|\ge 6 R_\ell,\\
                                 a_2 & \text{ if } |x| \le 2 R_\ell.
                                \end{array}
 \right.
\]
Now we chose the point $b_2$ from Step 3 to be $b_2 \coloneqq v(a_2)$. Thus, 
\[
 v(\phi_2(x - a_2) +a_2) = \left\{ \begin{array}{ll}
                                 v(x) & \text{ if } |x|\ge 6 R_\ell,\\
                                 b_2 & \text{ if } |x| \le 2 R_\ell.
                                \end{array}
 \right.
\]
Finally, we define
\[
 \widetilde{v}_\ell(x) = \left \{ \begin{array}{ll}
                  v(x) & |x| \ge 8 R_\ell,\\
                  v \circ (\phi_2(\cdot - a_2) +a_2) & 2 R_\ell \le |x| \le 8 R_\ell, \\
                  u_3(x) & |x|\le 2 R_\ell.
                 \end{array}
 \right.
\]
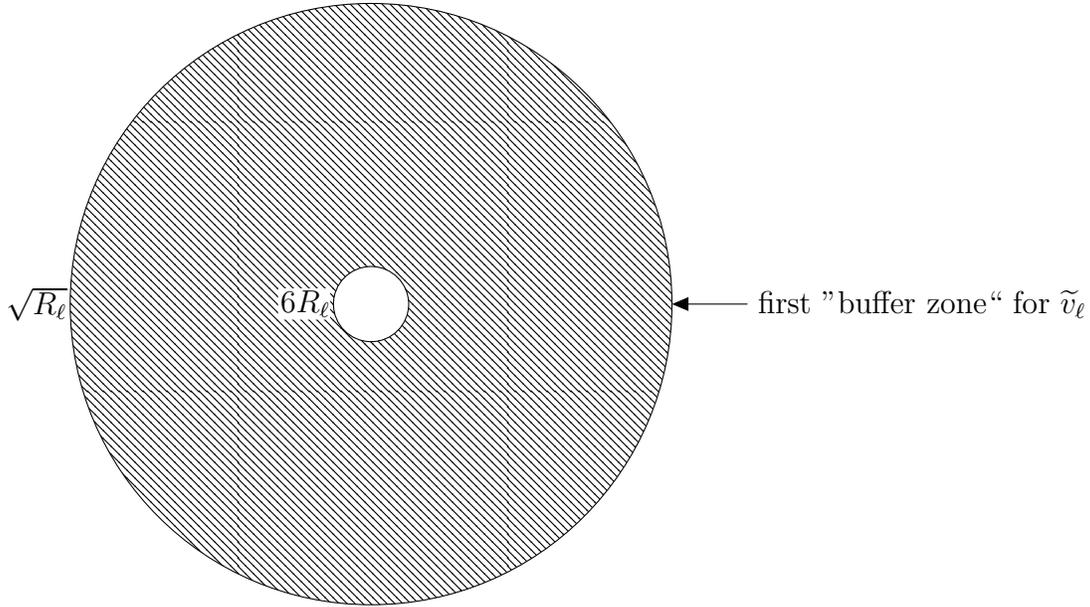
\begin{center}
\begin{tikzpicture}[line cap=round,line join=round,>=triangle 45,x=1cm,y=1cm]
\draw[pattern=north west lines] (0,0) circle (4cm);
\filldraw[fill=white, draw=black] (0,0) circle (0.5cm);
\filldraw[fill=white, draw=white] (-0.55,-0.2) rectangle (-1.23,0.2);
\node[anchor=east] at (-0.4,0) {${6R_\ell}$};
\node[left] at (-3.9,0) {$\sqrt{R_{\ell}}$};
\draw (4,0) edge[<-] node[scale=1] {} (5,0); 
\node[right] at (5,0) {first ''buffer zone`` for $\widetilde v_\ell$};
\end{tikzpicture}
\end{center}
We apply Lemma \ref{le:22} with $A=\Sigma$, $\rho = \sqrt{R_\ell}$, and $\eta = 6\sqrt{R_\ell}$, for sufficiently large $\ell$ we know $B(6R_\ell)\subset B(\sqrt{R_\ell})$ to obtain
\begin{equation}\label{eq:vtildefirstlem22}
\begin{split}
 E_{s,\frac{n}{s}}(\widetilde v_\ell, \Sigma) &\le \brac{1+\frac{C_1}{(1-6\sqrt{R_\ell})^{n+1}}}E_{s,\frac{n}{s}}(\widetilde v_\ell, B(\sqrt{R_\ell}))\\
 &\quad + \brac{1+\frac{C_1(6\sqrt{R_\ell})^n}{1-6\sqrt{R_\ell}}}E_{s,\frac{n}{s}}(\widetilde v_\ell, \Sigma\setminus B(6R_\ell)).
 \end{split}
\end{equation}
We note that $\widetilde v_\ell = v$ for $x\in \Sigma\setminus B(6R_\ell)$, so $E_{s,\frac{n}{s}}(\widetilde v_\ell, \Sigma\setminus B(6R_\ell))= E_{s,\frac{n}{s}}(v, \Sigma\setminus B(6R_\ell))$.

Next, we apply twice again Lemma \ref{le:22} to deal with the term $E_{s,\frac{n}{s}}(\widetilde v_\ell, B(\sqrt{R_\ell}))$. For the first application we take $A = B(\sqrt{R_\ell})$, $\rho = 8 R_\ell$, $\eta = \frac 34$ and  for the second $A = B(8 R_\ell)$, $\rho = 2 R_\ell$, $\eta = \frac12$,  we obtain
\begin{equation}\label{eq:vtildesecondthirdlem22}
\begin{split}
 E_{s,\frac{n}{s}}(\widetilde v_\ell, B(\sqrt{R_\ell})) 
 &\aleq  E_{s, \frac ns} (\widetilde v_\ell, B(8 R_\ell)) + E_{s,\frac ns}(\widetilde v_\ell, B(\sqrt{R_\ell}) \setminus B(6 R_\ell)) \\
 &\aleq  E_{s,\frac ns} (\widetilde v_\ell, B(2 R_\ell)) +  E_{s,\frac ns}(\widetilde v_\ell, B(8 R_\ell) \setminus B(R_\ell)) + E_{s,\frac ns}(\widetilde v_\ell, B(\sqrt{R_\ell})\setminus B(6 R_\ell)).   
\end{split}
 \end{equation}

\begin{center}
\begin{tikzpicture}[line cap=round,line join=round,>=triangle 45,x=1cm,y=1cm]
\draw[pattern=north west lines, pattern color=blue] (0,0) circle (4cm);
\filldraw[fill=white, draw=black] (0,0) circle (3cm);
\draw[pattern=dots, pattern color=red] (0,0) circle (1cm);
\filldraw[fill=white, draw=black] (0,0) circle (0.5cm);
\draw (0,0) circle (0.1cm);
\node[anchor=east] at (-3.9,0) {$8R_\ell$};
\filldraw[fill=white, draw=white] (-3.1,-0.2) rectangle (-3.75,0.2);
\node[anchor=east] at (-2.9,0) {${6R_\ell}$};
\draw (0.1,0) edge[<-] (6,-3);
\node[right] at (6,-3) {$B_{\rho_\ell}$};
\draw (0,1) edge[<-, bend left] node[scale=1] {} (6,2); 
\draw (4,0) edge[<-, bend left] node[scale=1] {} (6,2); 
\node[right] at (6,2) {''buffer zones`` for $\widetilde v_\ell$};
\node[below] at (0,-5) {Domain of the map $\widetilde v_\ell$};
\node[anchor=east] at (-0.4,0) {$R_\ell$};
\node[anchor=east] at (-0.9,0) {$2R_\ell$};
\end{tikzpicture}
\end{center}

Now, we note that $\widetilde v_\ell = v$ on $B(\sqrt{R_\ell})\setminus B(6 R_\ell)$, $\widetilde v_\ell = v \circ (\phi_2(\cdot - a_2)+a_2)$ on $B(8 R_\ell)\setminus B(R_\ell)$, and $\widetilde v_\ell = u_3$ on $B(R_\ell)$. Thus, 
\[
\begin{split}
E_{s,\frac ns} (\widetilde v_\ell, B(2 R_\ell)) &= E_{s,\frac ns} (u_3, B(2 R_\ell)) \\
E_{s,\frac ns}(\widetilde v_\ell, B(8 R_\ell) \setminus B(R_\ell)) &= E_{s,\frac ns}(v \circ (\phi_2(\cdot - a_2)+a_2), B(8 R_\ell) \setminus B(R_\ell))\\
E_{s,\frac ns}(\widetilde v_\ell, B(\sqrt{R_\ell})\setminus B(6 R_\ell)) &= E_{s,\frac ns}(v, B(\sqrt{R_\ell})\setminus B(6 R_\ell)). 
\end{split}
\]
Recall, that from Step 3, inequality \eqref{eq:u3estimate}, we know that
\begin{equation}\label{eq:vtildeu3}
 E_{s,\frac ns} (u_3, B(2 R_\ell)) \le E_{s,\frac ns} (u_3, \Sigma) \aleq  E_{s,\frac ns}(u, B(\rho_\ell)) +  E_{s,\frac ns}(\zeta_\ell, \Sigma). 
\end{equation}
We also have
\begin{equation}\label{eq:vtildemixed}
 E_{s,\frac ns}(v \circ (\phi_2(\cdot - a_2)+a_2) \aleq E_{s,\frac{n}{s}}(v, B(10 R_\ell)).
\end{equation}
Combining \eqref{eq:vtildefirstlem22}, \eqref{eq:vtildesecondthirdlem22}, \eqref{eq:vtildeu3}, and \eqref{eq:vtildemixed}, we get
\begin{equation}\label{eq:vtildeestimate}
\begin{split}
 E_{s,\frac ns} (\widetilde v_\ell, \Sigma) 
 &\le \brac{1+\frac{C_1(6\sqrt{R_\ell})^n}{1-6\sqrt{R_\ell}}}E_{s,\frac{n}{s}}(v, \Sigma\setminus B(6R_\ell))\\
 &\quad + C_2\brac{ E_{s,\frac ns}(u, B_{\rho_\ell}) +  E_{s,\frac ns}(\zeta_\ell, \Sigma) + E_{s,\frac{n}{s}}(v, B(10 R_\ell)) + E_{s,\frac ns}(v, B_{\sqrt{R_\ell}}\setminus B(6 R_\ell))}  
\end{split}
 \end{equation}
for a constant $C_2$ independent of $v$, $u$, $\ell$. 

\textsc{Step 5.} The only thing left to prove is that the map $\widetilde v_\ell$ is a good comparison map. We immediately verify that $\widetilde v_\ell \equiv u$ on $B(\rho_\ell)$. Finally, to show that $\widetilde v_\ell \sim u$ we recall that $v\sim u$ and thus it is enough to show that $\widetilde v_\ell \sim v$. We have
\[
 (v-\widetilde v)(x) = \left\{ \begin{array}{ll}
                                0 & \text{ if } |x|\ge 8 R_\ell\\
                                (v-\widetilde v)(x) & \text{ if } |x| \le 8 R_\ell.
                               \end{array}
\right.
\]
Thus, by Lemma \ref{le:22} we get
\[
 E_{s,\frac ns}(v-\widetilde v_\ell, \Sigma) \aleq E_{s,\frac ns}(v-\widetilde v_\ell, B(8 R_\ell)) \aleq E_{s,\frac ns}(v, B(8 R_\ell)) + E_{s,\frac ns}(\widetilde v_\ell, B(8 R_\ell)).  
\]
By taking $\ell$ large enough we can ensure, by the absolute continuity of the integral that the latter one is smaller than $\eps$, where $\eps$ is taken from Lemma \ref{la:samehomotopy}. Similarly, since $v$ and $\widetilde v$ differ only on a small set we verify that $\|v-\widetilde v_\ell\|_{L^1(\Sigma)}\ll\eps$ for sufficiently large $\ell$. Thus, from Lemma \ref{la:samehomotopy} we deduce that $\widetilde v_\ell \sim v$.

Combining the minimality outside of a point of $u$ with  with \eqref{eq:vtildeestimate} we get
\begin{equation}\label{eq:ucomparedtovtilde}
\begin{split}
 E_{s,\frac ns}(u,\Sigma) &\le E_{s,\frac ns}(\widetilde v_\ell,\Sigma) \\
 &\le \brac{1+\frac{C_1(6\sqrt{R_\ell})^n}{1-6\sqrt{R_\ell}}}E_{s,\frac{n}{s}}(v, \Sigma\setminus B(6R_\ell))\\
 &\quad + C_2\brac{ E_{s,\frac ns}(u, B(\rho_\ell)) +  E_{s,\frac ns}(\zeta_\ell, \Sigma) + E_{s,\frac{n}{s}}(v, B(10 R_\ell)) + E_{s,\frac ns}(v, B(\sqrt{R_\ell})\setminus B(6 R_\ell))}.
\end{split}
 \end{equation}
We observe that as $\ell\to\infty$ we get $\brac{1+\frac{C(6\sqrt{R_\ell})^n}{1-6\sqrt{R_\ell}}} \to 1$ and by the absolute continuity of the integral, since $B(\rho_\ell),\, B(10R_\ell),\,B(\sqrt{R_\ell})\setminus B(6R_\ell)$ shrink to $\{0\}$, we get as $\ell\to\infty$
\[
 E_{s,\frac ns}(u, B(\rho_\ell)) + E_{s,\frac{n}{s}}(v, B(10 R_\ell)) + E_{s,\frac ns}(v, B(\sqrt{R_\ell})\setminus B(6 R_\ell)) \to 0
\]
Finally, by Lemma \ref{le:32} we have
\[
 E_{s,\frac ns}(\zeta_\ell, \Sigma) \to 0.
\] 
Thus, passing with $\ell \to \infty$ in \eqref{eq:ucomparedtovtilde} we get
\[
 E_{s,\frac ns}(u,\Sigma) \le E_{s,\frac ns} (v,\Sigma).
\]
Thus, we can conclude that $u$ is minimizing in all of $\Sigma$ among all maps in the same homotopy class.
\end{proof}


As a corollary of \Cref{th:strongconvoutsideofpoints}, \Cref{th:removabilityminimizing} and then \Cref{th:reghomo} we obtain

\begin{theorem}\label{th:corollaryremovability}
There exists $s_0 > s$ such that the following holds.

Assume that $u_t \in W^{t,\frac{n}{s}}(\Sigma,\n)$ is a sequence of minimizers in a homotopy class $X$ that converges weakly to $u_s \in W^{s,\frac{n}{s}}(\Sigma,\n)$ in the $W^{s,\frac{n}{s}}$-topology.
Then $u_s \in W^{s_0,\frac{n}{s}}(\Sigma,\n)$.
\end{theorem}

We finish this section with a remark. We can remove discrete points in the equation, i.e., once we know that a map satisfies the equation of $W^{s,\frac ns}$-harmonic maps in $\Sigma \setminus A$, where $A$ is a set consisting of finitely many points, we know, that the equation is satisfied in $\Sigma$. Unfortunately, the lack of regularity theory in general does not allow us to conclude that the map is regular everywhere. But, in in view of \cite{S15,MS18,S16} if we have $W^{s,\frac{n}{s}}$-harmonic maps in $\Sigma \setminus A$ which maps into a sphere or a compact Lie group, or in view of \cite{DaLio-Riviere-1Dmfd} if we have $W^{\frac12, 2}$-harmonic maps on a line, we have regularity in all of $\Sigma$.

\begin{proposition}\label{pr:removepde}
Let $A$ be a finite set in $\Sigma$,  and let $u \in W^{s,\frac{n}{s}}$ be a $W^{s,\frac{n}{s}}$-harmonic map outside of $A$, i.e.,
\begin{equation}\label{eq:fractionalharmonicmapequationoutsideofpoints}
 \int_{\Sigma}  \int_{\Sigma} \frac{|u(x)-u(y)|^{\frac{n}{s}-2}((u(x)-u(y))\, \brac{\Pi(u(x)) \varphi(x)-\Pi(u(y)) \varphi(y)}}{|x-y|^{2n}}\dy \dx = 0 \quad \forall \varphi \in C_c^\infty(\Sigma \setminus A),
\end{equation}
where $\Pi(u)$ is the orthogonal projection onto the tangent space of $T_u \n$ for $u\in\n$, see \Cref{la:tubular}.

Then $u$ is a $W^{s,\frac{n}{s}}$-harmonic map in all of $\Sigma$, that is
\[
 \int_{\Sigma}  \int_{\Sigma} \frac{|u(x)-u(y)|^{\frac{n}{s}-2}((u(x)-u(y))\, \brac{\Pi(u(x)) \varphi(x)-\Pi(u(y)) \varphi(y)}}{|x-y|^{2n}}\dx\dy = 0 \quad \forall \varphi \in C_c^\infty(\Sigma ) . 
\]
\end{proposition}
\begin{proof}
For simplicity assume that $A = \{x_0\}$. Let $\varphi \in C^\infty_c(\Sigma)$ and let $\zeta_\ell\in C_c^\infty(B_{R_\ell})$ be as in Lemma~\ref{le:32}, that is
\[
 \zeta_\ell \equiv 1 \text{ on } B_{\rho_\ell}(x_0) \quad \text{ and } \quad [\zeta_\ell]_{W^{s,\frac ns}(\Sigma)} \to 0 \text{ as } \ell\to\infty.
\]
for a sequence $0<\rho_\ell<R_\ell\to 0$ as $\ell\to\infty$.

Thus, $\phi_\ell = \phi(1-\zeta_\ell)\in C_c^\infty(\Sigma\setminus\{x_0\})$ is an admissible test function and from \eqref{eq:fractionalharmonicmapequationoutsideofpoints} we get
\[
\begin{split}
 &\int_{\Sigma}  \int_{\Sigma} \frac{|u(x)-u(y)|^{\frac{n}{s}-2}((u(x)-u(y))\, \brac{\Pi(u(x)) \phi(x)-\Pi(u(y)) \phi(y)}}{|x-y|^{2n}}\dx\dy\\
 &=
 \int_{\Sigma} \int_{\Sigma} \frac{|u(x)-u(y)|^{\frac{n}{s}-2}((u(x)-u(y))\, \brac{\Pi(u(x)) \phi(x) \zeta_\ell(x)-\Pi(u(y)) \phi(y)\zeta_\ell(y)}}{|x-y|^{2n}}\dx\dy.
\end{split}
 \]
The latter one can be estimated in the following way.
\[
\begin{split}
 &\int_{\Sigma} \int_{\Sigma} \frac{|u(x)-u(y)|^{\frac{n}{s}-2}((u(x)-u(y))\, \brac{\Pi(u(x)) \phi(x) \zeta_\ell(x)-\Pi(u(y)) \phi(y)\zeta_\ell(y)}}{|x-y|^{2n}}\dx\dy\\ 
 &\le  \int_{\Sigma} \int_{\Sigma} \frac{|u(x)-u(y)|^{\frac{n}{s}-2}((u(x)-u(y))\, \Pi(u(x)) \phi(x)\brac{ \zeta_\ell(x)-\zeta_\ell(y)}}{|x-y|^{2n}}\dx\dy\\
 &\quad + \int_{\Sigma} \int_{\Sigma} \frac{|u(x)-u(y)|^{\frac{n}{s}-2}((u(x)-u(y))\, {\brac{\Pi(u(x))\phi(x)-\Pi(u(y))\phi(y)}\zeta_\ell(y)}}{|x-y|^{2n}}\dx\dy.
\end{split}
 \]

As for the first term we have by H\"older's inequality
\[
\begin{split}
 & \int_{\Sigma} \int_{\Sigma} \frac{|u(x)-u(y)|^{\frac{n}{s}-2}((u(x)-u(y))\, \Pi(u(x)) \phi(x)\brac{ \zeta_\ell(x)-\zeta_\ell(y)}}{|x-y|^{2n}}\dx\dy\\
 &\le  \brac{\int_{\Sigma} \int_{\Sigma} \frac{|u(x)-u(y)|^{\frac{n}{s}}}{|x-y|^{2n}}\dx \dy}^{\frac{n-s}{n}}[\zeta_\ell]_{W^{s,\frac{n}{s}}(\Sigma)} \xrightarrow{\ell\to\infty} 0.
\end{split}
 \]
As for, the second term we have
\[
\begin{split}
 &\int_{\Sigma} \int_{\Sigma} \frac{|u(x)-u(y)|^{\frac{n}{s}-2}((u(x)-u(y))\, {\brac{\Pi(u(x))\phi(x)-\Pi(u(y))\phi(y)}\zeta_\ell(y)}}{|x-y|^{2n}}\dx\dy\\
&\le \int_{B_{R_\ell}} \int_{\Sigma} \frac{|u(x)-u(y)|^{\frac{n}{s}-1}\abs{\Pi(u(x))\phi(x)-\Pi(u(y))\phi(y)}}{|x-y|^{2n}}\dx\dy \xrightarrow{\ell\to\infty} 0,
\end{split}
 \]
 by the absolute continuity of the integral. 
 
Hence, $u$ is a $W^{s,\frac{n}{s}}$-harmonic map, as
\[
\begin{split}
 &\int_{\Sigma} \int_{\Sigma} \frac{|u(x)-u(y)|^{\frac{n}{s}-2}((u(x)-u(y))\, \brac{\Pi(u(x)) \phi(x) -\Pi(u(y)) \phi(y)}}{|x-y|^{2n}}\dx\dy = 0
\end{split}
 \]
 for any $\phi\in C^\infty_c(\Sigma)$.
\end{proof}

\section{Balanced energy estimate for the non-scaling invariant norms}\label{s:balancing}
In this section we show the main advantage of approximating $W^{s,\frac{n}{s}}$-minimizers by $W^{t,\frac{n}{s}}$-minimizers. It does not avoid energy concentration in a single point, but energy cannot concentrate only in one point and vanish everywhere else. In some sense, the energy needs to be balanced. 

We will use \Cref{th:fractionalestimatesmalldiskbyremaining} {\it stellvertretend} for \cite[Lemma~5.3]{Sucks1} in our argument.

\begin{theorem}\label{th:fractionalestimatesmalldiskbyremaining}
Let $0 < s < s_0 < 1$ and $\rho_0 \in (0,\sqrt{\frac{4}{5}})$. There exists a constant $C=C(s,s_0,\rho_0)$ such that the following holds.

For any $t \in (s,s_0]$ let $u_t \in W^{t,\frac{n}{s}}(\S^n,\n)$ be a minimizing map in its own homotopy group. Then for any $y_0\in\S^n$ 
\begin{equation}\label{eq:noscaleinvarianceestimate}
 \int_{D(y_0,\rho)} \int_{\S^n} \frac{|u_t(x) - u_t(y)|^{\frac{n}{s}}}{|x-y|^{n+\frac{tn}{s}}} \dx\dy \le C\, \rho^{-n (\frac{t}{s}-1)}  \int_{\S^n \setminus D(y_0,\rho)}\int_{\S^n} \frac{|u_t(x) - u_t(y)|^{{\frac{n}{s}}}}{|x-y|^{n+\frac{tn}{s}}} \dx\dy. 
\end{equation}
Here, $D(a,r)\coloneqq B(a,r)\cap \S^n$ is the intersection of a ball centered at $a\in\S^n$ of radius $r$ intersected with the sphere. 
\end{theorem}

The arguments for \Cref{th:fractionalestimatesmalldiskbyremaining} carry over to the $W^{1,p}$-case, and for future reference we record

\begin{theorem}\label{th:fractionalestimatesmalldiskbyremainingW1p}
Let $p_0 \in (n,\infty)$ and $\rho_0 \in (0,\sqrt{\frac{4}{5}})$. There exists a constant $C=C(n,p_0,\rho_0)$ and some $\sigma > 0$ such that the following holds.

For any $p \in (n,p_0]$ let $u_p \in W^{1,p}(\S^n,\n)$ be a minimizing map in its own homotopy group. Then for any $y_0\in\S^n$ 
\[
 \int_{D(y_0,\rho)}  |\nabla u_p|^p \dx \le C\, \rho^{-\sigma(p-n)}  \int_{\S^n \setminus D(y_0,\rho)} |\nabla u_p|^p \dx. 
\]
Here, $D(a,r)\coloneqq B(a,r)\cap \S^n$ is the intersection of a ball centered at $a\in\S^n$ of radius $r$ intersected with the sphere. 
\end{theorem}

\begin{remark}
We are not aware of results similar to \Cref{th:fractionalestimatesmalldiskbyremainingW1p} or \Cref{th:fractionalestimatesmalldiskbyremaining} in the literature. However, it seems that a somewhat similar effect is underlying the arguments in the recent work by Lamm--Malchiodi--Micallef \cite{LMM19}.
\end{remark}

In order to prove Theorem \ref{th:fractionalestimatesmalldiskbyremaining} we will use the minimizing property of the mapping $u_t$ and compare its energy with a ''rescaled" version of $u_t$. In order to do so we will first change the coordinates into the spherical coordinates, then we will use the stereographic projection of the sphere and map the $n$-sphere to the hyperplane. Finally on the hyperplane we define the rescaling, which in the polar coordinates $(r,\omega)$, for $r>0,\ \omega\in\S^{n-1}$, on $\R^n$ is given simply by $r\mapsto \lambda r $, with a parameter $\lambda>1$.

As a quick motivation for using as the comparison map the rescaling we note that in the simple case, when we consider a minimizing map $v\in W^{t,\frac ns}(\R^n,\n)$ we get immediately by comparing with the rescaled map $v_\lambda\coloneqq v(\lambda \cdot)$ the following
\[
\begin{split}
 \int_{\R^n} \int_{\R^n} \frac{|v(x) - v(y)|^\frac ns}{|x-y|^{n+\frac{tn}{s}}} \dx \dy &\le \int_{\R^n} \int_{\R^n} \frac{|v_\lambda(x) - v_\lambda(y)|^\frac ns}{|x-y|^{n+\frac{tn}{s}}} \dx \dy\\
 &= \lambda^{n\brac{\frac{t-s}{s}}} \int_{\R^n} \int_{\R^n} \frac{|v(x) - v(y)|^\frac ns}{|x-y|^{n+\frac{tn}{s}}} \dx \dy,
 \end{split}
 \]
which is possible only if $v\equiv const$. Here, we emphasize that the last equality is true, because the energy is not scaling invariant and thus, if we would replace $E_{t,\frac ns}$ by $E_{s,\frac ns}$ there would not be an extra $\lambda$ term in front of the integral. 

Similarly in the case, when the the domain of the minimizing map is $\S^n$, since the $E_{t,\frac{n}{s}}$-energy is not conformally invariant, we already get an ``extra" term using the stereographic projection. Then again, an additional term appears after rescaling. Those extra terms can be estimated, accordingly to one of the three cases: integration over two balls $D(\rho)\times D(\rho)$, integration over the complements of the balls $\S^n \setminus D(\rho) \times \S^n \setminus D(\rho)$, and the mixed term $D(\rho)\times \S^n\setminus D(\rho)$. Which after a careful comparison of the energies gives the desired conclusion. 

The rescaling is performed in the following proposition which might be of independent interest.

\begin{proposition}\label{pr:rescaling}
Let $v\colon \S^n \to \R^M$, $n \geq 1$ and $\lambda > 0$.

If $n=1$, then we let $\tau\colon \R \to \S^1$ to be the inverse stereographic projection, namely
\[
 \tau(r) \coloneqq \brac{\frac{2r}{r^2+1},\frac{r^2-1}{r^2+1}}
\]
and set $v_\lambda \coloneqq v(\tau(\lambda\tau^{-1}(x)))$.

If $n \geq 2$, then we write $v=v(r,\omega)$, $r > 0$, $\omega \in S^{n-1}$ in terms of the usual stereographic projection (see below) of the punctured sphere $\S^n\setminus \{N\}$, where $ N\coloneqq (1,0,\ldots,0)\in\R^n$ is the north pole. In this case we set $v_\lambda \coloneqq v(\lambda r,\omega)$.

In both cases, for $r_{\lambda} \coloneqq \sqrt{\frac{4}{\lambda^2+1}}$, we have 
\begin{equation}\label{eq:rescaledenergyestimated:goal}
 \begin{split}
   \int_{\S^n}\int_{\S^n} \frac{|v_\lambda(x) - v_\lambda(y)|^\frac{n}{s}}{|x-y|^{n+\frac{tn}{s}}} \dx\dy 
  &\le (2\lambda)^{n\brac{\frac{t}{s}-1}} \int_{D(S,r_\lambda)} \int_{\S^n}\frac{|v(x) - v(y)|^\frac{n}{s}}{|x-y|^{n+\frac{tn}{s}}}\dx\\dy\\   
 &\quad  + \brac{\frac{2}{\lambda}}^{n\brac{\frac{t}{s}-1}} \int_{\S^n\setminus D(S,r_\lambda)} \int_{\S^n}\frac{|v(x) - v(y)|^\frac{n}{s}}{|x-y|^{n+\frac{tn}{s}}}\dx\dy,
  \end{split}
\end{equation}
where $S=(-1,0,\ldots,0)\in\S^n$ is the south pole.
\end{proposition}

The proofs of \Cref{pr:rescaling} are only slightly different for $n=1$ and $n\geq 2$. However, since they are very technical we give both of them in full detail. 
\begin{proof}[Proof of \Cref{pr:rescaling} for $n=1$]
Recall that for $\lambda > 0$ we have set 
\[
 v_\lambda(x) \coloneqq v(\tau(\lambda\tau^{-1}(x))).
\]
Here $\tau\colon \R \to \S^1$ is the inverse stereographic projection, namely
\[
 \tau(r) \coloneqq \brac{\frac{2r}{r^2+1},\frac{r^2-1}{r^2+1}}.
\]
Observe that
\[
\begin{split}
  |\tau(r)-\tau(R)|^2 
= \brac{R-r}^2 \frac{2}{R^2+1} \frac{2}{r^2+1},
 \end{split}
 \]
from which one obtains
\[
 |\tau'(r)| = \frac{2}{r^2+1}.
\]
Then, changing the variables, we compute
\[
\begin{split}
 [v_\lambda]_{W^{t,\frac 1s}(\S^1)}^\frac{1}{s}
 &=\int_{\S^1}\int_{\S^1} \frac{|v_\lambda(x)-v_\lambda(y)|^{\frac{1}{s}}}{|x-y|^{1+\frac{t}{s}}}\dx \,\dy\\
 &=\int_{\R}\int_{\R} \frac{|v(\tau(\lambda r))-v(\tau(\lambda R))|^{\frac{1}{s}}}{|\tau(r)-\tau(R)|^{1+\frac{t}{s}}}\, |\tau'(R)| |\tau'(r)|\dif R\, \dif r\\
 &=\int_{\R}\int_{\R} \frac{|v(\tau(\tilde{r}))-v(\tau(\tilde{R}))|^{\frac{1}{s}}}{|\tau(\lambda^{-1} \tilde{r})-\tau(\lambda^{-1} \tilde{R})|^{1+\frac{t}{s}}}\, |\tau'(\lambda^{-1} \tilde{R})| |\tau'(\lambda^{-1} \tilde r)| \lambda^{-2} \dif\tilde{R}\, \dif\tilde{r}\\
 &=\int_{\R}\int_{\R} \frac{|v(\tau(\tilde{r}))-v(\tau(\tilde{R}))|^{\frac{1}{s}}}{|\tau(\tilde{r})-\tau(\tilde{R})|^{1+\frac{t}{s}}} |\tau'(\tilde{R})| |\tau'(\tilde{r})| K_{\lambda}(\tilde{r},\tilde{R}) \dif \tilde R\, \dif \tilde r,
 \end{split}
\]
where 
\[
\begin{split}
K_{\lambda}(\tilde{r},\tilde{R}) 
&\coloneqq  \lambda^{-2} \brac{\frac{|\tau(\tilde{r})-\tau( \tilde{R})|}{|\tau(\lambda^{-1} \tilde{r})-\tau(\lambda^{-1} \tilde{R})|}}^{1+\frac{t}{s}}\, \frac{|\tau'(\lambda^{-1} \tilde{R})| |\tau'(\lambda^{-1} \tilde{r})|}{|\tau'(\tilde{R})| |\tau'(\tilde{r})|}\\
 &=\lambda^{\frac{t}{s}-1} \brac{\frac{
 ((\lambda^{-1} \tilde{R})^2+1) ((\lambda^{-1} \tilde{r})^2+1)
 }{
 \brac{\tilde{R}^2+1} \brac{\tilde{r}^2+1}
  }}^{\frac{\frac{t}{s}-1}{2}}\\
 &=\brac{\frac{ \tilde{r}^2+\lambda^{2} }{
 \lambda\brac{\tilde{r}^2+1}  }}^{\frac{\frac{t}{s}-1}{2}}
 \brac{\frac{ \tilde{R}^2+\lambda^{2} }{
 \lambda\brac{\tilde{R}^2+1}  }}^{\frac{\frac{t}{s}-1}{2}}.
   \end{split}
\]

Observe that for $|\tilde{r}| \le \lambda$ we have 
\[
 \frac{ \tilde{r}^2+\lambda^{2} }{
 \lambda\brac{\tilde{r}^2+1}  } \leq 2\lambda
\]
and for $|\tilde{r}| \ge \lambda$
\[
 \frac{ \tilde{r}^2+\lambda^{2} }{
 \lambda\brac{\tilde{r}^2+1}  } \leq \frac{2}{\lambda}.
\]
Thus,
\[
 K_{\lambda}(\tilde{r},\tilde{R}) \leq
 \begin{cases}
        (2\lambda)^{\frac{t}{s}-1}  \quad &|\tilde{r}|\leq\lambda,\ |\tilde{R}|\leq\lambda\\
        \frac{1}{2} (2\lambda)^{\frac{t}{s}-1} + \frac{1}{2}\brac{\frac{2}{\lambda}}^{\frac{t}{s}-1} \quad &|\tilde{r}|\geq\lambda,\ |\tilde{R}|\leq\lambda \text{ or } 
        |\tilde{r}|\leq\lambda,\ |\tilde{R}|\geq\lambda\\
        \brac{\frac{2}{\lambda}}^{\frac{t}{s}-1}  \quad &|\tilde{r}|\geq\lambda,\ |\tilde{R}|\geq\lambda.
\end{cases}
\]
In the case where $|\tilde{r}|\geq\lambda$ and $|\tilde{R}|\leq\lambda$ we have used the inequality $2ab \leq a^2 + b^2$.

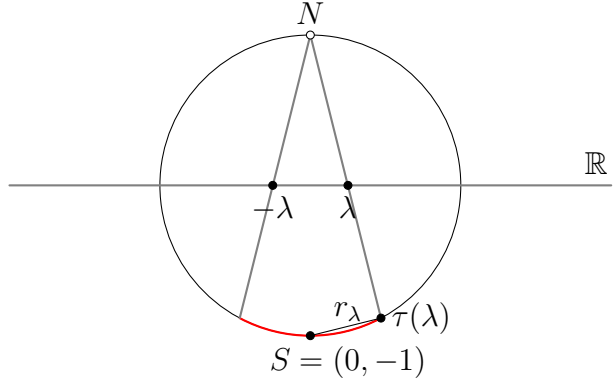
\begin{center}
\begin{tikzpicture}[line cap=round,line join=round,>=triangle 45,x=1cm,y=1cm]
\draw (0,0) circle (2cm);
\draw[red, thick] (242:2cm) arc (242:298:2);
\filldraw[black] (0,-2) circle (1.5pt) node[below,shift={(0.5,0)}] {$S=(0,-1)$};
\draw[gray, thick] (-4,0) -- (4,0);
\node[anchor=south] at (3.8,0) {$\R$};
\draw[gray, thick] (0,2) -- (242:2cm);
\draw[gray, thick] (0,2) -- (298:2cm);
\draw (0,-2) -- (298:2cm);
\filldraw[black] (-0.5,0) circle (1.5pt) node[anchor=north] {$-\lambda$};
\filldraw[black] (0.5,0) circle (1.5pt) node[anchor=north] {$\lambda$};
\node[anchor=south] at (0.5,-1.95) {$r_\lambda$};
\filldraw[black] (298:2cm) circle (1.5pt) node[right] {$\tau(\lambda)$};
\filldraw[fill=white, draw=black] (0,2) circle (1.5pt) node[anchor=south] {$N$};
\end{tikzpicture}
\end{center}

Observe that $\tau((-\lambda,\lambda)) = D(S,r_{\lambda})$ for $r_{\lambda} \coloneqq \sqrt{\frac{4\lambda^2}{\lambda^2+1}}$ and $\tau^{-1}(\R \setminus (-\lambda,\lambda)) = D(S,r_{\lambda})$, where $S = (0,-1)$ is the south pole of $\S^1$.

We thus conclude 
\[
\begin{split}
 \int_{\S^1}\int_{\S^1} \frac{|v_\lambda(x)-v_\lambda(y)|^{\frac{1}{s}}}{|x-y|^{1+\frac{t}{s}}} \dx \dy
 &\leq (2\lambda)^{\frac{t}{s}-1} \int_{D(S,r_{\lambda})}\int_{\S^1} \frac{|v(x)-v(y)|^{\frac{1}{s}}}{|x-y|^{1+\frac{t}{s}}} \dx\, \dy\\
 &\quad +\brac{\frac{2}{\lambda}}^{\frac{t}{s}-1} \int_{\S^1 \setminus D(S,r_{\lambda})}\int_{\S^1} \frac{|v(x)-v(y)|^{\frac{1}{s}}}{|x-y|^{1+\frac{t}{s}}} \dx\, \dy.
 \end{split}
\]
That is, \eqref{eq:rescaledenergyestimated:goal} is established and the proof of \Cref{pr:rescaling} for $n=1$ is finished.
\end{proof}

\begin{proof}[Proof of \Cref{pr:rescaling} for $n \geq 2$]
 We begin with introducing spherical coordinates. Since we are dealing with a double integral we will need separate coordinates to represent a point $x=(x_1,x_2\ldots, x_{n+1}) \in\S^n$ and to represent a point $y=(y_1,y_2,\ldots y_{n+1})\in\S^{n}$.  For $x$ we will use the coordinates 
\[
\begin{split}
(\vp,\omega), \quad &\text{ where } \vp\in (0,\pi),\ \omega\in\S^{n-1} \quad \text{ or }\\
( \vp,\omega_2,\ldots \omega_n), \quad &\text{ where } \vp,\, \omega_2,\,\ldots,\, \omega_{n-1}\in (0,\pi),\ \omega_n\in(0,2\pi) 
\end{split}
\]
whereas for $y$ we will use 
\[
\begin{split}
(\psi,\theta), \quad &\text{ where } \psi\in(0,\pi),\ \theta\in\S^{n-1} \quad \text{ or }\\
(\psi, \theta_2,\ldots,\theta_n)\quad &\text{ where } \psi,\, \theta_2,\ldots,\theta_{n-1}\in(0,\pi),\ \theta_{n}\in(0,2\pi).  
\end{split}
\]
The spherical coordinates are given by
\[
\begin{array}{rlrl}
 x_1 =& \cos \vp,   &y_1 =& \cos \psi,\\
 x_2 =& \sin \vp\cos \omega_2, & y_2 =& \sin \psi\cos \theta_2,\\
 x_3 =& \sin \vp \sin\omega_2 \cos \omega_3, & y_3 =& \sin \psi \sin\theta_2 \cos \theta_3,\\
 \vdots& & \vdots& \\
 x_n =& \sin\vp \sin \omega_2 \ldots \sin \omega_{n-1}\cos\omega_n, & y_n =& \sin\psi \sin \theta_2 \ldots \sin \theta_{n-1}\cos\theta_n,\\
 x_{n+1} =& \sin\vp \sin \omega_2 \ldots \sin \omega_{n-1} \sin\omega_n, & y_{n+1} =& \sin\psi \sin \theta_2 \ldots \sin \theta_{n-1} \sin\theta_n. 
\end{array}
\]
We recall that the volume element is given by 
\[
\begin{split}
 \dx &= \sin^{n-1}(\vp) \sin^{n-2}(\omega_2) \ldots \sin(\omega_{n-1})\dif \vp\dif \omega_2\,\ldots\dif \omega_n = \sin^{n-1}(\vp)\dif \vp\dif {\omega}\\
 \dy &= \sin^{n-1}(\psi) \sin^{n-2}(\theta_2) \ldots \sin(\theta_{n-1})\dif \psi\,\dif \theta_2\,\ldots\dif \theta_n = \sin^{n-1}(\psi)\dif \psi\dif {\theta}.
\end{split}
 \]
Now let us compute the squared distance $|x-y|^2$ in spherical coordinates
 \[
 \begin{split}
  |x-y|^2 &= \sum_{i=1}^{n+1} (x_i-y_i)^2 \\
  &= (\cos \vp - \cos\psi)^2 + (\sin\vp\cos\omega_2 - \sin\psi\cos\theta_2)^2 + \ldots \\
  &\quad \ldots + (\sin\vp \sin \omega_2 \ldots \sin \omega_{n-1} \sin\omega_n - \sin\psi \sin \theta_2 \ldots \sin \theta_{n-1} \sin\theta_n)^2\\
  & = 2 -2(\cos\vp\cos\psi + \sin\vp\sin\psi f(\omega,\theta)), 
 \end{split}
 \]
where $f(\omega,\theta)$ does not depend on $\vp$ and $\psi$, and is the sum of the remaining elements. We recall that the stereographic projection of the punctured sphere $\S^n\setminus \{N\}$, where $ N\coloneqq (1,0,\ldots,0)\in\R^n$ is the north pole  onto $\R^n$ is given by
\[
 (\vp,\omega) = \brac{2\arctan \frac 1r,\omega}, \quad  \quad (\psi,\theta) = \brac{2\arctan \frac 1R ,\theta}, 
\]
where $(r,\omega)$ and $(R,\theta)$ are polar coordinates on $\R^n$ with $r,\, R >0$ and $\omega,\,\theta \in \S^{n-1}$. We also recall that
\begin{equation}\label{eq:sphericaltopolarmultidim}
\begin{split}
  \sin \vp = \frac{2r}{r^2+1}, \quad \frac{\partial \vp}{\partial r} = -\frac{2}{r^2+1}, \quad \cos\vp = \frac{r^2-1}{r^2+1},\\
  \sin \psi = \frac{2R}{R^2+1}, \quad \frac{\partial \psi}{\partial R} = -\frac{2}{R^2+1}, \quad \cos\psi = \frac{R^2-1}{R^2+1}.
  \end{split}
\end{equation}

Let $v\in W^{t,\frac{n}{s}}(\S^n)$, we will compute its energy 
\[
 \int_{\S^n}\int_{\S^n} \frac{|v(x) - v(y)|^\frac{n}{s}}{|x-y|^{n+\frac{tn}{s}}} \dx\dy =   \int_{\S^n}\int_{\S^n} \frac{|v(x) - v(y)|^\frac{n}{s}}{|x-y|^{2\brac{\frac{n}{2}\brac{1+\frac{t}{s}}}}} \dx\dy
\]
 in polar coordinates. By a change of variable and using \eqref{eq:sphericaltopolarmultidim} we get

\[
 \begin{split}
 &\int_{\S^n}\int_{\S^n} \frac{|v(x) - v(y)|^\frac{n}{s}}{|x-y|^{2\brac{\frac{n}{2}\brac{1+\frac{t}{s}}}}} \dx\dy \\
 & = \int_{\S^{n-1}}\int_0^\pi \int_{\S^{n-1}}\int_0^\pi  \frac{|v(\vp,\omega) - v(\psi,\theta)|^\frac{n}{s}}{\brac{2 -2(\cos\vp\cos\psi + \sin\vp\sin\psi f(\omega,\theta))}^{{\frac{n}{2}\brac{1+\frac{t}{s}}}}} \sin^{n-1}(\vp) \sin^{n-1}(\psi)\dif \vp\dif {\omega}\dif \psi\dif {\theta}\\
 & = \int_{\S^{n-1}}\int_0^\infty \int_{\S^{n-1}}\int_0^\infty \frac{|v(r,\omega) - v(R,\theta)|^\frac{n}{s}}{\brac{\brac{\frac{2}{r^2+1}\frac{2}{R^2+1}}\brac{r^2+R^2+2rR \,f(\omega,\theta)}}^{{\frac{n}{2}\brac{1+\frac{t}{s}}}}}\\
 & \phantom{= \int_{\S^{n-1}}\int_0^\infty \int_{\S^{n-1}}\int_0^\infty}\quad \brac{\frac{2r}{r^2+1}}^{n-1} \brac{\frac{2R}{R^2+1}}^{n-1}\frac{2}{r^2+1} \frac{2}{R^2+1} \dif r\dif {\omega}\dif R\dif {\theta}\\
 & = \int_{\S^{n-1}}\int_0^\infty \int_{\S^{n-1}}\int_0^\infty \frac{|v(r,\omega) - v(R,\theta)|^\frac{n}{s}}{\brac{\brac{\frac{2}{r^2+1}\frac{2}{R^2+1}}\brac{r^2+R^2+2rR \,f(\omega,\theta)}}^{{\frac{n}{2}\brac{1+\frac{t}{s}}}}}\\
 & \phantom{= \int_{\S^{n-1}}\int_0^\infty \int_{\S^{n-1}}\int_0^\infty}\quad\brac{\frac{2}{r^2+1}}^{n} \brac{\frac{2}{R^2+1}}^{n} r^{n-1}R^{n-1}\dif r\dif {\omega}\dif R\dif {\theta}\\
 & = \int_{\S^{n-1}}\int_0^\infty \int_{\S^{n-1}}\int_0^\infty \frac{|v(r,\omega) - v(R,\theta)|^\frac{n}{s}}{{\brac{r^2+R^2+2rR \,f(\omega,\theta)}}^{{\frac{n}{2}\brac{1+\frac{t}{s}}}}}\\
 & \phantom{= \int_{\S^{n-1}}\int_0^\infty \int_{\S^{n-1}}\int_0^\infty}\quad \brac{\frac{2}{r^2+1}}^{\frac{n}{2}\brac{1-\frac{t}{s}}} \brac{\frac{2}{R^2+1}}^{\frac{n}{2}\brac{1-\frac{t}{s}}} r^{n-1}R^{n-1}\dif r\dif {\omega}\dif R\dif {\theta}.
 \end{split}
\]

In the latter we denote the integrand by
\[
 |v(r,R,\omega,\theta)|_{t,\frac n s}\coloneqq\frac{|v(r,\omega) - v(R,\theta)|^\frac{n}{s}}{{\brac{r^2+R^2+2rR \,f(\omega,\theta)}}^{{\frac{n}{2}\brac{1+\frac{t}{s}}}}} \brac{\frac{2}{r^2+1}}^{\frac{n}{2}\brac{1-\frac{t}{s}}} \brac{\frac{2}{R^2+1}}^{\frac{n}{2}\brac{1-\frac{t}{s}}} r^{n-1}R^{n-1}.
\]

Thus with this notation
\begin{equation}
 \begin{split}
 \int_{\S^n}\int_{\S^n} \frac{|v(x) - v(y)|^\frac{n}{s}}{|x-y|^{2\brac{\frac{n}{2}\brac{1+\frac{t}{s}}}}} \dx\dy=\int_{\S^{n-1}}\int_0^\infty \int_{\S^{n-1}}\int_0^\infty |v(r,R,\omega,\theta)|_{t,\frac n s}\dif r\dif {\omega}\dif R\dif {\theta}.
 \end{split}
\end{equation}

We consider the rescaling $v_\lambda(r,\omega) = v(\lambda r,\omega)$ and compute

\begin{equation}\label{eq:rescaledfractionalenergyfirst}
 \begin{split}
 &\int_{\S^n}\int_{\S^n} \frac{|v_\lambda(x) - v_\lambda(y)|^\frac{n}{s}}{|x-y|^{2\brac{\frac{n}{2}\brac{1+\frac{t}{s}}}}} \dx\dy \\
 & = \int_{\S^{n-1}}\int_0^\infty \int_{\S^{n-1}}\int_0^\infty \frac{|v(\lambda r,\omega) - v(\lambda R,\theta)|^\frac{n}{s}}{{\brac{r^2+R^2+2rR \,f(\omega,\theta)}}^{{\frac{n}{2}\brac{1+\frac{t}{s}}}}} \\
 & \phantom{= \int_{\S^{n-1}}\int_0^\infty \int_{\S^{n-1}}\int_0^\infty}\quad \brac{\frac{2}{r^2+1}}^{\frac{n}{2}\brac{1-\frac{t}{s}}} \brac{\frac{2}{R^2+1}}^{\frac{n}{2}\brac{1-\frac{t}{s}}} r^{n-1}R^{n-1}\dif r\dif {\omega}\dif R\dif {\theta}\\
&= \int_{\S^{n-1}}\int_0^\infty \int_{\S^{n-1}}\int_0^\infty \frac{|v(\tilde r,\omega) - v(\tilde R,\theta)|^\frac{n}{s}}{\lambda^{-n\brac{1+\frac ts}}{\brac{\tilde r^2+\tilde R^2+2\tilde r \tilde R \,f(\omega,\theta)}}^{{\frac{n}{2}\brac{1+\frac{t}{s}}}}}\\
 & \phantom{= \int_{\S^{n-1}}\int_0^\infty \int_{\S^{n-1}}\int_0^\infty}\quad \brac{\frac{2}{r^2+1}}^{\frac{n}{2}\brac{1-\frac{t}{s}}} \brac{\frac{2}{R^2+1}}^{\frac{n}{2}\brac{1-\frac{t}{s}}} \brac{\frac{\tilde r}{\lambda}}^{n-1}\brac{\frac{\tilde R}{\lambda}}^{n-1}\frac{1}{\lambda^2}\dif \tilde r\dif {\omega}\dif \tilde R\dif {\theta}\\
&= \int_{\S^{n-1}}\int_0^\infty \int_{\S^{n-1}}\int_0^\infty \frac{|v(\tilde r,\omega) - v(\tilde R,\theta)|^\frac{n}{s}}{{\brac{\tilde r^2+\tilde R^2+2\tilde r \tilde R \,f(\omega,\theta)}}^{{\frac{n}{2}\brac{1+\frac{t}{s}}}}}\\
 & \phantom{= \int_{\S^{n-1}}\int_0^\infty \int_{\S^{n-1}}\int_0^\infty}\quad \brac{\frac{2}{r^2+1}}^{\frac{n}{2}\brac{1-\frac{t}{s}}} \brac{\frac{2}{R^2+1}}^{\frac{n}{2}\brac{1-\frac{t}{s}}} {{\tilde r}}^{n-1}{{\tilde R}}^{n-1}\lambda^{n\brac{\frac ts -1}}\dif \tilde r\dif {\omega}\dif \tilde R\dif {\theta}.
 \end{split}
\end{equation}

We compute

\begin{equation}\label{eq:rlambdaestimatesnspherefractional}
\begin{split}
 &\brac{\frac{2}{r^2+1}}^{\frac{n}{2}\brac{1-\frac{t}{s}}} \brac{\frac{2}{R^2+1}}^{\frac{n}{2}\brac{1-\frac{t}{s}}}\lambda^{n\brac{\frac ts -1}}\\
 &=\brac{\frac{2}{\tilde r^2+1}}^{\frac{n}{2}\brac{1-\frac{t}{s}}} \brac{\frac{2}{\tilde R^2+1}}^{\frac{n}{2}\brac{1-\frac{t}{s}}}\brac{\frac{2}{\tilde r^2+1}}^{-\frac{n}{2}\brac{1-\frac{t}{s}}} \brac{\frac{2}{\tilde R^2+1}}^{-\frac{n}{2}\brac{1-\frac{t}{s}}}\brac{\frac{2}{r^2+1}}^{\frac{n}{2}\brac{1-\frac{t}{s}}}\\
 &\quad \brac{\frac{2}{R^2+1}}^{\frac{n}{2}\brac{1-\frac{t}{s}}}\lambda^{n\brac{\frac ts -1}}\\
 &= \brac{\frac{2}{\tilde r^2+1}}^{\frac{n}{2}\brac{1-\frac{t}{s}}} \brac{\frac{2}{\tilde R^2+1}}^{\frac{n}{2}\brac{1-\frac{t}{s}}} \brac{\frac{\tilde r^2 + \lambda^2}{\lambda(\tilde r^2 +1)}}^{\frac{n}{2}\brac{\frac ts -1}} \brac{\frac{\tilde R^2 + \lambda^2}{\lambda(\tilde R^2 +1)}}^{\frac{n}{2}\brac{\frac ts -1}}.
\end{split}
 \end{equation}

Combining \eqref{eq:rescaledfractionalenergyfirst} with \eqref{eq:rlambdaestimatesnspherefractional} we get

\begin{equation}\label{eq:rescaledfractionalenergynsphere}
\begin{split}
 &\int_{\S^n}\int_{\S^n} \frac{|v_\lambda(x) - v_\lambda(y)|^\frac{n}{s}}{|x-y|^{2\brac{\frac{n}{2}\brac{1+\frac{t}{s}}}}} \dx\\dy\\
 &=\int_{\S^{n-1}}\int_0^\infty \int_{\S^{n-1}}\int_0^\infty \frac{|v(\tilde r,\omega) - v(\tilde R,\theta)|^\frac{n}{s}}{{\brac{\tilde r^2+\tilde R^2+2\tilde r \tilde R \,f(\omega,\theta)}}^{{\frac{n}{2}\brac{1+\frac{t}{s}}}}} \brac{\frac{2}{\tilde r^2+1}}^{\frac{n}{2}\brac{1-\frac{t}{s}}} \brac{\frac{2}{\tilde R^2+1}}^{\frac{n}{2}\brac{1-\frac{t}{s}}}\tilde r^{n-1} \tilde R^{n-1}\\
 & \phantom{\int_{\S^{n-1}}\int_0^\infty \int_{\S^{n-1}}\int_0^\infty}\brac{\brac{\frac{\tilde r^2 + \lambda^2}{\lambda(\tilde r^2 +1)}}^{\frac{n}{2}\brac{\frac ts -1}} \brac{\frac{\tilde R^2 + \lambda^2}{\lambda(\tilde R^2 +1)}}^{\frac{n}{2}\brac{\frac ts -1}}}\dif \tilde r\dif {\omega}\dif \tilde R\dif {\theta}\\
  & = \int_{\S^{n-1}}\int_0^\infty \int_{\S^{n-1}}\int_0^\infty |v(\tilde r,\tilde R,\omega, \theta)|_{t,\frac{n}{s}} \ K_\lambda(\tilde r,\tilde R) \dif \tilde r\dif {\omega}\dif \tilde R\dif {\theta},
 \end{split}
\end{equation}
where
\[
 K_\lambda(\tilde r,\tilde R) = \brac{\brac{\frac{\tilde r^2 + \lambda^2}{\lambda(\tilde r^2 +1)}}^{\frac{n}{2}\brac{\frac ts -1}} \brac{\frac{\tilde R^2 + \lambda^2}{\lambda(\tilde R^2 +1)}}^{\frac{n}{2}\brac{\frac ts -1}}}.
\]

As in the 1 dimensional case, we have:

If $\tilde r \le \lambda$, then
\[
 \frac{\tilde r^2 + \lambda^2}{\lambda(\tilde r^2 +1)} \le 2\lambda.
\]

If $\tilde r\ge \lambda$, then
\[
\frac{\tilde r^2 + \lambda^2}{\lambda(\tilde r^2 +1)}\le \frac{2}{\lambda}.
\]
Thus,
\[
 K_\lambda(\tilde r, \tilde R) \le \begin{cases}
                                    (2\lambda)^{n\brac{\frac ts-1}} & \tilde r \le \lambda,\ \tilde R \le \lambda\\
                                    \frac12 (2\lambda)^{n\brac{\frac ts-1}} + \frac 12 \brac{\frac 2\lambda}^{n\brac{\frac ts-1}} & \tilde r \ge \lambda,\ \tilde R \le \lambda \text{ or } \tilde r \le \lambda,\ \tilde R \ge \lambda \\
                                    \brac{\frac 2\lambda}^{n\brac{\frac ts-1}} & \tilde r \ge \lambda,\ \tilde R \ge \lambda.
                                   \end{cases}
\]

This leads to
\[
 \begin{split}
  \int_{\S^n}\int_{\S^n}& \frac{|v_\lambda(x) - v_\lambda(y)|^\frac{n}{s}}{|x-y|^{2\brac{\frac{n}{2}\brac{1+\frac{t}{s}}}}} \dx\\dy\\
  & = \int_{\S^{n-1}}\int_0^\infty \int_{\S^{n-1}}\int_0^\infty |v(\tilde r,\tilde R,\omega, \theta)|_{t,\frac{n}{s}}\ K_\lambda(\tilde r, \tilde R)\dif \tilde r\dif {\omega}\dif \tilde R\dif {\theta}\\
   & \le (2\lambda)^{n\brac{\frac{t}{s}-1}} \int_{\S^{n-1}}\int_0^\lambda \int_{\S^{n-1}}\int_0^\lambda |v(\tilde r,\tilde R,\omega, \theta)|_{t,\frac{n}{s}} \dif \tilde r\dif {\omega}\dif \tilde R\dif {\theta} \\
&\quad + \brac{(2\lambda)^{n\brac{\frac ts-1}} + \brac{\frac 2\lambda}^{n\brac{\frac ts-1}}} \int_{\S^{n-1}}\int_0^\lambda \int_{\S^{n-1}}\int_\lambda^\infty |v(\tilde r,\tilde R,\omega, \theta)|_{t,\frac{n}{s}} \dif \tilde r\dif {\omega}\dif \tilde R\dif {\theta}  \\
&\quad + \brac{\frac{2}{\lambda}}^{n\brac{\frac{t}{s}-1}} \int_{\S^{n-1}}\int_\lambda^\infty \int_{\S^{n-1}}\int_\lambda^\infty |v(\tilde r,\tilde R,\omega, \theta)|_{t,\frac{n}{s}} \dif \tilde r\dif {\omega}\dif \tilde R\dif {\theta}\\
&= (2\lambda)^{n\brac{\frac{t}{s}-1}} \int_{\S^{n-1}}\int_0^\lambda \int_{\S^{n-1}}\int_0^\infty |v(\tilde r,\tilde R,\omega, \theta)|_{t,\frac{n}{s}} \dif \tilde r\dif {\omega}\dif \tilde R\dif {\theta}\\
&\quad +\brac{\frac{2}{\lambda}}^{n\brac{\frac{t}{s}-1}} \int_{\S^{n-1}}\int_\lambda^\infty \int_{\S^{n-1}}\int_0^\infty |v(\tilde r,\tilde R,\omega, \theta)|_{t,\frac{n}{s}} \dif \tilde r\dif {\omega}\dif \tilde R\dif {\theta}.
  \end{split}
\]
For $r_\lambda = \sqrt{\frac{4\lambda^2}{\lambda^2+1}}$ this inequality can be rephrased as\footnote{The circle centered at the origin of radius $\lambda$ in polar coordinates corresponds to the circle of radius $r_\lambda = \sqrt{\frac{4\lambda^2}{\lambda^2+1}}$ with center at $S$ in Euclidean coordinates. Indeed, one can compute it from the law of cosines: $r_\lambda^2=2-2\cos(\pi-\varphi_\lambda)$, where $\varphi_\lambda = 2\arctan \frac 1 \lambda$.}

\begin{equation}\label{eq:rescaledenergyestimated}
 \begin{split}
   \int_{\S^n}\int_{\S^n} \frac{|v_\lambda(x) - v_\lambda(y)|^\frac{n}{s}}{|x-y|^{n+\frac{tn}{s}}} \dx\dy 
  &\le (2\lambda)^{n\brac{\frac{t}{s}-1}} \int_{D(S,r_\lambda)} \int_{\S^n}\frac{|v(x) - v(y)|^\frac{n}{s}}{|x-y|^{n+\frac{tn}{s}}}\dx\\dy\\
 &\quad  + \brac{\frac{2}{\lambda}}^{n\brac{\frac{t}{s}-1}} \int_{\S^n\setminus D(S,r_\lambda)} \int_{\S^n}\frac{|v(x) - v(y)|^\frac{n}{s}}{|x-y|^{n+\frac{tn}{s}}}\dx\dy,
  \end{split}
\end{equation}
where $S=(-1,0,\ldots,0)\in\S^n$ is the south pole.
\end{proof}

Having \Cref{pr:rescaling} we are ready to proceed with the main theorem of this section.

\begin{proof}[Proof of Theorem~\ref{th:fractionalestimatesmalldiskbyremaining}]
Without loss of generality we may assume that $y_0 = S$, where $S=(-1,0,\ldots,0)$ is the south pole. 
Let $u_t \in W^{t,\frac{n}{s}}(\S^n)$ be the minimizing map from the assumptions of this Theorem. For $\lambda > 0$ take $(u_t )_\lambda$ from \Cref{pr:rescaling}. Observe that $\lambda \mapsto (u_t)_{\lambda}$ is a homotopy.

Since $u_t$ is a minimizer, we can compare the energies of $u_t$ and $(u_t)_{\lambda}$.
\begin{equation}\label{eq:comparisonfractionalenergyrescaled}
    \int_{\S^n}\int_{\S^n} \frac{|u_t (x) - u_t (y)|^\frac{n}{s}}{|x-y|^{n+\frac{tn}{s}}} \dx\dy \le     \int_{\S^n}\int_{\S^n} \frac{|(u_t )_\lambda(x) - (u_t )_\lambda(y)|^\frac{n}{s}}{|x-y|^{n+\frac{tn}{s}}} \dx\dy .
\end{equation}
Combining \eqref{eq:comparisonfractionalenergyrescaled} and \eqref{eq:rescaledenergyestimated:goal} we obtain
\[
\begin{split}
0 \leq &((2\lambda)^{n\brac{\frac{t}{s}-1}}-1) \int_{D(S,r_\lambda)} \int_{\S^n}\frac{|u_t(x) - u_t(y)|^\frac{n}{s}}{|x-y|^{n+\frac{tn}{s}}}\dx\\dy\\   
 &\quad  + \brac{\brac{\frac{2}{\lambda}}^{n\brac{\frac{t}{s}-1}}-1} \int_{\S^n\setminus D(S,r_\lambda)} \int_{\S^n}\frac{|u_t(x) - u_t(y)|^\frac{n}{s}}{|x-y|^{n+\frac{tn}{s}}}\dx\dy 
 \end{split}
\]
For $\lambda<2$ the expression 
\[
 C_{\lambda,t} \coloneqq \frac{\brac{\frac{2}{\lambda}}^{n\brac{\frac ts -1}}-1}{1-(2\lambda)^{n\brac{\frac{t}{s}-1}}}
\]
is positive for any $t\in(s,s_0]$.
Thus, for any $\lambda < 2$,
\[
 \int_{D(S,r_\lambda)} \int_{\S^n} \frac{|u_t (x) - u_t (y)|^\frac{n}{s}}{|x-y|^{n+\frac{tn}{s}}} \dx\dy \le C_{\lambda,t} \int_{\S^n\setminus D(S,r_\lambda)}\int_{\S^n} \frac{|u_t (x) - u_t (y)|^\frac{n}{s}}{|x-y|^{n+\frac{tn}{s}}} \dx\dy.
\]
We need to study the asymptotics of $C_{\lambda,t}$. Let $\lambda \in [0,\lambda_0]$ for some $\lambda_0 < 2$.
We have 
\[
 C_{\lambda,t} = \lambda^{-n\brac{\frac{t}{s}-1}} \frac{2^{n\brac{\frac ts -1}}-\lambda^{n\brac{\frac ts -1}}}{1-(2\lambda)^{n\brac{\frac ts-1}}} \le \lambda^{-n\brac{\frac ts-1}} \bar{C}_{\lambda_0,s_0}, 
\]
where we set 
\[
 \bar{C}_{\lambda_0,s_0} \coloneqq \max_{t \in [s,s_0], \lambda \in [0,\lambda_0]} \frac{2^{n\brac{\frac ts -1}}-\lambda^{n\brac{\frac ts -1}}}{1-(2\lambda)^{n\brac{\frac ts-1}}}.
\]
We need to show that $\bar{C}_{\lambda_0,s_0} <\infty$.
Since $(t,\lambda) \mapsto \lambda^{n(\frac{t}{s}-1)}C_{\lambda,t}$ is continuous in $[s,s_0]\times [0,\lambda_0]$ we only need to estimate $\lambda^{n(\frac{t}{s}-1)}C_{\lambda,t} $ at the asymptotic boundary of $[s,s_0]\times [0,\lambda_0]$.

First we observe 
\[
 \sup_{\lambda \in [0,\lambda_0]}\lim_{t \to s^+}  \frac{2^{n\brac{\frac ts -1}}-\lambda^{n\brac{\frac ts -1}}}{1-(2\lambda)^{n\brac{\frac ts-1}}} =  \sup_{\lambda \in [0,\lambda_0]}\frac{\log\brac{\frac{2}{\lambda}}}{\log(2\lambda)}  \in[0,1].
\]
Also
\[
 \sup_{\lambda \in [0,\lambda_0]}\lim_{t \to s_0^-}  \frac{2^{n\brac{\frac ts -1}}-\lambda^{n\brac{\frac ts -1}}}{1-(2\lambda)^{n\brac{\frac ts-1}}} =  \sup_{\lambda \in [0,\lambda_0]}\frac{2^{n\brac{\frac{s_0}{s}-1}}-\lambda^{n\brac{\frac{s_0}{s}-1}}}{1-(2\lambda)^{n\brac{\frac{s_0}{s}-1}}} < \infty.
\]
Next
\[
 \sup_{t \in [s,s_0]}  \lim_{\lambda \to 0} \frac{2^{n\brac{\frac ts -1}}-\lambda^{n\brac{\frac ts -1}}}{1-(2\lambda)^{n\brac{\frac ts-1}}} = \sup_{t \in [s,s_0]}  2^{n\brac{\frac{t}{s}-1}} = 2^{n\brac{\frac{s_0}{s}-1}}
\]
and it is easy to see that
\[
 \sup_{t \in [s,s_0]}  \lim_{\lambda \to \lambda_0} \frac{2^{n\brac{\frac ts -1}}-\lambda^{n\brac{\frac ts -1}}}{1-(2\lambda)^{n\brac{\frac ts-1}}} < \infty.
\]
In conclusion, we have shown that for any $\lambda_0 < 2$, $s_0 \in (s,1)$ for any $\lambda \in [0,\lambda_0]$, $t \in (s,s_0]$
\[
 \int_{D(S,r_\lambda)} \int_{\S^n} \frac{|u_t (x) - u_t (y)|^\frac{n}{s}}{|x-y|^{n + \frac{tn}{s}}} \dx\dy \le C_{\lambda_0,s_0} \lambda^{-n(\frac{t}{s}-1)} \int_{\S^n\setminus D(S,r_\lambda)}\int_{\S^n} \frac{|u_t (x) - u_t (y)|^\frac{n}{s}}{|x-y|^{n+\frac{tn}{s}}} \dx\dy.
\]
So for any $\rho_0 < \sqrt{\frac{4}{5}}$ let $\lambda_0< 2$ be such that $\rho_0 = \frac{2\lambda_0}{\sqrt{\lambda_0^2+1}}$. For any $\rho \in (0,\rho_0)$ there exists $\lambda \in (0,\lambda_0)$ such that $\rho = \frac{2\lambda}{\sqrt{\lambda^2+1}}$. We then have $\lambda  = \rho \frac{\sqrt{\lambda^2+1}}{2}$ and $0\le \frac{\sqrt{\lambda^2+1}}{2}\le \sqrt{\frac{5}{4}}$. This gives
\[
 \int_{D(S,\rho)} \int_{\S^n} \frac{|u_t (x) - u_t (y)|^\frac{n}{s}}{|x-y|^{n+\frac{tn}{s}}} \dx\dy \le C_{\lambda_0,s_0} \rho^{-n(\frac{t}{s}-1)} \int_{\S^n\setminus D(S,\rho)}\int_{\S^n} \frac{|u_t (x) - u_t (y)|^\frac{n}{s}}{|x-y|^{n+\frac{tn}{s}}} \dx\dy.
\]
This finishes the proof of \Cref{th:fractionalestimatesmalldiskbyremaining}.
\end{proof}

\section{Existence of \texorpdfstring{$W^{s,\frac{n}{s}}(\Sigma,\n)$}{Wsns}-minimizers if \texorpdfstring{$\pi_{n}(\n)=\{0\}$}{pin}}\label{s:sucks1}
The following theorem is a generalization of \cite[Theorem~5.1]{Sucks1}.

\begin{theorem}\label{th:main1}
Let $n\ge 2$ and $s\in(0,1)$ or $n=1$ and $s\le \frac 12$. Let $\n$ be compact, $\pi_n(\n) = 0$, and let $\Sigma$ be as before. Then there exists a minimizing $W^{t,\frac{n}{s}}$-harmonic map in every homotopy class of $C^0(\Sigma,\n)$ for any $t \in [s,1)$.
\end{theorem}

The assumption $\pi_n(\n)=0$ cannot be dropped as shown in an example by Eells and Wood \cite{Eells-Wood1976}: 
\begin{theorem}\label{th:eellswoodexample}
 There exists no harmonic map of degree one from $\mathbb T^2$ to $\S^2$. 
\end{theorem}
For a proof, that the infimum in  Theorem \ref{th:eellswoodexample} may not be attained in every homotopy class see also \cite[(9.2) Proposition]{Lemaire1978}.

\begin{proof}[Proof of Theorem \ref{th:main1}]
Fix a homotopy class $X  \subset C^0(\Sigma,\n)$. 

The statement for $t>s$ is clear.

Let $u_t$ be the minimizing harmonic maps within $X$ for $t \in (s,s_0)$. Here, $s_0 > s$ is taken from \Cref{th:reghomo}.

 We use Theorem \ref{th:strongconvoutsideofpoints} to infer that there is a map $u_s\in W^{s,\frac ns}(\Sigma,\n)$ for which on a subsequence (denoted the same)
\[
 u_t \xrightarrow{t\rightarrow s} u_s \quad \text{ strongly in } W^{s_0,\frac{n}{s}}(\Sigma\setminus \{x_1,\ldots,x_K\}).
\]
Moreover, by Theorem~\ref{th:corollaryremovability} isolated singularities can be removed and we deduce that $u_s\in W^{s_0,\frac ns}(\Sigma,\n)$. In order to conclude we will show that
\[
 u_t \xrightarrow{t\rightarrow s} u_s \quad \text{ strongly in } W^{s_0,\frac{n}{s}}(\Sigma).
\]
We denote $A=\{x_1,\ldots,x_K\}$. Consider $t$ close to $s$. Let $x_i\in A$ and take $\rho$ small enough, so that 
\[
B(x_i,2\Lambda \rho)\cap A = \{x_i\},
\]
where $\Lambda>1$ is the number taken from \Cref{la:smallnessonballs} so that the smallness of $E_{t,\frac ns}(v_t,B(x_i,2\Lambda \rho))<\frac \eps 4$ implies 
\[
 \int_{B(x_i,2\rho)}\int_{\Sigma} \frac{|v_t(x)-v_t(y)|^\frac ns}{|x-y|^{n+\frac ns}} \dif x \dif y <
 \frac \eps 2.
\]
with $2\Lambda \rho = \lambda(2\rho)^\frac 12$ and $\lambda=\lambda(\n,n,s,\eps)$.

We construct a comparison map $v_t$ such that
\[
 v_t = \begin{cases}
                       u_s & \text{ in } B(x_i,\rho)\\
                       u_t & \text{ outside of } B(x_i,2\rho).
                      \end{cases}
\]
In order to define $v_t$, we let $\eta_{B(x_i,\rho)}\in C_c^\infty(B(x_i,2\rho))$ be a standard cut-off function, such that $\eta_{B(x_i,\rho)}\equiv 1$ in $B(x_i,\rho)$.
We claim that for all $x\in\Sigma$ and $t$ sufficiently close to $s$ we have
\begin{equation}\label{eq:usutclosetomanifold}
 \dist((1-\eta_{B(x_i,\rho)}) u_t(x) + \eta_{B(x_i,\rho)} u_s(x),\n) \ll 1,
\end{equation}
This is true, because for $x$ outside of $B(x_i,2\rho)$ and for $x\in B(x_i,\rho)$ the distance is zero. On the remaining annulus $B(x_i,2\rho)\setminus B(x_i,\rho)$ we have $W^{s_0,\frac ns }$ and uniform convergence of $u_t$ to $u_s$ and thus taking $t$ sufficiently close to $s$ we have \eqref{eq:usutclosetomanifold}. Therefore, the map
\begin{equation}\label{eq:tildeutdefinition}
 v_t \coloneqq \begin{cases}
                        u_s(x) & \text{ for } x\in B(x_i,\rho)\\
                        \pi_\n\brac{(1-\eta_{B(x_i,\rho)}) u_t + \eta_{B(x_i,\rho)} u_s} & \text{ for } x\in B(x_i,2\rho)\setminus B(x_i,\rho)\\
                        u_t(x) & \text{ for } x\in \Sigma\setminus B(x_i,2\rho).
                       \end{cases}
\end{equation}
is well defined for $t$ sufficiently close to $s$. We observe that  $v_t\in W^{s_0,\frac{n}{s}}\cap C^0(\Sigma\cap B(x_i,2\Lambda\rho),\n)$. We also have
\begin{equation}\label{eq:utildeconvergenceonball}
 \lim_{t\to s^+}E_{t,\frac ns}(v_t, B(x_i, 2\Lambda\rho)) = E_{s,\frac ns}(u_s,B(x_i,2\Lambda\rho)).
\end{equation}
We observe that as $u_s\in W^{s_0,\frac ns}$ we have
\begin{equation}\label{eq:usestimateonball}
 E_{s,\frac ns}(u_s,B(x_i,2\Lambda\rho)) \le C \lambda^{n\frac{s_0-s}{s}}\rho^{n\frac{s_0-s}{2s}}E_{s_0,\frac ns}(u_s,B(x_i,2\Lambda\rho)) = \mathcal O(\rho^{n\frac{s_0-s}{2s}}) \text{ as } \rho\to0.
\end{equation}
Moreover, since $\pi_n(\n) = \{0\}$ we find that $u_t$ and $v_t$ must be homotopic. Indeed, since they coincide outside of $B(x_i, 2\rho)$ we can glue two copies of $B(x_i,2\rho)$ to an $\S^n$ with $u_t$ on the upper hemisphere $S^n_+$ and $v_t$ on the lower hemisphere $\S^n_-$ to construct a continuous map $u\colon \S^n \to \n$. Since $\pi_n(\n)$ is trivial, there exists a continuous extension $U: \B^{n+1} \to \n$, which readily leads to a homotopy of $u_t$ and $v_t$ on all of $\Sigma$.
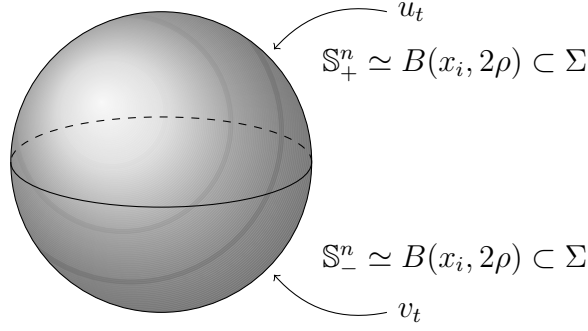
\begin{center}
\begin{tikzpicture}
  \shade[ball color = gray!40, opacity = 0.4] (0,0) circle (2cm);
  \draw (0,0) circle (2cm);
  \draw (-2,0) arc (180:360:2 and 0.6);
  \draw[dashed] (2,0) arc (0:180:2 and 0.6);
  \draw (1.5,-1.5) edge[<-, bend right] (3,-2);
  \node[right] at (3,-2) {$v_t$};
  \node[right] at (2,-1.3) {$\S^n_- \simeq B(x_i,2\rho)\subset \Sigma$};
  \draw (1.5,1.5) edge[<-, bend left] (3,2);
  \node[right] at (3,2) {$u_t$}; 
  \node[right] at (2,1.3) {$\S^n_+ \simeq B(x_i,2\rho)\subset \Sigma$};
\end{tikzpicture}
\end{center}
As $u_t$ is a minimizer in its homotopy class, we can compare the energies
\begin{equation}\label{eq:ututildecomparison}
 E_{t,\frac ns}(u_t,\Sigma) \le E_{t,\frac ns}(v_t,\Sigma).
\end{equation}
Decomposing the integrals into integration over $\Sigma\setminus B(x_i,2\rho)\times \Sigma\setminus B(x_i,2\rho)$, $\Sigma\setminus B(x_i,2\rho)$, $B(x_i,2\rho)\times B(x_i,2\rho)$ we obtain
\begin{equation}\label{eq:utildedecomposition}
 \begin{split}
  &\int_{\Sigma\setminus B(x_i,2\rho)}\int_{\Sigma\setminus B(x_i,2\rho)}\frac{|u_t(x) - u_t(y)|^\frac ns}{|x-y|^{n+\frac{nt}{s}}} \dif x \dif y + \int_{\Sigma}\int_{B(x_i,2\rho)} \frac{|u_t(x) - u_t(y)|^\frac ns}{|x-y|^{n+\frac{nt}{s}}} \dif x \dif y\\
  &\le \int_{\Sigma}\int_{\Sigma}\frac{|u_t(x) - u_t(y)|^\frac ns}{|x-y|^{n+\frac{nt}{s}}} \dif x \dif y
 \end{split}
\end{equation}
and
\begin{equation}\label{eq:utdecomposition}
 \begin{split}
&\int_{\Sigma}\int_{\Sigma}\frac{|v_t(x) - v_t(y)|^\frac ns}{|x-y|^{n+\frac{nt}{s}}} \dif x \dif y\\
&\le \int_{\Sigma\setminus B(x_i,2\rho)}\int_{\Sigma\setminus B(x_i,2\rho)}\frac{|v_t(x) - v_t(y)|^\frac ns}{|x-y|^{n+\frac{nt}{s}}} \dif x \dif y + 2 \int_{\Sigma}\int_{B(x_i,2\rho)}\frac{|v_t(x) - v_t(y)|^\frac ns}{|x-y|^{n+\frac{nt}{s}}} \dif x \dif y.
 \end{split}
\end{equation}

Thus, combining \eqref{eq:ututildecomparison}, \eqref{eq:utildedecomposition}, \eqref{eq:utdecomposition}, and $u_t$ and $v_t$ coincide outside $B(x_i,2\rho)$ we get
\begin{equation}\label{eq:ututildecomparisononsmall}
\int_{\Sigma}\int_{B(x_i,2\rho)} \frac{|u_t(x) - u_t(y)|^\frac ns}{|x-y|^{n+\frac{nt}{s}}} \dif x \dif y \le 2 \int_{\Sigma}\int_{B(x_i,2\rho)}\frac{|v_t(x) - v_t(y)|^\frac ns}{|x-y|^{n+\frac{nt}{s}}} \dif x \dif y. 
\end{equation}
From \eqref{eq:utildeconvergenceonball} and \eqref{eq:usestimateonball} we know that for $t$ sufficiently close to $s$ we have
\[
 \int_{B(x_i,2\rho)}\int_{B(x_i,2\rho)}\frac{|v_t(x) - v_t(y)|^\frac ns}{|x-y|^{2n}} \dx \dy \le \mathcal O(\rho^{n\frac{s_0-s}{2s}}) \text{ as } \rho\to 0
\]
and thus choosing $\rho$ sufficiently small we get from \Cref{la:smallnessonballs}
\[
 \int_{\Sigma}\int_{B(x_i,2\rho)}\frac{|v_t(x) - v_t(y)|^\frac ns}{|x-y|^{2n}} \dif x \dif y \le \frac{\eps}{2}.
\]
The latter inequality combined with \eqref{eq:ututildecomparisononsmall} gives for small $\rho$ and $t$ close to $s$
\begin{equation}\label{eq:utballholesmall}
 \int_{\Sigma}\int_{B(x_i,2\rho)} \frac{|u_t(x) - u_t(y)|^\frac ns}{|x-y|^{2n}} \dif x \dif y \le \eps.
\end{equation}
Therefore, applying the regularity result \Cref{co:strongconvergencesmallenergy} we get on a smaller disk $u_t \rightarrow u_s$ strongly in $W^{s_0,\frac ns}(B(x_i,\rho))$.

That is, we have found that $u_t$ converges on all of $\Sigma$ to $u_s$ uniformly, and in $W^{s_0,\frac{n}{s}}$. This readily implies that $u_s$ is a $W^{s,\frac{n}{s}}$-minimizer in $X$.
\end{proof}

\section{Existence of \texorpdfstring{$W^{s,\frac{n}{s}}(\S^n,\n)$}{Wsns}-minimizers if \texorpdfstring{$\pi_{n}(\n) \neq \{0\}$}{pin}}\label{s:sucks2}

In this section we assume that $\Sigma = \S^n$, $\pi_n(\n)\neq \{0\}$, and look for minimizers in the free homotopy classes of $C^0(\S^n,\n)$, which we denote $\pi_0 C^0(\S^n,\n)$. We prove the following theorem. 

\begin{theorem}\label{th:main2}
Let $s\in(0,1)$, $n\ge 2$ or $s\le \frac 12$ and $n=1$. There exists a set of free homotopy classes $X\subset \pi_0 C^0(\S^n,\n)$ with the following properties:
\begin{enumerate}
\item Each $\Gamma_i\in X$ contains a minimizing $W^{s,\frac ns}$-harmonic map.
\item Elements $\Gamma_i \in X$ form a generating set for $\pi_n(\n)$ acted on by $\pi_1(\n)$.
\end{enumerate}
\end{theorem}

We state three corollaries of \Cref{th:main2}, or rather \Cref{la:5.4} --- \Cref{th:main2} main's ingredient.

\begin{corollary}\label{co:existencenontrivial}
Let $\pi_{n}(\n) \neq \{0\}$. Then there exists a nontrivial $W^{s,\frac{n}{s}}(\S^n,\n)$-harmonic map.
\end{corollary}
\begin{proof}
 From \Cref{th:main2} we deduce that if all homotopy classes $\Gamma \in \pi_0 C^0(\S^n,\n)$ which have a $W^{s,\frac ns}$-minimizing harmonic map would be trivial then we would obtain that the set generated by them would be trivial, thus $\pi_n(\n) =\{0\}$, a contradiction. Thus, there must be a nontrivial homotopy class in which there is a minimizer. 
\end{proof}

In particular, we have 
\begin{corollary}\label{co:existencenontrivialforsphere}
There exists a number $k\in \Z$, $k\neq 0$ such that
\[
 \inf \left\{ E_{s,\frac ns}(u,\S^n)\colon\ u\in C^0 \cap W^{s,\frac{n}{s}}(\S^n,\S^n),\ \deg u =k\right\}
\]
is attained.
\end{corollary}

\begin{corollary}\label{co:existencenontrivialsmallenergy}
Let $s \in (0,1)$, $n\geq 1$ and $\n$ as above. There exists an $\eps = \eps(s,n,\n)$ such that the following holds:

Let 
\[
 \delta \coloneqq \inf\left \{E_{s,\frac{n}{s}}(u): \quad u \in C^\infty(\S^n,\n), \quad \text{$u$ is not homotopic to a constant}\right \}.
\]
Then $\delta > \eps$ and moreover if $\Gamma \in \pi_0C^0(\S^n,\n)$ satisfies
\[
 \inf_{u \in \Gamma \cap W^{s,\frac{n}{s}}(\S^n,\n)} E_{s,\frac ns}(u,\S^n) \leq \delta +\eps,                                                                                                                                                                                                                    
\]
then $\Gamma$ contains an $E^{s,\frac{n}{s}}$-harmonic map. 
\end{corollary}
Observe there is no a priori reason that a minimizing nontrivial homotopy class $\Gamma_0$ exists, i.e., $\Gamma_0$ such that
\[
\inf_{u \in \Gamma_0 \cap W^{s,\frac{n}{s}}(\S^n,\n)} E_{s,\frac ns}(u,\S^n)                                                                                                                                                                                                                   
= \inf\left \{E_{s,\frac{n}{s}}(u)\colon u \in C^\infty(\S^n,\n), \ \text{$u$ is not homotopic to a constant}\right \}.                                                                                                                                                 
\]
See \cite[Proposition 2.4 \& Theorem 1.2]{VanSchaftingen-gapfree}.

Before we begin the proof of \Cref{th:main2} let us recall a few facts about free homotopies and free homotopy decomposition in terms of homotopy groups. For definitions we refer the reader to the book
\cite{Hatcher} and for an explanation for an analyst we refer to \cite[III \textsection 17]{BT82} or \cite[Section 2.1]{VanSchaftingen-gapfree}. Here, we will adopt the notation of Sacks--Uhlenbeck \cite[Section 5]{Sucks1}.

Each $\gamma \in \pi_n(\n)$ determines a free homotopy class of maps from $\S^n$ into $\n$. As free homotopy does not depend on the choice of the base point two elements $\gamma, \gamma' \in \pi_n(\S^n)$ determine the same free homotopy class if and only if they belong to the same orbit 
\[
 \pi_1(\n)\gamma = \pi_1(\n)\gamma'
\]
under the usual action of $\pi_1(\n)$ on $\pi_n(\n)$. We denote by $\Gamma \in \pi_0 C^0(\S^n,\n)$ the free homotopy class that corresponds to $\pi_1(\n)\gamma$. For $\Gamma \in \pi_0 C^0(\S^n,\n)$, we will denote by $\gamma \in \pi_n(\n)$ any element for which $\pi_1(\n)\gamma$ corresponds to $\Gamma$, we will write $\gamma\in\Gamma$. 

For any $\alpha \in \pi_1(\n)$ and $\gamma_1,\, \gamma_2 \in \pi_n(\n)$ we have 
\[
 \alpha(\gamma_1 + \gamma_2) = \alpha \gamma_1 + \alpha \gamma_2.
\]
Moreover, for a given $\Gamma_i = \pi_1(\n)\gamma_1$, for $i=1,2,3$,  
\begin{equation}\label{eq:freehomotopysum}
\gamma_1 + \gamma_2 = \gamma_3 \quad \Rightarrow \quad \pi_1(\n) \gamma_3 \subset \pi_1(\n) \gamma_1 + \pi_1(\n) \gamma_2,
\end{equation}
because for any $\alpha\in \pi_1(\n)$ we have $\alpha \gamma_1 + \alpha\gamma_2 = \alpha \gamma_3$.

We also note that if $\pi_1(\n)$ was trivial then we could drop the action of $\pi_1(\n)$ on $\pi_n(\n)$. 

For a free homotopy class $\Gamma \in \pi_0C^0(\S^n,\n)$ we write 
\[
 \#\Gamma \coloneqq \inf_{u \in \Gamma \cap W^{s,\frac{n}{s}}(\S^n,\n)} E_{s,\frac ns}(u,\S^n).
\]
The following characterization will be needed in the proof.
\begin{lemma}\label{la:minX}
\[
 \#\Gamma = \lim_{t \to s^+} \inf_{u \in \Gamma \cap W^{t,\frac{n}{s}}(\S^n,\n)} E_{t,\frac ns}(u,\S^n).
\]
\end{lemma}
\begin{proof}
Let $u_t \in \Gamma$ be a minimizer in $\Gamma$ for $E_{t,\frac ns}(\cdot,\S^n)$. Then
\[
\#\Gamma  \leq E_{s,\frac ns}(u_t,\S^n) \leq \diam(\S^n)^{(t-s)\frac{n}{s}} E_{t,\frac{n}{s}}(u_t,\S^n),
\]
which readily leads to
\[
\#\Gamma \leq \liminf_{t \to s^+} \inf_{u \in \Gamma \cap W^{t,\frac{n}{s}}(\S^n,\n)} E_{t,\frac ns}(u,\S^n).
\]
On the other hand, by smooth approximation, Lemma \ref{la:smoothapprox}, we can approximate $u$ smoothly in its homotopy group, and thus combining it with the definition we get
\[
 \#\Gamma = \inf_{v \in \Gamma\cap W^{s,\frac{n}{s}}\cap C^\infty(\S^n,\n)} E_{s,\frac ns}(v,\S^n).
\]
For such a smooth $v\in \Gamma\cap W^{s,\frac{n}{s}}\cap C^\infty(\S^n,\n)$,
 \[
   \inf_{u \in \Gamma \cap W^{t,\frac{n}{s}}(\S^n,\n)} E_{t,\frac ns}(u,\S^n)= E_{t,\frac ns}(u_t,\S^n) \leq  E_{t,\frac ns}(v,\S^n).
 \]
 Since $\lim_{t \to s^+}E_{t,\frac ns}(v,\S^n)=E_{s,\frac ns}(v,\S^n)$, we conclude that for any smooth $v\in \Gamma\cap W^{s,\frac{n}{s}}\cap C^\infty(\S^n,\n)$,
\[
 \limsup_{t \to s^+} \inf_{u \in \Gamma \cap W^{t,\frac{n}{s}}(\S^n,\n)} E_{t,\frac ns}(u,\S^n) \leq E_{s,\frac ns}(v,\S^n).
\]
Taking the infimum over all $v\in\Gamma \cap W^{s,\frac ns}\cap C^\infty(\S^n,\n)$,
\[
 \limsup_{t \to s^+} \inf_{u \in \Gamma \cap W^{t,\frac{n}{s}}(\S^n,\n)} E_{t,\frac ns}(u,\S^n) \leq \#\Gamma.
\]
\end{proof}

Before we proceed to the proof of Theorem \ref{th:main2} we would like to note, as mentioned in the introduction, that in the case of harmonic maps, the Theorem can not be improved in general. Futaki constructed in \cite{Futaki80} a manifold with the following property.
\begin{theorem}
 There is a manifold $\n$ with the following property: there exists a homotopy component of $C^\infty(\S^2,\n)$ in which there is no minimizer of the Dirichlet energy.
\end{theorem}

\Cref{th:main2} follows \Cref{la:5.4} below as in \cite[Proof of Theorem 5.5]{Sucks1}. 
\begin{lemma}\label{la:5.4}
Let $s,\, n, \,\n$ be as in \Cref{th:main2}. There exists a $\theta = \theta(s,n,\n)$ such that the following holds.

Let $\Gamma_0 \in \pi_0C^0(\S^n,\n)$. Then at least one of the following cases holds:
\begin{enumerate}
\item There exists a minimizer of $E_{s,\frac ns}(\cdot,\S^n)$ in $\Gamma_0$.
\item For every $\delta>0$, there exist nontrivial free homotopy classes $\Gamma_1=\pi_1(\n)\gamma_1$ and $\Gamma_2=\pi_1(\n)\gamma_2$, such that 
\[
 \Gamma_0=\pi_1(\n)\gamma_0\subset \pi_1(\n)\gamma_1 + \pi_1(\n)\gamma_2
\]
such that
\begin{equation}\label{eq:sumofgamma12boundedby0plusdelta}
 \#\Gamma_1 + \#\Gamma_2 \leq \#\Gamma_0 + \delta,
\end{equation}
\[
 \theta < \#\Gamma_1 < \#\Gamma_0 -\frac \theta 2,
\]
\[
 \theta < \#\Gamma_2 < \#\Gamma_0 -\frac \theta 2. 
\]
\end{enumerate}
\end{lemma}

\begin{proof}[Proof of \Cref{la:5.4}]
Let $\{u_t\}$ be a sequence of $W^{t,\frac{n}{s}}$-maps, which minimize $E_{t,\frac ns}(\cdot,\S^n)$ in $\Gamma_0\cap W^{t,\frac ns}(\S^n,\n)$. Similar to the proof of \Cref{th:main1}, using \Cref{th:strongconvoutsideofpoints} we find that the sequence $\{u_t\}$ is uniformly bounded and using \Cref{th:strongconvoutsideofpoints} we get that on a subsequence $u_t$ converges to $u_s$ strongly in $W^{s_0,\frac{n}{s}}(\S^n \setminus A,\n)$, weakly in $W^{s,\frac ns}(\S^n,\n)$, and locally uniformly in $\S^n \setminus A$, where $A=\{x_1,\ldots,x_K\}$ is a set consisting of finite number of points. Moreover, by \Cref{th:corollaryremovability} we obtain that for an $s_0>s$ we have $u_s\in W^{s_0,\frac ns}(\Sigma,\n)$. Then, we have two possibilities.

\underline{\textsc{Case 1: There are no blowup points.}} For every point $x_i\in A$ and $t$ sufficiently close to $s$, there is a $\rho$ such that
\[
 E_{s,\frac{n}{s}}(u_t,B(x_i,\rho)) \le \eps,
\]
where $\eps>0$ in taken from \Cref{co:strongconvergencesmallenergy}. Then, \Cref{co:strongconvergencesmallenergy} implies that $u_t \xrightarrow{t\rightarrow s} u_s$ in $W^{s_0,\frac ns}(B(x_i,\frac\rho 2),\n)$ and we obtain $u_t \xrightarrow{t\rightarrow s} u_s$ strongly in $W^{s_0,\frac ns}(\S^n,\n)$. This implies that $u_s \in \Gamma_0$ and $u_s$ minimizes the energy $E_{s,\frac ns}(\cdot,\S^n)$.

\underline{\textsc{Case 2: There is a blowup point.}} We assume that there is a point $x_1 \in A$ such that 
\begin{equation}\label{eq:blowuppoint}
 \lim_{\alpha\to\infty}\limsup_{t\to s^+} \int_{B(x_1,2^{-\alpha})}\int_{\S^n} \frac{|u_t(x) - u_t(y)|^\frac ns}{|x-y|^{2n}} \dif x \dif y \ge \eps.
\end{equation}
Similarly as in the proof of \Cref{th:main1} we take a small enough radius $\rho \in (0,1)$, so that 
\begin{equation}\label{eq:smallradiuslambda}
B(x_1,\lambda \rho^\beta)\cap A = \{x_1\}, 
\end{equation}
where $\lambda = \lambda(\n,n,s)$ and $\beta>0$ will be chosen later (in the application of the smallness condition \Cref{la:smallnessonballs} and \Cref{la:smallnessoncompliments}).

We repeat the construction in \eqref{eq:tildeutdefinition}. Let $\eta_{B(x_1,\rho)}\in C_c^\infty(B(x_1,2\rho))$, $0\le \eta_{B(x_1,\rho)}\le 1$, and $\eta_{B(x_1,\rho)}\equiv 1$ in $B(x_1,\rho)$. For $x\in B(x_1,2\rho)$ we define 
\[
\tilde u_{s;t}\coloneqq \pi_\n((1-\eta_{B(x_1,\rho)})u_t + \eta_{B(x_1,\rho)}u_s). 
\]
From \eqref{eq:usutclosetomanifold} we know that the projection is well defined for $t$ sufficiently close to $s$. Let
\[
 v_t \coloneqq \begin{cases}
              u_t \quad &\mbox{ in } \S^n\setminus B(x_1,2\rho)\\
              \tilde u_{s;t}  &\mbox{ in } B(x_1,2\rho)
             \end{cases}
\]
and
\[
w_t \coloneqq \begin{cases}
              \tilde u_{s;t} \circ \tau  &\mbox{in }\S^n \setminus B(x_1,2\rho)\\
              u_t \quad &\mbox{in }B(x_1,2\rho),
             \end{cases}
\]
where $\tau: \S^n \setminus B(x_1,2\rho)\to B(x_1,2\rho)$ is a diffeomorphism, such that $|\nabla \tau| \simeq \frac{1}{\rho}$.

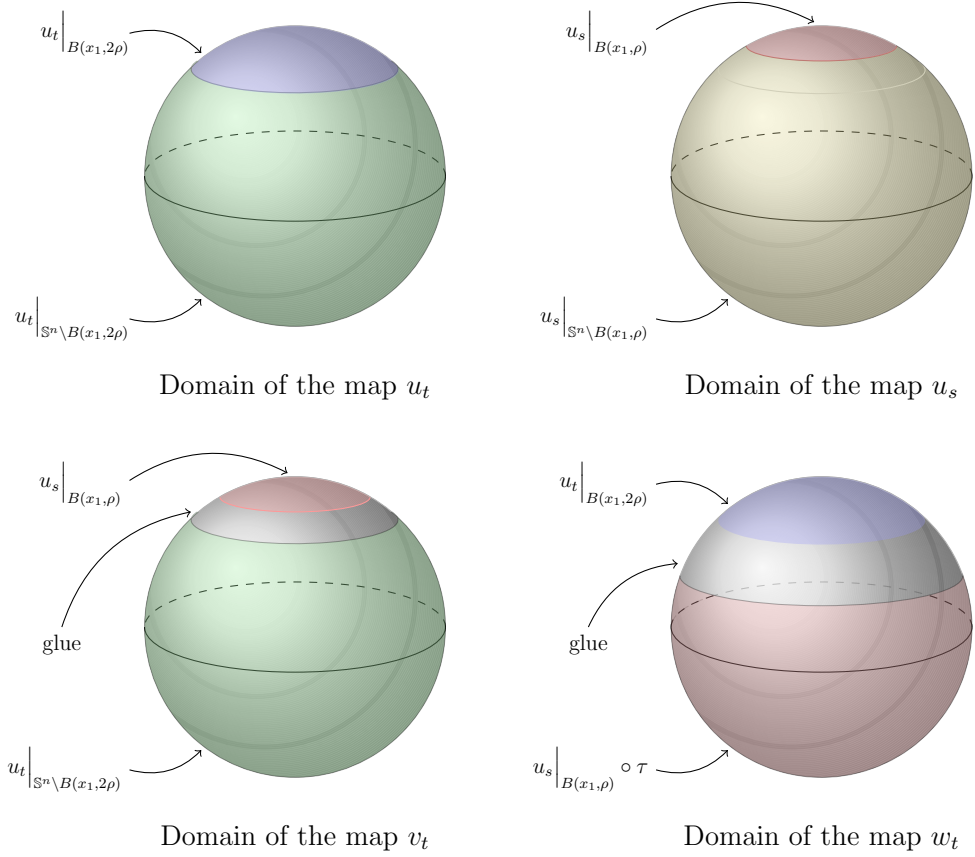
\begin{center}
\begin{tikzpicture}
\foreach \x/\xtext in {0,1}
\foreach \y/\xtext in {0,1}
  \shade[shift={(7*\x,0)},shift={(0,-6*\y)},ball color = gray!40, opacity = 0.4] (0,0) circle (2cm);

  \foreach \x/\xtext in {0,1}
  \foreach \y/\xtext in {0,1}
  \draw[shift={(7*\x,0)}, shift={(0,-6*\y)}] (-2,0) arc (180:360:2 and 0.6);
  
   \foreach \x/\xtext in {0,1}
   \draw[shift={(7*\x,0)},dashed] (2,0) arc (0:180:2 and 0.6);

   \draw[shift={(0,-6)},dashed] (2,0) arc (0:180:2 and 0.6);
   
\draw[shift={(7,-6)},dashed] (2,0) arc (0:45:2 and 0.6);
\draw[shift={(7,-6)},dashed] (-2,0) arc (180:135:2 and 0.6);
\draw[shift={(7,-6)},dashed,opacity=0.2] (2,0) arc (0:180:2 and 0.6);

\draw[transparent,name path=A1] (-1.29,1.52) arc [start angle=-200, end angle = 20, x radius = 13.75mm, y radius = 3.15mm];

\draw[red!40,shift={(7,0)},name path=A2] (120:2cm) arc (180:360:1 and 0.2);

\draw[gray!50,shift={(7,0)}] (-1.29,1.52) arc [start angle=-200, end angle = 20, x radius = 13.75mm, y radius = 3.15mm];

\draw[red!40,shift={(0,-6)},name path=A31] (120:2cm) arc (180:360:1 and 0.2);

\draw[transparent,shift={(0,-6)},name path=A3] (-1.29,1.52) arc [start angle=-200, end angle = 20, x radius = 13.75mm, y radius = 3.15mm];

\draw[transparent,shift={(7,-6)},name path=A4] (-1.29,1.52) arc [start angle=-200, end angle = 20, x radius = 13.75mm, y radius = 3.15mm];

\draw[gray!60,name path=B1] (50:2cm) arc (50:130:2cm);

\draw[gray!60,shift={(7,0)},name path=B2] (60:2cm) arc (60:120:2cm);

\draw[gray!60,shift={(0,-6)},name path=B3] (50:2cm) arc (50:130:2cm);

\draw[gray!60,shift={(0,-6)},name path=B31] (60:2cm) arc (60:120:2cm);

\draw[gray!60,shift={(7,-6)},name path=B4] (50:2cm) arc (50:130:2cm);

\draw[gray!60,shift={(7,-6)}] (20:2cm) arc (20:160:2cm);
\begin{scope}
        \draw[transparent,fill=green!30,opacity=0.3,](130:2cm) 
            arc (130:410:2)
            arc [start angle=20, end angle = -200,
    x radius = 13.75mm, y radius = 3.15mm];
    \end{scope}

    \begin{scope}
        \draw[shift={(7,0)},fill=yellow!40,opacity=0.3,](120:2cm) 
            arc (120:420:2)
            arc (360:180:1 and 0.2);
    \end{scope}
    
    \begin{scope}
        \draw[shift={(0,-6)},fill=green!30,opacity=0.3,](130:2cm) 
            arc (130:410:2)
            arc [start angle=20, end angle = -200,
    x radius = 13.75mm, y radius = 3.15mm];
    \end{scope}

    \begin{scope}
        \draw[shift={(7,-6)},fill=red!30,opacity=0.3](160:2cm) 
            arc (160:380:2)
            arc (360:180:1.883 and 0.4);
    \end{scope}

    \tikzfillbetween[of=B2 and A2, on layer=ft]{red!40, opacity=0.3};
    
    \tikzfillbetween[of=B1 and A1, on layer=ft]{blue!40, opacity=0.3};

    \tikzfillbetween[of=B31 and A31, on layer=ft]{red!40, opacity=0.3};
    
    \tikzfillbetween[of=B4 and A4, on layer=ft]{blue!40, opacity=0.3};

    \node[below, scale=0.9] at (0,-2.5) {Domain of the map $u_t$};

    \node[shift={(7,0)},below, scale=0.9] at (0,-2.5) {Domain of the map $u_s$};
    
    \node[shift={(0,-6)},below, scale=0.9] at (0,-2.5) {Domain of the map $v_t$};
    
    \node[shift={(7,-6)},below, scale=0.9] at (0,-2.5) {Domain of the map $w_t$};

      \draw[shift={(-0.1,0)}] (125:2cm) edge[<-, bend right] (-2.1,1.9);
    \node[left, scale=0.7] at (-2.1,1.9) {$u_t \Big \rvert_{B(x_1,2\rho)}$}; 

    \draw[shift={(-0.1,0)}] (235:2cm) edge[<-, bend left] (-2.1,-1.9);
    \node[left, scale=0.7] at (-2.1,-1.9) {$u_t \Big\rvert_{\S^n\setminus B(x_1,2\rho)}$};
    
     \draw[shift={(6.9,0.05)}] (90:2cm) edge[<-, bend right] (-2.1,1.9);
     \node[shift={(6.9,0)},left, scale=0.7] at (-2.1,1.9) {$u_s \Big\rvert_{B(x_1,\rho)}$};
    
    \draw[shift={(6.9,0)}] (235:2cm) edge[<-, bend left] (-2.1,-1.9);
    \node[shift={(6.9,0)},left, scale=0.7] at (-2.1,-1.9) {$u_s \Big\rvert_{\S^n\setminus B(x_1,\rho)}$};
    
    \draw[shift={(6.9,-6)}] (125:2cm) edge[<-, bend right] (-2.1,1.9);
     \node[shift={(6.9,-6)},left, scale=0.7] at (-2.1,1.9) {$u_t \Big\rvert_{B(x_1,2\rho)}$};
    
    \draw[shift={(6.9,-6)}] (235:2cm) edge[<-, bend left] (-2.1,-1.9);
    \node[shift={(6.9,-6)},left, scale=0.7] at (-2.1,-1.9) {$u_s \Big\rvert_{ B(x_1,\rho)}\circ\tau$};
    
    \draw[shift={(6.9,-6)}] (155:2cm) edge[<-, bend right] (-3,0);
    \node[shift={(6.9,-6)},below, scale=0.7] at (-3,0) {glue};
    
        \draw[shift={(-0.1,-5.95)}] (90:2cm) edge[<-, bend right] (-2.1,1.9);
     \node[shift={(-0.1,-6)},left, scale=0.7] at (-2.1,1.9) {$u_s \Big\rvert_{B(x_1,\rho)}$};
    
    \draw[shift={(-0.1,-6)}] (235:2cm) edge[<-, bend left] (-2.1,-1.9);
    \node[shift={(-0.1,-6)},left, scale=0.7] at (-2.1,-1.9) {$u_t \Big\rvert_{\S^n\setminus B(x_1,2\rho)}$};
    
    \draw[shift={(-0.1,-6)}] (130:2cm) edge[<-, bend right] (-3,0);
    \node[shift={(-0.1,-6)},below, scale=0.7] at (-3,0) {glue};
    \end{tikzpicture}
\end{center}

Let $\Gamma_1,\, \Gamma_2$ be the free homotopy classes determined respectively by $v_t$ and $w_t$. Then, we have as in \eqref{eq:freehomotopysum}
\[
 \pi_1(\n)\gamma_0 \subset \pi_1(\n)\gamma_1 + \pi_1(\n)\gamma_2.
\]
\underline{\textsc{Step 1.}} We will prove that
\[
 \lim_{t \to s^+} \abs{ E_{t,\frac ns}(v_t, \S^n)+ E_{t,\frac ns}(w_t, \S^n) - E_{t,\frac ns}(u_t, \S^n)}  =  \mathcal O(\rho^{n\frac{s_0-s}{s_0+s}}) \text{ as $\rho \to 0$}.
\]

Indeed, to see this we first decompose $\S^n\times \S^n$ into the compliment of the balls $\S^n \setminus B(x_1,2\rho) \times \S^n\setminus B(x_1,2\rho)$, the product of the balls $B(x_1,2\rho)\times B(x_1,2\rho)$, and the two mixed terms $B(x_1,2\rho) \times \S^n\setminus B(x_1,2\rho)$, $\S^n\setminus B(x_1,2\rho)\times B(x_1,2\rho)$. We recall that
\[
\begin{split}
 v_t &= u_t \quad \text{ on } \S^n\setminus B(x_1,2\rho)\\
 w_t &= u_t \quad \text{ on } B(x_1,2\rho).
\end{split}
 \]
Applying those observations we get
\[
\begin{split}
 &E_{t,\frac ns}(v_t,\S^n) +  E_{t,\frac ns}(w_t,\S^n) -  E_{t,\frac ns}(u_t,\S^n)\\
 &= E_{t,\frac ns}(v_t,B(x_1,2\rho)) + 2 \int_{\S^n\setminus B(x_1,2\rho)}\int_{B(x_1,2\rho)} \frac{|v_t(x) - v_t(y)|^{\frac ns}}{|x-y|^{n+\frac{nt}{s}}} \dif x \dif y + E_{t,\frac ns}(w_t,\S^n\setminus B(x_1,2\rho))\\ 
 &\quad + 2 \int_{\S^n\setminus B(x_1,2\rho)}\int_{B(x_1,2\rho)} \frac{|w_t(x) - w_t(y)|^{\frac ns}}{|x-y|^{n+\frac{nt}{s}}} \dif x \dif y - 2 \int_{\S^n\setminus B(x_1,2\rho)}\int_{B(x_1,2\rho)} \frac{|u_t(x) - u_t(y)|^{\frac ns}}{|x-y|^{n+\frac{nt}{s}}} \dif x \dif y.
 \end{split}
\]

Thus,
\begin{equation}\label{eq:vtwtminusutfirstestimate}
\begin{split}
 \big|E_{t,\frac ns}(v_t,\S^n) &+  E_{t,\frac ns}(w_t,\S^n) -  E_{t,\frac ns}(u_t,\S^n)\big|\\
 &\le 2\int_{\S^n} \int_{B(x_1,2\rho)} \frac{|v_t(x) - v_t(y)|^{\frac ns}}{|x-y|^{n+\frac{nt}{s}}} \dif x \dif y + 2\int_{\S^n} \int_{\S^n\setminus B(x_1,2\rho)} \frac{|w_t(x) - w_t(y)|^{\frac ns}}{|x-y|^{n+\frac{nt}{s}}} \dif x \dif y\\
 &\quad + 2 \int_{\S^n\setminus B(x_1,2\rho)}\int_{B(x_1,2\rho)} \frac{|u_t(x) - u_t(y)|^{\frac ns}}{|x-y|^{n+\frac{nt}{s}}} \dif x \dif y.
 \end{split}
\end{equation}

We begin with the estimate of the last term in \eqref{eq:vtwtminusutfirstestimate}. To do so, we observe that we can decompose
\begin{equation}\label{eq:utdecompostionrestterm}
\begin{split}
\int_{\S^n\setminus B(x_1,2\rho)}&\int_{B(x_1,2\rho)} \frac{|u_t(x) - u_t(y)|^{\frac ns}}{|x-y|^{n+\frac{nt}{s}}} \dif x \dif y\\
&\aleq E_{t,\frac ns}(u_t,B(x_1,3\rho)\setminus B(x_1,\rho)) + \int_{\S^n\setminus B(x_1,3\rho)} \int_{B(x_1,2\rho)} \frac{|u_t(x) - u_t(y)|^{\frac ns}}{|x-y|^{n+\frac{nt}{s}}} \dif x \dif y\\
&\quad+ \int_{B(x_1,3\rho)\setminus B(x_1,2\rho)} \int_{B(x_1,\rho)} \frac{|u_t(x) - u_t(y)|^{\frac ns}}{|x-y|^{n+\frac{nt}{s}}} \dif x \dif y \eqqcolon I_1 + I_2+I_3.
 \end{split}
\end{equation}

\underline{The estimate of $I_2$}: We will first estimate the term $I_2$. We start with noting that for $x\in B(x_1,2\rho)$ we have $u_t(x) = w_t(x)$ and for $y\in\S^n\setminus B(x_1,3\rho)$ we have $u_t(y) = v_t(y)$, thus
\begin{equation}\label{eq:utsplitintowtvt}
 |u_t(x) - u_t(y)|^\frac ns = |w_t(x) - v_t(y)|^\frac ns \aleq |w_t(x) - w_t(z)|^\frac{n}{s} + |w_t(z) - v_t(\tilde z)|^\frac{n}{s} |+ |v_t(\tilde z) - v_t(y)|^\frac{n}{s}. 
\end{equation}
Applying this to $I_2$ and integrating the inequality over $\barint_{B(x_1,3\rho)\setminus B(x_1,2\rho)}$ with respect to $\dif z$ and over $\barint_{B(x_1,2\rho)\setminus B(x_1,\frac{3\rho}{2})}$ with respect to $\dif \tilde z$, gives us

\[
 \begin{split}
  I_2 &\aleq \barint_{B(x_1,3\rho)\setminus B(x_1,2\rho)} \int_{\S^n\setminus B(x_1,3\rho)}\int_{B(x_1,2\rho)} \frac{|w_t(x) - w_t(z)|^\frac ns}{|x-y|^{n+\frac{nt}{s}}} \dif x \dif y \dif z\\
  &\quad +  \barint_{B(x_1, 2\rho)\setminus B(x_1,\frac{3\rho}{2})}  \int_{\S^n\setminus B(x_1,3\rho)} \int_{B(x_1,2\rho)} \frac{|v_t(\tilde z) - v_t(y)|^\frac ns}{|x-y|^{n+\frac{nt}{s}}} \dif x \dif y \dif \tilde z\\
  &\quad + \barint_{B(x_1,2\rho)\setminus B(x_1,\frac{3\rho}{2})}\barint_{B(x_1,3\rho)\setminus B(x_1,2\rho)} \int_{\S^n\setminus B(x_1,3\rho)} \int_{B(x_1,2\rho)} \frac{|w_t(z) - v_t(\tilde z)|^\frac ns}{|x-y|^{n+\frac{nt}{s}}} \dif x \dif y \dif z \dif \tilde z\\
  & \eqqcolon I_{2,1} + I_{2,2} + I_{2,3}.
 \end{split}
\]
Now we estimate $I_{2,1}$ and note that $|x-y|\ge \rho$, thus integrating over the $y$ variable
\begin{equation}\label{eq:I21estimate}
 \begin{split}
  I_{2,1}&=\barint_{B(x_1,3\rho)\setminus B(x_1,2\rho)} \int_{\S^n\setminus B(x_1,3\rho)} \int_{B(x_1,2\rho)} \frac{|w_t(x) - w_t(z)|^\frac ns}{|x-y|^{n+\frac{nt}{s}}} \dif x \dif y \dif z\\
  & \aleq \rho^{-\frac{nt}{s}} \barint_{B(x_1,3\rho)\setminus B(x_1,2\rho)} \int_{B(x_1,2\rho)} |w_t(x) - w_t(z)|^\frac{n}{s} \frac{|x-z|^{n+\frac{nt}{s}}}{|x-z|^{n+\frac{nt}{s}}} \dif x \dif z\\
  &\aleq \int_{B(x_1,3\rho)\setminus B(x_1,2\rho)} \int_{B(x_1,2\rho)} \frac{|w_t(x) - w_t(z)|^\frac{n}{s}}{|x-z|^{n+\frac{nt}{s}}} \dif x \dif z,
 \end{split}
\end{equation}
 in the last inequality we used $|x-z|\aleq \rho$. 
 
As for the term $I_{2,2}$ we observe that $|x-y| \ge 2\rho - |y|$
\[
 \begin{split}
  I_{2,2}&= \barint_{B(x_1, 2\rho)\setminus B(x_1,\frac{3\rho}{2} )} \int_{\S^n\setminus B(x_1,3\rho)} \int_{B(x_1,2\rho)}  \frac{|v_t(\tilde z) - v_t(y)|^\frac ns}{|x-y|^{n+\frac{nt}{s}}} \dif x \dif y \dif \tilde z\\
  &\aleq \rho^{n} \barint_{B(x_1, 2\rho)\setminus B(x_1,\frac{3\rho}{2} )} \int_{\S^n\setminus B(x_1,3\rho)} \frac{|v_t(\tilde z) - v_t(y)|^\frac ns}{|2\rho-|y||^{n+\frac{nt}{s}}}  \dif y \dif \tilde z.
 \end{split}
\]
For $y\in \S^n\setminus B(x_1,3\rho)$ and $\tilde z\in B(x_1,2\rho)\setminus B(x_1,\frac{3\rho}{2})$ we have
\[
 |y-\tilde z|\le \dist(y,B(x_1,2\rho)) + \dist(\tilde z,B(x_1,2\rho)) \le 2 \dist(y,B(x_1,2\rho)) \le 2 |2\rho - |y||.
\]
Thus,
\begin{equation}\label{eq:I22estimate}
 I_{2,2} \aleq \int_{B(x_1,2\rho)\setminus B(x_1,\frac{3\rho}{2})} \int_{\S^n\setminus B(x_1,3\rho)} \frac{|v_t(\tilde z) - v_t(y)|^{\frac ns}}{|\tilde z - y|^{n+\frac{nt}{s}}} \dif y \dif \tilde z.
\end{equation}
Next, we estimate $I_{2,3}$. We begin with the observation that $|x-y|\ge \rho$, from which we deduce
\begin{equation}\label{eq:I23estimate}
\begin{split}
 I_{2,3}&=\barint_{B(x_1,2\rho)\setminus B(x_1,\frac{3\rho}{2})}\barint_{B(x_1,3\rho)\setminus B(x_1,2\rho)} \int_{\S^n\setminus B(x_1,3\rho)} \int_{B(x_1,2\rho)} \frac{|w_t(z) - v_t(\tilde z)|^\frac ns}{|x-y|^{n+\frac{nt}{s}}} \dif x \dif y \dif z \dif \tilde z\\
 &\aleq \rho^{2n-n-\frac{nt}{s}} \barint_{B(x_1,2\rho)\setminus B(x_1,\frac{3\rho}{2})}\barint_{B(x_1,3\rho)\setminus B(x_1,2\rho)} |w_t(z) - v_t(\tilde z)|^\frac ns \dif z \dif \tilde z\\
 &\aleq \rho^{-n-\frac{nt}{s}}\int_{B(x_1,2\rho)\setminus B(x_1,\frac{3\rho}{2})} \int_{B(x_1,3\rho)\setminus B(x_1,2\rho)} |\tilde u_{s;t}\circ \tau (z) - \tilde u_{s;t}(\tilde z)|^\frac ns \dif z \dif \tilde z\\
 &\aleq 
 \rho^{-\frac{nt}{s}} \int_{B(x_1,2\rho)\setminus B(x_1,\frac{3\rho}{2})} \int_{B(x_1,2\rho)} |\tilde u_{s;t}(\bar z) - \tilde u_{s;t}(\tilde z)|^{\frac ns} \frac{|\bar z- \tilde z|^{n+\frac{nt}{s}}}{|\bar z- \tilde z|^{n+\frac{nt}{s}}}\dif \bar z \dif \tilde z\\
 &\aleq \rho^n \int_{B(x_1,2\rho)\setminus B(x_1,\frac{3\rho}{2})} \int_{B(x_1,2\rho)} \frac{|\tilde u_{s;t}(\bar z) - \tilde u_{s;t}(\tilde z)|^{\frac ns}}{|\bar z- \tilde z|^{n+\frac{nt}{s}}}\dif \bar z \dif \tilde z,
\end{split}
 \end{equation}
we have used the estimate $|\nabla \tau|\simeq \frac{1}{\rho}$ and $|\bar z - \tilde z| \aleq \rho$.

\underline{The estimate of $I_3$}: Similarly, we can estimate the term $I_3$ by noting that for $x\in B(x_1,\rho)$ we also have $u_t(x) = w_t(x)$ and for $y\in B(x_1,3\rho)\setminus B(x_1,2\rho)$ we have $u_t(y) = v_t(y)$. We use again the inequality \eqref{eq:utsplitintowtvt} and integrate over $\barint_{B(x_1,3\rho)\setminus B(x_1,2\rho)}$ with respect to $\dif z$ and over $\barint_{B(x_1,2\rho)\setminus B(x_1,\frac{3\rho}{2})}$ with respect to $\dif \tilde z$ to get
\begin{equation}\label{eq:I3termestimate}
\begin{split}
 I_3 &\aleq \int_{B(x_1,3\rho)\setminus B(x_1,2\rho)}\int_{B(x_1,\rho)} \frac{|w_t(x) - w_t(z)|^{\frac ns}}{|x-z|^{n+\frac{nt}{s}}} \dif x \dif z\\
 &\quad + \int_{B(x_1,2\rho)\setminus B(x_1,\frac{3\rho}{2})}\int_{B(x_1,3\rho)\setminus B(x_1,2\rho)} \frac{|v_t(\tilde z) - v_t(y)|^{\frac ns}}{|\tilde z-y|^{n+\frac{nt}{s}}} \dif y \dif \tilde z\\
 &\quad+ \rho^n \int_{B(x_1,2\rho)\setminus B(x_1,\frac{3\rho}{2})} \int_{B(x_1,2\rho)} \frac{|\tilde u_{s;t}(\bar z) - \tilde u_{s;t}(\tilde z)|^\frac{n}{s}}{|\bar z - \tilde z|^{n+\frac{nt}{s}}} \dif \bar z \dif \tilde z . 
\end{split}
 \end{equation}

 \underline{The estimate of $I_1$}: As for the term $I_1$ we note that on $B(x_1,3\rho)\setminus B(x_1, \rho)$ we have strong convergence of $u_t$ and thus as in \eqref{eq:usestimateonball}
\begin{equation}\label{eq:II4estimate}
\begin{split}
 \lim_{t\to s^+} E_{t,\frac ns} (u_t,B(x_1,3\rho)\setminus B(x_1, \rho))
 &= E_{s,\frac ns}(u_s,B(x_1,3\rho)\setminus B(x_1, \rho)) \\
 &\le C \rho^{n\frac{s_0 - s}{s}} E_{s_0,\frac ns}(u_s,B(x_1,2\rho))
 =\mathcal O(\rho^{n\frac{s_0 - s}{s}}) \text{ as } \rho\to 0.
\end{split}
 \end{equation}
Finally, combining \eqref{eq:vtwtminusutfirstestimate} with \eqref{eq:utdecompostionrestterm}, \eqref{eq:I21estimate}, \eqref{eq:I22estimate}, and \eqref{eq:I23estimate} we obtain
\begin{equation}\label{eq:vtwtminusutalmostfinalestimate}
\begin{split}
 &\big|E_{t,\frac ns}(v_t,\S^n) +  E_{t,\frac ns}(w_t,\S^n) -  E_{t,\frac ns}(u_t,\S^n)\big|\\
 &\aleq \int_{\S^n} \int_{B(x_1,2\rho)} \frac{|v_t(x) - v_t(y)|^{\frac ns}}{|x-y|^{n+\frac{nt}{s}}} \dif x \dif y + \int_{\S^n} \int_{\S^n\setminus B(x_1,2\rho)} \frac{|w_t(x) - w_t(y)|^{\frac ns}}{|x-y|^{n+\frac{nt}{s}}} \dif x \dif y\\
 &\quad  + \rho^n \int_{B(x_1,2\rho)\setminus B(x_1,\frac{3\rho}{2})} \int_{B(x_1,2\rho)} \frac{|\tilde u_{s;t}(x) - \tilde u_{s;t}(y)|^\frac{n}{s}}{|x - y|^{n+\frac{nt}{s}}} \dif x \dif y + E_{t,\frac ns}(u_t,B(x_1,3\rho)\setminus B(x_1,\rho))\\
 &\eqqcolon II_1 + II_2 + II_3 + II_4.
\end{split}
\end{equation}
The last term $II_4$ is just $I_1$ and was estimated in \eqref{eq:II4estimate}.

\underline{The estimate of $II_1$}: In order to estimate the term $\int_{\S^n} \int_{B(x_1,2\rho)} \frac{|v_t(x) - v_t(y)|^{\frac ns}}{|x-y|^{n+\frac{nt}{s}}} \dif x \dif y$ we will use \Cref{la:smallnessonballs}. Let us assume that $t\le 2s$ and let $\alpha$, $\beta$, and $\lambda$ be from \Cref{la:smallnessonballs}. We have assumed in \eqref{eq:smallradiuslambda} that $B(x_1,\lambda \rho^\beta)\cap A = \{x_1\}$, thus $v_t$ converges to $u_s$ strongly in $W^{s,\frac ns}$ on this ball and we have
\[
\begin{split}
 \lim_{t\to s^+} E_{t,\frac ns}(v_t,B(x_1,\lambda\rho^\beta)) 
 &= E_{s,\frac ns}(u_s,B(x_1,\lambda\rho^\beta))\\
 &\approx \brac{\lambda \rho^\beta}^{n\frac{s_0-s}{s}}E_{s_0,\frac ns}(u_s,B(x_1,\lambda \rho^\beta))\\
 &=\mathcal O(\rho^\alpha) \text{ as } \rho\to0,
 \end{split}
 \]
 where $\alpha = \beta n \frac{s_0-s}{s}$, recall from \Cref{la:smallnessonballs} that we also have $\beta = \frac12 \brac{1-\frac \alpha n}$, thus we take
 \[
  \alpha = n\frac{s_0-s}{s_0+s} \quad \text{and} \quad \beta = \frac12 \brac{\frac{s_0+s-n(s_0-s)}{s_0+s}},
 \]
here we can assume without loss of generality that $s_0 < s\frac{n+1}{n-1}$ and thus $\beta>0$. Therefore, we obtain
\[
 \lim_{t\to s^+} E_{t,\frac ns}(v_t,B(x_1,\lambda\rho^\beta)) = \mathcal O(\rho^{n\frac{s_0-s}{s_0+s}}) \text{ as } \rho\to 0.
\]
This implies, by \Cref{la:smallnessonballs},
\begin{equation}\label{eq:II1estimate}
 \int_{\S^n} \int_{B(x_1,2\rho)} \frac{|v_t(x) - v_t(y)|^{\frac ns}}{|x-y|^{n+\frac{nt}{s}}} \dif x \dif y = \mathcal O(\rho^{n\frac{s_0-s}{s_0+s}}) \text{ as } \rho\to0.
\end{equation}

\underline{The estimate of $II_2$}: Similarly, in order to estimate the second term on the right-hand side of \eqref{eq:vtwtminusutalmostfinalestimate} we will use \Cref{la:smallnessoncompliments} for $t\le 2s$ with $\sigma = n\frac{s_0-s}{s}$ and $\theta = 2+\frac{s_0-s}{s}>1$. We note that $\rho$ can be taken sufficiently small to ensure that $B(x_1,\tilde \lambda^{-1}\rho^\theta)\subset B(x_1,2\rho)$ (here $\tilde \lambda = \tilde \lambda(\n,n,s)$ is taken from \Cref{la:smallnessoncompliments}). We recall that $B(x_1,2\rho)\subset B(x_1,\lambda\rho^\beta)$ and by \eqref{eq:smallradiuslambda} we know that $B(x_1,2\rho)\setminus B(x_1,\tilde \lambda^{-1}\rho^\theta)\cap A = \emptyset$, thus $u_t$ converges strongly to $u_s$ in $B(x_1,2\rho)\setminus B(x_1,\tilde \lambda^{-1}\rho^\theta)$. We have
\begin{equation}\label{eq:limitutandtau}
\begin{split}
\lim_{t\to s^+} &E_{t,\frac ns} (w_t, \S^n\setminus B(x_1,\tilde \lambda^{-1}\rho^\theta))\\
&= \lim_{t\to s^+} E_{t,\frac ns}(\tilde u_{s;t}\circ \tau ,\S^n\setminus B(x_1,2\rho)) + \lim_{t\to s^+}E_{t,\frac ns}(u_t,B(x_1,2\rho)\setminus B(x_1,\tilde \lambda^{-1}\rho^\theta))\\
&\aleq \lim_{t\to s^+} \int_{B(x_1,2\rho} \int_{B(x_1,2\rho)} \frac{|\tilde u_{s;t}(\tilde x) - \tilde u_{s;t}(\tilde y)|^{\frac ns}}{|\tau^{-1}(\tilde x) - \tau^{-1}(\tilde y)|^{n+\frac{nt}{s}}} |\nabla \tau|^{2n}\dif \tilde x \dif \tilde y\\
&\quad + E_{s,\frac ns}(u_s,B(x_1,2\rho)\setminus B(x_1,\tilde \lambda^{-1}\rho^\theta))\\
 &\aleq \lim_{t\to s^+} \int_{B(x_1,2\rho)} \int_{B(x_1,2\rho)} \frac{|\tilde u_{s;t}(\tilde x) - \tilde u_{s;t}(\tilde y)|^{\frac ns}}{|\tilde x - \tilde y|^{n+\frac{nt}{s}}} |\nabla \tau|^{n\brac{\frac{t-s}{s}}}\dif \tilde x \dif \tilde y\\
 &\quad + E_{s,\frac ns}(u_s,B(x_1,2\rho)\setminus B(x_1,\tilde \lambda^{-1}\rho^\theta))\\
 &\approx E_{s,\frac ns} (u_s,B(x_1,2 \rho)) = \mathcal O(\rho^{n\frac{s_0-s}{s}}) \text{ as } \rho\to 0.
 \end{split}
\end{equation}

By \Cref{la:smallnessoncompliments} this implies the smallness of
\begin{equation}\label{eq:II2estimate}
 \lim_{t\to s^+} \int_{\S^n}\int_{\S^n\setminus B(x_1,2\rho)} \frac{|w_t(x) - w_t(y)|^{\frac ns}}{|x-y|^{n+\frac{nt}{s}}} \dif x \dif y = \mathcal O(\rho^{n\frac{s_0-s}{s}}) \text{ as } \rho\to 0.
\end{equation}

\underline{The estimate of $II_3$}: We immediately obtain
\begin{equation}\label{eq:II3estimate}
 \lim_{t\to s^+} \rho^n \int_{B(x_1,2\rho)\setminus B(x_1,\frac{3\rho}{2})}\int_{B(x_1,2\rho)} \frac{|\tilde u_{s;t}(x) - \tilde u_{s;t}(y)|^\frac ns}{|x-y|^{n+\frac{nt}{s}}} \dif x \dif y = \mathcal O(\rho^{n\brac{\frac{s_0-s}{s}+1}}) = \mathcal O(\rho^{n\frac{s_0-s}{s}}) \quad \text{ as } \rho\to 0.
\end{equation}

Finally, we note that $\mathcal O(\rho^{n\frac{s_0-s}{s}}) = \mathcal O(\rho^{n\frac{s_0-s}{s_0+s}})$ as $\rho\to 0$. Thus, passing with $t$ to the limit in \eqref{eq:vtwtminusutalmostfinalestimate} and using \eqref{eq:II1estimate}, \eqref{eq:II2estimate}, \eqref{eq:II3estimate}, and \eqref{eq:II4estimate} we obtain
\[
 \lim_{t\to s^+} \big|E_{t,\frac ns}(v_t,\S^n) +  E_{t,\frac ns}(w_t,\S^n) -  E_{t,\frac ns}(u_t,\S^n)\big| = \mathcal O(\rho^{n\frac{s_0-s}{s_0+s}}) \text { as } \rho\to 0.
\]

\underline{\textsc{Step 2.}} Here, we verify the inequality \eqref{eq:sumofgamma12boundedby0plusdelta}.

From Step 1 we obtain
\[
 \liminf_{t\to s^+}\brac{ E_{t,\frac ns}(v_t, \S^n) + E_{t,\frac ns}(w_t, \S^n)} = \liminf_{t\to s^+} E_{t,\frac ns}(u_t,\S^n) + o(1) \text{ as } \rho\to 0.
\]

In particular, we have by \Cref{la:minX},
\[\begin{split}
 \#\Gamma_1 + \#\Gamma_2 \leq& \liminf_{t \to s^+} E_{t,\frac ns}(v_t,\S^n) + \liminf_{t \to s^+}E_{t,\frac ns}(w_t,\S^n)\\
 \leq &
 \liminf_{t \to s^+} E_{t,\frac ns}(u_t,\S^n) + o(1) = \#\Gamma_0 + o(1) \text{ as } \rho\to0.
 \end{split}
\]
Choosing $\rho \ll \delta$ we obtain
\begin{equation}\label{eq:X1X2Xpdelta}
 \#\Gamma_1 + \#\Gamma_2 \leq \#\Gamma_0 + o(1) \le \#\Gamma_0 + \delta.
\end{equation}

\underline{\textsc{Step 3.} $\Gamma_2$ is nontrivial.} Indeed, if $\Gamma_2$ was trivial then $w_t$ would be homotopic to a constant, and by definition of $w_t$ this would imply that there is a homotopy between $u_t$ and $\tilde u_{s;t}$ in $B(x_1,2\rho)$. But $u_t$ and $\tilde u_{s;t}$ coincide outside on $\partial B(x_1,2\rho)$, so we would obtain that $u_t\sim v_t$. Since $u_t$ is a minimizer in its homotopy class we would get
\[
 E_{t,\frac{n}{t}}(u_t,\S^n) \le E_{t,\frac{n}{t}}(v_t,\S^n). 
\]
Similarly as in \eqref{eq:utballholesmall}, in the proof of \Cref{th:main1}, for small enough $\rho$, this would lead to the estimate
\[
\lim_{t\to s^+}\int_{\S^n}\int_{B(x_1,2\rho)} \frac{|u_t(x) - u_t(y)|^\frac ns}{|x-y|^{2n}} \dif x \dif y \le \eps,
\]
which is a a contradiction to \eqref{eq:blowuppoint}.

\underline{\textsc{Step 4.} $\Gamma_1$ is nontrivial}: Assume that $\Gamma_1$ is trivial, then $v_t$ is homotopic to a constant. That gives us a homotopy between then $u_t$ on $\S^n\setminus B(x_1,2\rho)$ and $\tilde u_{s;t}$ on $B(x_1,2\rho)$. Thus, we obtain that $u_t$ is homotopically equivalent to $\tilde u_{s;t} \circ \tau$ in $\S^n \setminus B(x_1,2\rho)$. Thus, $u_t$ is homotopic to $w_t$ and from the minimality of $u_t$ we get
\begin{equation}\label{eq:utlessthanwt}
 E_{t,\frac ns}(u_t,\S^n) \le E_{t,\frac ns}(w_t ,\S^n).
\end{equation}
Noting again, that 
\[
 \S^n\times \S^n = \brac{B(x_1,2\rho) \times B(x_1,2\rho)} \cup \brac{\S^n\setminus B(x_1,2\rho) \times B(x_1,2\rho)} \cup \brac{\S^n \times \S^n\setminus B(x_1,2\rho)}. 
\]
From \eqref{eq:utlessthanwt} and $u_t = w_t$ on $B(x_1,2\rho)$, we have
\[
\begin{split}
&\int_{\S^n\setminus B(x_1,2\rho)} \int_{B(x_1,2\rho)} \frac{|u_t(x) - u_t(y)|^\frac ns}{|x-y|^{n+\frac{nt}{s}}} \dif x \dif y + \int_{\S^n} \int_{\S^n\setminus B(x_1,2\rho)} \frac{|u_t(x) - u_t(y)|^\frac ns}{|x-y|^{n+\frac{nt}{s}}} \dif x \dif y\\
&\le   \int_{\S^n\setminus B(x_1,2\rho)} \int_{B(x_1,2\rho)} \frac{|w_t(x) - w_t(y)|^\frac ns}{|x-y|^{n+\frac{nt}{s}}} \dif x \dif y+ \int_{\S^n} \int_{\S^n\setminus B(x_1,2\rho)} \frac{|w_t(x) - w_t(y)|^\frac ns}{|x-y|^{n+\frac{nt}{s}}} \dif x \dif y,
\end{split}
\]
we also have
\[
\begin{split}
&\int_{\S^n} \int_{\S^n\setminus B(x_1,2\rho)} \frac{|u_t(x) - u_t(y)|^\frac ns}{|x-y|^{n+\frac{nt}{s}}} \dif x \dif y\\
&\le  \int_{\S^n\setminus B(x_1,2\rho)} \int_{B(x_1,2\rho)} \frac{|u_t(x) - u_t(y)|^\frac ns}{|x-y|^{n+\frac{nt}{s}}} \dif x \dif y + \int_{\S^n} \int_{\S^n\setminus B(x_1,2\rho)} \frac{|u_t(x) - u_t(y)|^\frac ns}{|x-y|^{n+\frac{nt}{s}}} \dif x \dif y
\end{split}
\]
and by the symmetry of the integral
\[
 \begin{split}
  &\int_{\S^n\setminus B(x_1,2\rho)} \int_{B(x_1,2\rho)} \frac{|w_t(x) - w_t(y)|^\frac ns}{|x-y|^{n+\frac{nt}{s}}} \dif x \dif y+ \int_{\S^n} \int_{\S^n\setminus B(x_1,2\rho)} \frac{|w_t(x) - w_t(y)|^\frac ns}{|x-y|^{n+\frac{nt}{s}}} \dif x \dif y\\
  & \le 2 \int_{\S^n} \int_{\S^n\setminus B(x_1,2\rho)} \frac{|w_t(x) - w_t(y)|^\frac ns}{|x-y|^{n+\frac{nt}{s}}} \dif x \dif y.
 \end{split}
\]
Thus,
\[
\int_{\S^n\setminus B(x_1,2\rho)} \int_{\S^n} \frac{|u_t(x) - u_t(y)|^{\frac ns}}{|x-y|^{n+\frac{nt}{s}}} \dif x \dif y 
\leq 2\int_{\S^n\setminus B(x_1,2\rho)} \int_{\S^n} \frac{|w_t(x) - w_t(y)|^{\frac ns}}{|x-y|^{n+\frac{nt}{s}}} \dif x \dif y. 
\]
In order to estimate the latter one, we will again use \Cref{la:smallnessoncompliments}. We recall from \textsc{Step 1}, \eqref{eq:limitutandtau} that for a $\lambda = \lambda(\n,n,s)>1$ we have 
\begin{equation}\label{eq:energycomparisonwithwt}
 \lim_{t\to s^+} E_{t,\frac ns} (w_t,\S^n\setminus B(x_1,\lambda^{-1}\rho^{2+\frac{s_0-s}{s}})) = \mathcal O(\rho^{n\frac{s_0-s}{s}}) \text { as } \rho\to 0.
 \end{equation}
Now from \Cref{la:smallnessoncompliments} the latter implies  
\[
 \lim_{t\to s^+} \int_{\S^n\setminus B(x_1,2\rho)} \int_{\S^n} \frac{|w_t(x) - w_t(y)|^{\frac ns}}{|x-y|^{n+\frac{nt}{s}}} \dif x \dif y = \mathcal O(\rho^{n\frac{s_0-s}{s}}) \text{ as } \rho\to0.
\]
Thus, passing with $t$ to $s$ in \eqref{eq:energycomparisonwithwt} we obtain
\[
\lim_{t\to s^+} \int_{\S^n\setminus B(x_1,2\rho)} \int_{\S^n} \frac{|u_t(x) - u_t(y)|^{\frac ns}}{|x-y|^{n+\frac{nt}{s}}} \dif x \dif y = \mathcal O(\rho^{n\frac{s_0-s}{s}}) \text{ as } \rho\to0.
\] 
That is, if $\Gamma_1$ was trivial, then for all $t$ sufficiently close to $s$ we would have
\begin{equation}\label{eq:22}
 \int_{\S^n\setminus B(x_1,2\rho)} \int_{\S^n} \frac{|u_t(x) - u_t(y)|^{\frac ns}}{|x-y|^{n+\frac{nt}{s}}} \dif x \dif y\le C \rho^{n\frac{s_0-s}{s}.}
\end{equation}
Then combining this with \Cref{th:fractionalestimatesmalldiskbyremaining} we obtain for all $t$ sufficiently close to $s$
\[
\begin{split}
\int_{B(x_1,2\rho)} \int_{\S^n} \frac{|u_t(x) - u_t(y)|^{\frac ns}}{|x-y|^{n+\frac{nt}{s}}} \dif x \dif y &\le C \rho^{-n\brac{\frac ts -1}}\int_{\S^n\setminus B(x_1,2\rho)} \int_{\S^n} \frac{|u_t(x) - u_t(y)|^{\frac ns}}{|x-y|^{n+\frac{nt}{s}}} \dif x \dif y\\
&\aleq \rho^{n\frac{s_0-t}{s}} \ll \eps.
\end{split}
\]
This contradicts \eqref{eq:blowuppoint}, so $\Gamma_1$ has to be also nontrivial.

\underline{\textsc{Step 5.} Estimate of $\#\Gamma_1$ and $\#\Gamma_2$.} Now since both $\Gamma_1$ and $\Gamma_2$ are nontrivial, we must have $\#\Gamma_1$, $\#\Gamma_2 > \theta$ for some $\theta > 0$,  since by \Cref{la:smallenergytrivial} we know that very small energy implies trivial homotopy class.

Moreover, choosing $\delta < \frac{\theta}{2}$ we also get from \eqref{eq:X1X2Xpdelta} that \[\#\Gamma_1 \leq \#\Gamma_0 + \delta - \#\Gamma_2 < \#\Gamma_0+\delta -\theta \le \#\Gamma_0 - \frac \theta 2\]
and similarly $\#\Gamma_2 < \#\Gamma_0-\frac{\theta}{2}$.
\end{proof}

The proof of \Cref{th:main2} follows now exactly as in \cite[Theorem 5.5]{Sucks1}, but for reader's convenience we repeat it here.

\begin{proof}[Proof of \Cref{th:main2}]
 Let $\theta>0$ be the number from \Cref{la:smallenergytrivial} such that $E_{s,\frac ns}(u,\S^n)<\theta$ implies trivial homotopy class. Without loss of generality, we may assume that $\theta$ is also the number from \Cref{la:5.4}. Let $P$ be the subgroup generated by the elements $\Gamma_i\in X$. Assume on the contrary that $P$ does not generate the whole $\pi_n(\n)$ acted on by $\pi_1(\n)$. Then, we would be able to find a class $\tilde \Gamma \not \in P$, such that for any $\Gamma'$ with $\#\Gamma'< \# \tilde \Gamma - \frac \theta2$ we have $\Gamma'\in P$.
  
Since there are no minimizing $W^{s,\frac ns}$-harmonic maps in $\tilde\Gamma$, applying \Cref{la:5.4} to $\tilde\Gamma$ we obtain that there exists two other nontrivial homotopy classes $\tilde \Gamma_1$ and $\tilde \Gamma_2$ such that
\[
 \pi_1(\n)\tilde \gamma \subset \pi_1(\n) \tilde \gamma_1 + \pi_1(\n) \tilde \gamma_2, \quad \#\tilde\Gamma_1 + \#\tilde \Gamma_2 < \#\tilde \Gamma + \frac \theta2, \quad \text{ and } \quad \#\tilde\Gamma_1, \, \#\tilde \Gamma_2 >\theta.
\]
This implies that $\#\tilde \Gamma_1,\, \#\tilde \Gamma_2< \#\tilde\Gamma - \frac \theta2$, so both sets $\pi_1(\n)\tilde \gamma_1,\, \pi_1(\n)\tilde \gamma_2 \in P$. Thus, we also have
\[
 \pi_1(\n)\tilde \gamma \subset \pi_1(\n) \tilde \gamma_1 + \pi_1(\n) \tilde \gamma_2 \subset P.
\]

\end{proof}

\appendix
\section{Observations on the smallness condition}
Let us remark that smallness conditions (that will be needed throughout the paper)
\[
 \int_{B(r)} \int_{B(r)} \frac{|u(x)-u(y)|^{p}}{|x-y|^{n+sp}}\dif x\dif y < \eps
\]
and
\[
 \int_{B(r)} \int_{\Sigma} \frac{|u(x)-u(y)|^{p}}{|x-y|^{n+sp}}\dif x\dif y < \eps
\]
are essentially equivalent. This is due to the following lemma. 
\begin{lemma}\label{la:smallnessonballs}
 Let $u\in W^{s,\frac nt}(\Sigma, \n)$, where $s\in(0,1)$, $t\ge s$ then there exists a $\lambda=\lambda(\n,n,s,\eps)>1$ such that
 \[
  \text{if }\int_{B(\lambda r^\frac st)} \int_{B(\lambda r^\frac st)} \frac{|u(x)-u(y)|^{\frac ns}}{|x-y|^{n+\frac{nt}{s}}}\dif x\dif y < \frac{\eps}{2},
  \text{ then }
  \int_{B(r)} \int_{\Sigma} \frac{|u(x)-u(y)|^{\frac ns}}{|x-y|^{n+\frac{nt}{s}}}\dif x\dif y < \eps.
 \]
where $0<r\le 1$. In particular if $t\le 2s$, then it suffices to take on the left-hand side of the inequality the integration over the ball $B(\lambda r^\frac 12)$.
 
 Moreover, there exists a $\lambda =\lambda(\n,n,s)$ such that if $\alpha>0$ and $0<\beta \coloneqq \frac st\brac{1- \frac \alpha n} <1$
 \[
  \int_{B(\lambda r^\beta)} \int_{B(\lambda r^\beta)} \frac{|u(x)-u(y)|^{\frac ns}}{|x-y|^{n+\frac{nt}{s}}}\dif x\dif y = \mathcal O (r^\alpha) \text{ as } r \to 0\]
then
\[
  \int_{B(r)} \int_{\Sigma} \frac{|u(x)-u(y)|^{\frac ns}}{|x-y|^{n+\frac{nt}{s}}}\dif x\dif y = \mathcal O (r^\alpha) \text{ as } r\to 0.
 \]
In particular, when $t\le 2s$, then it suffices to take $\beta = \frac12 \brac{1-\frac \alpha n}$.
\end{lemma}

\begin{proof}
 We begin with the decomposition
 \begin{equation}\label{eq:decompo1}
 \begin{split}
  \int_{B(r)}\int_{\Sigma} \frac{|u(x)-u(y)|^{\frac ns}}{|x-y|^{n+\frac{nt}{s}}}\dif x\dif y &= \int_{B(r)} \int_{B(\Lambda r)} \frac{|u(x)-u(y)|^{\frac ns}}{|x-y|^{n+\frac{nt}{s}}}\dif x\dif y\\
  &\quad + \int_{B(r)}\int_{\Sigma \setminus B(\Lambda r)} \frac{|u(x)-u(y)|^{\frac ns}}{|x-y|^{n+\frac{nt}{s}}}\dif x\dif y.
 \end{split}
 \end{equation}
We begin with the estimate of the second term. We have
\begin{equation}\label{eq:smallnessconditionwithbiglambda}
 \begin{split}
  \int_{B(r)}\int_{\Sigma \setminus B(\Lambda r)} \frac{|u(x)-u(y)|^{\frac ns}}{|x-y|^{n+\frac{nt}{s}}}\dif x\dif y &\aleq \|u\|_{L^{\infty}(\Sigma )}^{\frac ns} r^{n}\int_{|z|\ge r(\Lambda-1)} |z|^{-n-\frac{nt}{s}} \dif z\\
  &\aleq \|u\|_{L^{\infty}(\Sigma )}^{\frac ns} r^{n} \brac{r(\Lambda-1)}^{-\frac{nt}{s}}\\
  &=\|u\|_{L^{\infty}(\Sigma )}^{\frac ns} r^{n-\frac{nt}{s}}\Lambda^{-\frac{nt}{s}}\brac{1-\Lambda^{-1}}^{-\frac{nt}{s}}.
 \end{split}
\end{equation}
Thus, taking $\Lambda = \lambda r^{\frac{s}{t}-1}$ in \eqref{eq:smallnessconditionwithbiglambda} we get
\[
\begin{split}
 \int_{B(r)}\int_{\Sigma \setminus B(\Lambda r)} \frac{|u(x)-u(y)|^{\frac ns}}{|x-y|^{n+\frac{nt}{s}}}\dif x\dif y 
 &\aleq \|u\|_{L^\infty(\Sigma )}^{\frac ns}\lambda^{-\frac{nt}{s}}\brac{1-\frac{r^{1-\frac{s}{t}}}{\lambda}}^{-\frac{nt}{s}}\\
 &\aleq \|u\|_{L^\infty(\Sigma )}^{\frac ns} \lambda^{-n}\xrightarrow{\lambda\to\infty} 0,
\end{split}
 \]
where the estimate does not depend on $t$. 

As for the first term of \eqref{eq:decompo1}, with this choice of $\Lambda$ we observe that $B(\Lambda r) = B(\lambda r^{\frac st})$ and since $r<1$ and $s\le t$ we also have
\[
 B(r) \subset B(\lambda r^{\frac st})
\]
and thus
\[
 \int_{B(r)}\int_{B(\Lambda r)} \frac{|u(x) - u(y)|^\frac ns}{|x-y|^{n+\frac{nt}{s}}} \dif x \dif y \le \int_{B(\lambda r^\frac st)}\int_{B(\lambda r^\frac st)} \frac{|u(x) - u(y)|^\frac ns}{|x-y|^{n+\frac{nt}{s}}} \dif x \dif y <\frac \eps 2.
\]
This finishes the proof of the first part of the Lemma.

Similarly to get the second part we set $\Lambda = \lambda r^{\frac st -1 - \frac{s}{nt} \alpha}$ in \eqref{eq:smallnessconditionwithbiglambda} and obtain
\[
\begin{split}
  \int_{B(r)}\int_{\Sigma \setminus B(\Lambda r)} \frac{|u(x)-u(y)|^{\frac ns}}{|x-y|^{n+\frac{nt}{s}}}\dif x\dif y 
 &\aleq \|u\|_{L^\infty(\Sigma )}^{\frac ns}\lambda^{-\frac{nt}{s}}r^\alpha \brac{1-\frac{r^{1-\frac st \brac{1-\frac \alpha n}}}{\lambda}}^{-\frac{nt}{s}}\\
 &\aleq \|u\|_{L^\infty(\Sigma )}^{\frac ns} \lambda^{-n}r^\alpha.
\end{split}
 \]
 
 Now, it suffices to estimate the first term of \eqref{eq:decompo1}. With this choice of $\Lambda$ we have $\Lambda r = \lambda r^{\frac st\brac{1-\frac \alpha n}}$ and we observe that for $\alpha>0$ we have $\beta \coloneqq \frac st \brac{1-\frac \alpha n} <1$, thus since $0<r\le 1$ we have
 \[
  B(r)\subset B(\lambda r^\beta),
 \]
which gives by assumptions
\[
\int_{B(r)}\int_{B(\Lambda r)} \frac{|u(x) - u(y)|^\frac ns}{|x-y|^{n+\frac{nt}{s}}} \dif x \dif y \le \int_{B(\lambda r^\beta)}\int_{B(\lambda r^\beta)} \frac{|u(x) - u(y)|^\frac ns}{|x-y|^{n+\frac{nt}{s}}} \dif x \dif y = \mathcal O (r^\alpha) \text{ as } r\to 0.
\]

\end{proof}

Similarly we also have the following smallness condition.
\begin{lemma}\label{la:smallnessoncompliments}
 Let $u\in W^{s,\frac nt}(\Sigma, \n)$, $s\in(0,1)$, and $t\ge s$, then there exists a $\lambda = \lambda(\n,n,s)>1$ such that if for a $\sigma>0$ and $\theta \coloneqq \frac{t}{s}+\frac{\sigma}{n}>1$ we have
 \[
  \int_{\Sigma\setminus B(\lambda^{-1}r^\theta)} \int_{\Sigma\setminus B(\Lambda^{-1} r^\theta)} \frac{|u(x)-u(y)|^{\frac ns}}{|x-y|^{n+\frac{nt}{s}}}\dif x\dif y = \mathcal O(r^\sigma) \text{ as } r\to 0
 \]
then
\[
 \int_{\Sigma\setminus B(r)} \int_{\Sigma} \frac{|u(x)-u(y)|^{\frac ns}}{|x-y|^{n+\frac{nt}{s}}}\dif x\dif y = \mathcal O(r^\sigma) \text{ as } r\to 0.
\]
In particular if $t\le 2s$, it suffices to take $\theta = 2+\frac \sigma n$.
\end{lemma}

\begin{proof}
We begin with the decomposition
 \begin{equation}\label{eq:decompo2}
 \begin{split}
  \int_{\Sigma \setminus B(r)}\int_{\Sigma } \frac{|u(x)-u(y)|^{\frac ns}}{|x-y|^{n+\frac{nt}{s}}}\dif x\dif y &= \int_{\Sigma \setminus B(r)} \int_{B(\Lambda^{-1}r)} \frac{|u(x)-u(y)|^{\frac ns}}{|x-y|^{n+\frac{nt}{s}}}\dif x\dif y\\
  &\quad + \int_{\Sigma \setminus B(r)}\int_{\Sigma \setminus B(\Lambda^{-1}r)} \frac{|u(x)-u(y)|^{\frac ns}}{|x-y|^{n+\frac{nt}{s}}}\dif x\dif y.
 \end{split}
 \end{equation}
We begin with the estimate of the first term. We have
\[
\begin{split}
 \int_{\Sigma \setminus B(r)}\int_{B(\Lambda^{-1}r)} \frac{|u(x)-u(y)|^{\frac ns}}{|x-y|^{n+\frac{nt}{s}}}\dif x\dif y
 &\aleq \|u\|_{L^\infty(\Sigma )}^\frac ns (\Lambda^{-1}r)^{n}\int_{|z|\ge r(1-\Lambda^{-1})} |z|^{-n-\frac{nt}{s}} \dif z\\
 &\aleq \|u\|_{L^\infty(\Sigma )}^\frac ns (\Lambda^{-1}r)^n \brac{r(1-\Lambda^{-1})}^{-\frac{nt}{s}}\\
 &=\|u\|_{L^\infty(\Sigma )}^\frac ns r^{n-\frac{nt}{s}}\Lambda^{-n}\brac{1-\Lambda^{-1}}^{-\frac{nt}{s}}.
\end{split}
 \]
Thus taking $\Lambda=\lambda r^{1-\frac{t}{s}-\frac{\sigma}{n}}$, we have
\[
\begin{split}
 \int_{\Sigma \setminus B(r)}\int_{B(\Lambda^{-1}r)} \frac{|u(x)-u(y)|^{\frac ns}}{|x-y|^{n+\frac{nt}{s}}}\dif x\dif y 
 &\aleq \|u\|_{L^\infty(\Sigma )}^\frac ns \lambda^{-n} r^\sigma\brac{1-\frac{r^{\frac ts + \frac \sigma n-1}}{\lambda}}^{-\frac{nt}{s}}\\
 &\aleq \|u\|_{L^\infty(\Sigma )}^\frac ns \lambda^{-n}r^\sigma.
\end{split}
 \]
where the estimate does not depend on $t$.

As for the first term of \eqref{eq:decompo2}, for this choice of $\Lambda$, we have $\Lambda^{-1} r = \lambda^{-1} r^{\frac ts +\frac \sigma n}$ and since $t\ge s$ we have $\theta \coloneqq \frac ts + \frac \sigma n >1$, thus for sufficiently small $r$
\[
 B(\Lambda^{-1}r) \subset B(r), \quad \text{ thus } \Sigma \setminus B(r) \subset \Sigma \setminus B(\Lambda^{-1}r).
\]
This implies
\[
\begin{split}
 \int_{\Sigma \setminus B(r)}\int_{\Sigma \setminus B(\Lambda^{-1}r)} \frac{|u(x)-u(y)|^{\frac ns}}{|x-y|^{n+\frac{nt}{s}}}\dif x\dif y &\le \int_{\Sigma \setminus B(\Lambda^{-1}r)}\int_{\Sigma \setminus B(\Lambda^{-1}r)} \frac{|u(x)-u(y)|^{\frac ns}}{|x-y|^{n+\frac{nt}{s}}}\dif x\dif y\\
 &= \mathcal O(r^\sigma), \text{ as } r\to 0.
\end{split}
 \]

\end{proof}

\section{A Sobolev-type estimate for Gagliardo-type spaces}
Here we record the Sobolev estimates we are using throughout the paper. All of them essentially follow the theory of Triebel--Lizorkin and Besov spaces, cf. \cite{RunstSickel,T83,GModern} and their Sobolev embedding theory.
\begin{theorem}\label{th:sobolev}
Assume that $s \in (0,1)$, $t \in (s,1)$, and $p, p^\ast \in (1,\infty)$ with 
\[
s - \frac{n}{p^\ast} = t - \frac{n}{p}.
\]
Then 
\begin{enumerate}
 \item 
If $f \in \dot{W}^{t,p}(\R^n)$
\begin{equation}\label{eq:sob1}
 [f]_{W^{s,p^\ast}(\R^n)} \aleq [f]_{W^{t,p}(\R^n)} 
\end{equation}
and
\begin{equation}\label{eq:sob2}
 \brac{\int_{\R^n}\brac{\int_{\R^n} \frac{|f(x)-f(y)|^p}{|x-y|^{n+sp}}\, \dy}^{\frac{p^\ast}{p}} \dx}^{\frac{1}{p^\ast}} \aleq [f]_{W^{t,p}(\R^n)}. 
\end{equation}
Moreover
\begin{equation}\label{eq:sob2b}
 \|\laps{s} f\|_{L^{p^\ast}(\R^n)} \aleq [f]_{W^{t,p}(\R^n)}. 
\end{equation}
\item If $f \in W^{t,p}(B)$ for some ball $B \subset \R^n$, we have (with a constant independent 
of the specific ball)
\begin{equation}\label{eq:sob3}
 [f]_{W^{s,p^\ast}(B)} \aleq [f]_{W^{t,p}(B)}.
\end{equation}
and
\begin{equation}\label{eq:sob4}
 \brac{\int_{B}\brac{\int_{B} \frac{|f(x)-f(y)|^p}{|x-y|^{n+sp}}\, \dy}^{\frac{p^\ast}{p}}}^{\frac{1}{p^\ast}} \aleq [f]_{W^{t,p}(B)}. 
\end{equation}
\end{enumerate}
\end{theorem}

\begin{proof}
The statements are consequences of the theory of Besov spaces $\dot{B}^{s}_{p,q}$ and Triebel--Lizorkin spaces $\dot{F}^s_{p,q}$.

The first estimate \eqref{eq:sob1} follows from Sobolev embedding for Triebel--Lizorkin spaces $\dot{F}^{t}_{p,p}(\R^n) \hookrightarrow \dot{F}^{s}_{p^\ast,p^\ast}(\R^n)$, see \cite{J77} or \cite[Theorem 2.71]{T83}. We then have by the characterization of $W^{s,p}$ in terms of Triebel--Lizorkin spaces $F^{s}_{p,p}$, see \cite[\textsection 2.6, Proposition 3, p.95]{RunstSickel} and \cite[\textsection 2.1.2, Proposition, p.14]{RunstSickel}, 
\[
 [f]_{W^{s,p^\ast}(\R^n)} \aeq \|f\|_{\dot{F}^{s}_{p^\ast, p^\ast}(\R^n)} \aleq \|f\|_{\dot{F}^{t}_{p, p}(\R^n)} \aeq [f]_{W^{t,p}(\R^n)}.
\]
For the second estimate \eqref{eq:sob2} we first recall the following well-known integral inequality (which follows from Riesz duality and Fubini's theorem) for any $r \geq1$,
\[
 \brac{\int_{\R^n} \brac{\int_{\R^n} |G(x,h)| dh}^r \dx}^{\frac{1}{r}} \leq \int_{\R^n}\brac{\int_{\R^n} |G(x,h)|^r \dx}^{\frac{1}{r}} dh.
\]
Applying this to $G(x,h) \coloneqq \frac{|f(x)-f(x+h)|^p}{|h|^{n+sp}}$ and $r = \frac{p^\ast}{p} > 1$ we have
\[
\begin{split}
 &\brac{\int_{\R^n}\brac{\int_{\R^n} \frac{|f(x)-f(y)|^p}{|x-y|^{n+sp}}\, \dy}^{\frac{p^\ast}{p}} \dx}^{\frac{1}{p^\ast}} \\
 &=\brac{\int_{\R^n}\brac{\int_{\R^n} \frac{|f(x)-f(x+h)|^p}{|h|^{n+sp}}\, \dif h}^{\frac{p^\ast}{p}} \dx }^{\frac{1}{p^\ast}} \\
 &\leq\brac{\int_{\R^n} \frac{1}{|h|^{n+sp}} \brac{\int_{\R^n} |f(x)-f(x+h)|^{p^\ast} \dx}^{\frac{p}{p^\ast} }dh}^{\frac{1}{p}}\\
 &\aeq\|f\|_{\dot{B}^{s}_{p^\ast,p}(\R^n)}.
 \end{split}
\]
In the last step we  used the integral identification of the Besov space $\dot{B}^s_{p^\ast,p}$, see again \cite[\textsection 2.6, Proposition 3, p.95]{RunstSickel} and \cite[\textsection 2.1.2, Proposition, p.14]{RunstSickel}.

By Sobolev embedding for Besov spaces,  \cite{J77}, we have $\dot{B}^s_{p,p}(\R^n) \hookrightarrow \dot{B}^{s}_{p^\ast,p}(\R^n)$ and moreover $\dot{B}^s_{p,p}(\R^n) = \dot{F}^s_{p,p}(\R^n)$, see \cite[\textsection 2.1, Remark 6., p.10]{RunstSickel}. Again by the characterization of $W^{s,p}$ in terms of Triebel--Lizorkin spaces $F^{s}_{p,p}$, see \cite[\textsection 2.6, Proposition 3, p.95]{RunstSickel} and \cite[\textsection 2.1.2, Proposition, p.14]{RunstSickel}, we arrive at 
\[
\begin{split}
 \brac{\int_{\R^n}\brac{\int_{\R^n} \frac{|f(x)-f(y)|^p}{|x-y|^{n+sp}}\, \dy}^{\frac{p^\ast}{p}} \dx}^{\frac{1}{p^\ast}} 
 \aleq\|f\|_{\dot{F}^{t}_{p,p}(\R^n)} \aeq [f]_{W^{t,p}(\R^n)}.
 \end{split}
\]
This establishes \eqref{eq:sob2}.

As for \eqref{eq:sob2b}, by \cite[\textsection 2.6, Proposition 3, p.95]{RunstSickel} and \cite[\textsection 2.1.2, Proposition, p.14]{RunstSickel}, we have 
\[
 \|\laps{s} f\|_{L^{p^\ast}(\R^n)} \aeq \|f\|_{\dot{F}^{s}_{p^\ast,2}(\R^n)}.
\]
Sobolev embedding for Triebel--Lizorkin spaces also implies $\dot{F}^{t}_{p,2}(\R^n) \hookrightarrow \dot{F}^{s}_{p^\ast,p^\ast}(\R^n)$ since $t > s$, see \cite{J77}, so we have
\[
 \|\laps{s} f\|_{L^{p^\ast}(\R^n)} \aeq \|f\|_{\dot{F}^{s}_{p^\ast,2}(\R^n)} \aleq \|f\|_{\dot{F}^{t}_{p,p}(\R^n)} \aeq [f]_{W^{t,p}(\R^n)}.
\]

As for  \eqref{eq:sob3} and  \eqref{eq:sob4}, by rescaling we may assume without loss of generality that $B = B(0,1)$.  

Let $f \in W^{t,p}(B(0,1))$. We may assume that $(f)_{B(0,1)} = 0$ otherwise we consider $f-(f)_{B(0,1)}$ instead. $B(0,1)$ is an extension domain, so there exists $\tilde{f} \in W^{t,p}(\R^n)$, and 
\[
 [\tilde{f}]_{W^{t,p}(\R^n)} \aleq \|f\|_{L^p(B(0,1))} + [f]_{W^{t,p}(B(0,1))} \aleq [f]_{W^{t,p}(B(0,1))}.
\]
In the last step we used Poincar\'e lemma. Applying \eqref{eq:sob1} to $\tilde{f}$ we obtain \eqref{eq:sob3}, and applying \eqref{eq:sob2} to $\tilde{f}$ we obtain \eqref{eq:sob4}.
\end{proof}

\section{A Luckhaus-type lemma}
The Luckhaus' Lemma, \cite[Lemma~1]{L88} or \cite[Section 2.6, Lemma~1]{S96}, provides a way to glue together two maps in different regions with a precise estimate on the Sobolev norms. This is an important tool in the theory of harmonic maps --- in particular in the supercritical space. 
In \cite{BRSV20} there is a $1$-dimension fractional version of this Lemma. We extend this here to any dimension, which might be a useful result in its own right. Observe that the estimate is somewhat suboptimal for $W^{s,p}$-spaces with $sp < n-1$ (where Luckhaus' original Lemma develops its full force). We also make no effort to obtain an optimal estimate with respect to $\delta$ as $\delta \to 0$, since this is not what we need. So one might argue that the following is only reminiscent of the Luckhaus' Lemma.

\begin{lemma}\label{la:frluckhaus}
Let $n\geq 1$ $s \in (0,1)$, $p \in (1,\infty)$, $r > 0$, $u,v : \R^n \to \R^M$ such that $u,v \Big |_{\partial B(r)}$ are continuous and\footnote{For $n=1$ the term $\int_{\partial B(r)} \int_{\partial B(r)} \frac{|u(\theta)-u(\omega)|^{p}}{|\theta-\omega|^{n-1+sp}} \dif \theta \dif \omega$ is not needed.}
\[
 \int_{\partial B(r)} \int_{\R^n} \frac{|u(\theta)-u(y)|^{p}}{|\theta-y|^{n+sp}} \dy \dif \theta + \int_{\partial B(r)} \int_{\partial B(r)} \frac{|u(\theta)-u(\omega)|^{p}}{|\theta-\omega|^{n-1+sp}}  \dif \theta \dif \omega < \infty
 \]
 as well as
 \[
\int_{\partial B(r)} \int_{\R^n} \frac{|v(\theta)-v(y)|^{p}}{|\theta-y|^{n+sp}} \dy \dif \theta  <\infty.
 \]
For any $\delta \in (0,\frac{1}{4})$ we set 
\[
w(x) \coloneqq
\begin{cases}
u(x) \quad &|x| \ge r\\
(1-\eta(|x|)) u(\theta) + \eta(|x|) v(\theta) \quad &\theta = r\frac{x}{|x|},\ |x| \in ((1-\delta)r,r)\\
v(x/(1-\delta)) \quad & |x|\le (1-\delta)r,
 \end{cases}
\]
where $\eta: \R_+ \to [0,1]$ is smooth with $\eta(t) = 0$ for $t \geq (1-\frac{\delta}{2})r$ and $\eta \equiv 1$ on $[0,(1-\frac{3}{4} \delta)r]$, $|\eta'| \leq \frac{100}{\delta r}$.

Then
\begin{itemize}
\item For any $R \geq r$, where $K \coloneqq u(B(R)) \cup v(B(R))$,
 \begin{equation}\label{eq:frl:dist}
  \sup_{x \in B(R)} \dist(w(x),K) \leq \sup_{\theta \in \partial B(r)} |u(\theta)-v(\theta)|.
 \end{equation}
 \item We have for $\sigma \coloneqq \max\{p-1,sp\}$
 \begin{equation}\label{eq:frl:west}
 \begin{split}
  [w]_{W^{s,p}(B(2r))}^p &\leq [u]_{W^{s,p}(B(2r)\setminus B(r))}^p+[v]_{W^{s,p}(B(r))}^p\\
  &\quad +\delta^{-sp} r \brac{\int_{\partial B(r)} \int_{B(2r)\setminus B(r)} \frac{|u(\theta) - u(y)|^{p}}{|\theta-y|^{n+sp}} \dy \dif\theta +\int_{\partial B(r)} \int_{B(r)} \frac{|v(\theta) - v(y)|^{p}}{|\theta-y|^{n+sp}}  \dy \dif\theta}\\
  &\quad +\delta r {\int_{\partial B(r)} \int_{\partial B(r)} \frac{|u(\theta) - u(\omega)|^{p}}{|\theta-\omega|^{n-1+sp}} \dif\theta \dif \omega} + \delta^{-\sigma} r^{n-sp}\, \|v-u\|_{L^\infty(\partial B(r))}^p.
\end{split}
 \end{equation}
\end{itemize}
\end{lemma}
\begin{proof}
Estimate \eqref{eq:frl:dist} is almost obvious, indeed for $|x|>r$ or $|x| < (1-\delta)r$ we have  $w(x)=u(x)$ or $w(x) = v(x/(1-\delta))$ so $\dist(w(x),K) = 0$ unless $|x| \in ((1-\delta)r,r)$.

If $|x| \in ((1-\delta)r,r)$ then 
\[
 \dist(w(x),K) \leq |w(x)-u(rx/|x|)| \leq |u(rx/|x|)-v(rx/|x|)| \leq \sup_{\theta \in \partial B(r)} |u(\theta)-v(\theta)|.
\]
This establishes \eqref{eq:frl:dist}.

We now provide the estimate \eqref{eq:frl:west}. We assume from now on $n\geq 2$, and refer to the (very similar) case $n=1$ to \cite{BRSV20}.

We have 
\[
[w]_{W^{s,p}(B(2r))}^p \leq [u]_{W^{s,p}(B(2r)\setminus B(r))}^p+ I + II + 2III + 2IV +2V,
\]
where
\[
\begin{split}
 I &\coloneqq [v(\cdot/(1-\delta))]_{W^{s,p}(B((1-\delta)r))}^p\\
 II &\coloneqq \int_{B(r) \setminus B((1-\delta)r)} \int_{B(r) \setminus B((1-\delta)r)} \frac{|w(x)-w(y)|^{p}}{|x-y|^{n+sp}} \dx \dy\\
 III &\coloneqq\int_{B(2r)\setminus B(r)} \int_{B((1-\delta)r)} \frac{|w(x)-w(y)|^{p}}{|x-y|^{n+sp}} \dx \dy\\
IV&\coloneqq\int_{B(2r)\setminus B(r)} \int_{B(r) \setminus B((1-\delta)r)} \frac{|w(x)-w(y)|^{p}}{|x-y|^{n+sp}} \dx \dy\\
 V &\coloneqq \int_{B((1-\delta)r)} \int_{B(r) \setminus B((1-\delta)r)}\frac{|w(x)-w(y)|^{p}}{|x-y|^{n+sp}} \dx \dy.
 \end{split}
\]
First we observe 
\[
 I = (1-\delta)^{n-sp} [v]_{W^{s,p}(B(r))}^p \aleq [v]_{W^{s,p}(B(r))}^p.
\]
Next for $II$ we observe that for $x,\,y\in B(r)\setminus B((1-\delta)r)$ we have
\[
\begin{split}
 w(x)-w(y) 
 &= (1-\eta(|x|)) u(rx/|x|) + \eta(|x|) v(rx/|x|) - \brac{(1-\eta(|y|)) u(ry/|y|) + \eta(|y|) v(ry/|y|) }\\
 &=(1-\eta(|x|)) (u(rx/|x|)-u(ry/|y|)) + \eta(|x|) (v(rx/|x|)-v(ry/|y|))\\
 &\quad+\brac{\eta(|x|)- \eta(|y|)} \brac{v(ry/|y|)-u(ry/|y|)}.
 \end{split}
\]
Thus,
\begin{equation}\label{eq:frL:IIest}
 \begin{split}
II 
&\aleq  \int_{B(r) \setminus B((1-\delta)r)} \int_{B(r) \setminus B((1-\delta)r)} \frac{|u(rx/|x|)-u(ry/|y|)|^{p}}{|x-y|^{n+sp}} \dx \dy\\
&\quad +\int_{B(r) \setminus B((1-\delta)r)} \int_{B(r) \setminus B((1-\delta)r)} \frac{|v(rx/|x|)-v(ry/|y|)|^{p}}{|x-y|^{n+sp}} \dx \dy\\
& \quad +\|v-u\|_{L^\infty(\partial B(r))}^p \int_{B(r) \setminus B((1-\delta)r)} \int_{B(r) \setminus B((1-\delta)r)}\frac{|\eta(|x|)-\eta(|y|)|^{p}}{|x-y|^{n+sp}} \dx \dy. 
 \end{split}
\end{equation}
Firstly, we deal with the last term of \eqref{eq:frL:IIest} and observe 
\begin{equation}\label{eq:frL:245}
\begin{split}
 &\int_{B(r) \setminus B((1-\delta)r)} \int_{B(r) \setminus B((1-\delta)r)}\frac{|\eta(|x|)-\eta(|y|)|^{p}}{|x-y|^{n+sp}} \dx \dy\\
 &\aleq (\delta r)^{-p} \int_{B(r) \setminus B((1-\delta)r)} \int_{B(r) \setminus B((1-\delta)r)}\frac{1}{|x-y|^{n+(s-1)p}} \dx \dy\\
 &\aleq(\delta r)^{-p} r^{-(s-1)p} |B(r) \setminus B((1-\delta)r)|\\
 &\aleq \delta^{1-p} r^{n-sp}.
 \end{split}
\end{equation}
Next, we estimate the first term of \eqref{eq:frL:IIest}. We have with the help of \Cref{la:superdifficult}
\begin{equation}\label{eq:mimimimimimi}
\begin{split}
 &\int_{B(r) \setminus B((1-\delta)r)} \int_{B(r) \setminus B((1-\delta)r)} \frac{|u(rx/|x|)-u(ry/|y|)|^{p}}{|x-y|^{n+sp}} \dx \dy\\
 & = \int_{\S^{n-1}}\int_{\S^{n-1}} |u(r \theta) - u(r\omega)|^p \int_{(1-\delta)r}^r\int_{(1-\delta)r}^r \frac{1}{|\rho_1 \theta - \rho_2 \omega|^{n+sp}} \rho_1^{n-1} \rho_2^{n-1}\dif \rho_1 \dif \rho_2 \dif \omega \dif \theta\\
 &\aleq 
\delta r \int_{\S^{n-1}}\int_{\S^{n-1}} \frac{|u(r\theta)-u(r\omega)|^p}{|r\theta-r\omega|^{{n-1+sp}}}\, r^{n-1}\dif \omega\, r^{n-1}\dif \theta\\
& \approx \delta r \int_{\partial B(r)} \int_{\partial B(r)}\frac{|u(\theta)-u(\omega)|^p}{|\theta-\omega|^{{n-1+sp}}} \dif \omega \dif \theta.
 \end{split}
\end{equation}

Moreover, for the second term of \eqref{eq:frL:IIest} we observe that if $x,y \in B(r)\setminus B((1-\delta)r)$, then for $z_{x,y} \coloneqq \frac{x+y}{2}$ we have $|z_{x,y}-x| \aeq |x-y| \aeq |z_{x,y}-y|$, so
\begin{equation}\label{eq:mimi-mimi}
\begin{split}
 &\int_{B(r) \setminus B((1-\delta)r)} \int_{B(r) \setminus B((1-\delta)r)} \frac{|v(rx/|x|)-v(ry/|y|)|^{p}}{|x-y|^{n+sp}} \dx \dy\\
 &\aleq 2\int_{B(r) \setminus B((1-\delta)r)} \int_{B(r) \setminus B((1-\delta)r)} \frac{|v(rx/|x|)-v(z_{x,y})|^{p}}{|x-y|^{n+sp}} \dx \dy\\
  &\aeq \int_{B(r) \setminus B((1-\delta)r)} \int_{B(r) \setminus B((1-\delta)r)} \frac{|v(rx/|x|)-v(z_{x,y})|^{p}}{|x-z_{x,y}|^{n+sp}} \dx \dy\\
  &\aleq \int_{B(r) \setminus B((1-\delta)r)} \int_{B(r)} \frac{|v(rx/|x|)-v(z)|^{p}}{|x-z|^{n+sp}} \dx \dz\\
 &\aleq \delta r \int_{\partial B(r)} \int_{B(r)} \frac{|v(\theta)-v(z)|^{p}}{|\theta-z|^{n+sp}} \dif\theta \dz.
 \end{split}
\end{equation}
In the second to last step we used the transformation $y \mapsto z_{x,y}$.

Plugging \eqref{eq:frL:245}, \eqref{eq:mimimimimimi}, and \eqref{eq:mimi-mimi} into \eqref{eq:frL:IIest} we have shown
\[
 \begin{split}
II &\aleq   \delta r \int_{\partial B(r)}\int_{\partial B(r)} \frac{|u(\theta)-u(\omega)|^p}{|\theta-\omega|^{{n-1+sp}}} \dif \omega \dif \theta\\
&\quad +\delta r \int_{\partial B(r)} \int_{B(r)} \frac{|v(\theta)-v(z)|^{p}}{|\theta-z|^{n+sp}} \dif\theta \dz +\delta^{1-p} r^{n-sp}\|v-u\|_{L^\infty(\partial B(r))}^p . 
 \end{split}
\]
Next we estimate $III$. For any $\theta \in \partial B(r)$ we have
\[
\begin{split}
  III &=\int_{B(2r)\setminus B(r)} \int_{B((1-\delta)r)} \frac{\abs{v(x/(1-\delta))
  -u(y)}^{p}}{|x-y|^{n+sp}} \dx \dy\\
  &\aleq\int_{B(2r)\setminus B(r)} \int_{B((1-\delta)r)} \frac{\abs{v(x/(1-\delta))
  -v(\theta)}^{p}}{|x-y|^{n+sp}} \dx \dy\\
  &\quad +\int_{B(2r)\setminus B(r)} \int_{B((1-\delta)r)} \frac{\abs{u(\theta)-u(y)}^{p}}{|x-y|^{n+sp}} \dx \dy +\int_{B(2r)\setminus B(r)} \int_{B((1-\delta)r)} \frac{\abs{u(\theta)-v(\theta)}^{p}}{|x-y|^{n+sp}} \dx \dy\\
  &\aleq (\delta r)^{-sp} \int_{B(r)} \abs{v(x)
  -v(\theta)}^{p} \dx\\
  &\quad +(\delta r)^{-sp}  \int_{B(2r)\setminus B(r)} \abs{u(\theta)-u(y)}^{p} \dy +r^n (\delta r)^{-sp} \abs{u(\theta)-v(\theta)}^{p}.
   \end{split}
\]
Multiplying this inequality by $|\partial B(r)|^{-1} \aeq r^{1-n}$ and integrating in $\theta$ on $\partial B(r)$ we find
\[
\begin{split}
  III  &\aleq\delta^{-sp} r \brac{\int_{\partial B(r)} \int_{B(r)} \frac{\abs{v(x)
  -v(\theta)}^{p}}{|x-\theta|^{n+sp}} \dx \dif \theta +\int_{\partial B(r)} \int_{B(2r)/B(r)} \frac{\abs{u(y)
  -u(\theta)}^{p}}{|y-\theta|^{n+sp}} \dy \dif \theta}\\
  &\quad +\delta^{-sp} r^{n-sp} \|u-v\|_{L^\infty(\partial B(r))}^p.
   \end{split}
\]

Now we estimate $IV$
\[
\begin{split}
  IV &=\int_{B(2r)\setminus B(r)} \int_{B(r) \setminus B((1-\delta)r)} \frac{\abs{
  \brac{(1-\eta(|x|)) u(rx/|x|) + \eta(|x|) v(rx/|x|)}
  -u(y)}^{p}}{|x-y|^{n+sp}} \dx \dy\\
  &\aleq \int_{B(2r)\setminus B(r)} \int_{B(r) \setminus B((1-\delta)r)} \frac{\abs{u(rx/|x|)-u(y)}^{p}}{|x-y|^{n+sp}} \dx \dy\\
  &\quad +\int_{B(2r)\setminus B(r)} \int_{B(r) \setminus B((1-\delta)r)} \frac{|\eta(|x|)|^p\abs{
   u(rx/|x|) -v(rx/|x|)}^{p}}{|x-y|^{n+sp}} \dx \dy.
  \end{split}
\]
Observe that for $y \in B(2r)\setminus B(r)$ and $x \in B(r)\setminus B((1-\delta)r)$ we have $|x-y| \ageq \abs{rx/|x|-y}$.
\[
\begin{split}
 &\int_{B(2r)\setminus B(r)} \int_{B(r) \setminus B((1-\delta)r)} \frac{\abs{u(rx/|x|)-u(y)}^{p}}{|x-y|^{n+sp}} \dx \dy\\
 &\aleq \int_{(1-\delta)r}^r \int_{B(2r)\setminus B(r)}  \int_{\partial B(r)} \frac{\abs{u(\theta)-u(y)}^{p}}{|\theta-y|^{n+sp}}\brac{\frac{\rho}{r}}^{n-1} \dif\theta \dy \dif\rho\\ 
 &\aeq \delta r \int_{B(2r)\setminus B(r)}  \int_{\partial B(r)} \frac{\abs{u(\theta)-u(y)}^{p}}{|\theta-y|^{n+sp}} \dif\theta \dy.
 \end{split}
\]
Also, we know that for $y\in B(2r)\setminus B(r)$ we have $\eta(|y|)=0$, and thus we estimate
\[
\begin{split}
 &\int_{B(2r)\setminus B(r)} \int_{B(r) \setminus B((1-\delta)r)} \frac{|\eta(|x|)|^p\abs{
   u(rx/|x|) -v(rx/|x|)}^{p}}{|x-y|^{n+sp}} \dx \dy\\
   &=\int_{B(2r)\setminus B(r)} \int_{B(r) \setminus B((1-\delta)r)} \frac{|\eta(|x|)-\eta(|y|)|^p\abs{
   u(rx/|x|) -v(rx/|x|)}^{p}}{|x-y|^{n+sp}} \dx \dy\\
   &\aleq \|u-v\|_{L^\infty(\partial B(r))}^p \int_{B(2r)\setminus B(r)} \int_{B(r) \setminus B((1-\delta)r)} \frac{|\eta(|x|)-\eta(|y|)|^p}{|x-y|^{n+sp}} \dx \dy\\
   &\aleq \delta^{1-sp} r^{n-p} \|u-v\|_{L^\infty(\partial B(r))}^p .
\end{split}
\]
In the last step we argued similar to \eqref{eq:frL:245}.

So we have shown
\[
 \begin{split}
IV &\aleq  \delta r \int_{B(2r)\setminus B(r)}  \int_{\partial B(r)} \frac{\abs{u(\theta)-u(y)}^{p}}{|\theta-y|^{n+sp}} \dif\theta \dy +\delta^{1-p} r^{n-sp} \|u-v\|_{L^\infty(\partial B(r))}^p .
 \end{split}
\]

With essentially the same argument we get the estimate of $V$
\[
 \begin{split}
V &\aleq   \delta r \int_{B(r)}  \int_{\partial B(r)} \frac{\abs{v(\theta)-v(y)}^{p}}{|\theta-y|^{n+sp}} \dif\theta \dy +\delta^{1-p} r^{n-sp} \|u-v\|_{L^\infty(\partial B(r))}^p.
 \end{split}
\]
Combining the estimates on $I$, $II$, $III$, $IV$, and $V$ we obtain inequality \eqref{eq:frl:west}. This concludes the proof of \Cref{la:frluckhaus}.
\end{proof}

Above we used the following 
\begin{lemma}\label{la:superdifficult}
For any $\alpha > 1$ there exists a constant $C(\alpha)$ such that for any $R > 0$, $\lambda \in (0,1)$ and any $\theta,\omega \in \S^{n-1}$,
 \[
 \int_{\lambda R}^{R} \int_{\lambda R}^{R} |r\theta-\rho \omega|^{-\alpha} \dif r \dif\rho \leq C(\alpha) (1-\lambda)\lambda^{1-\alpha}\, R\, |R\theta-R\omega|^{1-\alpha}.
\]
\end{lemma}
\begin{proof}
Observe that 
\[
  \int_{\lambda R}^{R} \int_{\lambda R}^{R} |r\theta-\rho \sigma|^{-\alpha}\, \dif r \dif \rho
  =R^{2-\alpha} \int_{\lambda }^{1} \int_{\lambda }^{1} |r\theta-\rho \sigma|^{-\alpha}\, \dif r \dif\rho
\]
so it suffices to prove the claim for $R = 1$, which we will assume from now on.

Furthermore, we observe
\[
\begin{split}
 |r\theta-\rho \omega|^2 &= r^2 + \rho^2 - 2 r\rho \langle \theta,\omega \rangle\\
&= r^2 + \rho^2 - 2 r\rho + 2r\rho (1-\langle \theta,\omega \rangle)\\
&= (r-\rho)^2 + r\rho |\theta-\omega|^2.
 \end{split}
\]
Now observe that for $r,\rho \geq \lambda$
\[
 |r\theta-\rho \omega| \ageq  \max \{|r-\rho|,\lambda |\theta-\omega|\}.
\]
and thus for any $\alpha > 0$
\[
 |r\theta-\rho \omega|^{-\alpha} \aleq  \min \{|r-\rho|^{-\alpha},\lambda^{-\alpha} |\theta-\omega|^{-1}\}.
\]
Then we split
\[
\begin{split}
 &\int_{\lambda}^1\int_{\lambda}^1|r\theta-\rho \omega|^{-\alpha} \dif r \dif \rho\\
 &=\int_{\lambda}^1\int_{\lambda}^1 \chi_{\{|r-\rho| < \lambda |\theta-\omega|\}} |r\theta-\rho \omega|^{-\alpha} \dif r \dif \rho+\int_{\lambda}^1\int_{\lambda}^1 \chi_{\{|r-\rho| > \lambda |\theta-\omega|\}} |r\theta-\rho \omega|^{-\alpha} \dif r \dif \rho\\
&\aleq\lambda^{-\alpha}\int_{\lambda}^1\int_{\lambda}^1 \chi_{\{|r-\rho| < \lambda |\theta-\omega|\}} |\theta-\omega|^{-\alpha} \dif r \dif \rho+\int_{\lambda}^1\int_{\lambda}^1 \chi_{\{|r-\rho| > \lambda |\theta-\omega|\}} (r-\rho)^{-\alpha} \dif r \dif \rho.
\end{split}
\]
Observe that 
\[
\begin{split}
 &\lambda^{-\alpha}\int_{\lambda}^1\int_{\lambda}^1 \chi_{|r-\rho| < \lambda |\theta-\omega|} |\theta-\omega|^{-\alpha} \dif r \dif \rho \\
 &\leq  \lambda^{-\alpha}|\theta-\omega|^{-\alpha} \int_{\lambda}^1 \mathcal{L}^1(\{\rho: |r-\rho| < \lambda |\theta-\omega|\})\\
 &= 2(1-\lambda) \lambda^{1-\alpha} |\theta-\omega|^{1-\alpha}.
 \end{split}
\]
Morover, since $\alpha > 1$,
\[
\begin{split}
  &\int_{\lambda}^1\int_{\lambda}^1 \chi_{|r-\rho| > \lambda |\theta-\omega|} (r-\rho)^{-\alpha} \dif r \dif \rho\\
  &\leq \int_{\lambda}^1\int_{|r-\rho|>\lambda |\theta-\omega|} (r-\rho)^{-\alpha} \dif r \dif \rho\\
  &=\int_{\lambda}^1 \frac{1}{1-\alpha} \lambda^{1-\alpha} |\theta-\omega|^{1-\alpha} \dif \rho\\
  &=\frac{(1-\lambda)\lambda^{1-\alpha}}{(1-\alpha)}\, |\theta-\omega|^{1-\alpha}.
  \end{split}
\]
We now conclude.
\end{proof}

\begin{remark}
While the formulation of \Cref{la:frluckhaus} suffices for our purposes, let us remark that in the assumptions and in the inequality \eqref{eq:frl:west} the term
\[
  r \int_{\partial B(r)} \int_{\partial B(r)} \frac{|u(\theta)-u(\sigma)|^{p}}{|\theta-\sigma|^{{n-1+sp}}}  \dif \theta \dif \sigma
  \]
 can be replaced by  
 \[
  r \int_{\partial B(r)} \int_{\R^n \setminus B(r)} \frac{|u(\theta)-u(y)|^{p}}{|\theta-y|^{n+sp}}  \dif y \dif \theta.
  \]
  
Indeed, the only modification that has to be made is in the estimate of \eqref{eq:mimimimimimi}. This can be done in the following way:
\[
\begin{split}
 &\int_{B(r) \setminus B((1-\delta)r)} \int_{B(r) \setminus B((1-\delta)r)} \frac{|u(rx/|x|)-u(ry/|y|)|^{p}}{|x-y|^{n+sp}} \dx \dy\\
 &\aleq\int_{B(r) \setminus B((1-\delta)r)} \int_{B(r) \setminus B((1-\delta)r)} \chi_{|x-y| < \frac{1}{100}r}\frac{|u(rx/|x|)-u(ry/|y|)|^{p}}{|x-y|^{n+sp}} \dx \dy\\
 &\quad +r^{-n-sp} \int_{B(r) \setminus B((1-\delta)r)}\int_{B(r) \setminus B((1-\delta)r)} |u(rx/|x|) - (u)_{A(r)}|^{p}  \dx \dy.
\end{split}
\]
For the second term we have
\[
 \begin{split}
  &r^{-n-sp} \int_{B(r) \setminus B((1-\delta)r)}\int_{B(r) \setminus B((1-\delta)r)} |u(rx/|x|) - (u)_{A(r)}|^{p} \dx \dy\\
  &\aleq r^{n-sp}|B(r) \setminus B((1-\delta)r)|\int_{B(r) \setminus B((1-\delta)r)} |u(rx/|x|) - (u)_{A(r)}|^{p} \dx\\
  &\approx \delta r^{-sp} \int_{(1-\delta r)}^r \int_{\partial B(r)} |u(\theta) - (u)_{A(r)}|^p \dif \theta\\
  &\aleq \delta^2 r^{1-sp}  r^{-n}\int_{\partial B(r)}\int_{B(2r)\setminus B(r)} \frac{|u(\theta) - u(z)|^p}{|\theta -z|^{n+sp}}|\theta -z|^{n+sp} \dif \theta \dz \\
  &\aleq \delta^2 r \int_{\partial B(r)}\int_{B(2r)\setminus B(r)} \frac{|u(\theta) - u(z)|^p}{|\theta -z|^{n-1+sp}} \dif \theta \dz.
 \end{split}
\]
For the first term we observe 
\[
\begin{split}
 &\int_{B(r) \setminus B((1-\delta)r)} \int_{B(r) \setminus B((1-\delta)r)} \chi_{\{|x-y| < \frac{1}{100}r\}}\frac{|u(rx/|x|)-u(ry/|y|)|^{p}}{|x-y|^{n+sp}} \dx \dy\\
 &\aleq \int_{B(r)\setminus B((1-\delta)r)} \int_{B(r)\setminus B((1-\delta)r)} \chi_{\{|x-y| < \frac{1}{100}r\}}\frac{|u(rx/|x|)-u(z_{x,y})|^{p}}{|x-y|^{n+sp}} \dx \dy\\
&\aleq  \int_{B(r)\setminus B((1-\delta)r)} \int_{B(2r)\setminus B(r)} \frac{|u(rx/|x|)-u(z)|^{p}}{|rx/|x|-z|^{n+sp}} \dx \dz\\
&\aleq \delta r \int_{\partial B(r)} \int_{B(2r)\setminus B(r)} \frac{|u(\theta)-u(z)|^{p}}{|\theta-z|^{n+sp}} \dz \dif \theta,
 \end{split}
\]
where we have chosen an intermediate point $z_{x,y}$ with the following properties: $z_{x,y}\in B(2r)\setminus B(r)$, $|r x/|x| - z_{x,y}|\approx |ry/|y| - z_{x,y}| \approx |x-y|$, $z_{x,x} = rx/|x|$, and $z_{y,y} = ry/|y|$. This can be done by using a diffeomorphism $\tau\colon \mathcal C\cap \{|x-y|<\frac{1}{100}r\} \to \mathcal K$ that transforms a cone $\mathcal C$ centered at the origin that contains the ball $\{|x-y|<\frac{1}{100}r\}$ intersected with the annulus $B(2r)\setminus B(r)$, into a convex set $\mathcal K$. Then we can take as the intermediate point $z_{x,y}$ the preimage $\tau^{-1}$ of the convex combination of the image (under the diffeomorphism $\tau$) of the points $rx/|x|$ and $r y/|y|$. This is quite technical, so for convenience of the authors, we leave the details to the reader.
\end{remark}

\bibliographystyle{abbrv}%
\bibliography{bib}%

\begin{thebibliography}{100}

\bibitem{Adams-Hedberg}
D.~R. Adams and L.~I. Hedberg.
\newblock {\em Function spaces and potential theory}, volume 314 of {\em
  Grundlehren der Mathematischen Wissenschaften [Fundamental Principles of
  Mathematical Sciences]}.
\newblock Springer-Verlag, Berlin, 1996.

\bibitem{RandomRemy}
A.~Aftalion and R.~Rodiac.
\newblock One dimensional phase transition problem modelling striped spin orbit
  coupled bose-einstein condensates, 2017.

\bibitem{AMN16}
I.~Agol, F.~C. Marques, and A.~Neves.
\newblock Min-max theory and the energy of links.
\newblock {\em J. Amer. Math. Soc.}, 29(2):561--578, 2016.

\bibitem{BZ11}
J.~M. Ball and A.~Zarnescu.
\newblock Orientability and energy minimization in liquid crystal models.
\newblock {\em Arch. Ration. Mech. Anal.}, 202(2):493--535, 2011.

\bibitem{BGMM03}
J.~R. Banavar, O.~Gonzalez, J.~H. Maddocks, and A.~Maritan.
\newblock Self-interactions of strands and sheets.
\newblock {\em J. Statist. Phys.}, 110(1-2):35--50, 2003.

\bibitem{BR14}
Y.~Bernard and T.~Rivi\`ere.
\newblock Energy quantization for {W}illmore surfaces and applications.
\newblock {\em Ann. of Math. (2)}, 180(1):87--136, 2014.

\bibitem{Bethuel-approximation}
F.~Bethuel.
\newblock The approximation problem for {S}obolev maps between two manifolds.
\newblock {\em Acta Math.}, 167(3-4):153--206, 1991.

\bibitem{Bethuel-Zheng}
F.~Bethuel and X.~M. Zheng.
\newblock Density of smooth functions between two manifolds in {S}obolev
  spaces.
\newblock {\em J. Funct. Anal.}, 80(1):60--75, 1988.

\bibitem{BRS16}
S.~Blatt, P.~Reiter, and A.~Schikorra.
\newblock Harmonic analysis meets critical knots. {C}ritical points of the
  {M}\"{o}bius energy are smooth.
\newblock {\em Trans. Amer. Math. Soc.}, 368(9):6391--6438, 2016.

\bibitem{BRS19}
S.~{Blatt}, P.~{Reiter}, and A.~{Schikorra}.
\newblock {On O'hara knot energies I: Regularity for critical knots}.
\newblock {\em arXiv e-prints}, page arXiv:1905.06064, May 2019.

\bibitem{BRSV20}
S.~Blatt, P.~Reiter, A.~Schikorra, and N.~Vorderobermeier.
\newblock On tangent-point energies for curves.
\newblock {\em in preparation}, 2020.

\bibitem{BH93}
B.~Bojarski and P.~Haj\l{}asz.
\newblock Pointwise inequalities for {S}obolev functions and some applications.
\newblock {\em Studia Math.}, 106(1):77--92, 1993.

\bibitem{BT82}
R.~Bott and L.~W. Tu.
\newblock {\em Differential forms in algebraic topology}, volume~82 of {\em
  Graduate Texts in Mathematics}.
\newblock Springer-Verlag, New York-Berlin, 1982.

\bibitem{BPVS15}
P.~Bousquet, A.~C. Ponce, and J.~Van~Schaftingen.
\newblock Strong density for higher order {S}obolev spaces into compact
  manifolds.
\newblock {\em J. Eur. Math. Soc. (JEMS)}, 17(4):763--817, 2015.

\bibitem{Branding20}
V.~Branding.
\newblock On the {E}volution of {R}egularized {D}irac-{H}armonic {M}aps from
  {C}losed {S}urfaces.
\newblock {\em Results Math.}, 75(2):Paper No. 57, 30, 2020.

\bibitem{BL17}
L.~Brasco and E.~Lindgren.
\newblock Higher {S}obolev regularity for the fractional {$p$}-{L}aplace
  equation in the superquadratic case.
\newblock {\em Adv. Math.}, 304:300--354, 2017.

\bibitem{BLS18}
L.~Brasco, E.~Lindgren, and A.~Schikorra.
\newblock Higher {H}\"{o}lder regularity for the fractional {$p$}-{L}aplacian
  in the superquadratic case.
\newblock {\em Adv. Math.}, 338:782--846, 2018.

\bibitem{Brezis-theinterplay}
H.~Brezis.
\newblock The interplay between analysis and topology in some nonlinear {PDE}
  problems.
\newblock {\em Bull. Amer. Math. Soc. (N.S.)}, 40(2):179--201, 2003.

\bibitem{BrezisCoronH}
H.~Brezis and J.-M. Coron.
\newblock Convergence of solutions of {$H$}-systems or how to blow bubbles.
\newblock {\em Arch. Rational Mech. Anal.}, 89(1):21--56, 1985.

\bibitem{Brezis-Li}
H.~Brezis and Y.~Li.
\newblock Topology and {S}obolev spaces.
\newblock {\em J. Funct. Anal.}, 183(2):321--369, 2001.

\bibitem{Brezis-mironescu-fractionaldensity}
H.~Brezis and P.~Mironescu.
\newblock Density in {$W^{s,p}(\Omega;N)$}.
\newblock {\em J. Funct. Anal.}, 269(7):2045--2109, 2015.

\bibitem{BC17}
H.-Q. Bui and T.~Candy.
\newblock A characterization of the {B}esov-{L}ipschitz and
  {T}riebel-{L}izorkin spaces using {P}oisson like kernels.
\newblock In {\em Functional analysis, harmonic analysis, and image processing:
  a collection of papers in honor of {B}j\"{o}rn {J}awerth}, volume 693 of {\em
  Contemp. Math.}, pages 109--141. Amer. Math. Soc., Providence, RI, 2017.

\bibitem{CS07}
L.~Caffarelli and L.~Silvestre.
\newblock An extension problem related to the fractional {L}aplacian.
\newblock {\em Comm. Partial Differential Equations}, 32(7-9):1245--1260, 2007.

\bibitem{DL13}
F.~Da~Lio.
\newblock Fractional harmonic maps into manifolds in odd dimension {$n>1$}.
\newblock {\em Calc. Var. Partial Differential Equations}, 48(3-4):421--445,
  2013.

\bibitem{DL15bubble}
F.~Da~Lio.
\newblock Compactness and bubble analysis for 1/2-harmonic maps.
\newblock {\em Ann. Inst. H. Poincar\'{e} Anal. Non Lin\'{e}aire},
  32(1):201--224, 2015.

\bibitem{DLPoho}
F.~Da~Lio.
\newblock Some remarks on {P}ohozaev-type identities.
\newblock In {\em Bruno {P}ini {M}athematical {A}nalysis {S}eminar 2018},
  volume~9 of {\em Bruno Pini Math. Anal. Semin.}, pages 115--136. Univ.
  Bologna, Alma Mater Stud., Bologna, 2018.

\bibitem{DaLio-Riviere-1Dmfd}
F.~Da~Lio and T.~Rivi\`ere.
\newblock Sub-criticality of non-local {S}chr\"{o}dinger systems with
  antisymmetric potentials and applications to half-harmonic maps.
\newblock {\em Adv. Math.}, 227(3):1300--1348, 2011.

\bibitem{DLR11}
F.~Da~Lio and T.~Rivi\`ere.
\newblock Three-term commutator estimates and the regularity of
  {$\frac12$}-harmonic maps into spheres.
\newblock {\em Anal. PDE}, 4(1):149--190, 2011.

\bibitem{DLS14}
F.~Da~Lio and A.~Schikorra.
\newblock {$n/p$}-harmonic maps: regularity for the sphere case.
\newblock {\em Adv. Calc. Var.}, 7(1):1--26, 2014.

\bibitem{DLS17}
F.~Da~Lio and A.~Schikorra.
\newblock On regularity theory for {$n/p$}-harmonic maps into manifolds.
\newblock {\em Nonlinear Anal.}, 165:182--197, 2017.

\bibitem{DNPV12}
E.~Di~Nezza, G.~Palatucci, and E.~Valdinoci.
\newblock Hitchhiker's guide to the fractional {S}obolev spaces.
\newblock {\em Bull. Sci. Math.}, 136(5):521--573, 2012.

\bibitem{D31}
J.~Douglas.
\newblock Solution of the problem of {P}lateau.
\newblock {\em Trans. Amer. Math. Soc.}, 33(1):263--321, 1931.

\bibitem{Duzaar-Fuchs}
F.~Duzaar and M.~Fuchs.
\newblock On removable singularities of {$p$}-harmonic maps.
\newblock {\em Ann. Inst. H. Poincar\'{e} Anal. Non Lin\'{e}aire},
  7(5):385--405, 1990.

\bibitem{DK98}
F.~Duzaar and E.~Kuwert.
\newblock Minimization of conformally invariant energies in homotopy classes.
\newblock {\em Calc. Var. Partial Differential Equations}, 6(4):285--313, 1998.

\bibitem{DM10}
F.~Duzaar and G.~Mingione.
\newblock Local {L}ipschitz regularity for degenerate elliptic systems.
\newblock {\em Ann. Inst. H. Poincar\'{e} Anal. Non Lin\'{e}aire},
  27(6):1361--1396, 2010.

\bibitem{Eells-Wood1976}
J.~Eells and J.~C. Wood.
\newblock Restrictions on harmonic maps of surfaces.
\newblock {\em Topology}, 15(3):263--266, 1976.

\bibitem{E10}
L.~C. Evans.
\newblock {\em Partial differential equations}, volume~19 of {\em Graduate
  Studies in Mathematics}.
\newblock American Mathematical Society, Providence, RI, second edition, 2010.

\bibitem{FHW94}
M.~H. Freedman, Z.-X. He, and Z.~Wang.
\newblock M\"{o}bius energy of knots and unknots.
\newblock {\em Ann. of Math. (2)}, 139(1):1--50, 1994.

\bibitem{Futaki80}
A.~Futaki.
\newblock Nonexistence of minimizing harmonic maps from {$2$}-spheres.
\newblock {\em Proc. Japan Acad. Ser. A Math. Sci.}, 56(6):291--293, 1980.

\bibitem{G20}
F.~Gaia.
\newblock {Noether Theorems for Lagrangians involving fractional Laplacians},
  2020.

\bibitem{G19}
N.~Garofalo.
\newblock Fractional thoughts.
\newblock In {\em New developments in the analysis of nonlocal operators},
  volume 723 of {\em Contemp. Math.}, pages 1--135. Amer. Math. Soc.,
  Providence, RI, 2019.

\bibitem{Lasic}
L.~Giacomelli, M.~\L{}asica, and S.~Moll.
\newblock Regular 1-harmonic flow.
\newblock {\em Calc. Var. Partial Differential Equations}, 58(2):Paper No. 82,
  24, 2019.

\bibitem{Giaquinta}
M.~Giaquinta.
\newblock {\em Multiple integrals in the calculus of variations and nonlinear
  elliptic systems}, volume 105 of {\em Annals of Mathematics Studies}.
\newblock Princeton University Press, Princeton, NJ, 1983.

\bibitem{GM99}
O.~Gonzalez and J.~H. Maddocks.
\newblock Global curvature, thickness, and the ideal shapes of knots.
\newblock {\em Proc. Natl. Acad. Sci. USA}, 96(9):4769--4773, 1999.

\bibitem{GModern}
L.~Grafakos.
\newblock {\em Modern {F}ourier analysis}, volume 250 of {\em Graduate Texts in
  Mathematics}.
\newblock Springer, New York, third edition, 2014.

\bibitem{H96}
P.~Haj{\l}asz.
\newblock Sobolev spaces on an arbitrary metric space.
\newblock {\em Potential Anal.}, 5(4):403--415, 1996.

\bibitem{H18}
P.~Haj\l{}asz.
\newblock The {$(n+1)$}-{L}ipschitz homotopy group of the {H}eisenberg group
  {$\mathbb{H}^n$}.
\newblock {\em Proc. Amer. Math. Soc.}, 146(3):1305--1308, 2018.

\bibitem{HS14}
P.~Haj\l{}asz and A.~Schikorra.
\newblock Lipschitz homotopy and density of {L}ipschitz mappings in {S}obolev
  spaces.
\newblock {\em Ann. Acad. Sci. Fenn. Math.}, 39(2):593--604, 2014.

\bibitem{HST14}
P.~Haj\l{}asz, A.~Schikorra, and J.~T. Tyson.
\newblock Homotopy groups of spheres and {L}ipschitz homotopy groups of
  {H}eisenberg groups.
\newblock {\em Geom. Funct. Anal.}, 24(1):245--268, 2014.

\bibitem{HL03}
F.~Hang and F.~H. Lin.
\newblock Topology of {S}obolev mappings. {II}.
\newblock {\em Acta Math.}, 191(1):55--107, 2003.

\bibitem{Hatcher}
A.~Hatcher.
\newblock {\em Algebraic topology}.
\newblock Cambridge University Press, Cambridge, 2002.

\bibitem{polyharmonic}
W.~{He}, R.~{Jiang}, and L.~{Lin}.
\newblock {Existence of polyharmonic maps in critical dimensions}.
\newblock {\em arXiv e-prints}, page arXiv:1911.00849, Nov. 2019.

\bibitem{J77}
B.~Jawerth.
\newblock Some observations on {B}esov and {L}izorkin-{T}riebel spaces.
\newblock {\em Math. Scand.}, 40(1):94--104, 1977.

\bibitem{JLZ19}
J.~Jost, L.~Liu, and M.~Zhu.
\newblock Energy identity for a class of approximate {D}irac-harmonic maps from
  surfaces with boundary.
\newblock {\em Ann. Inst. H. Poincar\'{e} Anal. Non Lin\'{e}aire},
  36(2):365--387, 2019.

\bibitem{Kawai-Nakauchi-Takeuchi}
S.~Kawai, N.~Nakauchi, and H.~Takeuchi.
\newblock On the existence of {$n$}-harmonic spheres.
\newblock {\em Compositio Math.}, 117(1):33--43, 1999.

\bibitem{K10}
S.~Kolasi\'{n}ski.
\newblock Regularity of weak solutions of {$n$}-dimensional {$H$}-systems.
\newblock {\em Differential Integral Equations}, 23(11-12):1073--1090, 2010.

\bibitem{KS97}
R.~B. Kusner and J.~M. Sullivan.
\newblock M\"{o}bius energies for knots and links, surfaces and submanifolds.
\newblock In {\em Geometric topology ({A}thens, {GA}, 1993)}, volume~2 of {\em
  AMS/IP Stud. Adv. Math.}, pages 570--604. Amer. Math. Soc., Providence, RI,
  1997.

\bibitem{KMS14}
T.~Kuusi, G.~\href{https://www.youtube.com/watch?v=aUGigYuOGNg}{Mingione}, and
  Y.~\href{https://www.youtube.com/watch?v=yKOKZsb5guo}{Sire}.
\newblock A fractional {G}ehring lemma, with applications to nonlocal
  equations.
\newblock {\em Atti Accad. Naz. Lincei Rend. Lincei Mat. Appl.},
  25(4):345--358, 2014.

\bibitem{Lamm2010}
T.~Lamm.
\newblock Energy identity for approximations of harmonic maps from surfaces.
\newblock {\em Trans. Amer. Math. Soc.}, 362(8):4077--4097, 2010.

\bibitem{LMM19}
T.~{Lamm}, A.~{Malchiodi}, and M.~{Micallef}.
\newblock {A gap theorem for $\alpha$-harmonic maps between two-spheres}.
\newblock {\em arXiv e-prints}, page arXiv:1903.10217, Mar. 2019.

\bibitem{LaurainPetrides}
P.~Laurain and R.~Petrides.
\newblock Regularity and quantification for harmonic maps with free boundary.
\newblock {\em Adv. Calc. Var.}, 10(1):69--82, 2017.

\bibitem{LaurainRivierebiharmonic}
P.~Laurain and T.~Rivi\`ere.
\newblock Energy quantization for biharmonic maps.
\newblock {\em Adv. Calc. Var.}, 6(2):191--216, 2013.

\bibitem{LR14}
P.~Laurain and T.~Rivi\`ere.
\newblock Angular energy quantization for linear elliptic systems with
  antisymmetric potentials and applications.
\newblock {\em Anal. PDE}, 7(1):1--41, 2014.

\bibitem{Lemaire1978}
L.~Lemaire.
\newblock Applications harmoniques de surfaces riemanniennes.
\newblock {\em J. Differential Geometry}, 13(1):51--78, 1978.

\bibitem{LS20}
E.~Lenzmann and A.~Schikorra.
\newblock Sharp commutator estimates via harmonic extensions.
\newblock {\em Nonlinear Anal.}, 193:111375, 2020.

\bibitem{LR02}
F.-H. Lin and T.~Rivi\`ere.
\newblock Energy quantization for harmonic maps.
\newblock {\em Duke Math. J.}, 111(1):177--193, 2002.

\bibitem{LW14}
F.-H. Lin and C.~Wang.
\newblock Recent developments of analysis for hydrodynamic flow of nematic
  liquid crystals.
\newblock {\em Philos. Trans. R. Soc. Lond. Ser. A Math. Phys. Eng. Sci.},
  372(2029):20130361, 18, 2014.

\bibitem{L06}
P.~Lindqvist.
\newblock {\em Notes on the {$p$}-{L}aplace equation}, volume 102 of {\em
  Report. University of Jyv\"{a}skyl\"{a} Department of Mathematics and
  Statistics}.
\newblock University of Jyv\"{a}skyl\"{a}, Jyv\"{a}skyl\"{a}, 2006.

\bibitem{L88}
S.~Luckhaus.
\newblock Partial {H}\"{o}lder continuity for minima of certain energies among
  maps into a {R}iemannian manifold.
\newblock {\em Indiana Univ. Math. J.}, 37(2):349--367, 1988.

\bibitem{MN14}
F.~C. Marques and A.~Neves.
\newblock Min-max theory and the {W}illmore conjecture.
\newblock {\em Ann. of Math. (2)}, 179(2):683--782, 2014.

\bibitem{MRS17}
K.~Mazowiecka, R.~Rodiac, and A.~Schikorra.
\newblock Epsilon-regularity for p-harmonic maps at a free boundary on a
  sphere.
\newblock {\em APDE (to appear)}, 2017.

\bibitem{MS18}
K.~Mazowiecka and A.~Schikorra.
\newblock Fractional div-curl quantities and applications to nonlocal geometric
  equations.
\newblock {\em J. Funct. Anal.}, 275(1):1--44, 2018.

\bibitem{MP20}
V.~Millot and M.~Pegon.
\newblock Minimizing 1/2-harmonic maps into spheres.
\newblock {\em Calc. Var. Partial Differential Equations}, 59(2):Paper No. 55,
  2020.

\bibitem{MPS19}
V.~Millot, M.~Pegon, and A.~Schikorra.
\newblock Partial regularity for fractional harmonic maps into spheres, 2019.

\bibitem{MS15}
V.~Millot and Y.~Sire.
\newblock On a fractional {G}inzburg-{L}andau equation and 1/2-harmonic maps
  into spheres.
\newblock {\em Arch. Ration. Mech. Anal.}, 215(1):125--210, 2015.

\bibitem{MSY18}
V.~Millot, Y.~Sire, and H.~Yu.
\newblock Minimizing fractional harmonic maps on the real line in the
  supercritical regime.
\newblock {\em Discrete Contin. Dyn. Syst.}, 38(12):6195--6214, 2018.

\bibitem{M07}
P.~Mironescu.
\newblock Sobolev maps on manifolds: degree, approximation, lifting.
\newblock In {\em Perspectives in nonlinear partial differential equations},
  volume 446 of {\em Contemp. Math.}, pages 413--436. Amer. Math. Soc.,
  Providence, RI, 2007.

\bibitem{MiSi15}
P.~Mironescu and W.~Sickel.
\newblock A {S}obolev non embedding.
\newblock {\em Atti Accad. Naz. Lincei Rend. Lincei Mat. Appl.},
  26(3):291--298, 2015.

\bibitem{MonteilVanSchaftingen}
A.~Monteil and J.~Van~Schaftingen.
\newblock Uniform boundedness principles for {S}obolev maps into manifolds.
\newblock {\em Ann. Inst. H. Poincar\'{e} Anal. Non Lin\'{e}aire},
  36(2):417--449, 2019.

\bibitem{M48}
C.~B. Morrey, Jr.
\newblock The problem of {P}lateau on a {R}iemannian manifold.
\newblock {\em Ann. of Math. (2)}, 49:807--851, 1948.

\bibitem{M11}
R.~Moser.
\newblock Intrinsic semiharmonic maps.
\newblock {\em J. Geom. Anal.}, 21(3):588--598, 2011.

\bibitem{Mou-Yang-1996}
L.~Mou and P.~Yang.
\newblock Regularity for {$n$}-harmonic maps.
\newblock {\em J. Geom. Anal.}, 6(1):91--112, 1996.

\bibitem{OH91}
J.~O'Hara.
\newblock Energy of a knot.
\newblock {\em Topology}, 30(2):241--247, 1991.

\bibitem{OH92}
J.~O'Hara.
\newblock Family of energy functionals of knots.
\newblock {\em Topology Appl.}, 48(2):147--161, 1992.

\bibitem{Parker}
T.~H. Parker.
\newblock Bubble tree convergence for harmonic maps.
\newblock {\em J. Differential Geom.}, 44(3):595--633, 1996.

\bibitem{DLP17}
A.~Pigati and F.~D. Lio.
\newblock Free boundary minimal surfaces: a nonlocal approach, 2017.

\bibitem{R17}
T.~Rivi\`ere.
\newblock A viscosity method in the min-max theory of minimal surfaces.
\newblock {\em Publ. Math. Inst. Hautes \'{E}tudes Sci.}, 126:177--246, 2017.

\bibitem{R18}
J.~Roberts.
\newblock A regularity theory for intrinsic minimising fractional harmonic
  maps.
\newblock {\em Calc. Var. Partial Differential Equations}, 57(4):Paper No. 109,
  68, 2018.

\bibitem{RS14}
X.~Ros-Oton and J.~Serra.
\newblock The {P}ohozaev identity for the fractional {L}aplacian.
\newblock {\em Arch. Ration. Mech. Anal.}, 213(2):587--628, 2014.

\bibitem{RunstSickel}
T.~Runst and W.~Sickel.
\newblock {\em Sobolev spaces of fractional order, {N}emytskij operators, and
  nonlinear partial differential equations}, volume~3 of {\em De Gruyter Series
  in Nonlinear Analysis and Applications}.
\newblock Walter de Gruyter \& Co., Berlin, 1996.

\bibitem{Sucks1}
J.~Sacks and K.~Uhlenbeck.
\newblock The existence of minimal immersions of {$2$}-spheres.
\newblock {\em Ann. of Math. (2)}, 113(1):1--24, 1981.

\bibitem{Sucks2}
J.~Sacks and K.~Uhlenbeck.
\newblock Minimal immersions of closed {R}iemann surfaces.
\newblock {\em Trans. Amer. Math. Soc.}, 271(2):639--652, 1982.

\bibitem{S06}
C.~Scheven.
\newblock Partial regularity for stationary harmonic maps at a free boundary.
\newblock {\em Math. Z.}, 253(1):135--157, 2006.

\bibitem{S12}
A.~Schikorra.
\newblock Regularity of {$n/2$}-harmonic maps into spheres.
\newblock {\em J. Differential Equations}, 252(2):1862--1911, 2012.

\bibitem{S13}
A.~Schikorra.
\newblock A note on regularity for the {$n$}-dimensional {$H$}-system assuming
  logarithmic higher integrability.
\newblock {\em Analysis (Berlin)}, 33(3):219--234, 2013.

\bibitem{S15}
A.~Schikorra.
\newblock Integro-differential harmonic maps into spheres.
\newblock {\em Comm. Partial Differential Equations}, 40(3):506--539, 2015.

\bibitem{SLp15}
A.~Schikorra.
\newblock {$L^p$}-gradient harmonic maps into spheres and {$SO(N)$}.
\newblock {\em Differential Integral Equations}, 28(3-4):383--408, 2015.

\bibitem{Seps15}
A.~Schikorra.
\newblock {$\varepsilon$}-regularity for systems involving non-local,
  antisymmetric operators.
\newblock {\em Calc. Var. Partial Differential Equations}, 54(4):3531--3570,
  2015.

\bibitem{S16}
A.~Schikorra.
\newblock Nonlinear commutators for the fractional {$p$}-{L}aplacian and
  applications.
\newblock {\em Math. Ann.}, 366(1-2):695--720, 2016.

\bibitem{Sconformal}
A.~{Schikorra}.
\newblock {Limits of conformal immersions under a bound on a fractional normal
  curvature quantity}.
\newblock {\em arXiv e-prints}, page arXiv:1812.03494, Dec 2018.

\bibitem{SS17}
A.~Schikorra and P.~Strzelecki.
\newblock Invitation to {$H$}-systems in higher dimensions: known results, new
  facts, and related open problems.
\newblock {\em EMS Surv. Math. Sci.}, 4(1):21--42, 2017.

\bibitem{SU1}
R.~Schoen and K.~Uhlenbeck.
\newblock A regularity theory for harmonic maps.
\newblock {\em J. Differential Geom.}, 17(2):307--335, 1982.

\bibitem{SU2}
R.~Schoen and K.~Uhlenbeck.
\newblock Boundary regularity and the {D}irichlet problem for harmonic maps.
\newblock {\em J. Differential Geom.}, 18(2):253--268, 1983.

\bibitem{Schoen-Yau1979}
R.~Schoen and S.~T. Yau.
\newblock Existence of incompressible minimal surfaces and the topology of
  three-dimensional manifolds with nonnegative scalar curvature.
\newblock {\em Ann. of Math. (2)}, 110(1):127--142, 1979.

\bibitem{S96}
L.~Simon.
\newblock {\em Theorems on regularity and singularity of energy minimizing
  maps}.
\newblock Lectures in Mathematics ETH Z\"{u}rich. Birkh\"{a}user Verlag, Basel,
  1996.
\newblock Based on lecture notes by Norbert Hungerb\"{u}hler.

\bibitem{So17}
E.~Soultanis.
\newblock Homotopy classes of {N}ewtonian maps.
\newblock {\em Rev. Mat. Iberoam.}, 33(3):951--994, 2017.

\bibitem{S93}
E.~M. Stein.
\newblock {\em Harmonic analysis: real-variable methods, orthogonality, and
  oscillatory integrals}, volume~43 of {\em Princeton Mathematical Series}.
\newblock Princeton University Press, Princeton, NJ, 1993.
\newblock With the assistance of Timothy S. Murphy, Monographs in Harmonic
  Analysis, III.

\bibitem{S85}
M.~Struwe.
\newblock On the evolution of harmonic mappings of {R}iemannian surfaces.
\newblock {\em Comment. Math. Helv.}, 60(4):558--581, 1985.

\bibitem{svdm:imc}
P.~Strzelecki, M.~Szuma{\'n}ska, and H.~von~der Mosel.
\newblock Regularizing and self-avoidance effects of integral {M}enger
  curvature.
\newblock {\em Ann. Sc. Norm. Super. Pisa Cl. Sci. (5)}, 9(1):145--187, 2010.

\bibitem{svdm:somemenger}
P.~Strzelecki, M.~Szuma\'nska, and H.~von~der Mosel.
\newblock On some knot energies involving {M}enger curvature.
\newblock {\em Topology Appl.}, 160(13):1507--1529, 2013.

\bibitem{StrzvdM-2013}
P.~Strzelecki and H.~von~der Mosel.
\newblock Menger curvature as a knot energy.
\newblock {\em Phys. Rep.}, 530(3):257--290, 2013.

\bibitem{SvdM13}
P.~Strzelecki and H.~von~der Mosel.
\newblock Tangent-point repulsive potentials for a class of non-smooth
  {$m$}-dimensional sets in {$\mathbb{R}^n$}. {P}art {I}: {S}moothing and
  self-avoidance effects.
\newblock {\em J. Geom. Anal.}, 23(3):1085--1139, 2013.

\bibitem{Tartar}
L.~Tartar.
\newblock {\em An introduction to {S}obolev spaces and interpolation spaces},
  volume~3 of {\em Lecture Notes of the Unione Matematica Italiana}.
\newblock Springer, Berlin; UMI, Bologna, 2007.

\bibitem{T83}
H.~Triebel.
\newblock {\em Theory of function spaces}, volume~78 of {\em Monographs in
  Mathematics}.
\newblock Birkh\"{a}user Verlag, Basel, 1983.

\bibitem{VanSchaftingen-gapfree}
J.~Van~Schaftingen.
\newblock Estimates by gap potentials of free homotopy decompositions of
  critical {S}obolev maps.
\newblock {\em Adv. Nonlinear Anal.}, 9(1):1214--1250, 2020.

\bibitem{WY14}
S.~Wenger and R.~Young.
\newblock Lipschitz homotopy groups of the {H}eisenberg groups.
\newblock {\em Geom. Funct. Anal.}, 24(1):387--402, 2014.

\bibitem{White-homotopy}
B.~White.
\newblock Homotopy classes in {S}obolev spaces and the existence of energy
  minimizing maps.
\newblock {\em Acta Math.}, 160(1-2):1--17, 1988.

\end{thebibliography}

\end{document}